\pdfoutput=1
\documentclass[
	,12pt%
	,pagesize%
	,headings=small%
	,paper=a4%
	,parskip=false%
	,abstract=true
	,toc=bibliography
	,DIV=12%
]{scrartcl}

\addtokomafont{sectioning}{\normalfont\bfseries}
\setkomafont{title}{\normalfont}
\setkomafont{subtitle}{\normalfont}
\setkomafont{author}{\normalfont}
\setkomafont{date}{\normalfont}
\addtokomafont{pageheadfoot}{\scshape\small}
\setkomafont{caption}{\footnotesize}
\setkomafont{captionlabel}{\usekomafont{caption}\bfseries}
\setcapindent{0pt}

\usepackage{anyfontsize}
\usepackage{xspace}

\usepackage[utf8]{inputenc}
\usepackage[T1]{fontenc}
\usepackage{csquotes}
\usepackage[english]{babel}
\usepackage{graphicx}
\graphicspath{%
	{/}%
	{Pictures/}%
	{Pictures/Fig_GeodesicShooting_G}%
	{Pictures/Fig_GeodesicShooting_H1}%
}
\usepackage[export]{adjustbox}	% key=value interface
\usepackage[dvipsnames]{xcolor}
\usepackage{tikz-cd}
\usepackage[nointlimits]{amsmath}
\usepackage{amssymb,mathrsfs,mathtools}
\usepackage{newtxtext}
\usepackage[varvw]{newtxmath}

\usepackage{upgreek}
\usepackage{braket}
\usepackage{cancel}
\usepackage{siunitx}
\usepackage{aliascnt} %alias counter
\usepackage[amsmath,thmmarks]{ntheorem}
\usepackage[expansion=true,protrusion=true]{microtype}
\usepackage{xkeyval}
\usepackage{authblk}
\usepackage{multicol}
\usepackage{enumitem}
\usepackage{longtable}
\usepackage{hyphenat}

\usepackage[final,%
	pdftex,%
	bookmarks,%
	bookmarksdepth=3,%
	breaklinks=true,%
	colorlinks=true,%
	urlcolor=NavyBlue,%
	linkcolor=NavyBlue,%
	citecolor=ForestGreen,%	
]{hyperref}%
\usepackage[all]{hypcap}

\AtBeginEnvironment{abstract}{\footnotesize}
\makeatletter
\patchcmd{\@maketitle}{\huge}{\Large}{}{}
\makeatother

\newcommand{\mednegskip}{\vspace{-1ex}}
\newcommand{\bignegskip}{\vspace{-2ex}}

\newcommand{\BoundedLinOps}{L}

\newcommand{\HSpace}{\Bessel[s][][\Domain][\AmbSpace]}

\newcommand{\DenomExp}{p}

\newcommand{\Energy}{E}
\newcommand{\Dirichlet}{\cE}
\newcommand{\Curve}{\gamma}

\DeclareDocumentCommand{\Path}{ o }{
	\IfValueTF{#1}{%
		\varGamma_{\!#1}%
	}{%
		\varGamma%
	}%
}

\DeclareDocumentCommand{\dotPath}{ o }{
	\IfValueTF{#1}{%
		\dot{\varGamma}_{\!#1}%
	}{%
		{\dot{\varGamma}}%
	}%
}

\DeclareDocumentCommand{\ddotPath}{ o }{
	\IfValueTF{#1}{%
		\ddot{\varGamma}_{\!#1}%
	}{%
		{\ddot{\varGamma}}%
	}%
}

\newcommand{\NProj}{P}
\newcommand{\Riesz}{J}
\newcommand{\Mfld}{\mathcal{M}}
\newcommand{\BoundedSet}{\mathcal{U}}
\newcommand{\ArcLength}{L}
\newcommand{\PathLength}{\mathcal{L}}
\newcommand{\PathEnergy}{\mathcal{E}}

\newcommand{\ConstraintMap}{\mathcal{\varPhi}}
\newcommand{\ConstraintMfld}{\cN}

\newcommand{\MorreyConst}{C_{\mathrm{Morrey}}}
\newcommand{\Morrey}{\mathrm{M}}
\newcommand{\UnitInterval}{\intervalcc{0,1}}

\newcommand{\SmashedSqrt}[1]{\sqrt{\smash[b]{#1}}}

\newcommand{\wrt}{w.r.t.\xspace}

\DeclareMathOperator{\BiLip}{BiLip}
\DeclareMathOperator{\distor}{distor}

\DeclareDocumentCommand{\dist}{ o o o }{
	\IfValueTF{#1}{
		\varrho_{#1}\IfValueTF{#2}{(#2,#3)}{}
	}{
		\varrho\IfValueTF{#2}{(#2,#3)}{}
	}
}

\DeclareDocumentCommand{\distC}{ o o }{
	\varrho_{\gamma}\IfValueTF{#1}{(#1,#2)}{}
}

\DeclareDocumentCommand{\Speed}{ o }{
	\IfValueTF{#1}{
	  	h_{#1}
	}{
      	h
	}
}

\DeclareDocumentCommand{\InvSpeed}{ o }{
	\IfValueTF{#1}{
	  	H_{#1}
	}{
      	H
	}
}

\newcommand{\Op}[2]{\mathcal{R}_{#1}^{#2}}

\newcommand{\diffop}[2]{\delta_{#1}^{#2}}
\newcommand{\Circle}{\mathbb{T}}

\newcommand{\TP}{\mathrm{TP}}

\newcommand{\loc}{\mathrm{loc}}

\newcommand{\AmbDim}{m} % dimension of ambient space
\newcommand{\AmbSpace}{{\R^\AmbDim}}
\newcommand{\Domain}{{\Circle}}
\newcommand{\FunDomain}{{\varSigma}}

\newcommand{\Sphere}{\mathbb{S}}

\DeclareDocumentCommand{\Hess}{ O{} }{\operatorname{Hess}_{#1}}

\newcommand{\nospaceperiod}{\makebox[0pt][l]{\,.}}

\DeclareMathOperator*{\esssup}{ess\,sup}

\DeclareMathOperator{\Lip}{Lip}

\DeclareDocumentCommand{\converges}{ o o }{
	\mathbin{%
		\IfValueTF{#1}{%
			\IfValueTF{#2}{%
				{\xlongrightarrow[\smash{#2}]{\;\;\phantom{w}\;\;}}
			}{
				{\xlongrightarrow{\;\;\phantom{w}\;\;}}
			}
		}{%
			\longrightarrow
		}%
	}%
}

\DeclareDocumentCommand{\wconverges}{ o o }{
	\mathbin{%
		\IfValueTF{#1}{%
			\IfValueTF{#2}{%
				{\xrightharpoonup[\smash{#2}]{\;\;\mathrm{w}\;\;}}
			}{
				{\xrightharpoonup{\;\;\mathrm{w}\;\;}}
			}
		}{%
			\longrightharpoonup
		}%
	}%
}

\newcommand{\trans}{^\transp}

\newcommand{\dual}{'^{\!}}

\newcommand{\pinv}{^{\dagger\!}}

\newcommand{\mymathcal}{\mathcal}

\newcommand{\cB}{{\mymathcal{B}}}

\newcommand{\cD}{{\mymathcal{D}}}
\newcommand{\cE}{{\mymathcal{E}}}

\newcommand{\cG}{{\mymathcal{G}}}
\newcommand{\cH}{{\mymathcal{H}}}

\newcommand{\cL}{{\mymathcal{L}}}

\newcommand{\cN}{{\mymathcal{N}}}

\newcommand{\cP}{{\mymathcal{P}}}

\newcommand{\cU}{{\mymathcal{U}}}
\newcommand{\cV}{{\mymathcal{V}}}

\newcommand{\cX}{{\mymathcal{X}}}
\newcommand{\cY}{{\mymathcal{Y}}}
\newcommand{\cZ}{{\mymathcal{Z}}}

\DeclareMathOperator{\Laplacian}{\Delta}
\newcommand{\dd}{\mathop{}\!\mathrm{d}}

\newcommand{\ee}{{\mathrm{e}}}

\newcommand{\ceq}{\coloneqq}
\newcommand{\qec}{\eqqcolon}
\newcommand{\R}{{\mathbb{R}}}

\newcommand{\N}{\mathbb{N}}
\newcommand{\Z}{{\mathbb{Z}}}
\newcommand{\Q}{{\mathbb{Q}}}

\newcommand{\Ninfty}{\N \cup \braces{\infty}}

\DeclareMathOperator{\id}{id}

\DeclarePairedDelimiterXPP{\pars}[1]{\mathop{}}{\lparen}{\rparen}{}{#1}
\DeclarePairedDelimiterXPP{\abs}[1]{\mathop{}}{\lvert}{\rvert}{}{#1}
\DeclarePairedDelimiterXPP{\norm}[1]{\mathop{}}{\lVert}{\rVert}{}{#1}
\DeclarePairedDelimiterXPP{\seminorm}[1]{\mathop{}}{\lbrack}{\rbrack}{}{#1}
\DeclarePairedDelimiterXPP{\inner}[1]{\mathop{}}{\langle}{\rangle}{}{#1}
\DeclarePairedDelimiterXPP{\iinner}[1]{\mathop{}}{\langle\!\langle}{\rangle\!\rangle}{}{#1}
\DeclarePairedDelimiterXPP{\brackets}[1]{\mathop{}}{\lbrack}{\rbrack}{}{#1}
\DeclarePairedDelimiterXPP{\braces}[1]{\mathop{}}{\lbrace}{\rbrace}{}{#1}

\DeclarePairedDelimiterXPP{\floor}[1]{\mathop{}}{\lfloor}{\rfloor}{}{#1}
\DeclarePairedDelimiterXPP{\ceil}[1]{\mathop{}}{\lceil}{\rceil}{}{#1}

\DeclarePairedDelimiterXPP{\intervalcc}[1]{\mathop{}}{\lbrack}{\rbrack}{}{#1}
\DeclarePairedDelimiterXPP{\intervalco}[1]{\mathop{}}{\lbrack}{\rparen}{}{#1}
\DeclarePairedDelimiterXPP{\intervaloc}[1]{\mathop{}}{\lparen}{\rbrack}{}{#1}
\DeclarePairedDelimiterXPP{\intervaloo}[1]{\mathop{}}{\lparen}{\rparen}{}{#1}

\DeclarePairedDelimiterXPP{\myset}[2]{\mathop{}}{\lbrace}{\rbrace}{}{#1\,\delimsize\vert\,\mathopen{}#2}

\DeclareMathOperator{\Hom}{Hom}    %Homomorphisms

\DeclareMathOperator{\sgn}{sgn}
\newcommand{\vol}{{\omega}}
\newcommand{\dvol}{\dd\vol}
\newcommand{\pull}{{\#}}

\newcommand{\sdfrac}[2]{\mbox{\small$\displaystyle\frac{#1}{#2}$}}

\newcommand{\qtext}[1]{\quad\text{#1}\quad}
\newcommand{\qqtext}[1]{\qquad\text{#1}\qquad}
	
\newcommand{\qand}{\qtext{and}}
\newcommand{\qor}{\qtext{or}}

\newcommand{\qqand}{\qqtext{and}}

\newcommand{\DiffGroup}{\cG}
\newcommand{\QuotientMap}{\pi}
\newcommand{\QuotientSpace}{\cB}
\newcommand{\SoboSlobo}{Sobolev--Slobodeckij\xspace}

\newcommand{\spacephantom}{\vphantom{a}}

\DeclareDocumentCommand{\Sobo}{ O{} O{\spacephantom} o o}{
	\IfValueTF{#3}{
	  \IfValueTF{#4}{
	  	W^{#1}_{#2}(#3;#4)
	  }{
	  	W^{#1}_{#2}(#3)
	  }
	}{
	  W^{#1}_{#2}
	}
}

\DeclareDocumentCommand{\Emb}{ O{} O{\spacephantom} o o}{
	\IfValueTF{#3}{
	  \IfValueTF{#4}{
	  	\mathrm{Emb}^{#1}_{#2}(#3;#4)
	  }{
		\mathrm{Emb}^{#1}_{#2}(#3)
	  }
	}{
		\mathrm{Emb}^{#1}_{#2}
	}
}

\DeclareDocumentCommand{\Imm}{ O{} O{\spacephantom} o o}{
	\IfValueTF{#3}{
	  \IfValueTF{#4}{
	  	\mathrm{Imm}^{#1}_{#2}(#3;#4)
	  }{
		\mathrm{Imm}^{#1}_{#2}(#3)
	  }
	}{
		\mathrm{Imm}^{#1}_{#2}
	}
}

\DeclareDocumentCommand{\Diff}{ O{} O{\spacephantom} o o}{
	\IfValueTF{#3}{
	  \IfValueTF{#4}{
	  	\mathrm{Diff}^{#1}_{#2}(#3;#4)
	  }{
		\mathrm{Diff}^{#1}_{#2}(#3)
	  }
	}{
		\mathrm{Diff}^{#1}_{#2}
	}
}

\DeclareDocumentCommand{\Bessel}{ O{} O{\spacephantom} o o}{
	\IfValueTF{#3}{
	  \IfValueTF{#4}{
	  	H^{#1}_{#2}(#3;#4)
	  }{
	  	H^{#1}_{#2}(#3)
	  }
	}{
	  H^{#1}_{#2}
	}
}

\DeclareDocumentCommand{\Holder}{ O{} O{\spacephantom} o o}{
	\IfValueTF{#3}{
	  \IfValueTF{#4}{
	  	C^{#1}_{#2}(#3;#4)
	  }{
	  	C^{#1}_{#2}(#3)
	  }
	}{
	  C^{#1}_{#2}
	}
}
\DeclareDocumentCommand{\HolderC}{ O{} O{\spacephantom} }{\Holder[#1][#2][\Circle][\AmbSpace]}

\DeclareDocumentCommand{\Lebesgue}{ O{} O{\spacephantom} o o}{
	\IfValueTF{#3}{
	  \IfValueTF{#4}{
	  	L^{#1}_{#2}(#3;#4)
	  }{
	  	L^{#1}_{#2}(#3)
	  }
	}{
	  L^{#1}_{#2}
	}
}
\DeclareDocumentCommand{\LebesgueC}{ O{} O{\spacephantom} }{\Lebesgue[#1][#2][\Circle][\AmbSpace]}

\newcommand{\prx}{\pi_1}
\newcommand{\pry}{\pi_2}

\newcommand{\pullx}{\prx^\pull}
\newcommand{\pully}{\pry^\pull}

\usepackage[capitalise,nameinlink]{cleveref}
\crefname{equation}{}{}

\newlist{claims}{enumerate}{10}
\setlist[claims]{label*=\arabic*.,ref=\arabic*}
\crefname{claimsi}{Claim}{Claim}
\Crefname{claimsi}{Claim}{Claims}

\newlist{conditions}{enumerate}{10}
\setlist[conditions]{label*=\arabic*.,ref=\arabic*}
\crefname{conditionsi}{condition}{conditions}
\Crefname{conditionsi}{Condition}{Conditions}

\newlist{failuremodes}{enumerate}{10}
\setlist[failuremodes]{label*=(\roman*),ref=(\roman*)}
\crefname{failuremodesi}{}{}
\Crefname{failuremodesi}{}{}

\newcommand{\mynewtheorem}[4] %{BEZEICHNER}{COUNTER}{TITEL}{PLURAL} - for correct enumeration with \cref from the cleveref package
{%
\newaliascnt{#1}{#2}%
\newtheorem{#1}[#1]{#3}%
\aliascntresetthe{#1}%
\crefname{#1}{#3}{#4}%
\Crefname{#1}{#3}{#4}%
}

\mynewtheorem{theorem}{basetheorem}{Theorem}{Theorems} 
\mynewtheorem{lemma}{basetheorem}{Lemma}{Lemmata} 
\mynewtheorem{proposition}{basetheorem}{Proposition}{Propositions} 
\mynewtheorem{corollary}{basetheorem}{Corollary}{Corollaries}
\mynewtheorem{prelemma}{basetheorem}{Pre-lemma}{Pre-lemmata} 
\mynewtheorem{problem}{basetheorem}{Problem}{Problems}

\theoremstyle{break}
\mynewtheorem{btheorem}{basetheorem}{Theorem}{Theorems} 
\mynewtheorem{blemma}{basetheorem}{Lemma}{Lemmata}
\mynewtheorem{bproposition}{basetheorem}{Proposition}{Propositions}
\mynewtheorem{bcorollary}{basetheorem}{Corollary}{Corollaries}
\mynewtheorem{bproblem}{basetheorem}{Problem}{Problems}

\theoremstyle{plain}
\theorembodyfont{\normalfont}
\mynewtheorem{definition}{basetheorem}{Definition}{Definitions}
\mynewtheorem{example}{basetheorem}{Example}{Examples}
\mynewtheorem{remark}{basetheorem}{Remark}{Remarks}

\theoremstyle{break}
\mynewtheorem{bdfn}{basetheorem}{Definition}{Definitions}
\mynewtheorem{bex}{basetheorem}{Example}{Examples}
\mynewtheorem{brem}{basetheorem}{Remark}{Remarks}

\theoremheaderfont{\itshape}
\theorembodyfont{\upshape}
\theoremstyle{nonumberplain}
\theoremseparator{.}
\theoremsymbol{\ensuremath{\Box}}	
\newtheorem{proof}{Proof}

\theoremstyle{empty}

\title{On a Complete Riemannian Metric \\ on the Space of Embedded Curves}

\author[1]{Elias Döhrer}
\author[1]{Philipp Reiter}
\author[1,2,3,4]{Henrik Schumacher}

\affil[1]{Chemnitz University of Technology, Chemnitz, Germany}
\affil[2]{University of Georgia, Athens, Georgia, USA}
\affil[3]{RWTH Aachen University, Aachen, Germany}
\affil[4]{Graz University of Technology, Institute of Geometry, Graz, Austria}

\begin{document}

\maketitle

\begin{abstract}
	\begin{small}
		\noindent
		We propose a new strong Riemannian metric on the manifold of (parametrized) embedded curves of regularity $H^s$, $s\in(3/2,2)$.
We highlight its close relationship to the (generalized) tangent-point energies and employ it to show that this metric is \emph{complete} in the following senses:
(i)~bounded sets are relatively compact with respect to the weak $H^s$ topology;
(ii)~every Cauchy sequence with respect to the induced geodesic distance converges;
(iii)~solutions of the geodesic initial-value problem exist for all times. Moreover, we prove the existence of length-minimizing geodesics between every pair of curves in the same path component (i.e., in the same knot class).
As a by-product, we show $C^\infty$-smoothness of the tangent-point energies in the Hilbert case.
		\medskip
		
		\noindent
		\textbf{MSC-2020 classification:} 
		58B20; % Riemannian, Finsler and other geometric structures on infinite-dimensional manifolds
		58D10; % Spaces of embeddings and immersions
		58E10 % Variational problems in applications to the theory of geodesics (problems in one independent variable)
		\end{small}
\end{abstract}

% !TEX root = Main.tex

\section{Introduction}

Maintaining the topology of objects under the influence of certain forces or deformations is a crucial aspect in modeling that occurs in all fields of engineering science and in particular in computer graphics.
For instance, in elasticity theory it is essential to avoid interpenetration of matter~\cite{MR616782,MR4160797}. 
This can be realized via regularization by nonlocal self-repulsive terms, see~\cite{Kroemer_Reiter_2023}, for example.
Similar approaches have been proposed in \cite{10.1145/3450626.3459757,IPC}; these operate on already discretized domains and surfaces in $\R^3$ and their barrier functions are based on logarithmized distances between mesh elements.

These approaches are aimed at obtaining static equilibrium configurations or at simulating the dynamics of elastic bodies.
In the present work we are more interested in finding continuous families of deformations that are in some sense ``optimal''.
One motivation for this is key frame animation, where one seeks a smooth interpolation between several given poses.
For aesthetic reasons and for the sake of plausibility it is desirable to avoid ``singular'' phenomena such as pinching, breaking, and self-intersection.
Moreover, this type of motion should appear ``natural'' in the sense that it avoids any detour while connecting the two configurations in an appealing and plausible way.

In order to illustrate this idea, we consider a curve belonging to the figure-eight knot class, see \cref{fig:FigureEight}.
It is well-known that it is ambient isotopic to its mirror image, i.e., both curves can be continuously deformed into each other without self-intersections or pull-tight of knotted loops.
However, it is not quite obvious a priori how this could be done, let alone in a ``cost-efficient'' way. 
It is not even clear what ``cost'' should actually mean in this context.

\begin{figure}[t]%
    \centering%
    \newcommand{\incl}[1]{%
        \includegraphics[
            width=0.125\textwidth,
            trim = 100 60 100 30, 
            clip = true, 
            angle = 0
        ]{#1}%
    }%
    \incl{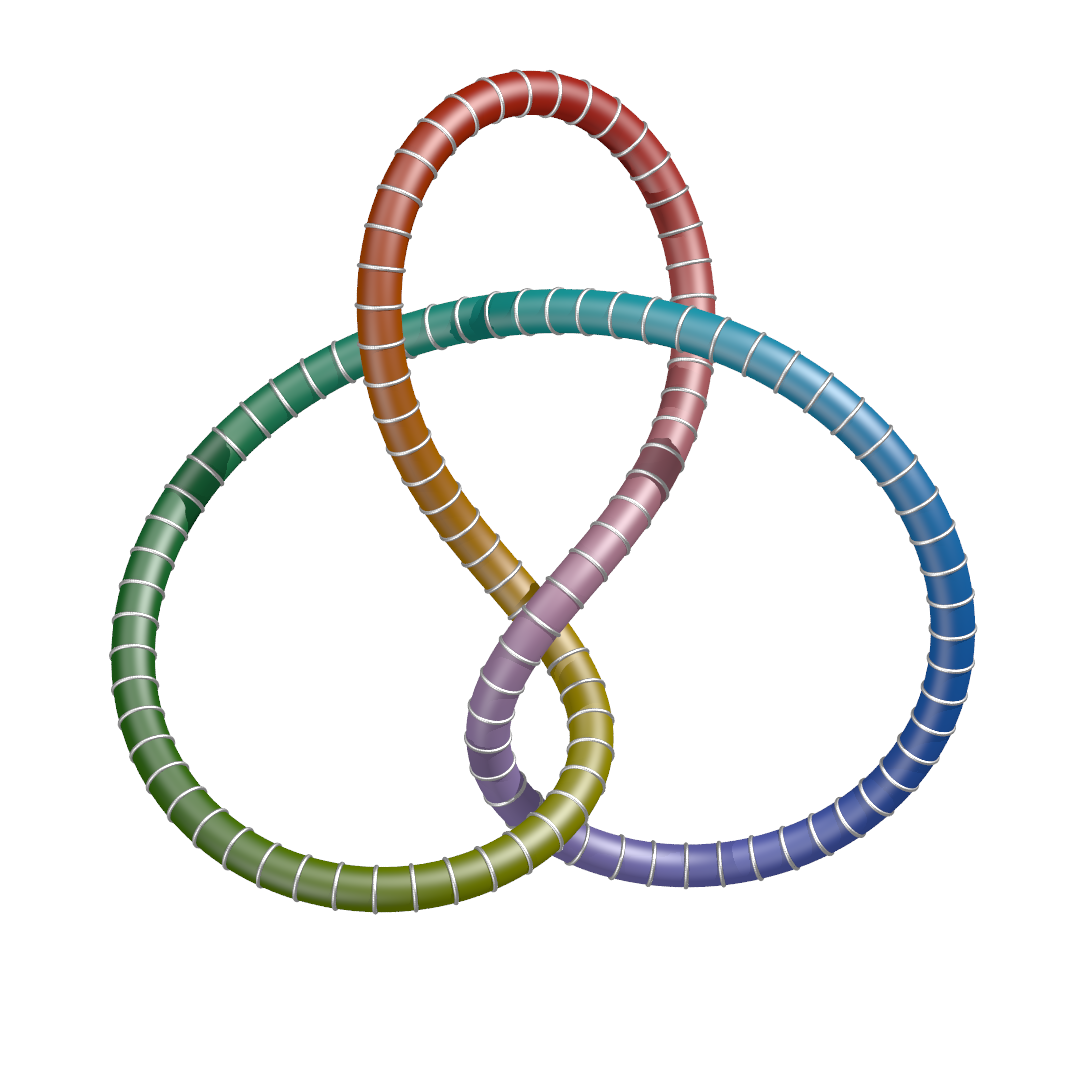}%
    \incl{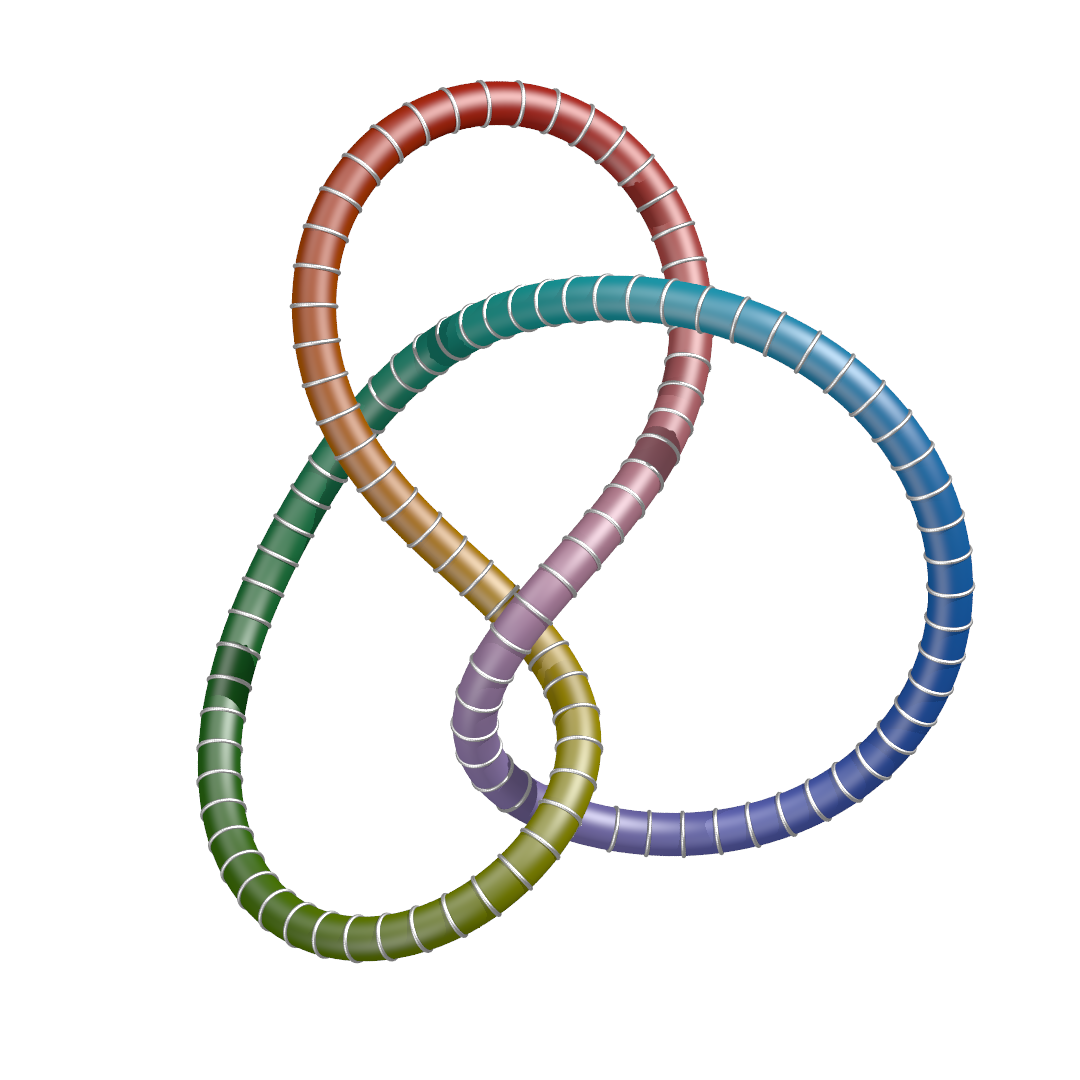}%
    \incl{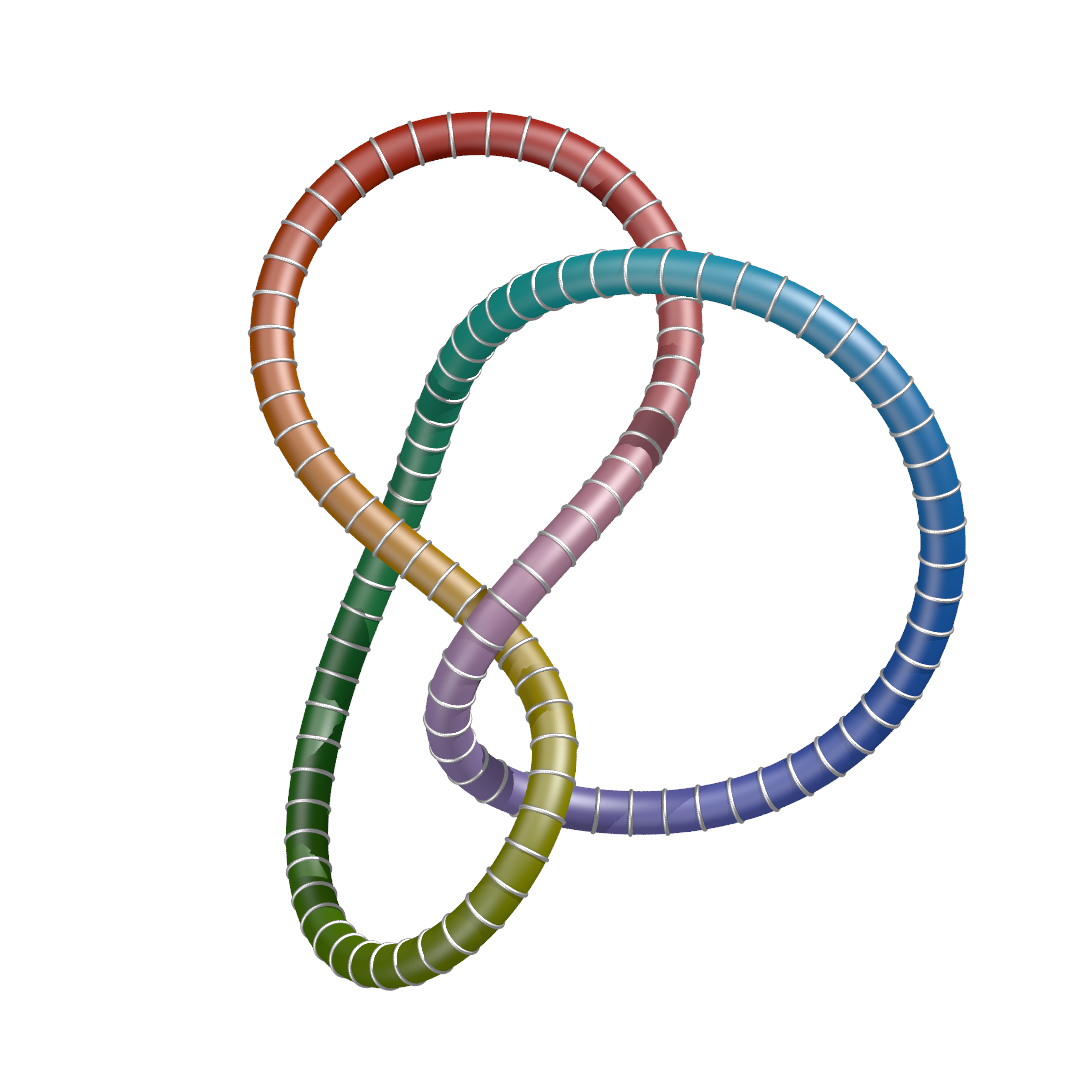}%
    \incl{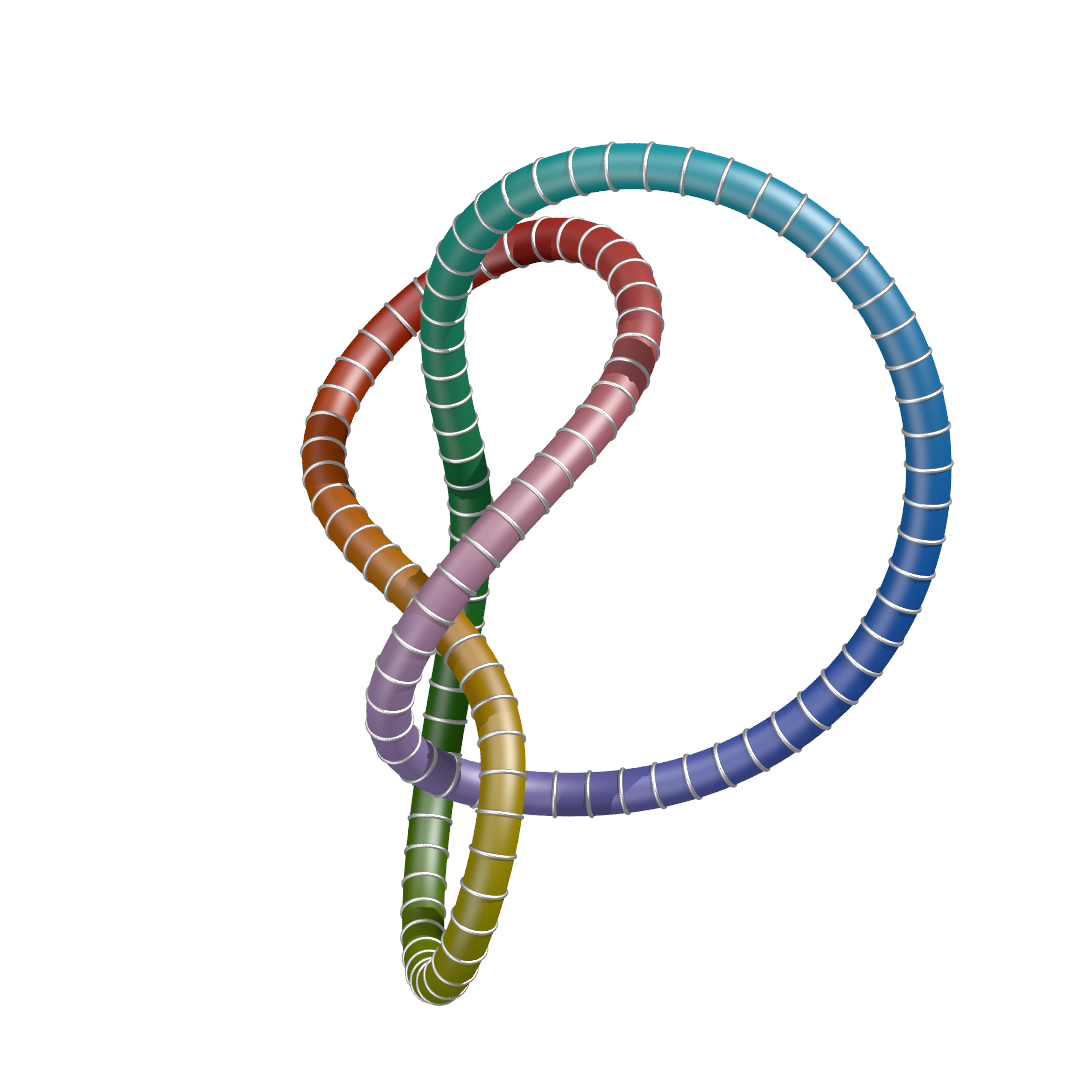}%
    \incl{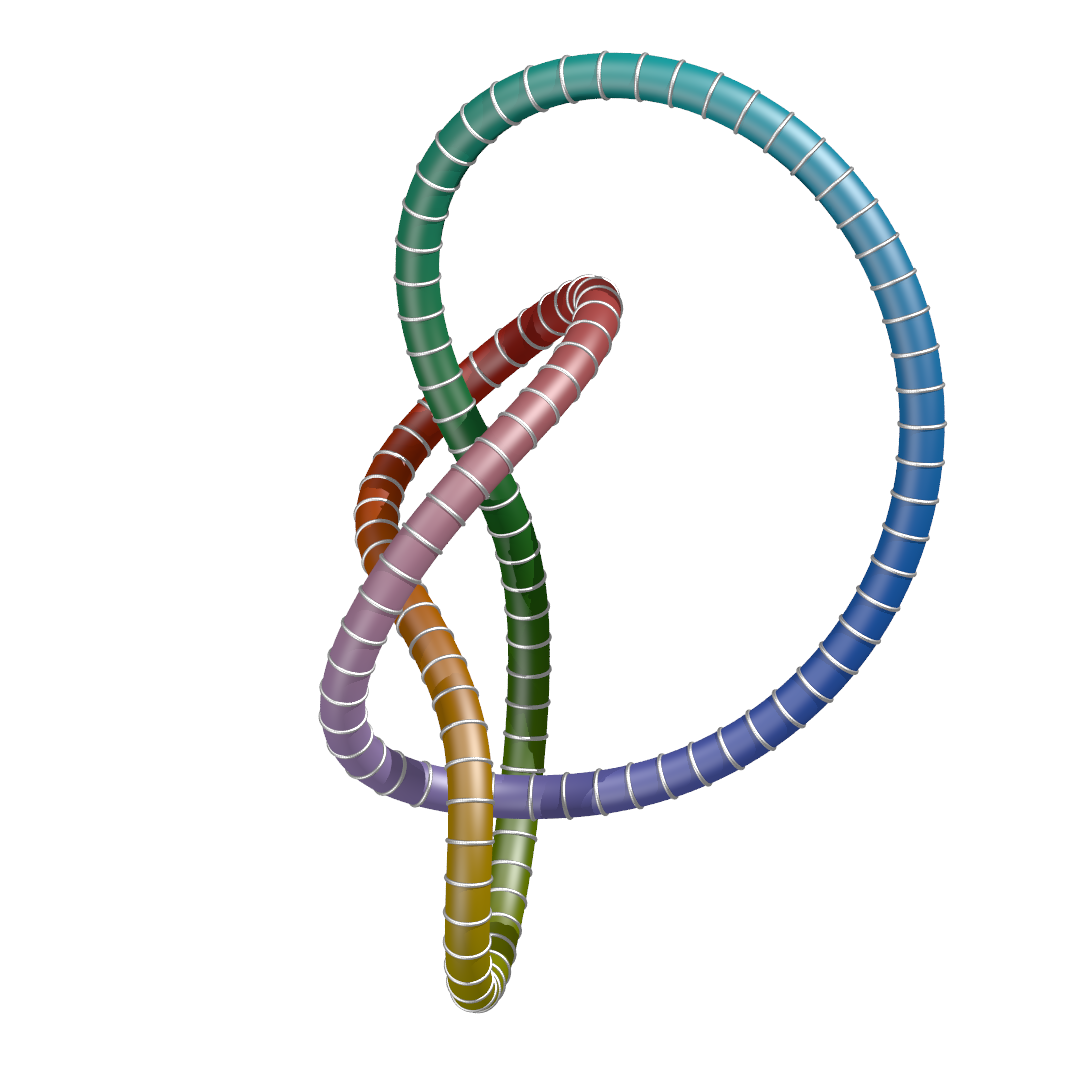}%
    \incl{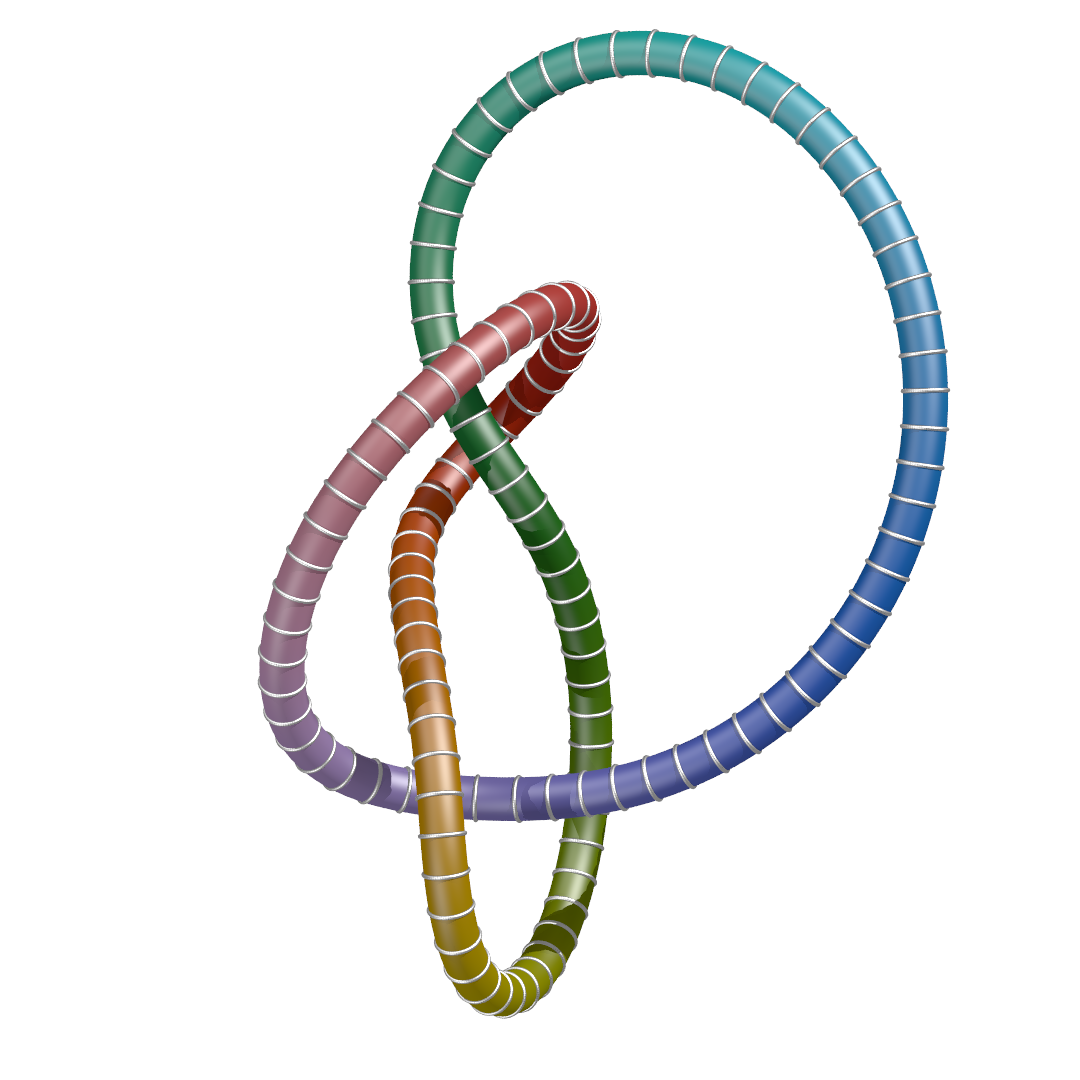}%
    \incl{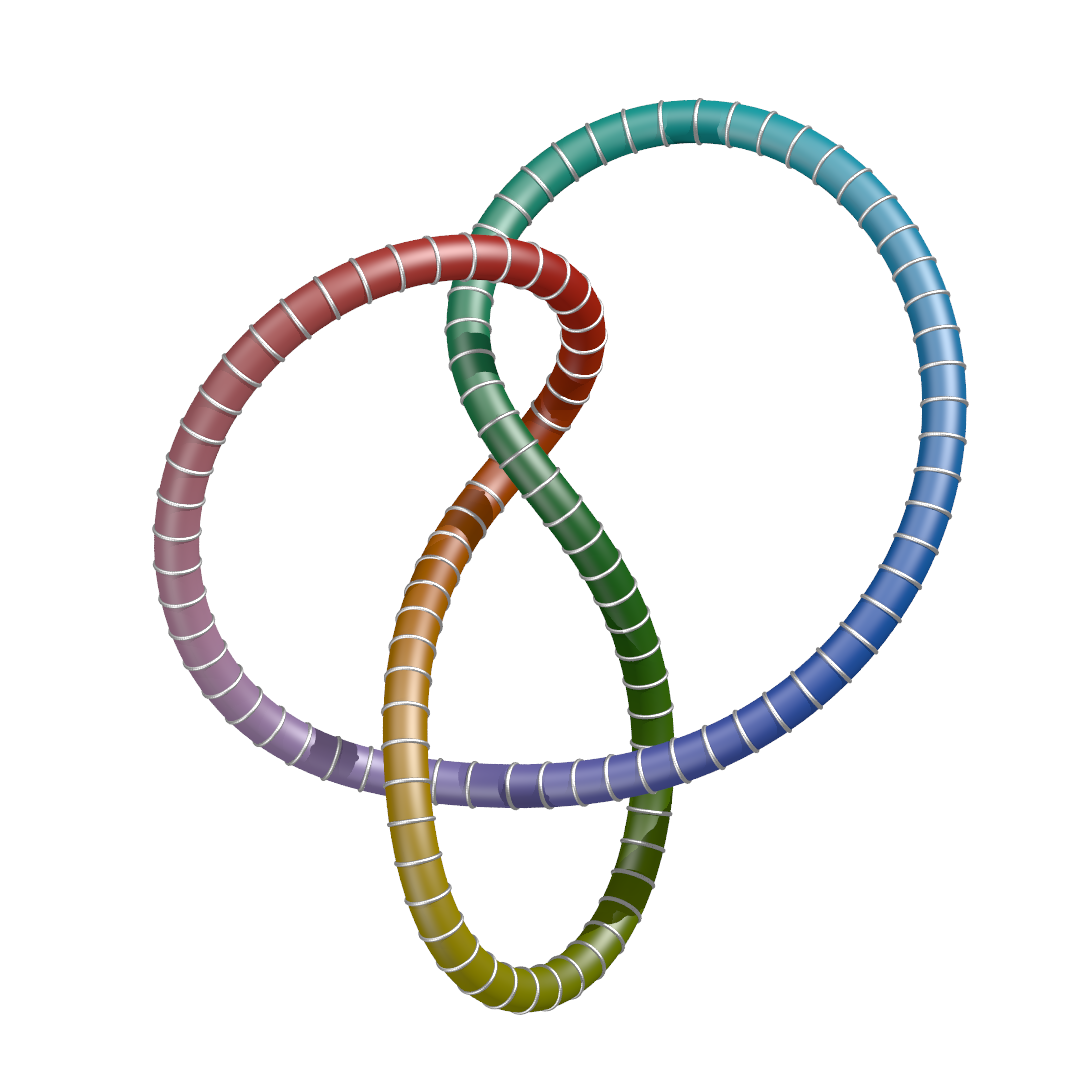}%
    \incl{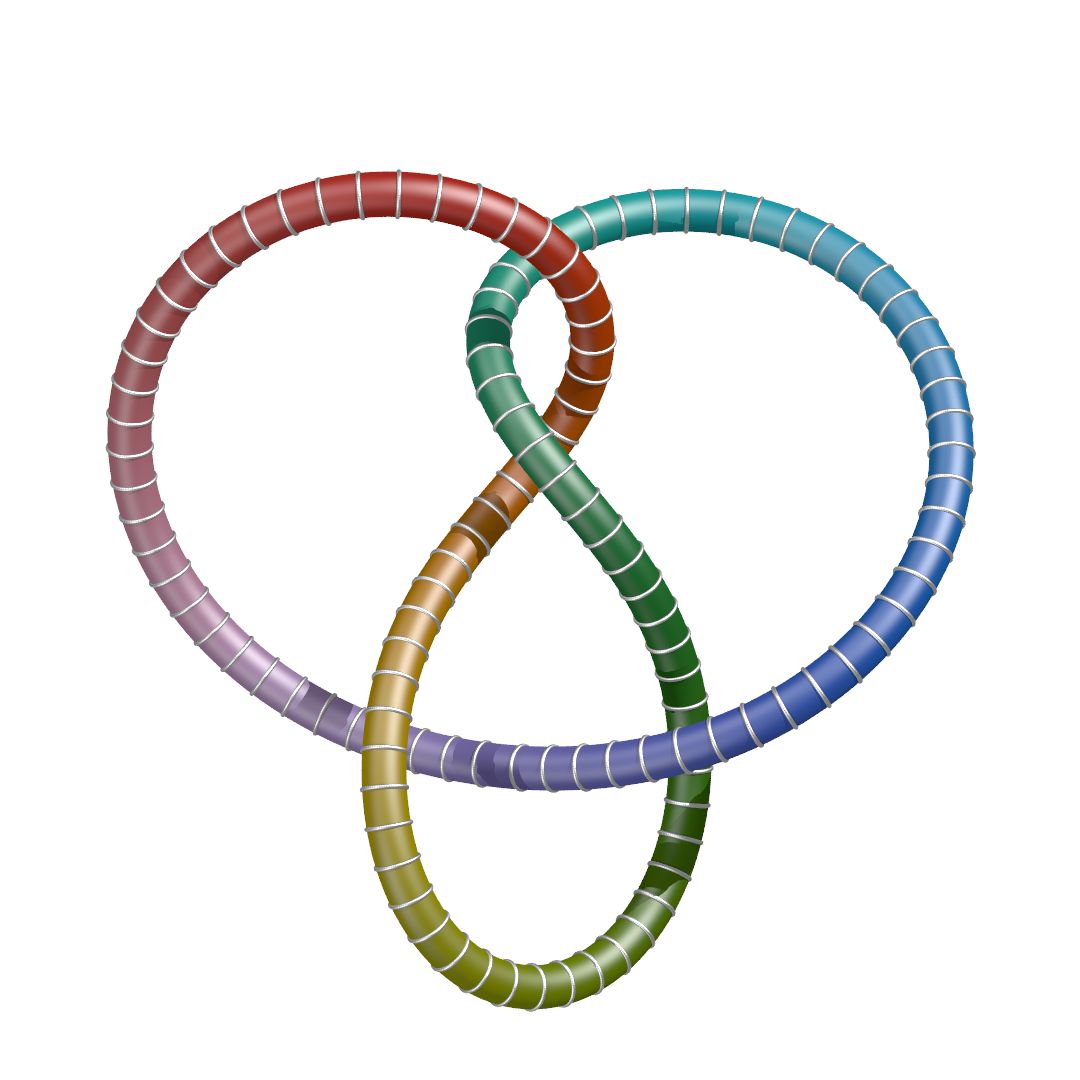}%
    \caption{A shortest path between a figure-eight knot (outmost left) and its mirror image (outmost right) with respect to the metric $G$ for $s = 7/4$. This illustrates that this knot is \emph{positive amphicheiral}, i.e., it is isotopic to its mirror image. The color coding indicates the parametrization; the markings are equidistant in the coordinate domain.}
    \label{fig:FigureEight}
\end{figure}

The purpose of this paper is to provide such a concept of ``cost'' and ``natural motions'' by proposing a suitable Riemannian metric $G$ on the (infinite dimensional) manifold~$\Mfld$ of embeddings into the Euclidean space~$\AmbSpace$.
Throughout this paper, the term \emph{embedding} will denote an immersion that is also a homeomorphism onto its image. (The only exception will be embeddings in the sense of the Sobolev and Morrey embedding theorems.)
The requirement of embeddedness serves as our operationalization of impenetrability of the parametrized geometric objects.

For the sake of simplicity we limit ourselves to one-dimensional submanifolds without boundary, i.e., closed curves embedded in~$\AmbSpace$.
These can be represented by parametrizations~$\gamma \colon \Circle \to \AmbSpace$, where $\Circle = \R \slash \Z$.
The higher-dimensional case will be addressed in a forthcoming paper.
The metric $G$ assigns a ``cost'' to any path in the space $\Mfld$; the so-called \emph{path length}.
Among other properties, we will show that there is a path of minimal length between any two given configurations within the same isotopy class.
Finite element discretization and  numerical optimization of the variational problem then yield  motions as shown in \cref{fig:FigureEight,,fig:Trefoil,,fig:PerkoPair}.

In order to realize the shortest distances, we study paths $\Path$ mapping an interval $\intervalco{0,T}$ to the space~$\Mfld$ of closed embedded curves.
Generally speaking, there are plenty of failure modes for extending $\Path$ to $\intervalcc{0,T}$:
as $t \nearrow T$ the curves $\Path(t)$ could
    \begin{failuremodes}
        \itemsep0em 
        \item \label{item:LoseRegularity} lose regularity and fail to be tame (e.g., develop kinks);
        \item \label{item:BlowUp} blow up locally, i.e., $\abs{\partial_x \Path(t)(x)} \longrightarrow \infty$;
        \item \label{item:Collapse} collapse locally, eventually failing to be immersed, i.e., $\abs{\partial_x \Path(t)(x)} \longrightarrow 0$;
        \item \label{item:WildTravel} travel wildly through the ambient space $\AmbSpace$, e.g., the center of mass of $\Path(t)$ might diverge;
        \item \label{item:PullTight} develop pull-tights (a small knotted arc shrinks to a point, see \cref{fig:PullTight}); or
        \item \label{item:SelfIntersections} develop self-intersections.
    \end{failuremodes}
In particular, in $\R^3$ the last two failure modes may cause the isotopy type of $\Path(t)$ to change as $t$ varies.

\begin{figure}[ht]%
    \centering%
    \newcommand{\incl}[2]{\begin{tikzpicture}%
        \node[inner sep=0pt] (fig) at (0,0) {\includegraphics[	trim = 0 0 0 0, 
        clip = true,  
        angle = 0,
        width = 0.12\textwidth]{#1}};
    \end{tikzpicture}}
    \incl{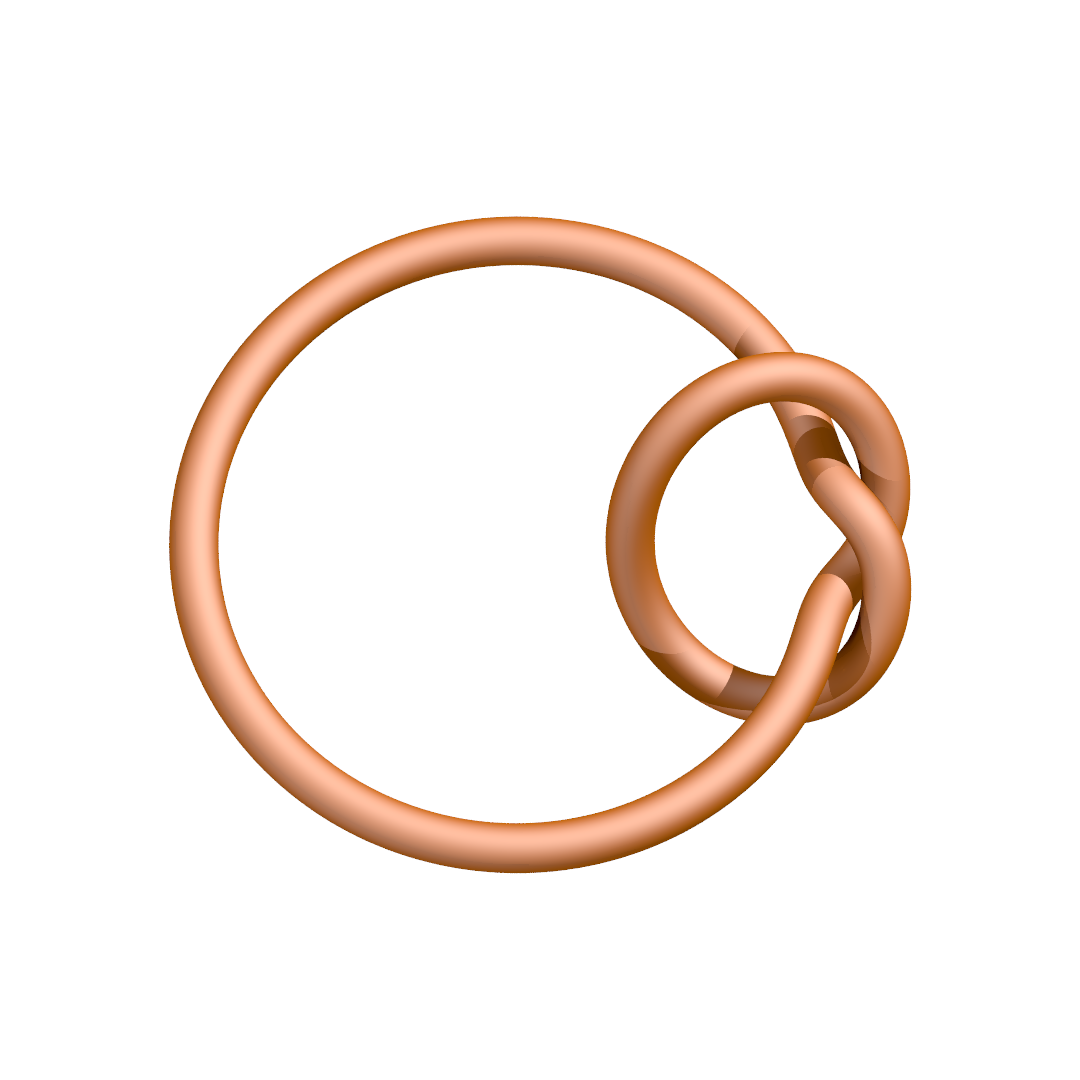}{a}
    \incl{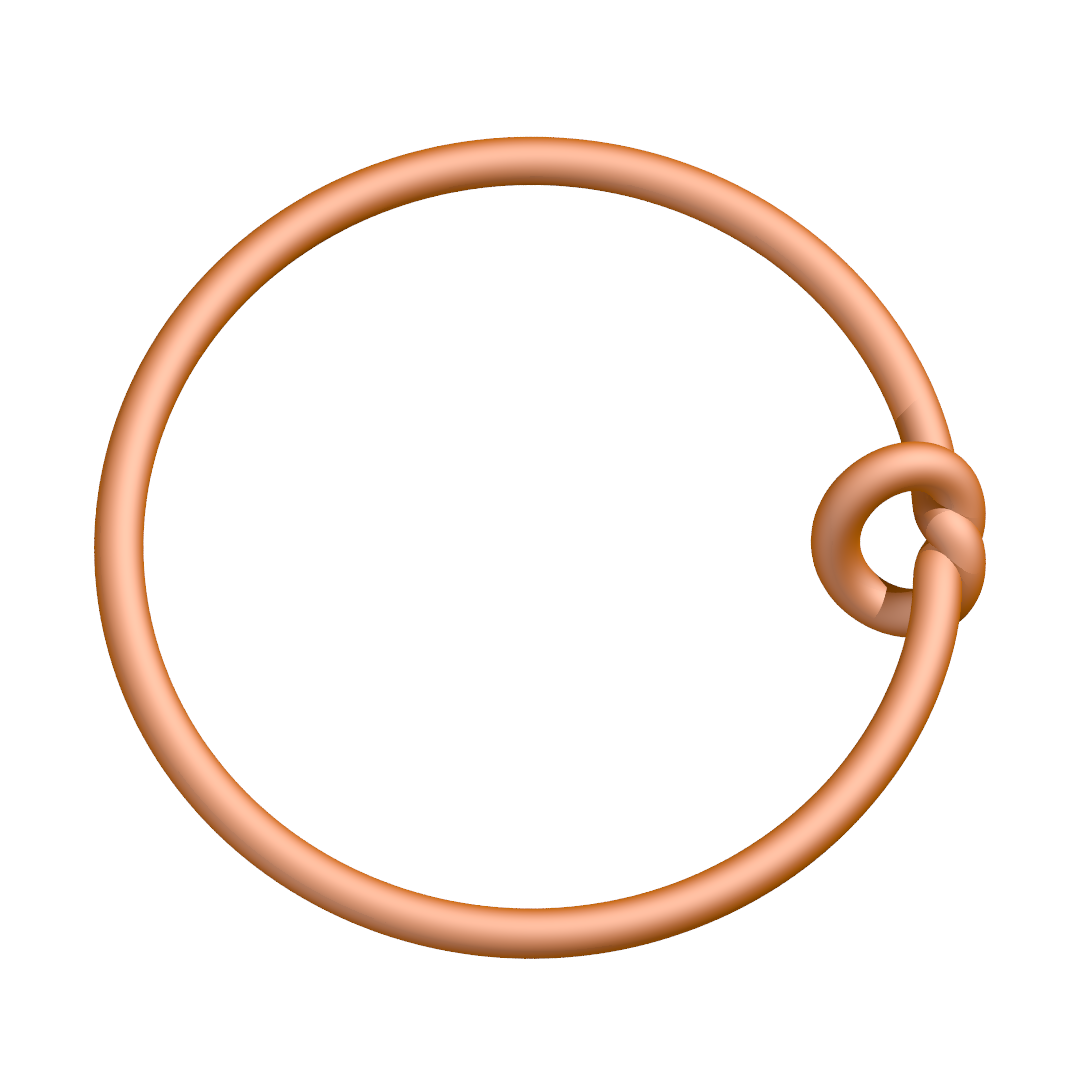}{b}
    \incl{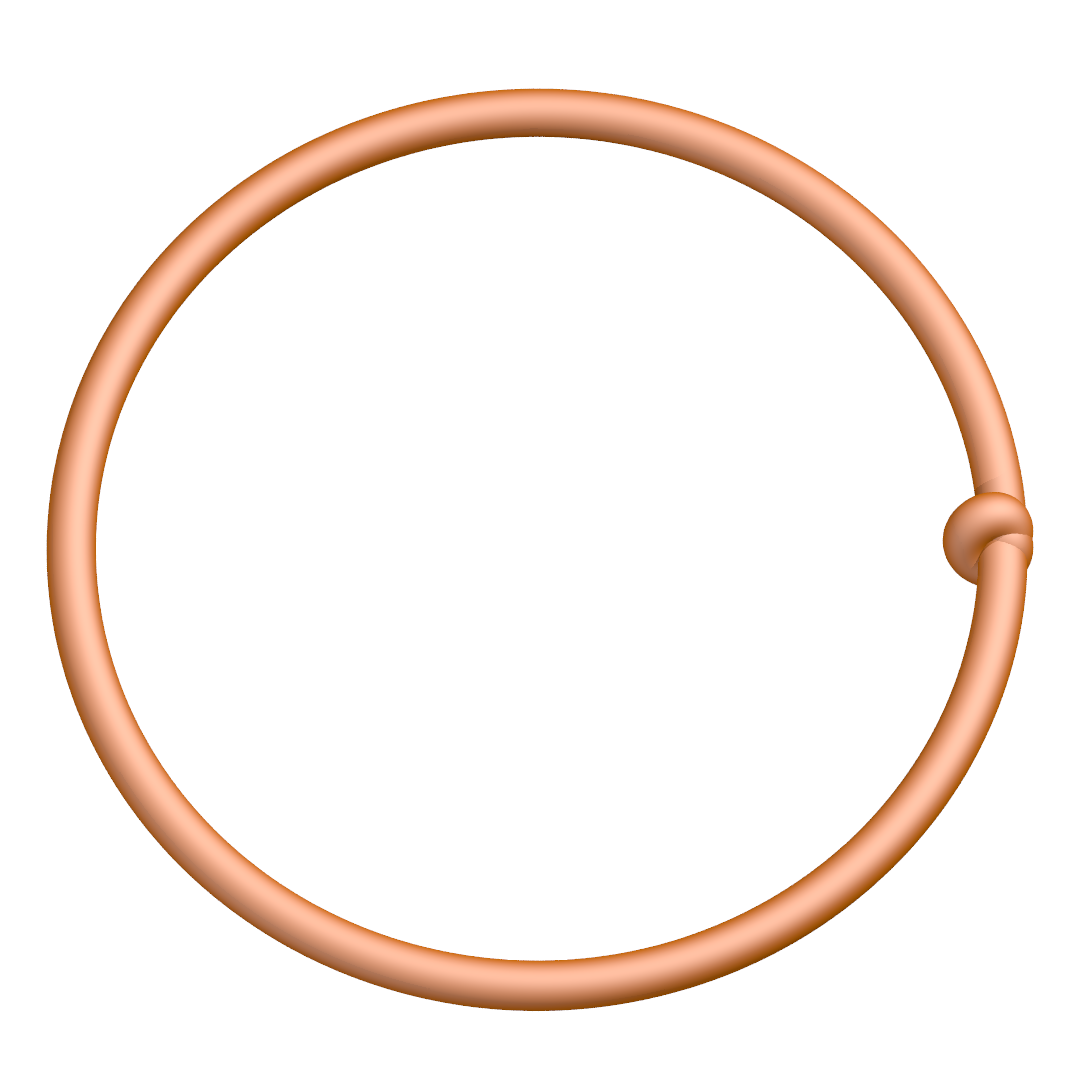}{b}
    \incl{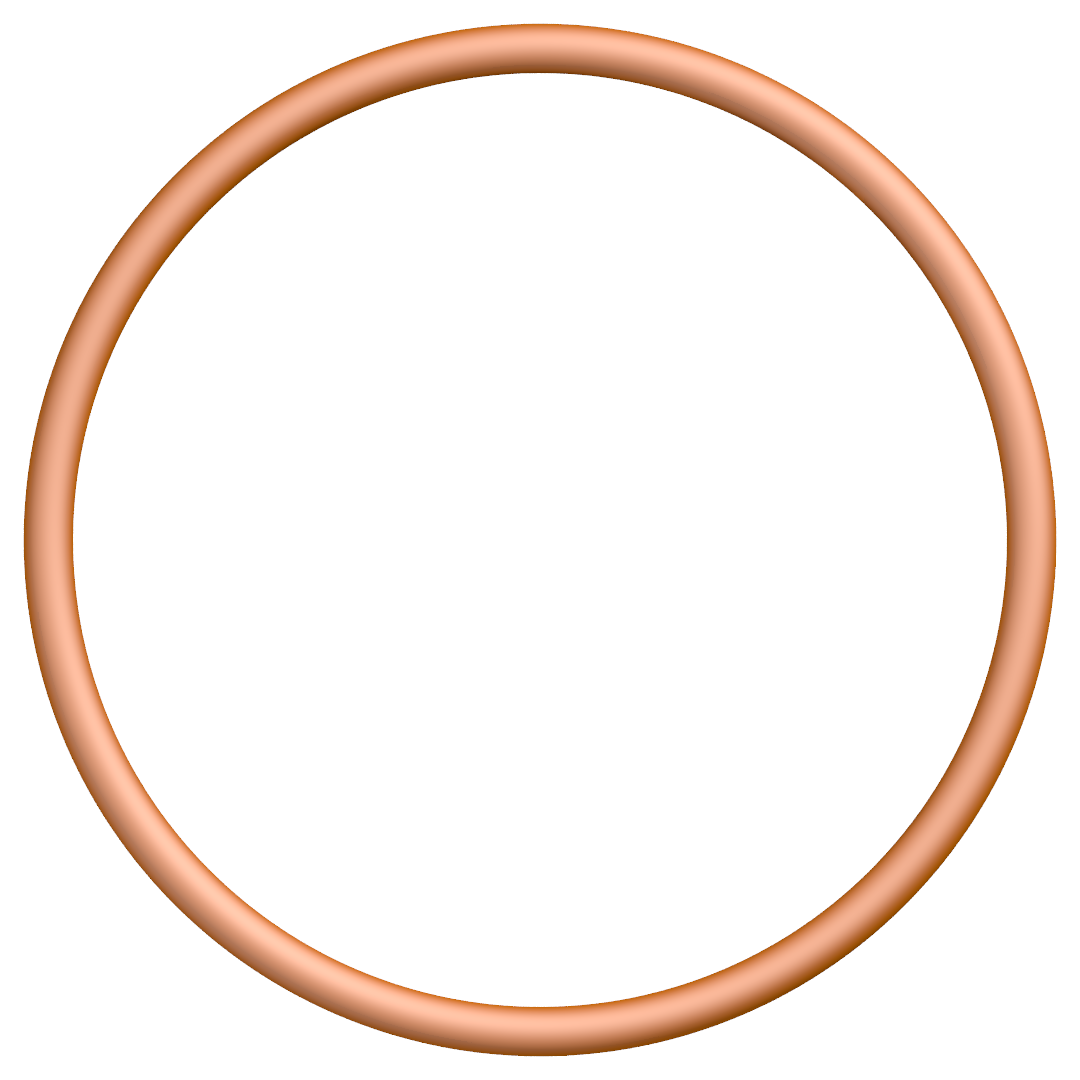}{c}
    \caption{Pull-tight of the trefoil knot. The limiting curve is an unknot. As we will show, this cannot happen for a path of finite length with respect to the metric $G$.}
    \label{fig:PullTight}
\end{figure}

According to an impressive series of papers authored, in different collaborations, by Bauer, Bruveris, Harms, Heslin, Kolev, Maor, Michor, Mumford, Preston, and many others
(see~\cite{MR3745697,bauerharmsmichor2,zbMATH07845668,MR3264258,MR2148075} and references therein)
all of these failure modes except the last one (formation of self-intersections) can be addressed by choosing a Riemannian metric that builds on geometrized versions of inner products on the Sobolev space $\Bessel[s]$, $s > 3/2$.
The above papers investigate the space $\Imm(\Sphere^1;\AmbSpace)$ of immersions, which actually includes
all embeddings by the above definition.
It also contains the group of diffeomorphisms of the circle $\Diff(\Sphere^1)$, which in some sense serves as a model case for $\Imm(\Sphere^1;\AmbSpace)$.
Especially,
$\Lebesgue[2]$- and $\Bessel[1]$-metrics on the aforementioned spaces have been studied intensively, see~\cite{MR3055801,MR3411680,MR4142275,MR1668586,MR3265830,MR2900792,MR3995369}.

The $\Bessel[1]$-metrics are also referred to as ``elastic metrics''. They are particularly appealing for applications because their geodesics can be written in closed form via their \emph{square-root velocity representation}, see \cite{MR4737374,Mio:2007im,5601739}.

\begin{figure}[t]%
    \centering%
    \newcommand{\incl}[1]{%
        \includegraphics[
            width=0.165\textwidth,
            trim = 80 60 30 30, 
            clip = true, 
            angle = 0
        ]{#1}%
    }%
    \incl{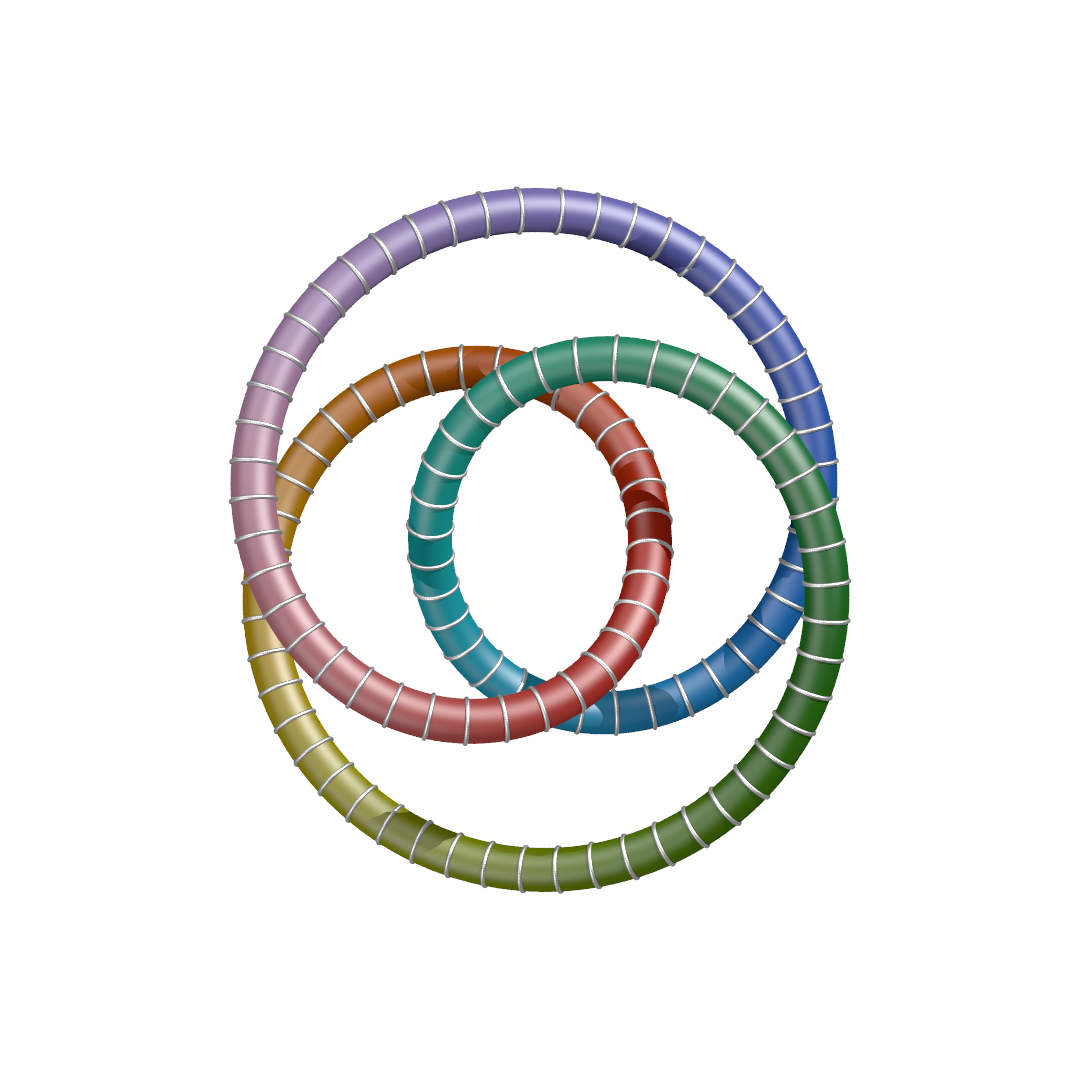}%
    \incl{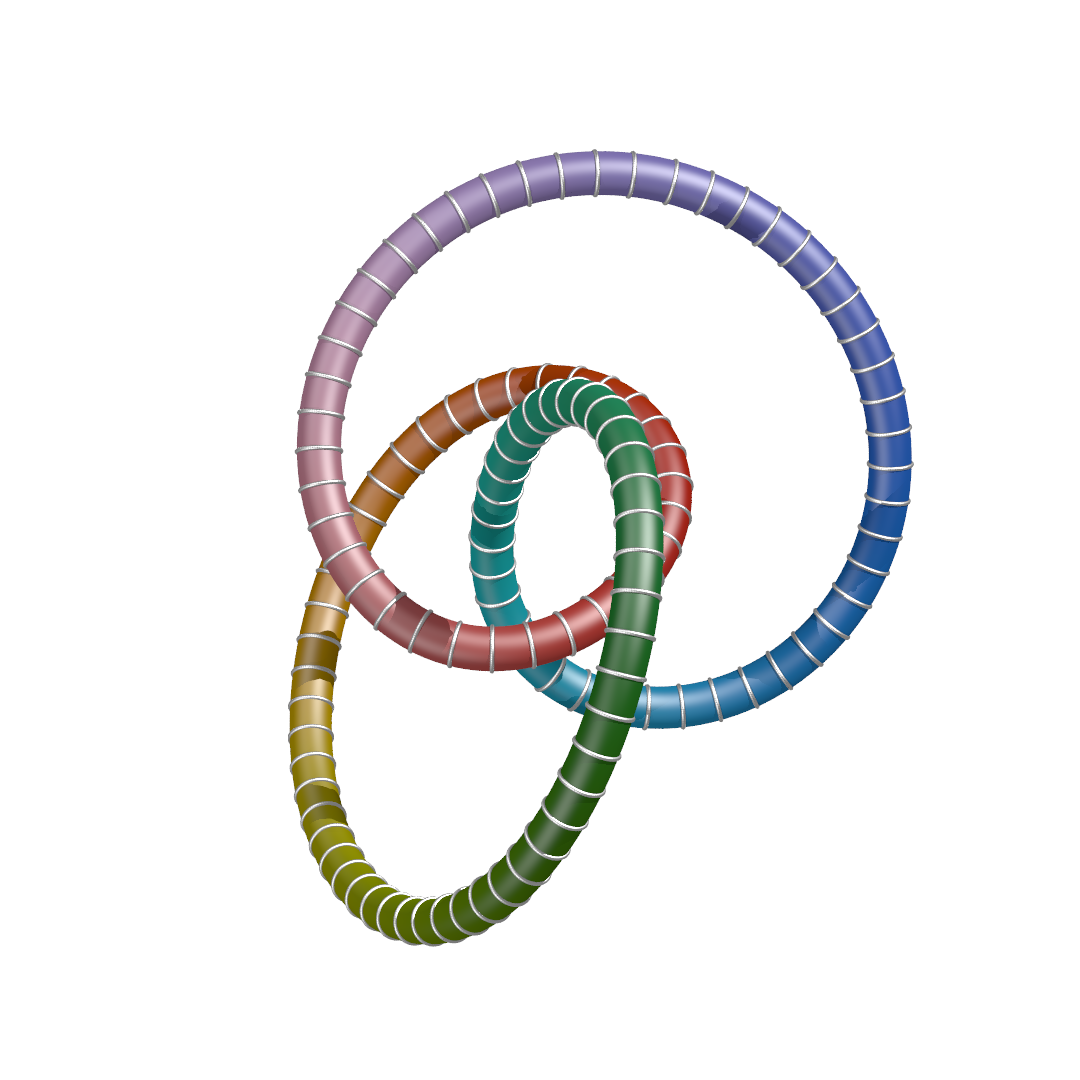}%
    \incl{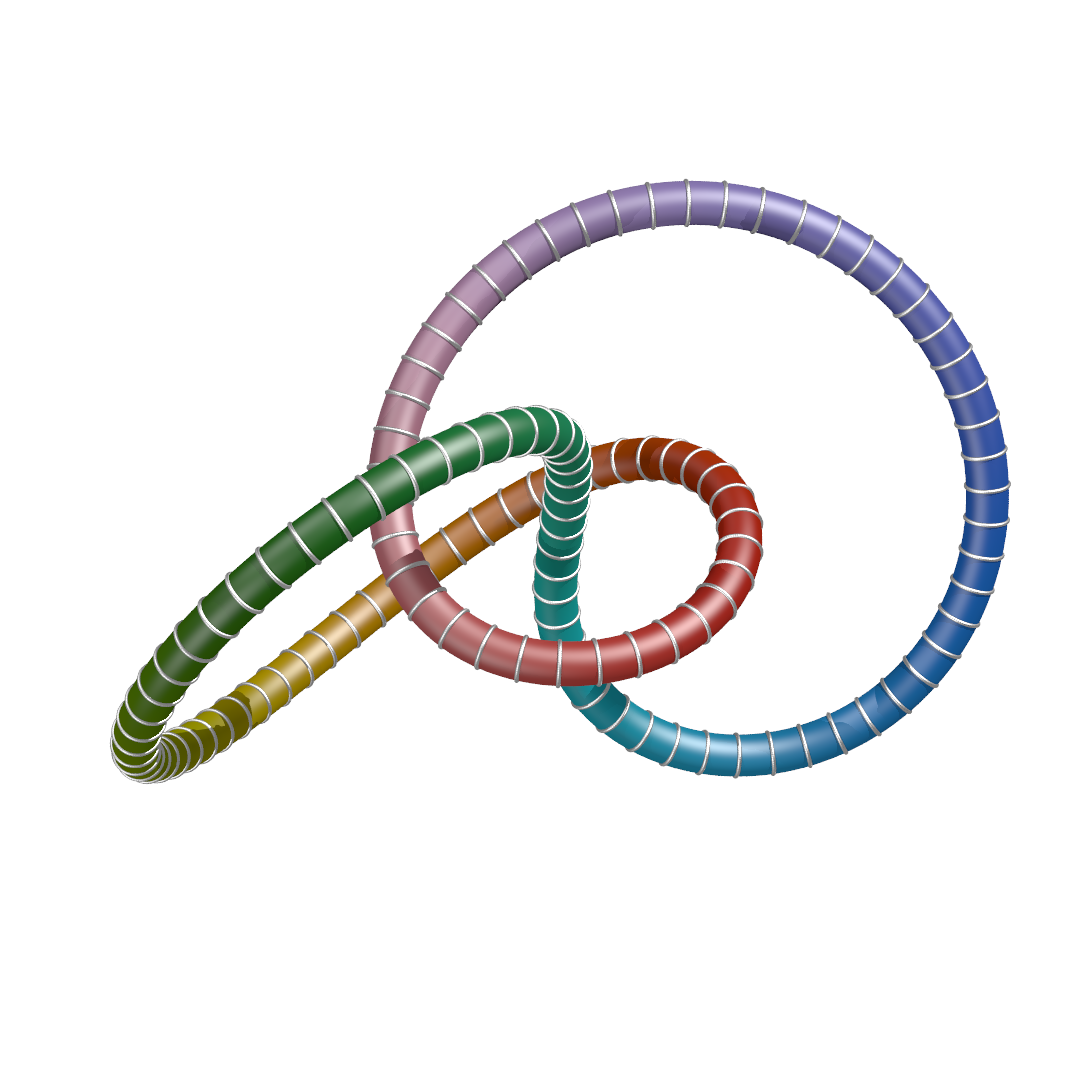}%
    \incl{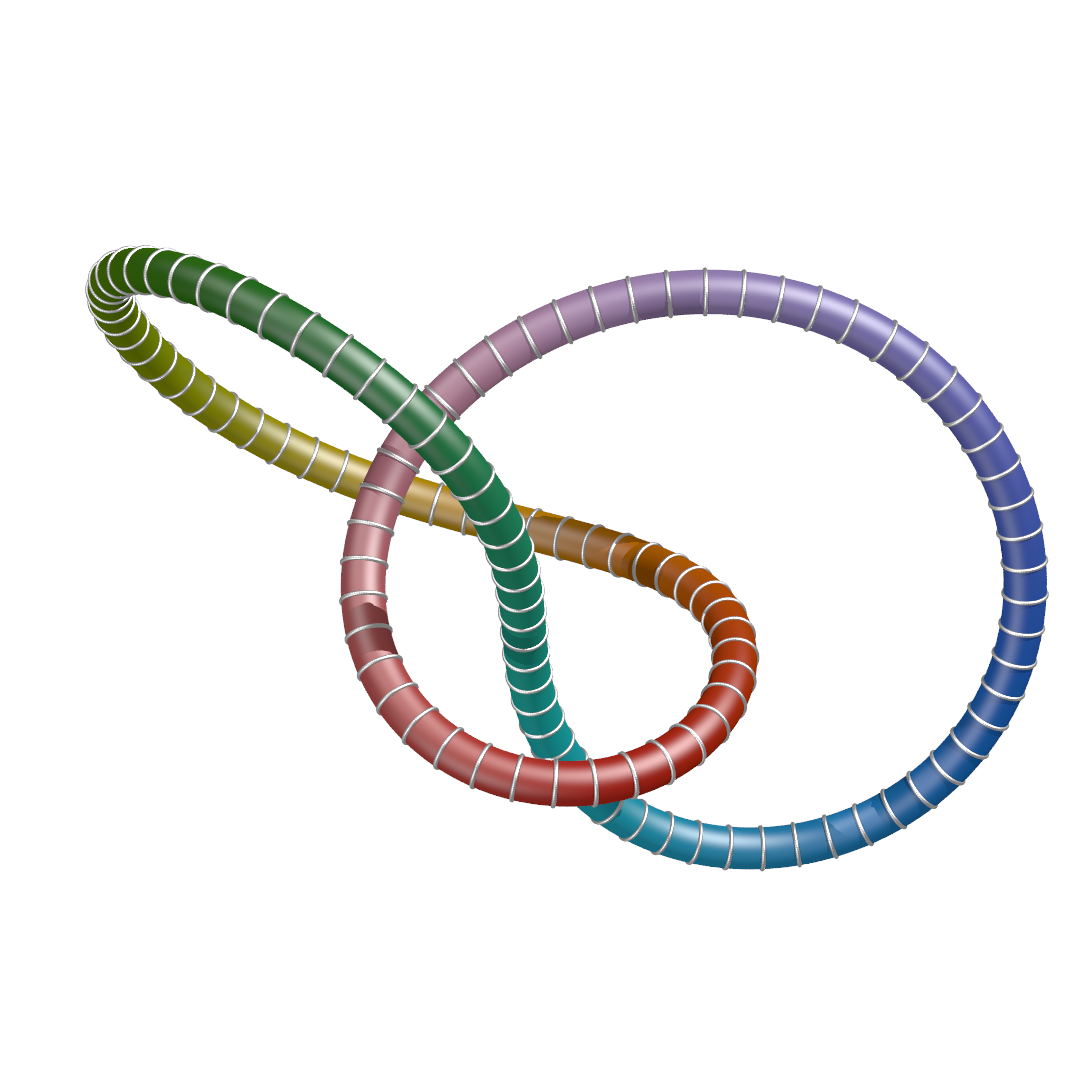}%
    \incl{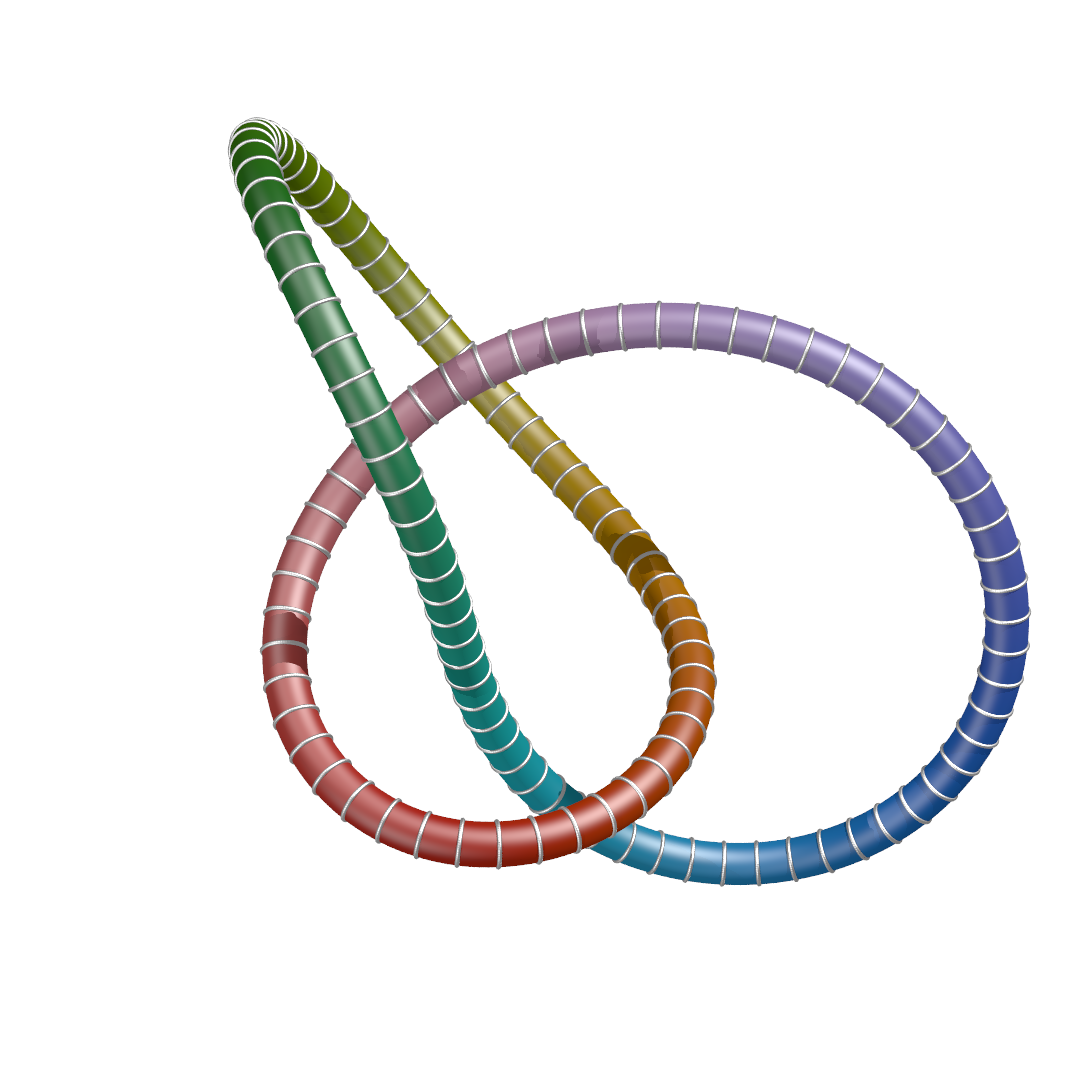}%
    \incl{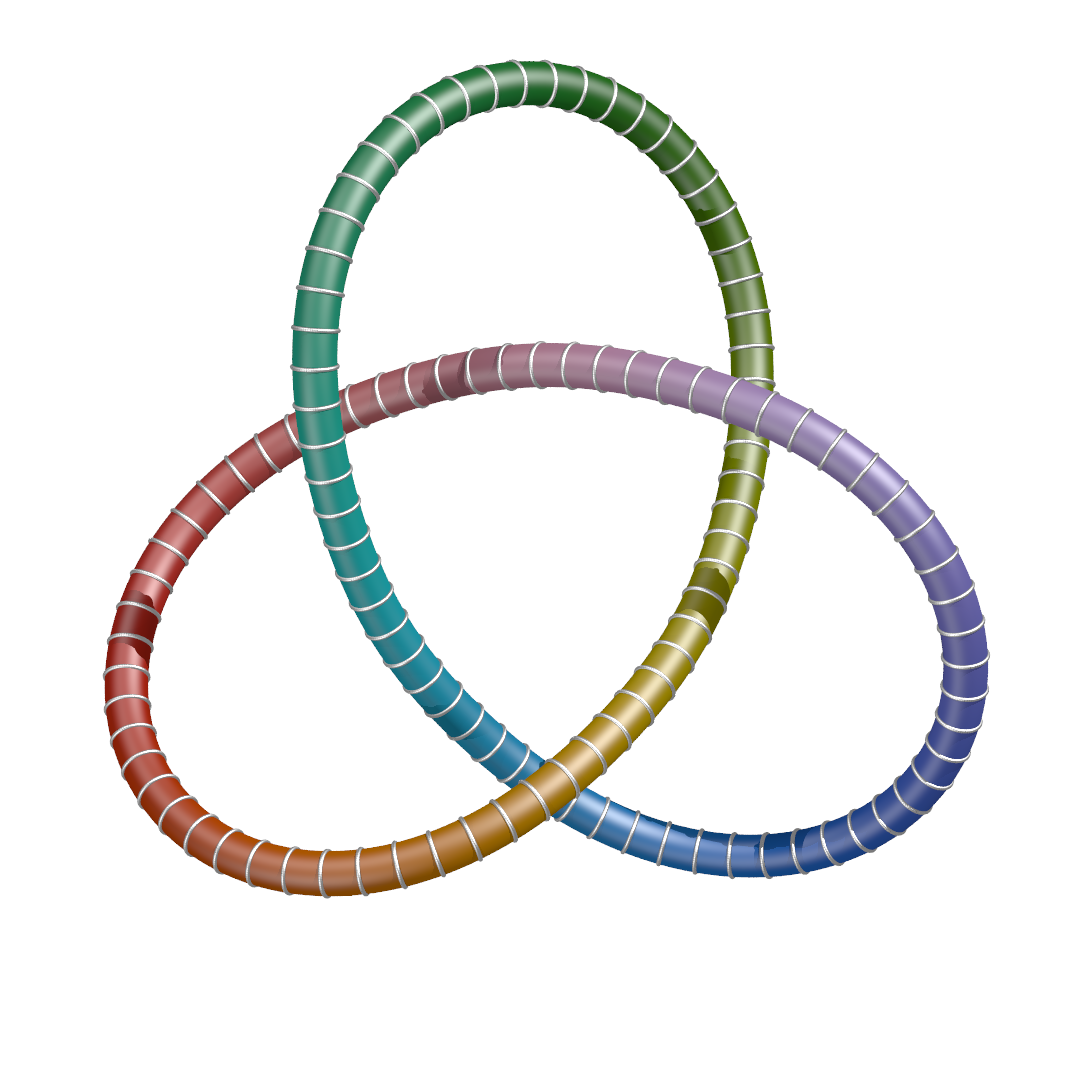}%
    \caption{A shortest path between a $(-3,2)$ torus knot (outmost left) and a $(2,-3)$ torus knot (outmost right) with respect to the metric $G$ for $s = 7/4$. The color coding indicates the parametrization; the markings are equidistant in the coordinate domain.}
    \label{fig:Trefoil}
\end{figure}

Higher-order Sobolev metrics and their geodesic properties have also been studied for immersed manifolds, see~\cite{MR2888014,bauerharmsmichor2,2512.01566}.
In \cite{HeRuSc14} a metric $G_{\mathrm{shell}}$ of flavor $\Bessel[2]$ on the space of immersed surfaces is deduced from membrane and bending energies of thin elastic shells.
The associated path length is the integral over the elastic strain rates. 
So this path energy is akin to the amount of work one has to invest to plastically deform the surface as prescribed by the path. As a result, geodesics with respect to this metric look particularly plausible and efficient.

\begin{figure}[t]%
    \centering%
    \newcommand{\incl}[1]{%
        \includegraphics[
            width=0.165\textwidth,
            trim = 120 80 120 40, 
            clip = true, 
            angle = 0
        ]{#1}%
    }%
    \incl{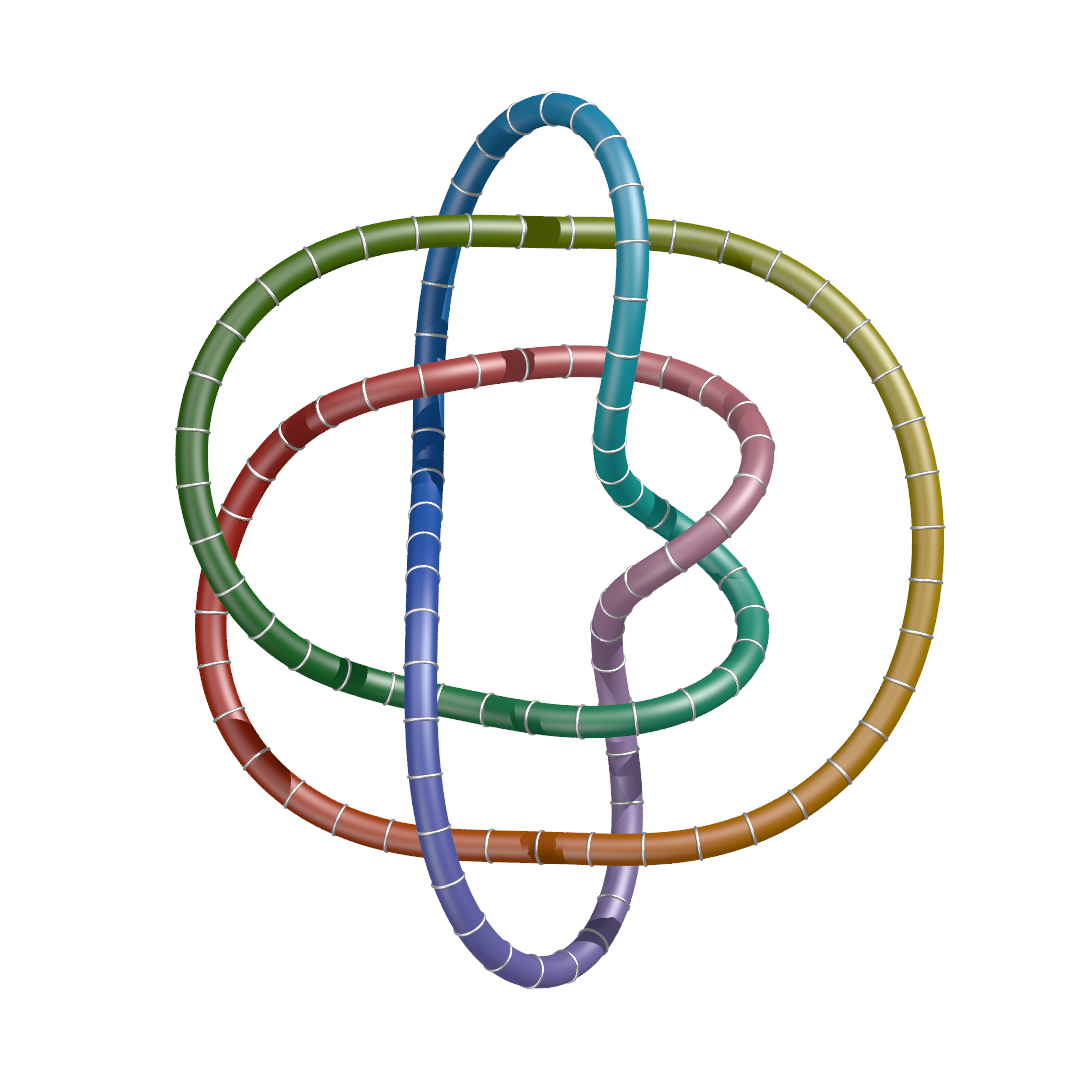}%
    \incl{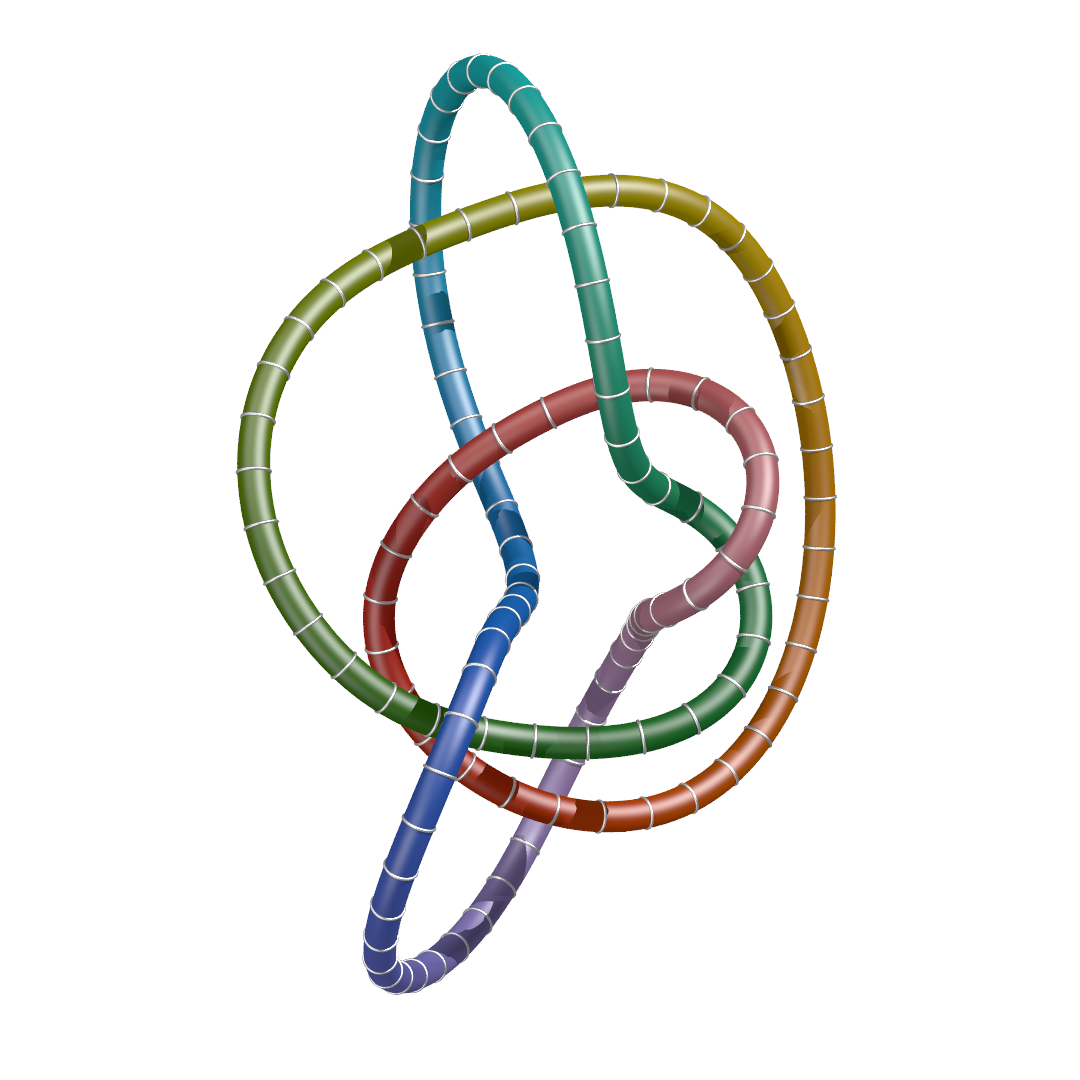}%
    \incl{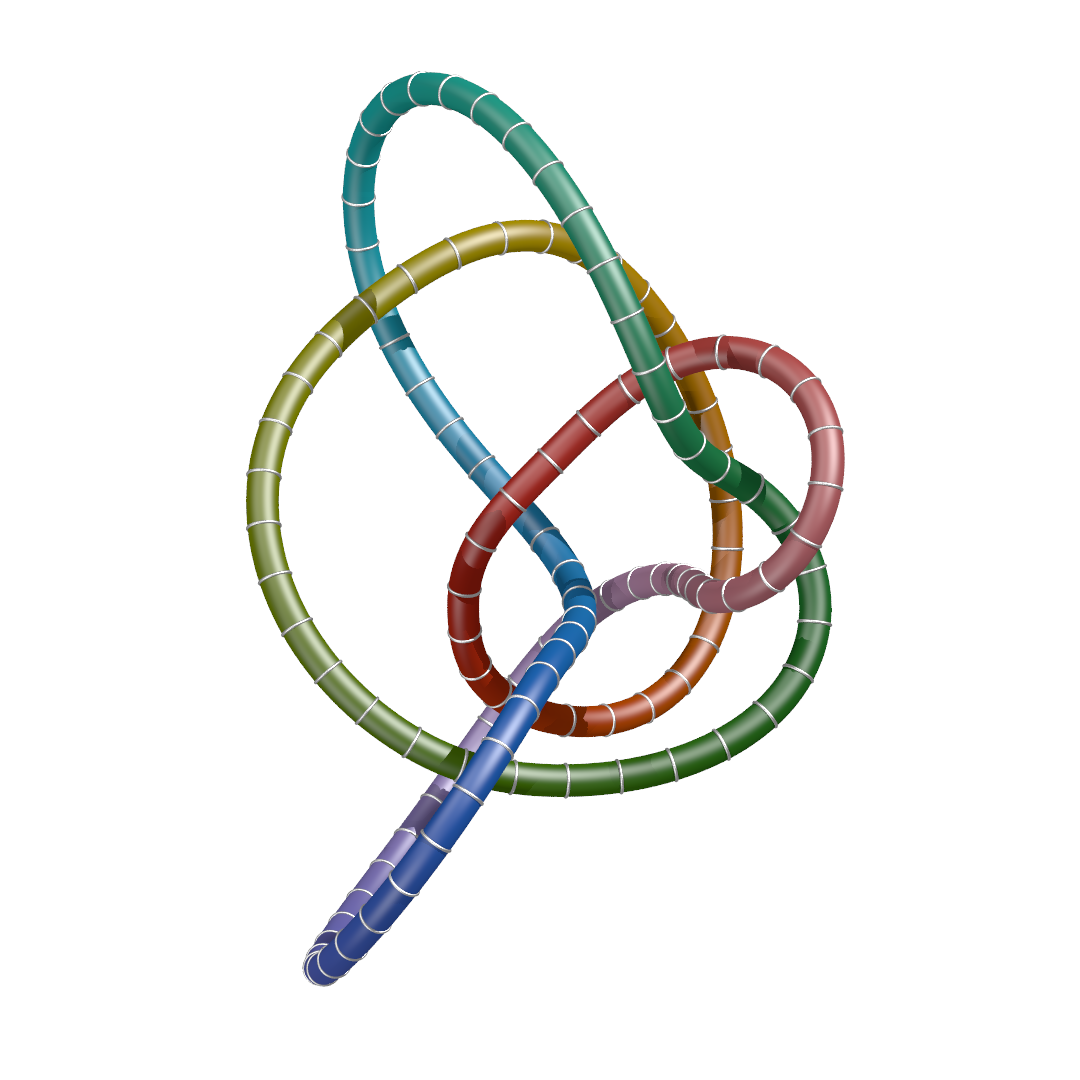}%
    \incl{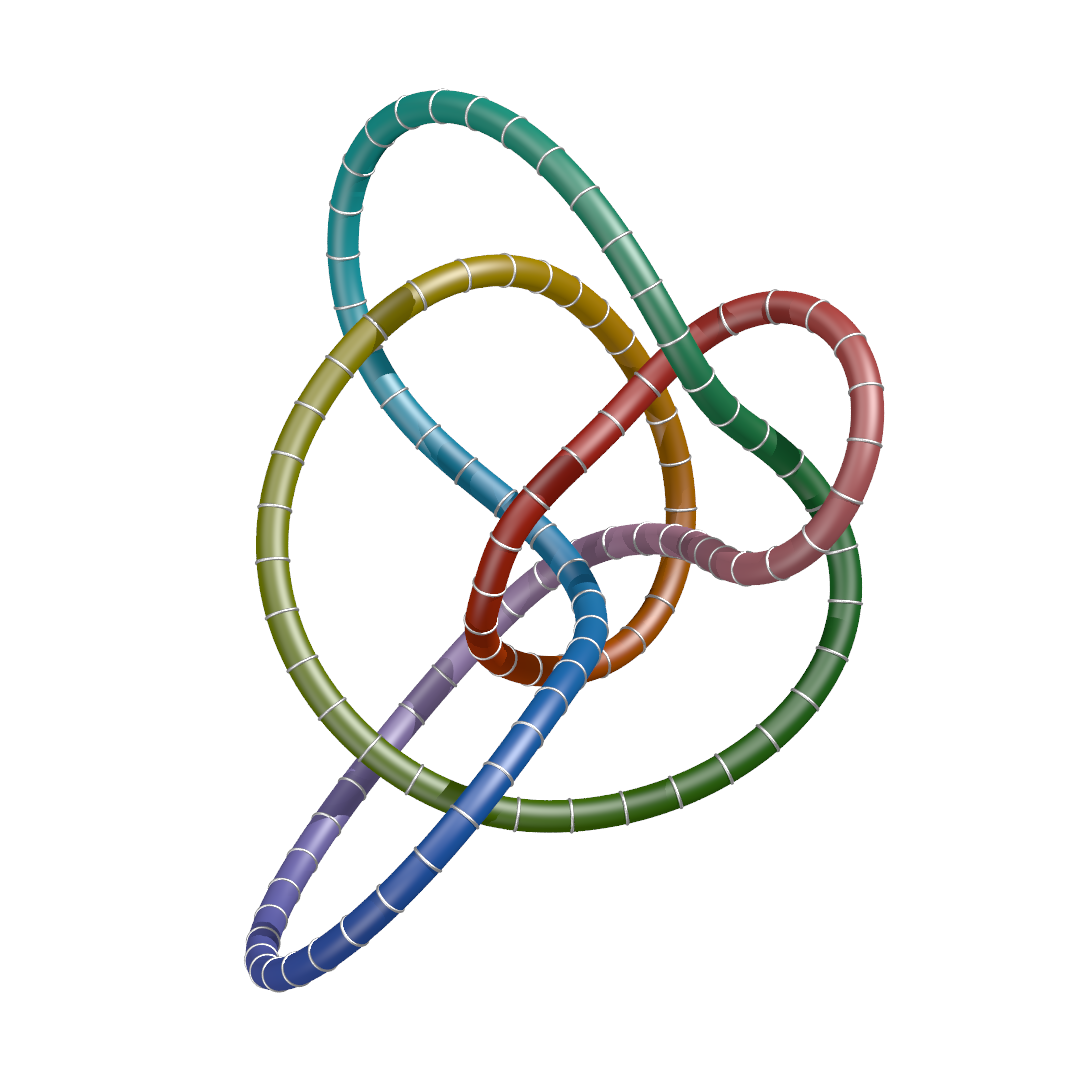}%
    \incl{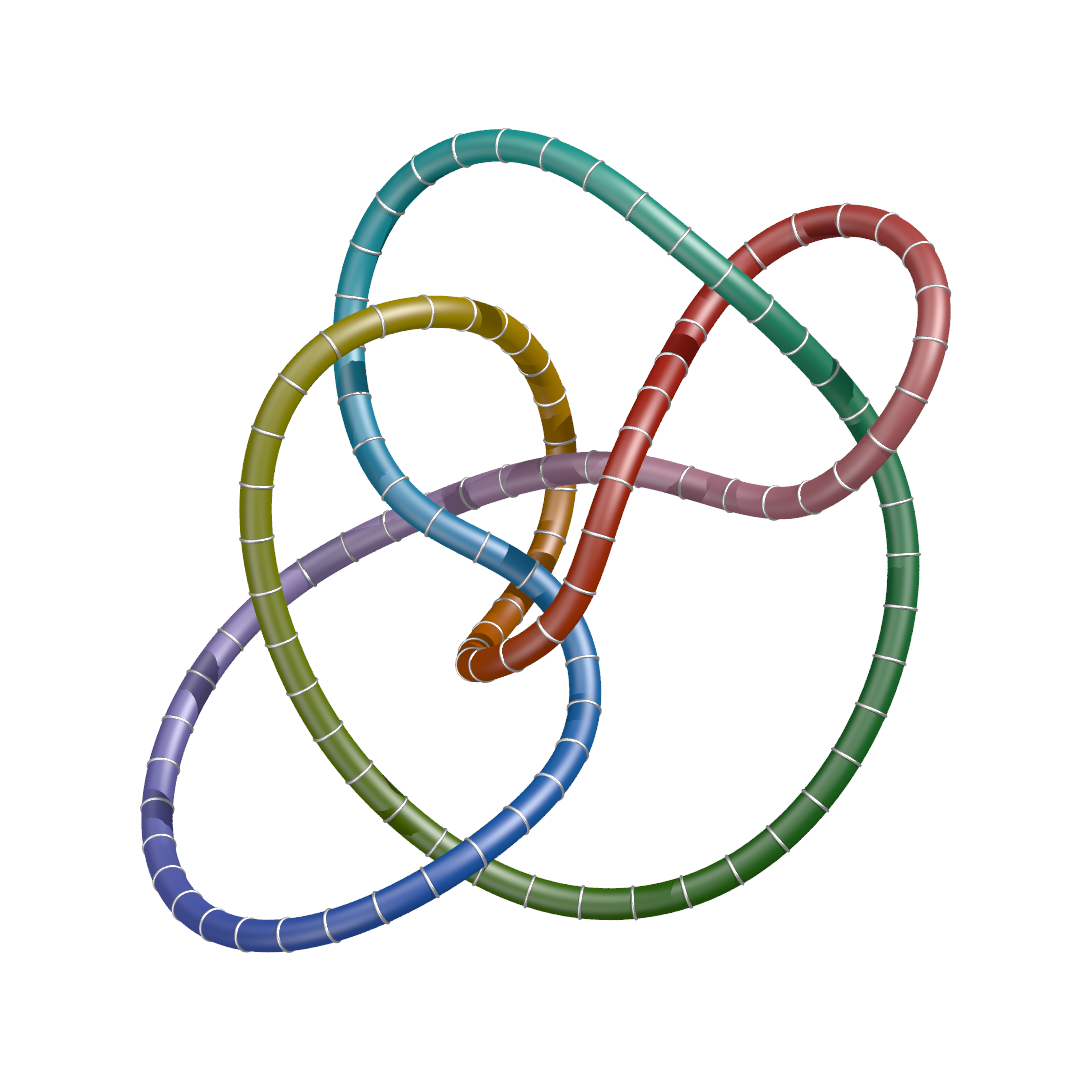}%
    \incl{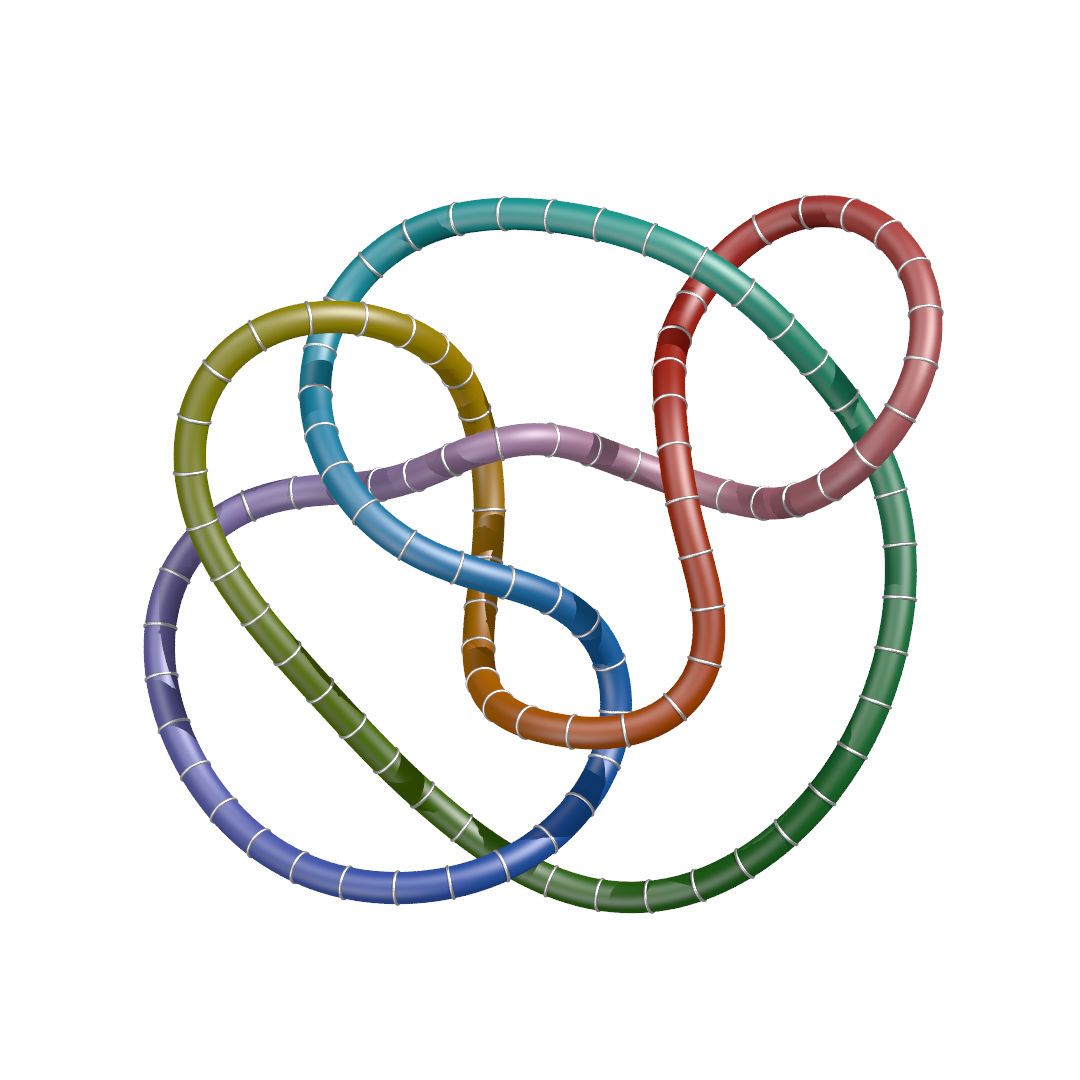}%
    \caption{A shortest path between the knots from the Perko pair with respect to the metric $G$ for $s=7/4$. The color coding indicates the parametrization; the markings are equidistant in the coordinate domain.}
    \label{fig:PerkoPair}
\end{figure}

However, dealing with paths of immersed submanifolds is not sufficient for modeling impenetrability of physical objects.
Instead, one has to guarantee that each submanifold along such a path stays embedded.
To our knowledge, there is only one construction for this so far:
In \cite{10.1145/3658174} the \emph{elastic shell metric} $G_{\mathrm{shell}}$ from \cite{HeRuSc14} was augmented by a rank-one modification to the \emph{repulsive shell metric}
$G_{\mathrm{rep}} = G_{\mathrm{shell}} + D\Energy \otimes D\Energy$,
where $\Energy$ is the tangent-point energy \emph{of surfaces}.
We point out that this was analyzed only on \emph{discrete} surfaces, and hence, only in a finite-dimensional setting.
This kind of augmentation technique goes back to Gordon \cite{zbMATH03362632}. 
For example, it has been employed by Herzog and Loayza-Romero with a suitable barrier function to construct complete metrics on the space of planar triangulations~\cite{zbMATH07603624}.
A caveat of Gordon's construction is that the function~$\Energy$ ought to be proper, i.e., preimages of compact sets under $\Energy$ need to be compact. 
So it is not obvious how to generalize it to infinite dimensions.
Notice that the tangent-point energy is invariant under reparametrizations (see below); so it cannot be proper, even not in a weak sense.

The metric $G$ that we will construct here also involves a (generalized) tangent-point energy~$\Energy$, albeit in a more subtle way and only for curves, not for surfaces. 
While $G_{\mathrm{shell}}$ has fairly little to do with $\Energy$, the metric $G$ is specifically tailored to be a good preconditioner for $\Energy$ by resembling the essential features of the Hessian of $\Energy$. In particular, $G$ metrizes the energy space of $\Energy$.
In fact, the idea for $G$ is inspired by experience made with numerical optimization of self-avoidance energies, see
\cref{fig:Gradients} and \cite{2103.10408,reiterschumacher1,10.1145/3478513.3480521,2006.07859}.
We would like to stress that the use of $\Energy$ automatically leads us to working with fractional Sobolev spaces.

Before we go into further detail on the tangent-point energy~$\Energy$ and the Riemannian metric~$G$, some general words on infinite-dimensional manifolds are in order.

\subsection{Completeness and regularity in infinite-dimensional Riemannian manifolds}

In finite dimensions, the \emph{Hopf--Rinow theorem}
states that for a finite-dimensional smooth Riemannian manifold~$M$
metric and geodesic completeness as well as the Heine--Borel property are
equivalent;
moreover they imply the existence of length-minimizing geodesics.

The situation of infinite-dimensional manifolds differs quite a lot from this classical setting.
For instance, Riesz's lemma applied to an infinite-dimensional Hilbert space demonstrates that a ball of positive radius cannot be relatively compact anymore. 
So the Heine--Borel property cannot hold true without modifications.
The failure of the Hopf--Rinow theorem in infinite dimensions has also been illustrated by
Grossman's ellipsoid; see McAlpin~\cite{MR2614999},
Grossman~\cite{MR188943}, and Schmeding~\cite[Example 4.43]{zbMATH07575923}.
Further counterexamples have been provided, e.g., by Atkin~\cite{zbMATH03582926, zbMATH00992413}.
In certain instances, $\Bessel[1]$ metrics on immersions may fail to be metrically complete; their metric completion turns out to be in~$\Sobo[1,1]$, see~\cite{MR4737374}.

Consequently, the best we can hope for in infinite dimensions---and what we actually achieve in this paper---is showing all the four statements of the Hopf--Rinow theorem separately: a~relaxed version of the Heine--Borel property (which we call the \emph{Banach--Alaoglu property}), metric completeness, geodesic completeness, and existence of minimizing geodesics.

\begin{figure}[t]%
    \centering%
    \newlength{\heigth}
    \setlength{\heigth}{0.12\textwidth}
    \newcommand{\incla}[1]{%
        \includegraphics[
            height=\heigth,
            trim = 620 180 620 180, 
            clip = true, 
            angle = 0
        ]{#1}%
    }%
    \newcommand{\inclb}[1]{%
        \includegraphics[
            height=\heigth,
            trim = 560 180 560 180, 
            clip = true, 
            angle = 0
        ]{#1}%
    }%
    \newcommand{\inclc}[1]{%
        \includegraphics[
            height=\heigth,
            trim = 520 180 520 180, 
            clip = true, 
            angle = 0
        ]{#1}%
    }%    
    \newcommand{\incld}[1]{%
        \includegraphics[
            height=\heigth,
            trim = 460 180 460 180, 
            clip = true, 
            angle = 0
        ]{#1}%
    }%   
    \newcommand{\incle}[1]{%
        \includegraphics[
            height=\heigth,
            trim = 400 180 400 180, 
            clip = true, 
            angle = 0
        ]{#1}%
    }%    
    \newcommand{\incl}[1]{%
        \includegraphics[
            height=\heigth,
            trim = 360 180 360 180, 
            clip = true, 
            angle = 0
        ]{#1}%
    }%
    \incla{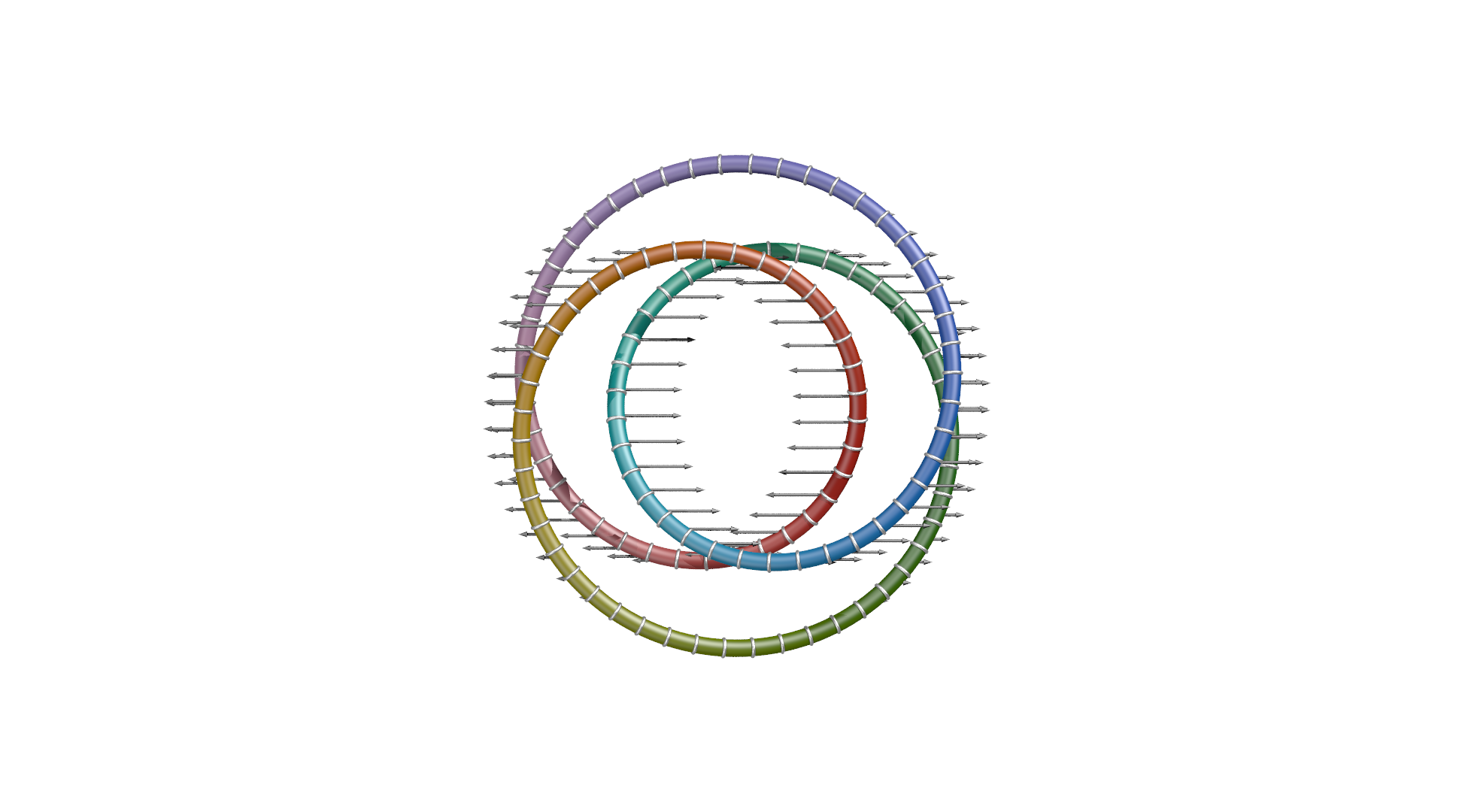}%
    \inclb{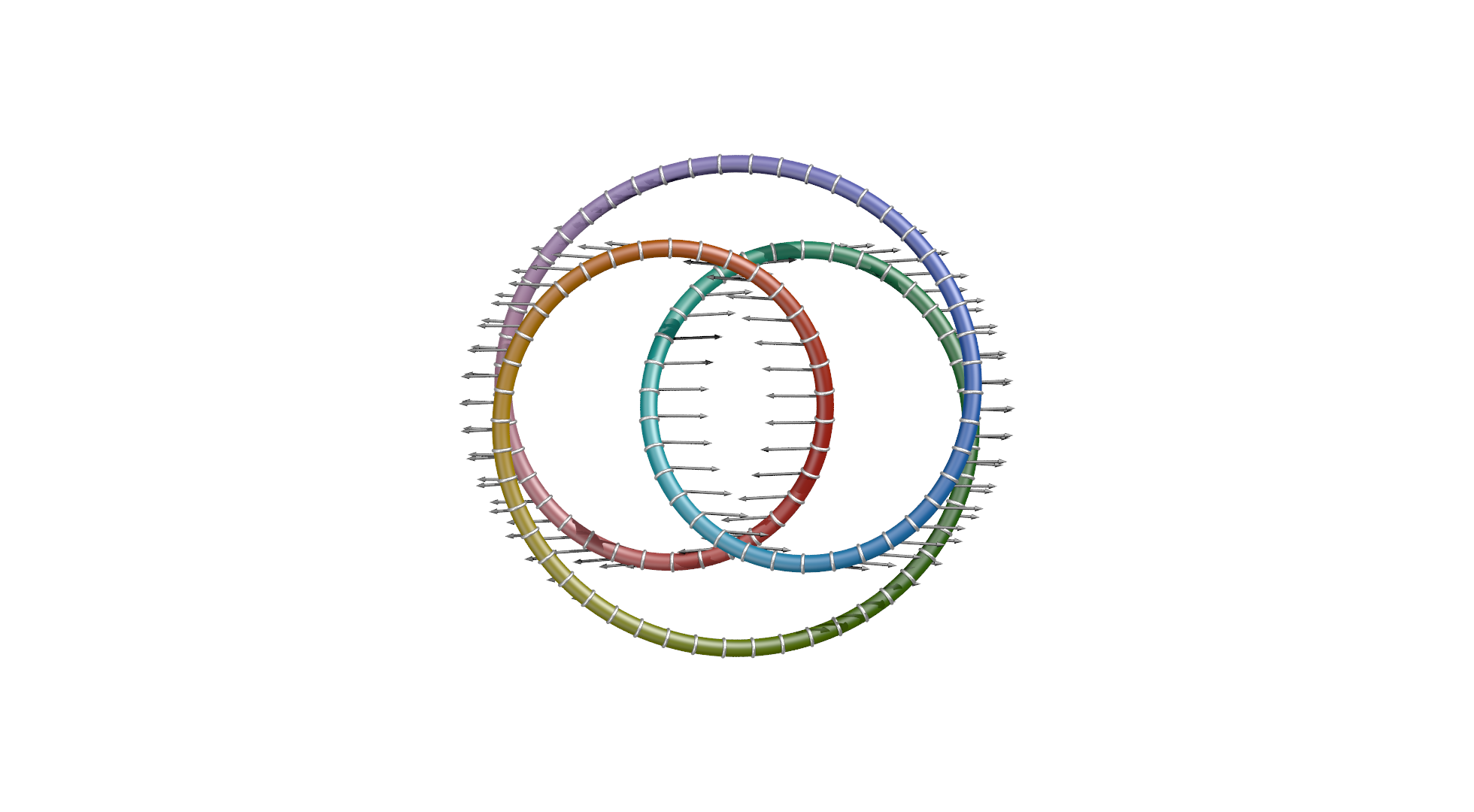}%
    \inclc{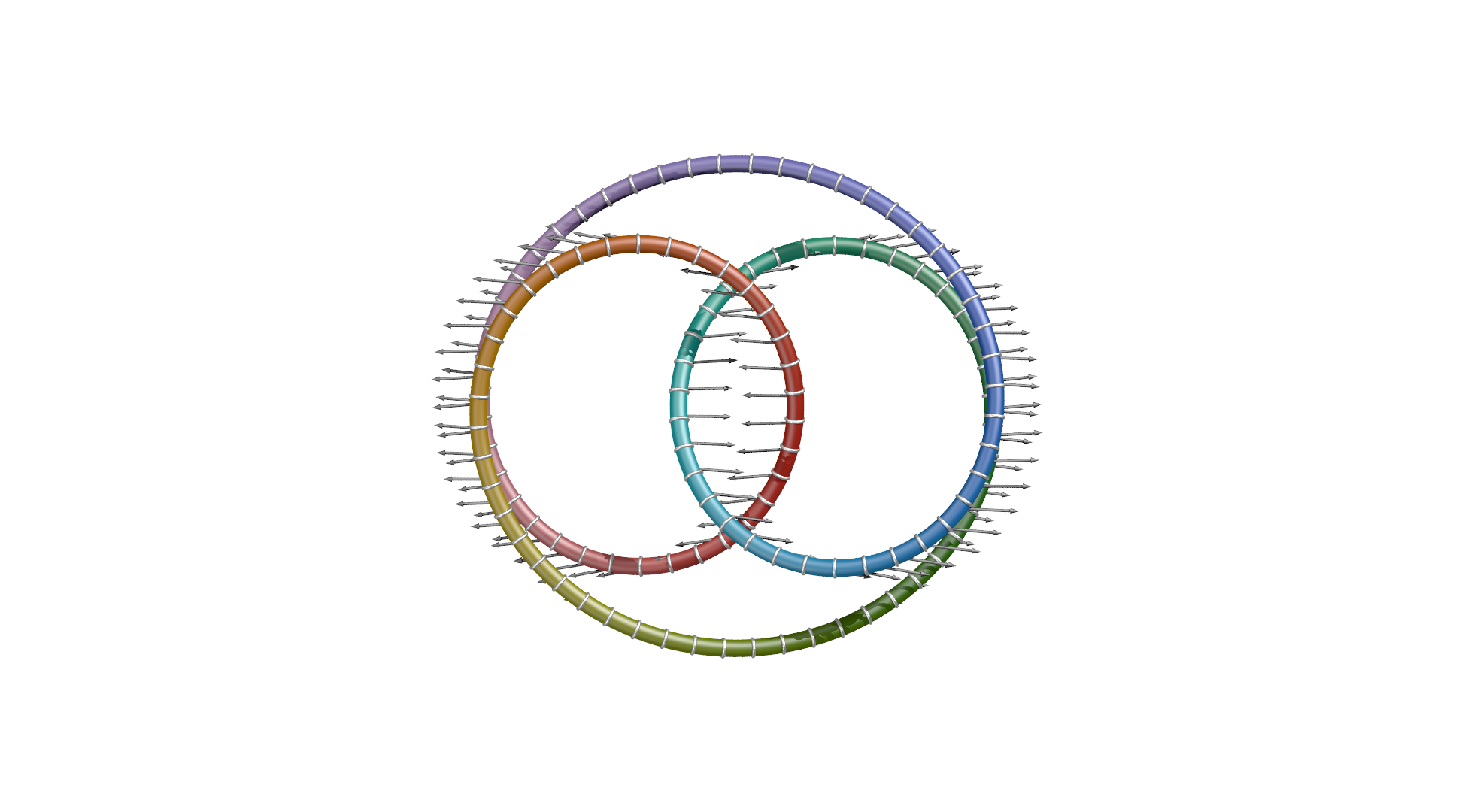}%
    \incld{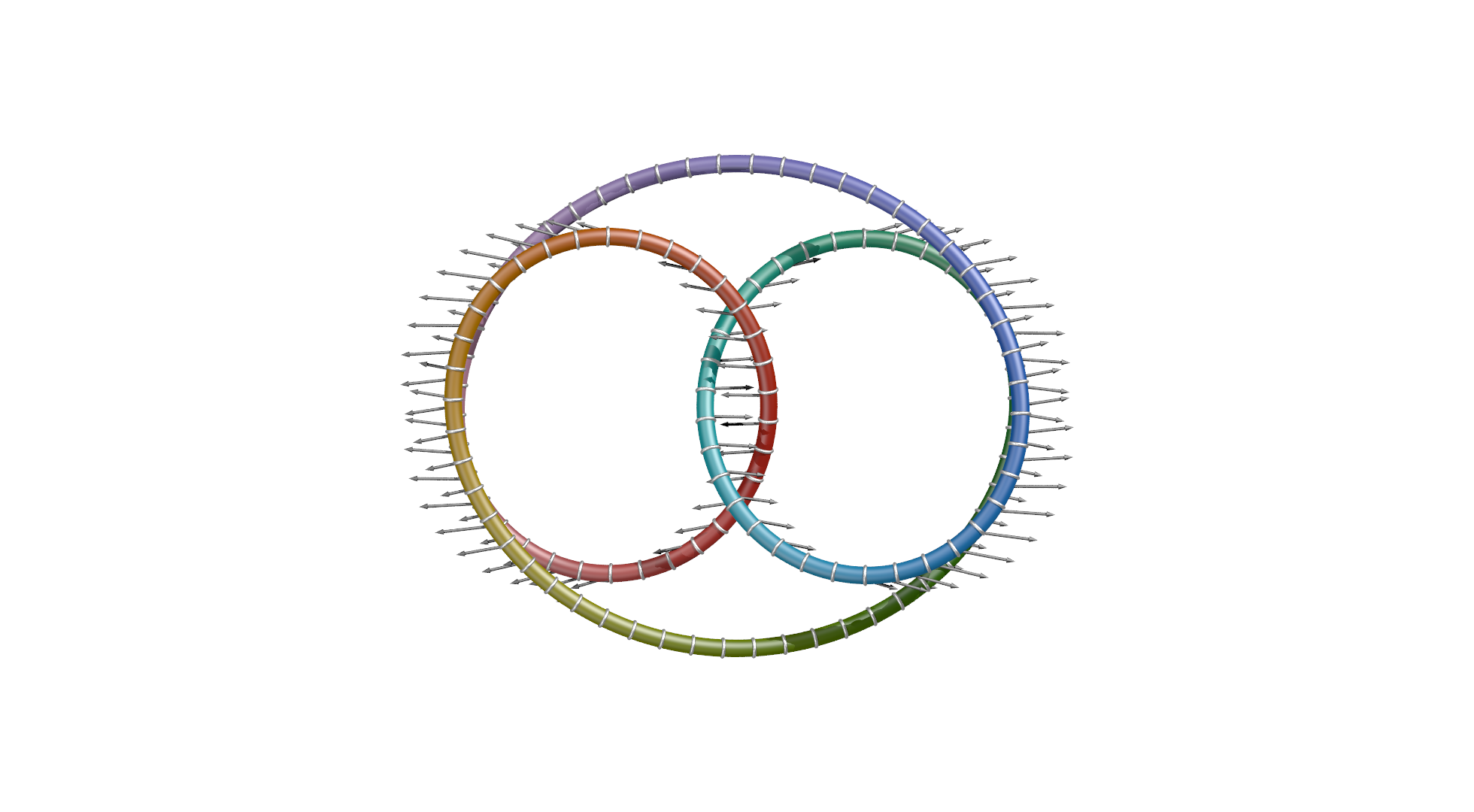}%
    \incle{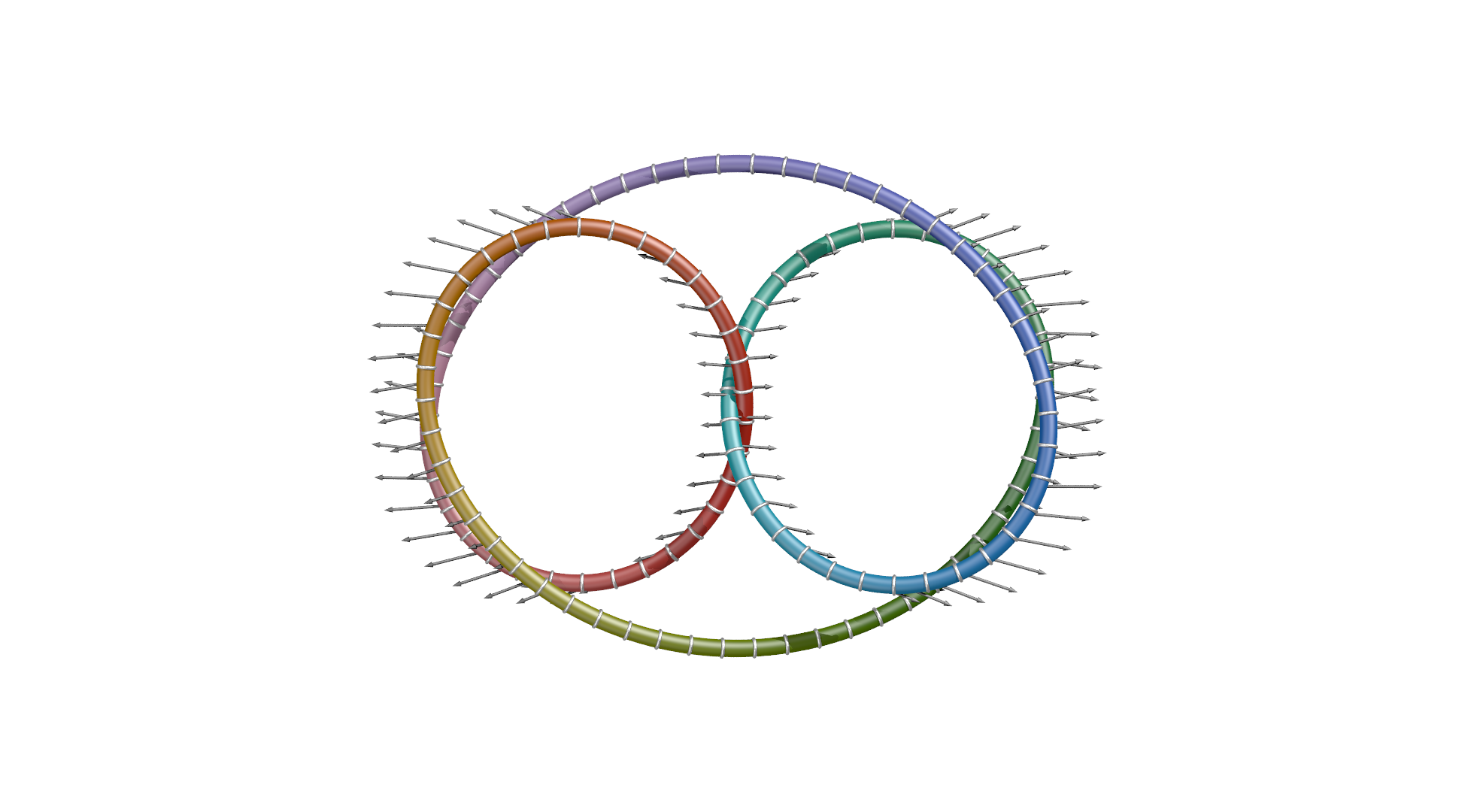}%
    \incl{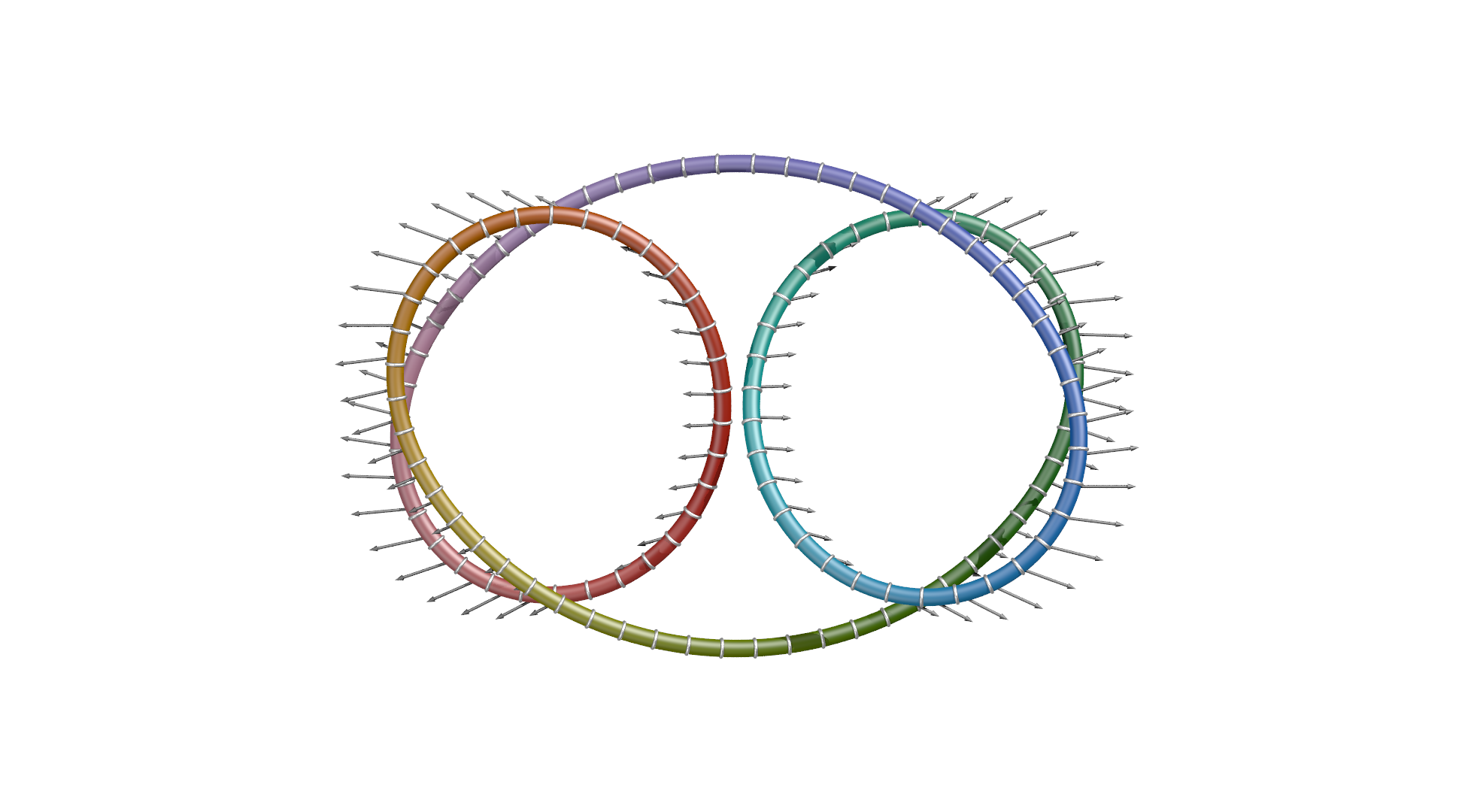}%
    \\
    \incla{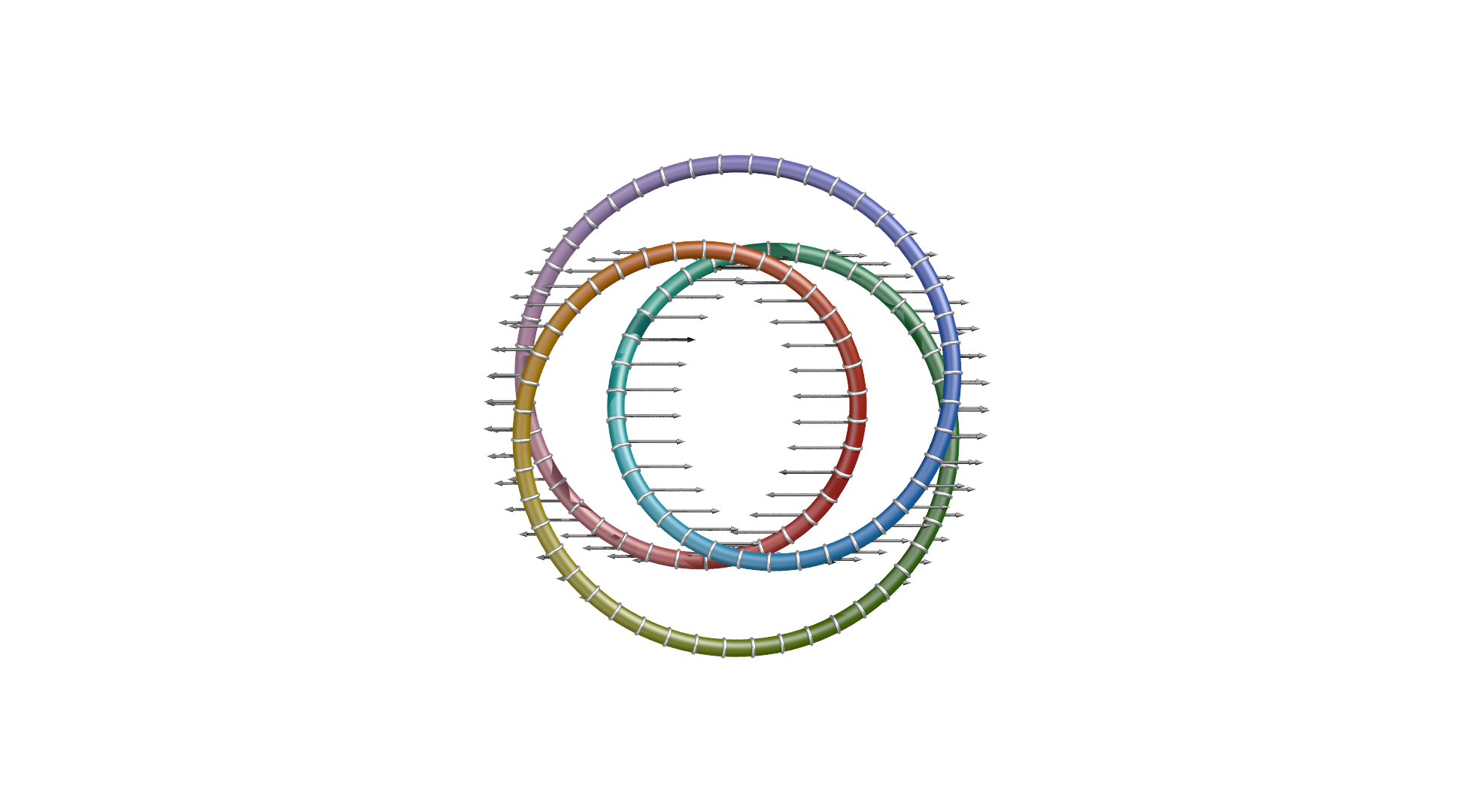}%
    \inclb{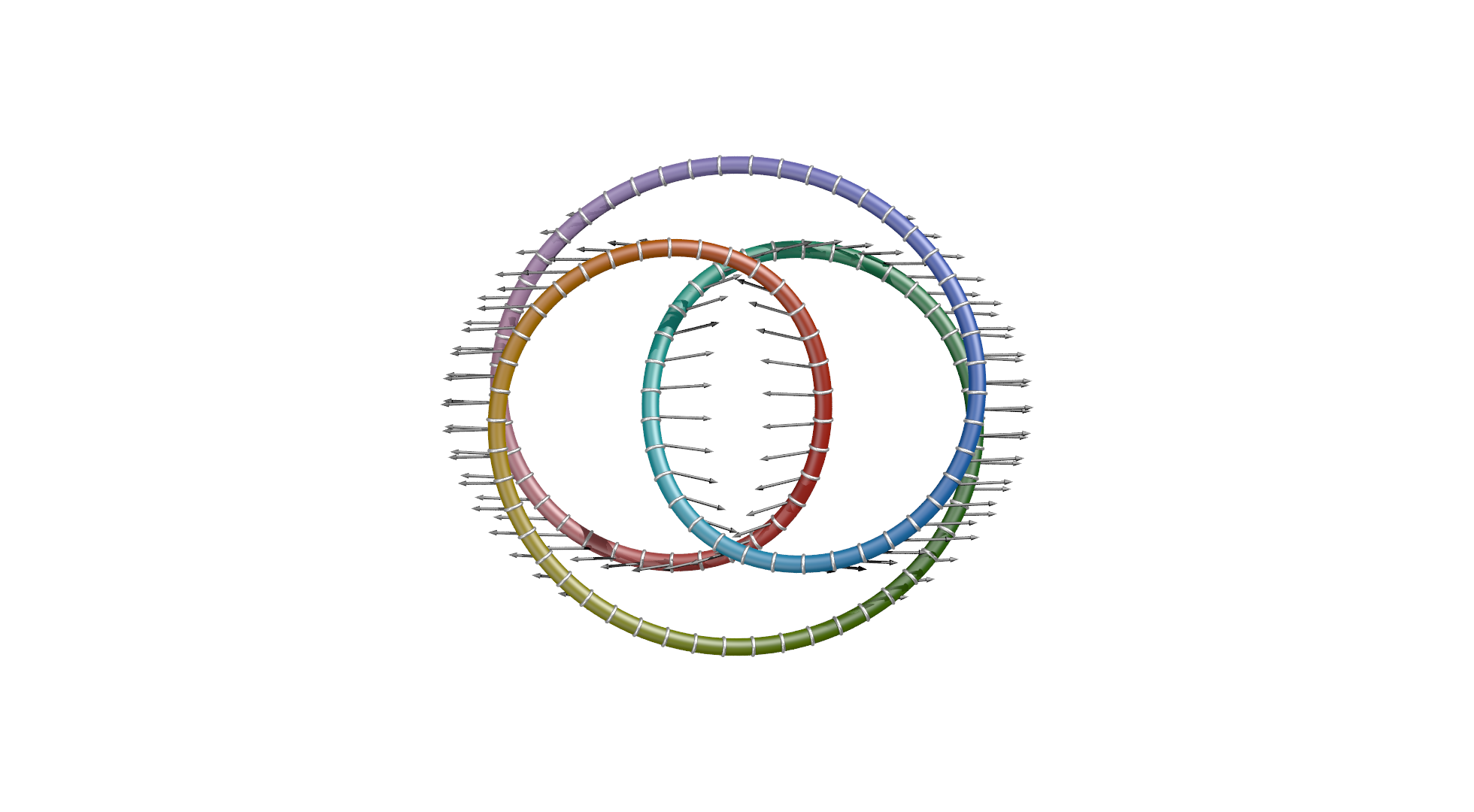}%
    \inclc{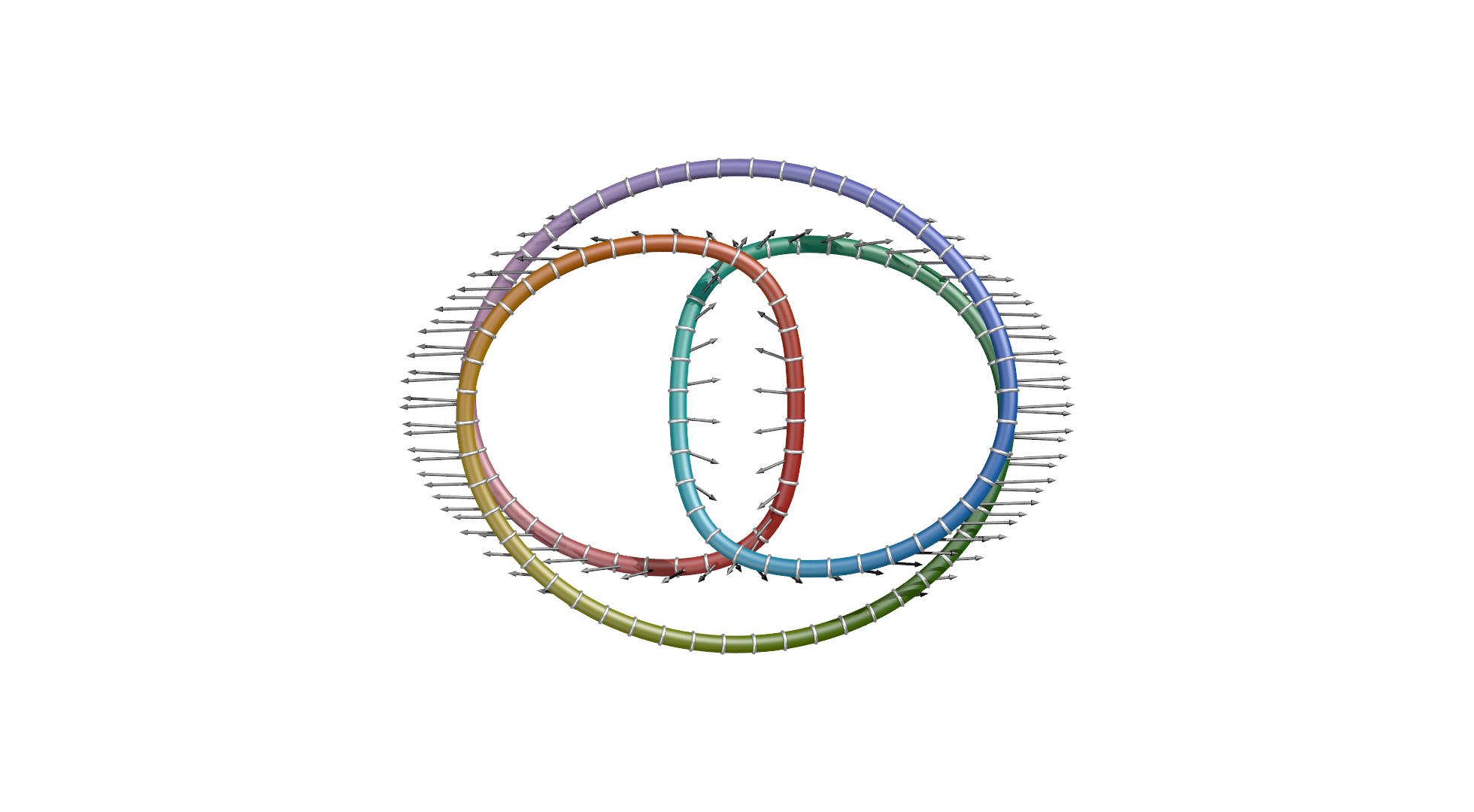}%
    \incld{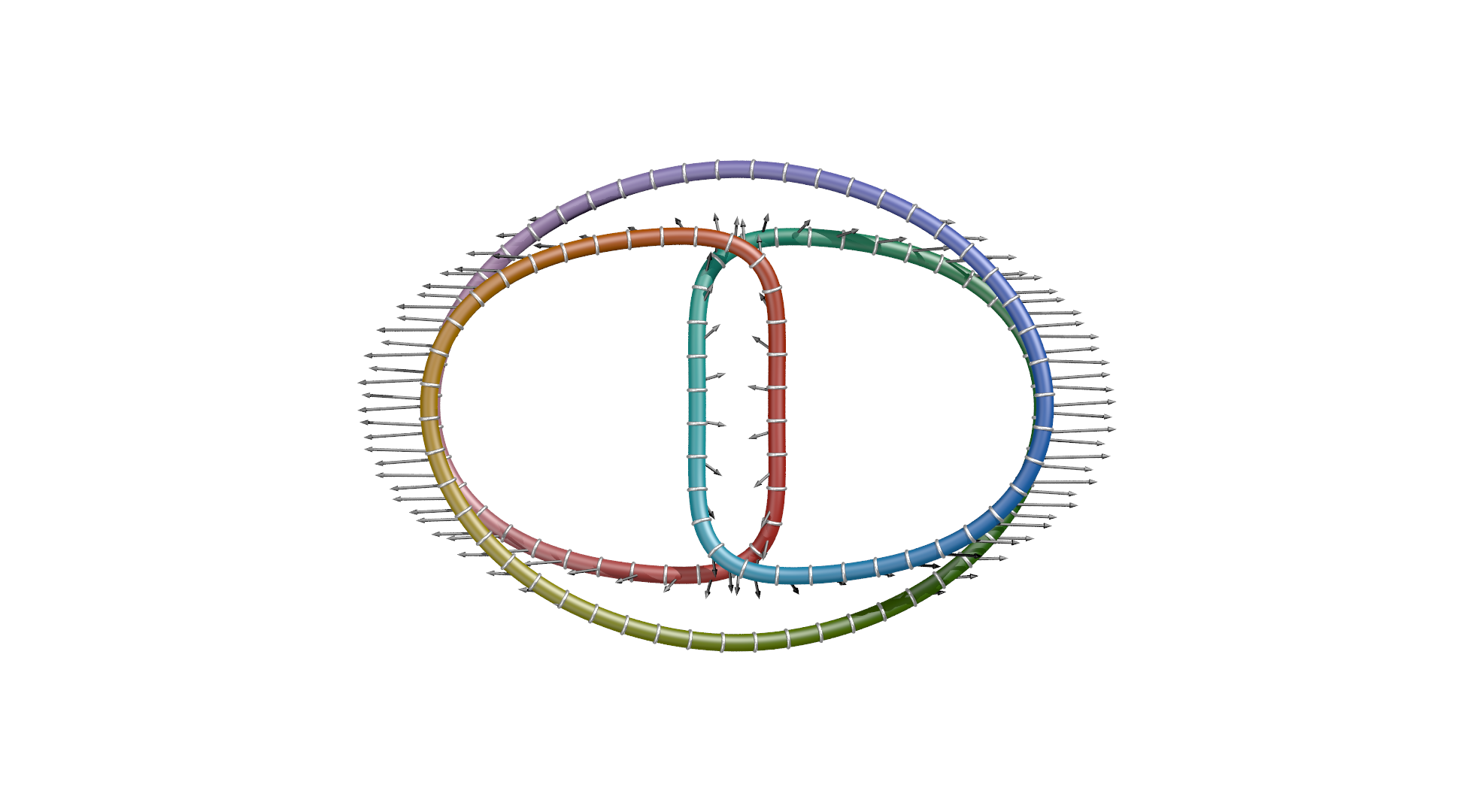}%
    \incle{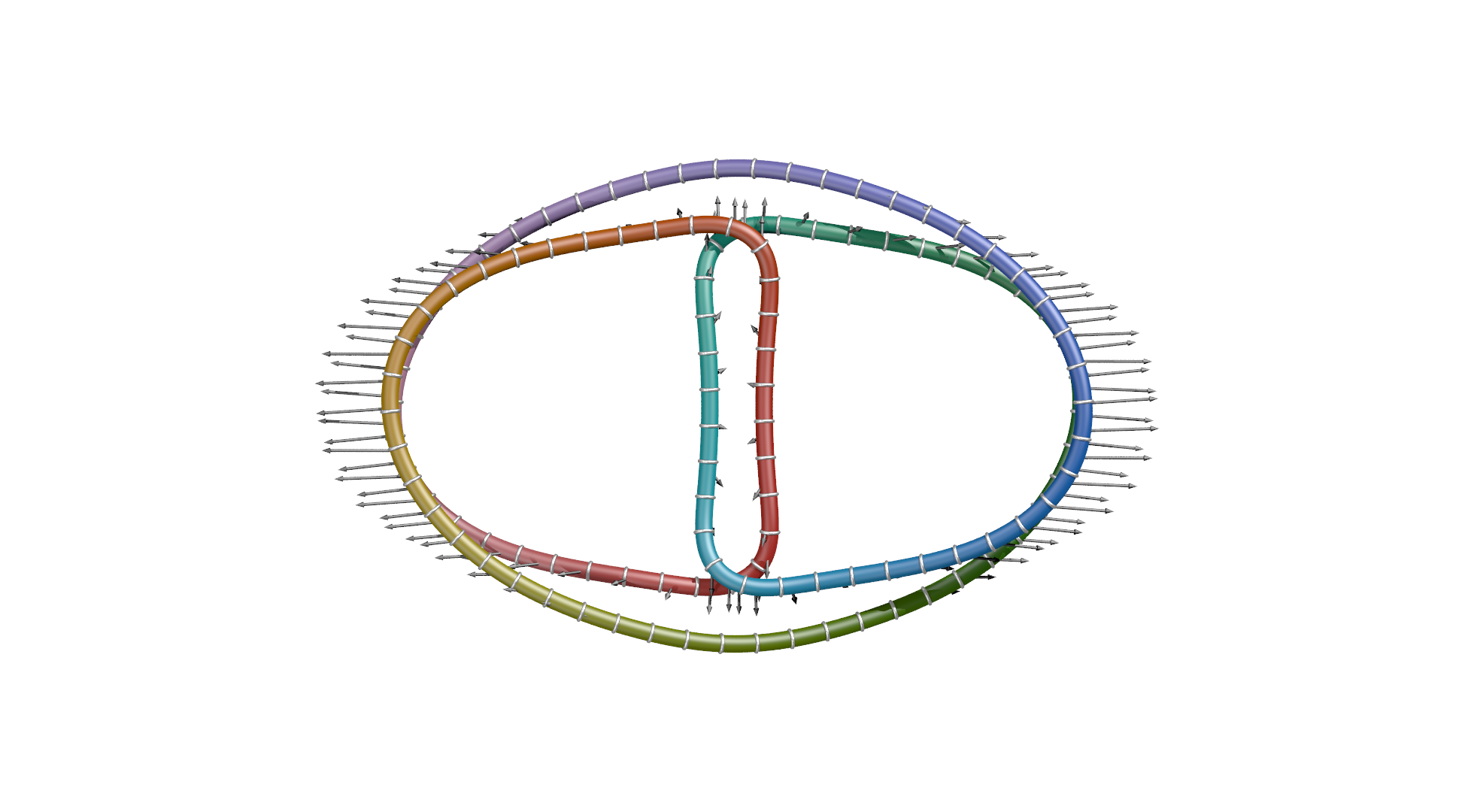}%
    \incl{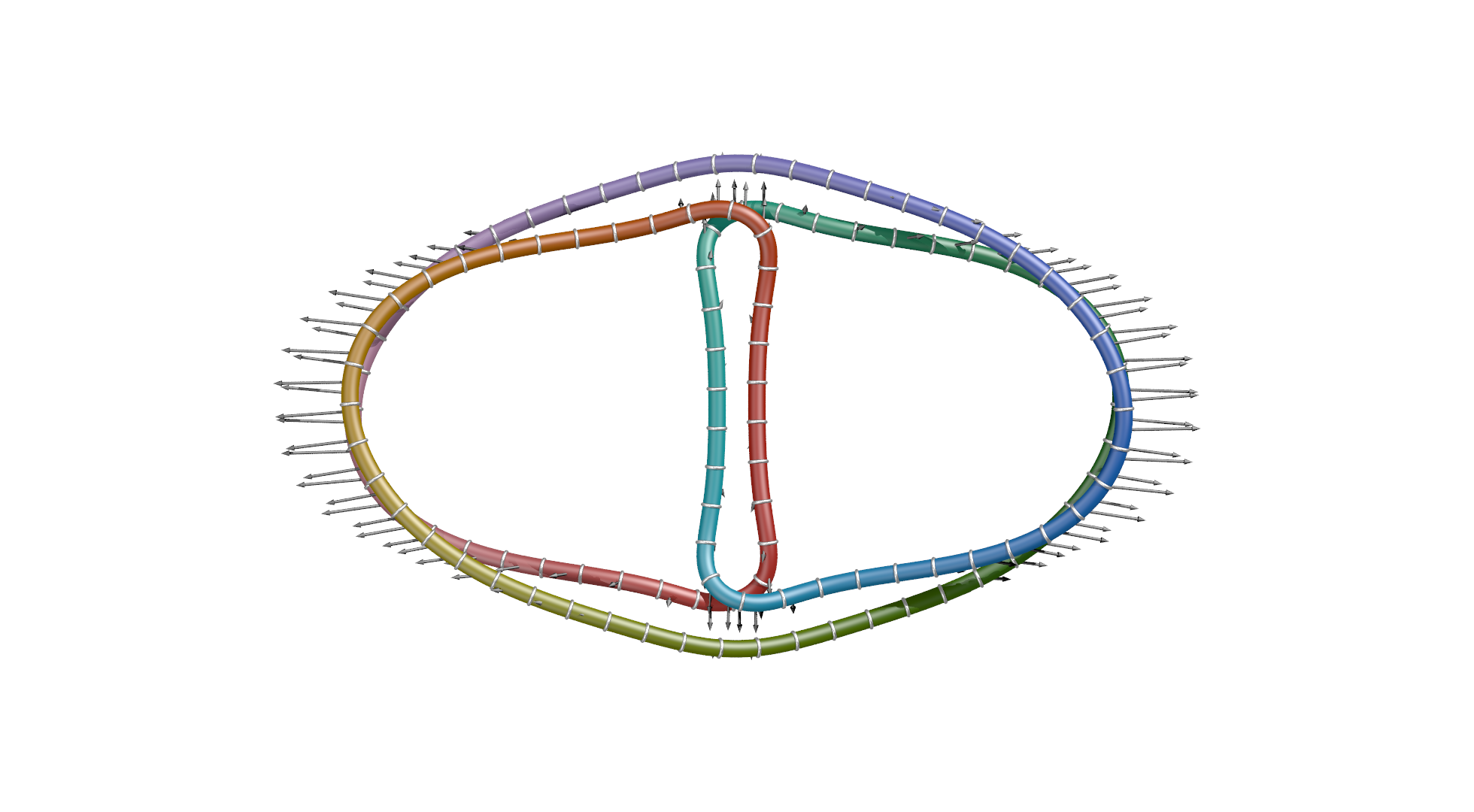}%
    \caption{Numerical simulations of the initial value problem for geodesics.    
    \emph{Top row:} A few snapshots of an $\Bessel[1]$-geodesic starting at a $(3,2)$-torus knot. The current velocity fields along the curves are indicated by arrows.
    After finite time, a collision arises and the knot type changes to the unknot's. 
    \emph{Bottom row:} Snapshots of the $G$-geodesic for $s = 7/4$ with the same initial state and initial velocities. Note that the velocities on the strands in the center are reduced more and more as the strands approach each other. Consequently, the movement in this region stops, and a collision is prevented.}
    \label{fig:GeodesicShooting}
\end{figure}

\newpage

The results on metrics on the space of immersions cited above reveal the crucial role of regularity.
In fact, counter-intuitive phenomena may occur if the metric employed is too weak, see \cite{MR2148075} for $\Lebesgue[2]$-metrics.
Sobolev metrics of higher order have been found suitable for establishing (metric and geodesic) completeness~\cite{zbMATH06502944,MR3669789,MR3264258}.
Recently, also fractional Sobolev $\Bessel[s]$-metrics have been studied~\cite{MR3745697,bauerharmsmichor2,zbMATH07845668}. 
These are defined in terms of Fourier series and conjugation of Fourier multipliers with reparametrization operators.
Effectively, this corresponds to fractional powers $(-\Delta_\gamma)^s$ of the geometric Laplace--Beltrami operator~$-\Delta_\gamma$, $\gamma\in\Imm(\Sphere^1,\AmbSpace)$.
A full picture of the current knowledge about $\Diff(\Sphere^1)$ and $\Imm(\Sphere^1,\AmbSpace)$ is provided in Bauer et al.~\cite[Table~1]{zbMATH07845668}.
Here the threshold $s=3/2$ plays a special role; it can be traced back to a heuristic argument for diffeomorphism groups by Ebin and Marsden~\cite{MR271984}.
In this light, it is no surprise that we will also require regularity $\Bessel[s]$ with $s > 3/2$.

We circumvent using fractional powers of $-\Delta_\gamma$ by employing integro-differential or Gagliardo-like inner products for several reasons:
Firstly, our metric is inspired by the tangent-point energies whose energy spaces are much closer related to Gagliardo norms ``with variable coefficients'' than to Fourier multipliers;
details can be found below.
Secondly, we prefer a situation that does not involve reparametrization operators (while still being reparametrization invariant).
This results in a conceptionally less involved and more flexible setting at the expense of being analytically more intricate.
Thirdly, avoiding fractional powers of $-\Delta_\gamma$ is crucial for numerical applications because forming fractional powers of a positive-definite matrix is very expensive.
Our inner products of Gagliardo type can be seen as compact perturbations of the operator $(-\Delta_\gamma)^s$. 
They are non-local and thus their discretized Gram matrices are dense, even though the finite-element discretization of the Laplacian $-\Delta_\gamma$ is sparse.
Fortunately, the \emph{action} of these matrices can be approximated efficiently with so-called \emph{tree codes} like the Barnes--Hut method or the Fast Multipole Method; and the resulting linear equations can be solved quickly by iterative methods and sophisticated preconditioning strategies, see~\cite{10.1145/3478513.3480521,2006.07859}. 
So our results here are not only of purely theoretical relevance;
they can also be implemented and  simulated, see 
\cref{fig:FigureEight,,fig:Trefoil,,fig:PerkoPair,,fig:GeodesicShooting,,fig:Borromean}.

\subsection{Tangent-point energies}
\label{sec:TangentPointEnergies}

\begin{figure}[t]
    \centering
    \includegraphics[width=\textwidth]{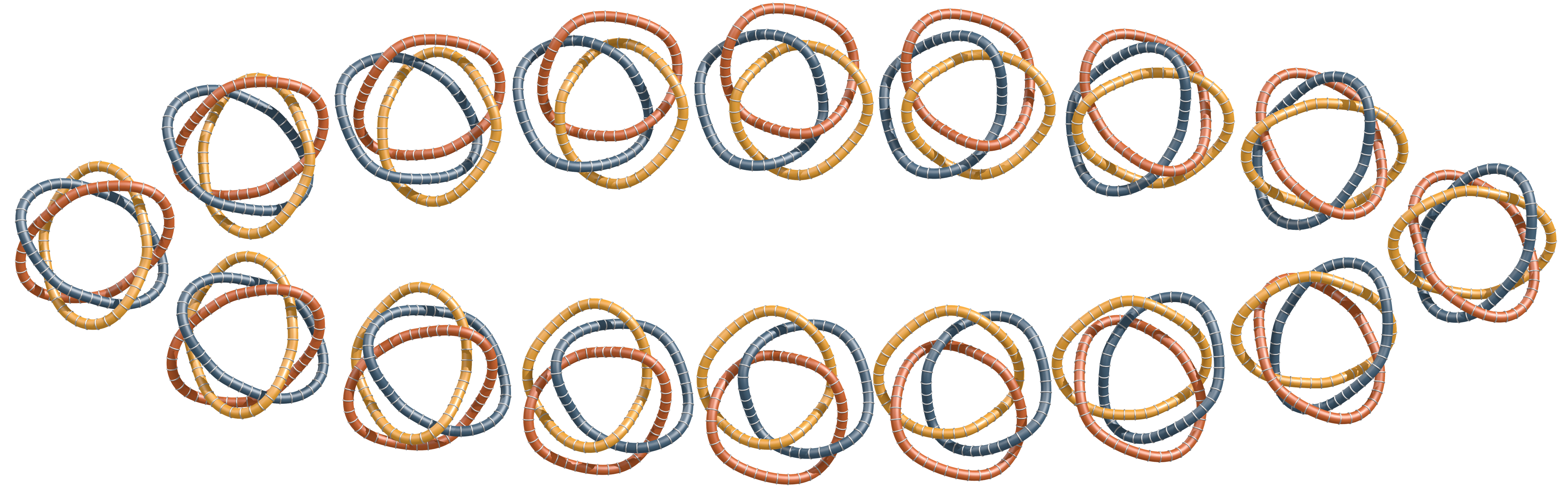}
    \caption{A nontrivial geodesic loop in the space of embedded Borromean rings with respect to the metric $G$ for $s=7/4$. 
    This illustrates that the fundamental group of this embedding space is nontrivial. 
    We thank Jason Cantarella for bringing this fact to our attention. For more details on the homotopy type of this space see \cite{zbMATH07402702}, in particular, Corollary~5.4 therein (the Borromean link is hyperbolic).
    The Borromean rings form a multiple-component link. So, strictly speaking, this is not within the realm of the present paper.
    However, with some ample modifications, our theory can be extended to cover this more general case as well.}
    \label{fig:Borromean}
\end{figure}

\begin{figure}[t]%
    \centering%
    \newcommand{\incl}[2]{\begin{tikzpicture}%
        \node[inner sep=0pt] (fig) at (0,0) {\includegraphics[	trim = 0 0 0 0, 
        clip = true,  
        angle = 0,
        width = 0.2\textwidth]{#1}};
        \node[above right= -.4ex] at (fig.south west) {\begin{footnotesize}(#2)\end{footnotesize}};%
    \end{tikzpicture}}
    \vphantom{.}\hfill
    \incl{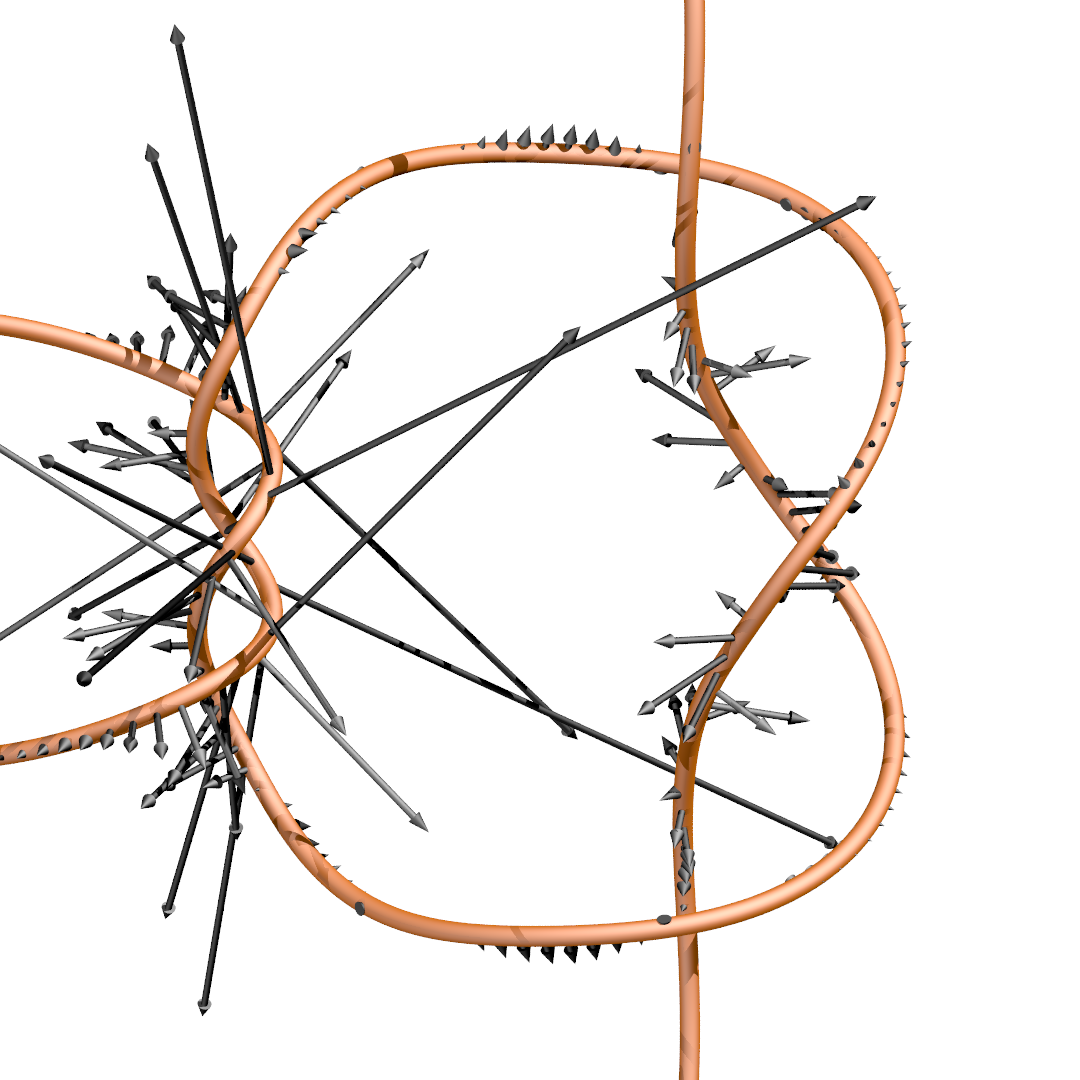}{a} \hfill
    \incl{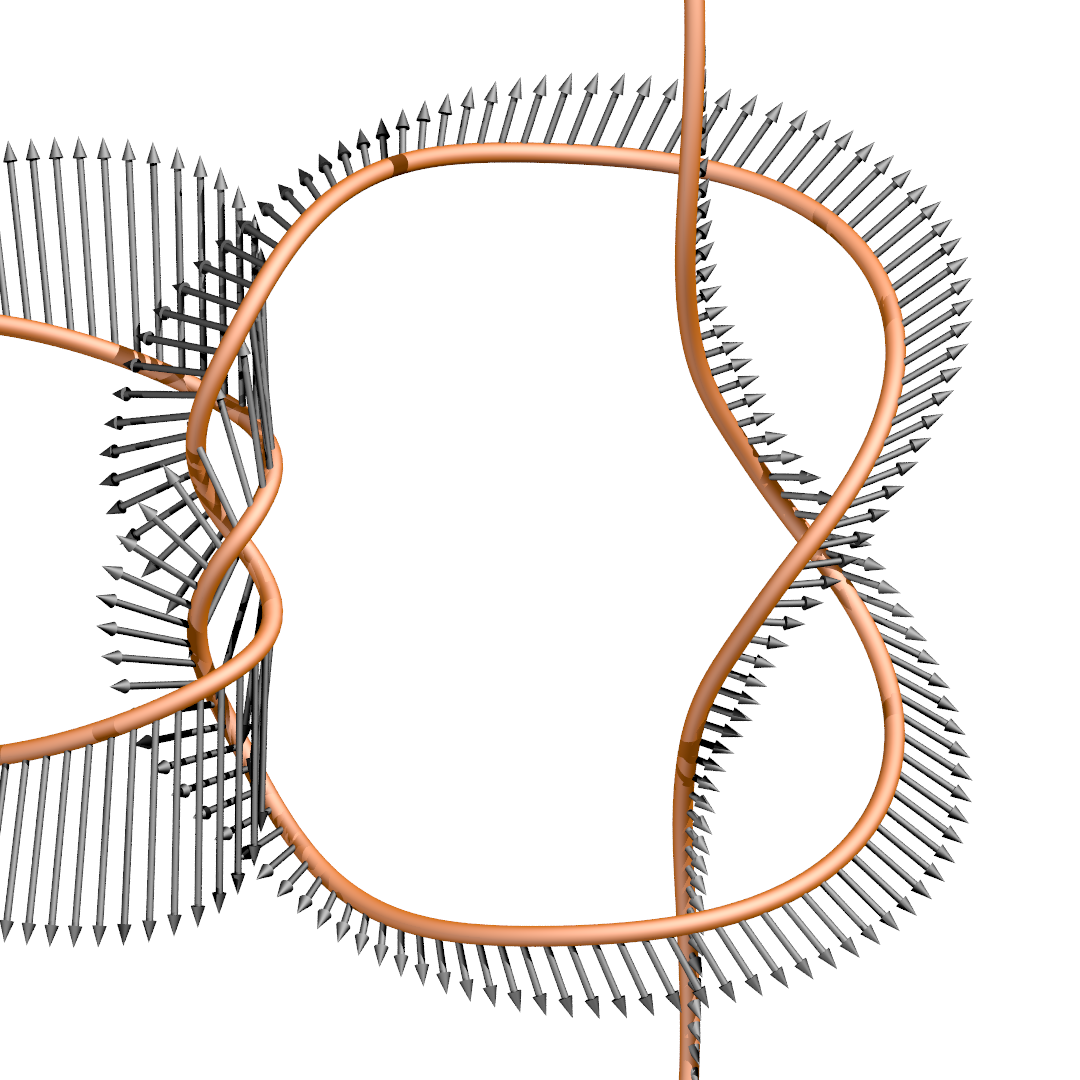}{b} \hfill
    \incl{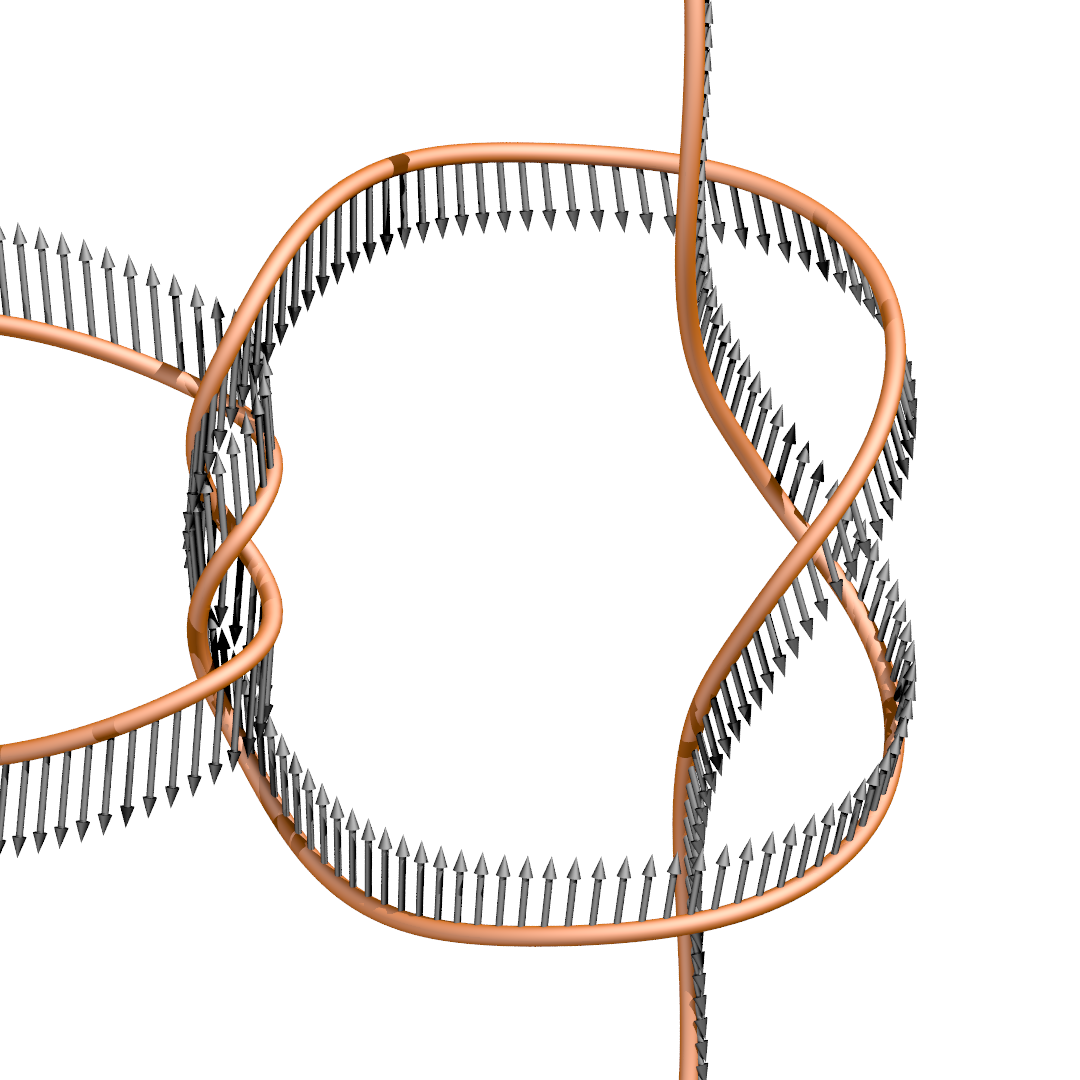}{c} \hfill
    \incl{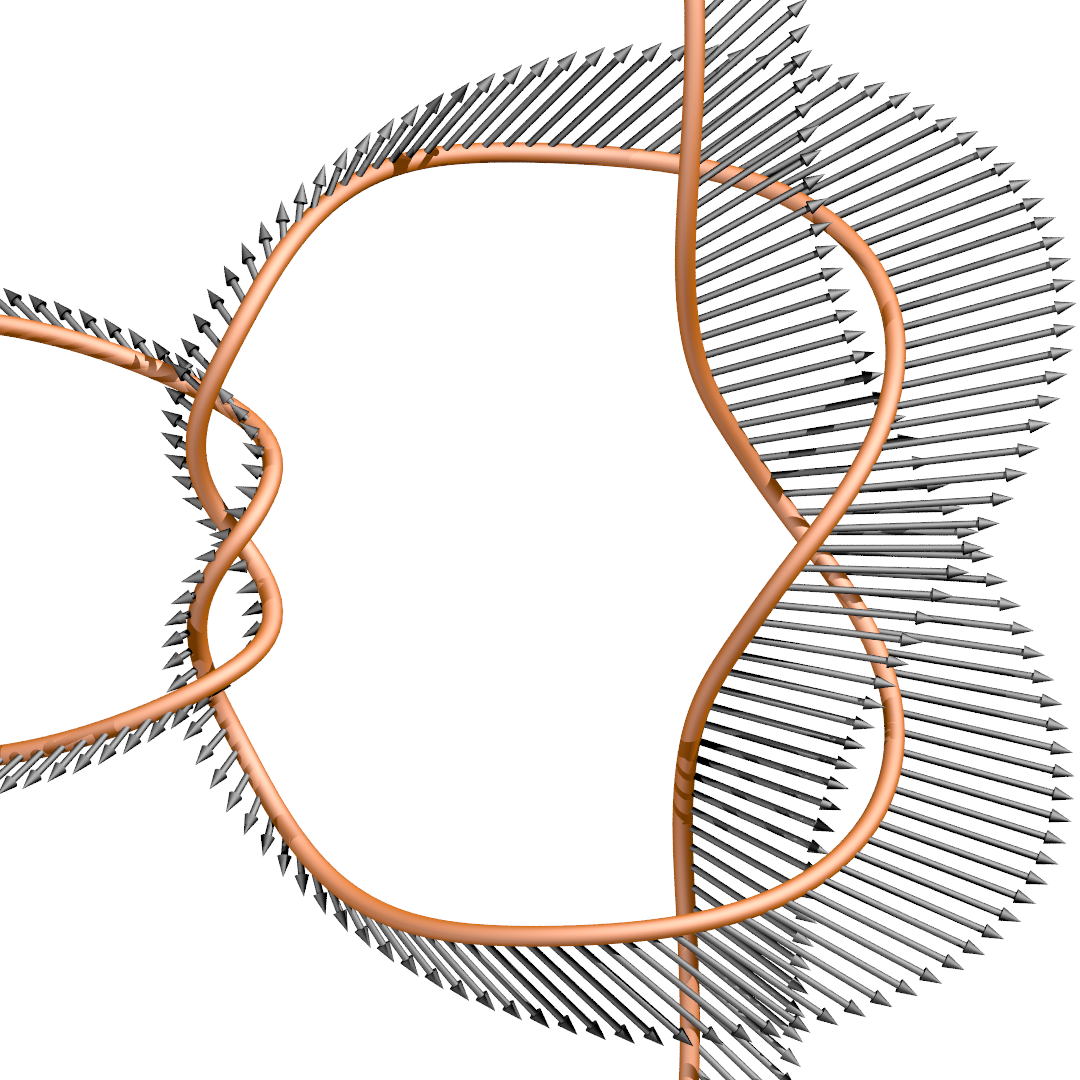}{d} \hfill \vphantom{.}
    \caption{Various downward gradients of the tangent-point energy $\Energy$ for $s = 7/4$ plotted as arrows along the same four-tangle. 
    (a)~The $\Lebesgue[2]$-gradient is very sensitive to local changes, which makes explicit gradient descent methods unstable unless tiny step sizes are used. 
    It has particularly high magnitude in tight or highly curved spots.
    (b)~The $\Bessel[1]$-gradient and (c)~the $\Bessel[s]$-gradient;
    these lead to sightly more stable gradient steps since they smooth out the $\Lebesgue[2]$-gradient. Nonetheless, the magnitude can still be large in tight spots such as the left-hand twisting.
    (d)~The $G$-gradient adapts to local features, as the terms $B^{2}$ and $B^{3}$ (see \cref{eq:B2} and \cref{eq:B3}) add some penalty for variations that have unfavorable behavior where the tangent-point energy density is high.}
    \label{fig:Gradients}
\end{figure}

Our metric is based on the concept of self-avoidance energies, more precisely on the family of tangent-point energies.
In fact, they provide some  ``measure of embeddedness''.
From the perspective of applications, the main feature of the those energies is that they can be used to model impenetrability, see~\cite{MR4182084,MR4273107,MR3800032,10.1145/3658174,10.1145/3478513.3480521,2006.07859}.

The classical tangent-point energies have been defined by Gonzalez and Maddocks \cite{GonzalezMaddocks} in an attempt to construct functionals on the space of knots that blow up when a knot approaches the boundary of its path component;
a special case already appears in an earlier paper by Buck and Orloff~\cite{MR1317077}.
For every $q \geq 1$ one defines the tangent-point energy of a parametrized closed curve $\gamma \colon \Circle \to \AmbSpace$ by
\begin{equation}
    \TP^{(q)}(\gamma)
    \ceq
    2^{-q}
    \int_{\Circle} \int_{\Circle}
    \pars[\Big]{\sdfrac{1}{R_{\TP}(\gamma)(x,y)}}^q
    \abs{\gamma'(x)} \dd x
    \abs{\gamma'(y)} \dd y
    .
    \label{eq:DfnTPp}
\end{equation}
Here, $R_{\TP}(\gamma)(x,y)$ is the so-called \emph{tangent-point radius}. 
It is the radius of the smallest sphere in $\AmbSpace$ through $\gamma(x)$ and $\gamma(y)$ that is also tangent to $\gamma'(x)$.
Fortunately, there is a closed-form expression for this radius if $\gamma(x) \neq \gamma(y)$:
\begin{equation}
    R_{\TP}(\gamma)(x,y)
    \ceq
    \sdfrac{\abs{\gamma(y) - \gamma(x)}^2}{2 \abs{\NProj_\gamma(x) \pars{\gamma(y) - \gamma(x)}}}
    .
    \label{eq:DfnTPRadius}
\end{equation}
Here $\NProj_\gamma(x)$ denotes the orthogonal projector onto $\gamma'(x)^\perp$.
If $\gamma$ is twice continuously differentiable, then for $y \to x$ the radius $R_{\TP}(\gamma)(x,y)$ converges to the radius of curvature of $\gamma$ at $x$.
For this reason, $1/R_{\TP}(\gamma)$ is often called a nonlocal generalization of curvature, and $\TP^{(q)}(\gamma)$ is often referred to as a nonlocal curvature energy.

A~rigorous analysis of regularization and self-avoidance effects of $\TP^{(q)}$ has been carried out by Strzelecki and von der Mosel~\cite{strzeleckivdm}.
Their most important result:
Let $q > 2$, and suppose that $\gamma \colon \Circle \to \AmbSpace$ is a parametrization by arc length of a rectifiable curve with finite length and with finite energy $\TP^{(q)}(\gamma) < \infty$. 
Then $\gamma \in \Holder[1,\alpha][][\Circle][\AmbSpace]$
and 
$\gamma(\Circle) \subset \AmbSpace$ is a
one-dimensional submanifold of class ~$\Holder[1,\alpha]$, where $\alpha = 1- 2/q$.
Notice that the only reason $\gamma$ might not be injective is that it may cover its image multiple times.
Neglecting this detail for a moment, we can informally say that bounded tangent-point energy prevents formation of self-intersections. (We will make this more precise later.)
In contrast, a curve with self-intersections may have finite energy $\TP^{(q)}$ for $q < 2$; this can be checked directly on two compact line segments that meet perpendicularly in a cross pattern. So $\TP^{(q)}$ cannot be self-avoidant for $q < 2$.
The case $q=2$ is special, as $\TP^{(2)}$ is invariant with respect to rescaling. This poses a number of analytical difficulties~\cite{blatt-reiter-schikorra-vorderobermeier}.
In particular, pull-tights with finite energy $\TP^{(2)}$ may occur, see \cref{sec:PullTight}.

In~\cite{zbMATH06214305}, Blatt improved the regularity statement by showing:
If $\gamma \in \Holder[1][][\Circle][\AmbSpace]$ is an embedding and parametrized by arc length and if $q > 2$, then $\TP^{(q)}(\gamma)$  is finite if and only if $\gamma$ is a member of the \SoboSlobo space $\Sobo[s,q][][\Domain][\AmbSpace]$, where $s = 2 - 1/q$.
In this sense, this fractional Sobolev space is the natural domain of $\TP^{(q)}$.
By construction, $\TP^{(q)}(\gamma)$ is invariant under reparametrization; so we may reparametrize $\gamma$ quite wildly and obtain another curve $\gamma \circ \varphi \not \in \Sobo[s,q][][\Domain][\AmbSpace]$ but with $\TP^{(q)}(\gamma \circ \varphi) = \TP^{(q)}(\gamma) < \infty$. So the classical concept of ``domain'' is not really appropriate here.
Generalizations of $\TP^{(q)}$ for higher-dimensional submanifolds have been studied in~\cite{zbMATH06214305,kolasinskistrzeleckivdm,strzeleckivdm2}.

Because of the requirement $q > 2$, it becomes apparent now that the classical tangent-point energies prevent us from working in Hilbert spaces and thus from constructing strong Riemannian metrics.
With this very problem in mind, Blatt and the second author introduced the family of \emph{generalized} tangent-point energies~\cite{blattreiter1}
by decoupling powers in the integrand of the tangent-point functional as follows:
\begin{equation}
    \TP^{(p,q)}(\gamma)
    \ceq
    \int_{\Circle} \int_{\Circle}
    \sdfrac{\abs{\NProj_\gamma(x) \pars{\gamma(y) - \gamma(x)}}^q}{\abs{\gamma(y) - \gamma(x)}^p}
    \abs{\gamma'(x)} \dd x
    \abs{\gamma'(y)} \dd y
    .
    \label{eq:DfnTPpq}
\end{equation}
Notice that $\TP^{(q)} = \TP^{(2q,q)}$ is a ``classical'' tangent-point energy.
For $q + 2 < p < 2\,q + 1$, the energy $\TP^{(p,q)}$ is still ``repulsive'', 
and the appropriate domain is $\Sobo[s,q][][\Domain][\AmbSpace] \subset \Holder[1,\alpha][][\Domain][\AmbSpace]$, where $s \ceq (p - 1)/q$, and $\alpha \ceq (p - q - 2)/q$.
If the condition $p < 2\,q + 1$ is violated, then the integrand \cref{eq:DfnTPpq} is too singular: all smooth curves that are not contained in a straight line have infinite energy, (see \cite[Proposition~2.6]{blattreiter1}). 
The significance of the condition $q + 2 < p$ is two-fold: (i) In the case $q + 2 > p$ a curve with finite energy may indeed have self-intersections (which can be checked again by investigating a pair of line segments forming a cross). (ii) If $q + 2 \geq p$, then the energy $\TP^{(p,q)}$  may fail to prevent pull-tight, see \cref{sec:PullTight}.

It has been shown in \cite{blattreiter1} that the energies $\TP^{(p,q)}$ for $q + 2 < p < 2\,q + 1$ are continuously differentiable functionals.
As a byproduct of our analysis, we will strengthen this result and show that the tangent-point energies in the Hilbert case ($q =2$) are smooth, see \cref{thm:DE}. 
The scale-invariant case ($q + 2 = p$) has been discussed in~\cite{blatt-reiter-schikorra-vorderobermeier}.
For a broad overview on the energies $\TP^{(p,q)}$ and their relationship to other nonlocal curvature energies, we refer to the survey~\cite{zbMATH07063699}.

\medskip

From now on, we will focus on those elements of the generalized tangent-point energy family
whose domains turn out to be Hilbert spaces while maintaining the self-avoidance property.
These are just the elements of the one-parameter family $\TP^{(p,2)}$, $4 < p < 5$.
The corresponding domains are $\HSpace = \Sobo[s,2][][\Circle][\AmbSpace]$, $s = (p-1)/2$, so that $3/2 < s < 2$.
For $q = 2$ the regularity theory for critical points of~$\TP^{(p,2)}$ subject to a length constraint is completely understood:
after reparametrization to arc length such a critical point is always smooth (see~\cite{blattreiter1}) and real analytic (see~\cite{zbMATH07531070}).
For $q \neq 2$, it is still open whether critical points enjoy any higher regularity.
The Hilbert case $q = 2$ is so special because 
the leading order term of the nonlinear operator $\gamma \mapsto D\TP^{(p,2)}(\gamma)$ happens to be a \emph{semilinear} elliptic operator of order $2s$. Thus, elliptic bootstrapping can be applied to achieve higher regularity of critical points.
This elliptic operator also happens to be symmetric; so up to a compact perturbation it gives naturally rise to a (strong) inner product on $\HSpace$. This is what we exploit in the construction of our metric.

\subsection{Definition of the metric}
\label{sec:DefinitionMetric}

We are going to construct our Riemannian metric on the smooth manifold
\begin{align}\label{eq:defManifold}
    \Mfld \ceq \braces[\Big]{
        \gamma \in \Bessel[s][][\Domain][\AmbSpace]
        \, \Big| \,
        \text{$\gamma$ is an embedding}
    }.
\end{align}

Because $\HSpace$ embeds continuously into $\Holder[1][][\Circle][\AmbSpace]$ (see \cref{thm:MorreyInequality}) and because the set of embeddings is open in $\Holder[1]$, (see, e.g.,~\cite[Thm.~2.1.4]{MR448362}), the set $\Mfld$ is an open subset of~$\HSpace$. Thus, $\Mfld$ is indeed a smooth manifold modelled on the Hilbert space~$\HSpace$.

In order to construct our Riemannian metric, we first need to introduce a couple of operators.
We start with the following $\gamma$-dependent differential operator $\cD_\gamma$ that maps a $C^1$-mapping $u \colon \Domain \to \AmbSpace$ to
\begin{align}
    \cD_\gamma  u (x)
    \ceq 
    \dd u(x) \dd \gamma(x)^\dagger 
    = 
    \sdfrac{u'(x)}{\abs{\gamma'(x)}} \sdfrac{\gamma'(x)\trans}{\abs{\gamma'(x)}}
    \in 
    \BoundedLinOps(\AmbSpace;\AmbSpace)
    \label{eq:OpD}
    .
\end{align}
Here $\dd u(x) \colon T_x \Circle \to \AmbSpace$ denotes the differential at point $x$,
the symbol $\dagger$ denotes the Moore--Penrose pseudoinverse,
and $\BoundedLinOps(X;Y)$ denotes the space of continuous linear operators between two topological vector spaces $X$ and $Y$.
This allows us to define the $\gamma$-dependent operator $\smash{\Op{\gamma}{s}}$ that will be pivotal for the whole exposition in the paper.
It maps $u$ to the measurable mapping $\Op{\gamma}{s} u \colon \Domain^2 \to \AmbSpace$ defined by
\begin{align}
    \Op{\gamma}{s} u(x,y)
    \ceq 
    \sdfrac{
        u(y) - u(x) - \cD_\gamma u (x) \pars{\gamma(y) - \gamma(x)} 
    }{
        \abs{\gamma(y) - \gamma(x)}^s
    }
    \in 
    \AmbSpace
    \label{eq:Op}
    .
\end{align}
The normal projector can be written as $\NProj_\gamma(x) = \id_{\AmbSpace} - \cD_\gamma \gamma(x)$. For $u = \gamma$, we get
\[
    \Op{\gamma}{s} \gamma(x,y)
    =
    \sdfrac{\NProj_\gamma(x) \pars{\gamma(y) - \gamma(x)}}{\abs{\gamma(y) - \gamma(x)}^s}
    .
\]
It might help the geometric intuition to observe that $\Op{\gamma}{2} \gamma(x,y)$ is the curvature vector of the tangent-point sphere at $x$.
Moreover, if $\gamma$ is of class $\Holder[2]$, then $\lim_{y \to x} \Op{\gamma}{2} \gamma(x,y) = \kappa_\gamma(x)$, the curvature vector of $\gamma$ at $x$.

We denote the line element by
\begin{equation}
    \dvol_\gamma(x) \ceq \abs{\gamma'(x)} \dd x
    \label{eq:vol}
\end{equation}
and define the $\gamma$-dependent measure $\mu_\gamma$ by
\begin{align}
    \dd \mu_\gamma(x,y) 
    \ceq 
    \sdfrac{\dvol_\gamma(y) \dvol_\gamma(x)}{\abs{\gamma(y) - \gamma(x)}}
    .
    \label{eq:muGeometric}
\end{align}
Since $p = 2\,s + 1$, this allows us to rewrite the tangent-point energy as follows:
\begin{equation}
    \label{eq:E}
    \Energy(\gamma) 
    \ceq
    \TP^{(p,2)}(\gamma)
    =
    \int_{\Circle} \int_{\Circle}
    \sdfrac{\abs{\NProj_\gamma(x) \pars{\gamma(y) - \gamma(x)}}^2}{\abs{\gamma(y) - \gamma(x)}^{2s+1}}
    \dvol_\gamma(x)
    \dvol_\gamma(y)
    =
    \iint_{\Circle^2} \abs{\Op{\gamma}{s} \gamma}^2 \dd \mu_\gamma  
    .
\end{equation}
At $\gamma \in \Mfld$ and for tangent vectors $u$, $v \in T_\gamma \Mfld = \Bessel[s][][\Domain][\AmbSpace]$, we define the metric $G$ in terms of five bilinear forms whose precise definitions will be detailed out next:
\begin{equation}
    G_\gamma(u,v)
    \ceq 
    2 \, B_\gamma^{1}(u,v)
    +
    (2\,s+1)
    \, B_\gamma^{2}(u,v)
    +
    B_\gamma^{3}(u,v)
    +
    \inner{u,v}_{\Bessel[1](\gamma)}
    +
    \inner{u,v}_{\Lebesgue[2](\gamma)}
    .
    \label{eq:G}
\end{equation}
The main ingredients here are the following three bilinear forms:
\begin{align}
    B_\gamma^{1}(u,v)
    &\ceq 
    \iint_{\Domain^2}
        \inner[\Big]{ \Op{\gamma}{s}u(x,y), \Op{\gamma}{s}v(x,y) } 
    \dd \mu_\gamma(x,y),
    \label{eq:B1}
    \\
    B_\gamma^{2}(u,v)
    &\ceq 
    \iint_{\Domain^2}
            \abs[\big]{ \Op{\gamma}{s}\gamma(x,y) }^2 
            \inner[\Big]{ 
                \sdfrac{u(y)-u(x)}{\abs{\gamma(y) - \gamma(x)}}, 
                \sdfrac{v(y)-v(x)}{\abs{\gamma(y) - \gamma(x)}} 
            } 
    \dd \mu_\gamma(x,y),
    \label{eq:B2}
    \\
    B_\gamma^{3}(u,v)
    &\ceq 
    \iint_{\Domain^2}
            \abs[\big]{ \Op{\gamma}{s}\gamma(x,y) }^2 
            \pars[\Big]{
                \inner{ D_\gamma u(x), D_\gamma v(x)}
                +
                \inner{ D_\gamma u(y), D_\gamma v(y)}
            }
    \dd \mu_\gamma(x,y),
    \label{eq:B3} 
\end{align}
where 
\begin{equation}
    \label{eq:DfnDiffWrtArcLength}
    D_\gamma u(x) \ceq \sdfrac{u'(x)}{\abs{\gamma'(x)}}
\end{equation}
is the derivative of $u$ \emph{with respect to arc length}.
Here and in the following, $\abs{\cdot}$ and $\inner{\cdot,\cdot}$ refer to the Euclidean norm and the Euclidean inner product for vectors---or to the Frobenius norm and the Frobenius inner product for tensors.

As we will see later in \cref{thm:DE}, all the terms $\smash{B_\gamma^{1}}$, $\smash{B_\gamma^{2}}$, and $\smash{B_\gamma^{3}}$ appear very naturally in the first variation of $\Energy$.
Among these, the term $\smash{B_\gamma^{1}}$ is the one that topologizes $\Bessel[s][][\Domain][\AmbSpace]$, so it is responsible for making this a \emph{strong} Riemannian metric (see \cref{thm:GIsStrong}).
Moreover, one can show that $\smash{B_\gamma^{1}(u,v)}$ is the weak formulation of a compact perturbation of the 
fractional Laplacian $(-\Laplacian_\gamma)^s = (-D_\gamma^{*} D_\gamma)^s$, where $*$ denotes the $L^2$-adjoint (see \cite[Proposition~4.1]{blattreiter1}).

The terms $\smash{B_\gamma^{2}}$ and $\smash{B_\gamma^{3}}$ have the responsibility to push curves with self-intersections ``towards infinity''.
Metrics very similar to $\smash{B_\gamma^{1} + B_\gamma^{2}}$ have been used as preconditioners for numerical optimization via gradient descent for the Möbius energy of curves \cite{reiterschumacher1}, for the tangent-point energy of curves \cite{2006.07859}, and for the tangent-point energy of surfaces \cite{10.1145/3478513.3480521}.
In fact, these numerical works are the main motivation for considering this particular metric $G$.
Not only do we know that $G$ is an excellent preconditioner for $\Energy$ (and compact perturbations of it); we also observed in these experiments that including the term $B_\gamma^{2}$ made collision detection in the line search almost obsolete: we almost never bounced into the boundary of $\Mfld$. This boundary is of course the set of curves with self-intersections (or the results of the other failure modes \cref{item:LoseRegularity,item:BlowUp,item:Collapse,item:PullTight}).
This observation lead us to wonder whether the employed metric (or some minor modification thereof) could be geodesically complete. The present paper is the affirmative answer to this.

Alas, the three $B$-terms are not enough. 
Note that $\smash{B_\gamma^{1}(u,u) + B_\gamma^{2}(u,u) + B_\gamma^{3}(u,u)}$ vanishes if $u$ is a constant function, hence $\smash{B_\gamma^{1} + B_\gamma^{2} + B_\gamma^{3}}$ is not positive-definite, and we have vulnerability towards failure mode \cref{item:WildTravel}. This can be coped with in various ways, but the easiest is to add also the geometric $\Lebesgue[2]$-term
\begin{equation}
    \inner{u,v}_{\Lebesgue[2](\gamma)}
    \ceq 
    \int_\Domain \inner{ u(x), v (x) } \dvol_\gamma(x)    
    .
    \label{eq:L2}
\end{equation}
This assigns a non-vanishing ``cost'' to constant vectors fields, which generate the translations in~$\AmbSpace$.
Finally, we also add the geometric $\Bessel[1]$ term
\begin{equation}
    \inner{u,v}_{\Bessel[1](\gamma)}
    \ceq 
    \int_\Domain \inner{ D_\gamma u(x), D_\gamma v (x) } \dvol_\gamma(x)
    .
    \label{eq:H1}
\end{equation}
Adding $\inner{u,v}_{\Bessel[1](\gamma)}$ to the metric grants some good control over the arc length functional (see \cref{lem:ArcLengthLipschitz}). 
Eventually, it will also help us to guard against failure modes \cref{item:BlowUp} and \cref{item:Collapse} (see \cref{lem:BiLipBound}).

In principle, Sobolev inequalities should allow for bounding $\smash{\inner{u,u}_{\Bessel[1](\gamma)}}$ in terms of $\smash{B_\gamma^{1}(u,u) + B_\gamma^{2}(u,u) + B_\gamma^{3}(u,u)}$ (and actually in terms of $\smash{B_\gamma^{1}(u,u)}$ alone). 
However, the constants in these inequalities depend on $\gamma$ in such an unfavorable way that this prevents us from using Grönwall inequalities unless we introduce some artificial control over the arc length functional (e.g., by a hard constraint).    

\subsection{Main results and outline of the paper}
\label{sec:MainResults}

We strive to make the exposition as self-contained as possible.
To this end, we require some thorough preparation before we can prove the main results in  Sections~\ref{sec:Completeness}, \ref{sec:ExistenceMinimizingGeodesics},~and~\ref{sec:ShapeSpace}.
Thus, this paper is organized as follows.

In \cref{sec:Preliminaries} we first introduce the \SoboSlobo spaces because they form the domains for the generalized tangent-point energies. 
We limit our attention to the Hilbert case.
We also define $\Curve$-dependent \SoboSlobo spaces along with corresponding inner products based on certain geometric quantities.
The induced norms turn out to be equivalent to the classical $\Curve$-independent notion (see \cref{lem:HsgammaHs}).
So the $\Curve$-dependent \SoboSlobo spaces coincide with the classical \SoboSlobo space $\Bessel[s][][\Circle][\AmbSpace]$.
Finally, we prove a geometric Morrey inequality (see \cref{thm:MorreyInequalityGeometricGeometric}), quote the uniform distortion estimate (see \cref{thm:UniformBiLipschitz}) from~\cite{blattreiter1}, and deduce \cref{lem:BoundedEnergyAndVelocityImpliesEmbeddedness}, which will play a central role in all our analysis.

The next section is devoted to studying the family of operators $\Op{}{s}$ that has been introduced in~\cref{eq:Op}.
It constitutes the main building block of our metric~$G$.
We start by proving some elementary results on reparametrizations. These demonstrate that our metric $G$ is actually a \emph{geometric} entity (see \cref{thm:GIsInvariant}).
Subsequently, we derive bounds on $\Op{}{s}$ in terms of \SoboSlobo norms and distortion (see \cref{lem:RsEquiBoundedII}).
Paying particular attention to deriving quantitative bounds is fundamental for later applications, e.g., for proving that $\Op{}{s}$ is uniformly bounded on small $\Bessel[s]$-balls
and for establishing completeness properties of~$G$.
In \cref{thm:RsIsSmooth} we show that $\gamma \mapsto \Op{\gamma}{s}$ is smooth.

In \cref{sec:GenProp} we derive several foundational properties of the metric $G$:
First we prove that $G$ is a \emph{strong} metric (see \cref{thm:GIsStrong}), i.e., it is a Riemannian metric such that the induced norm 
$\norm{u}_{\smash{G_\gamma}} \ceq \sqrt{ G_\gamma(u,u)}$ for $u \in T_\gamma \Mfld$
generates the right topology on the tangent space $T_\gamma \Mfld$.
Moreover, we show that the bilinear forms~\cref{eq:B1}--\cref{eq:B3} depend smoothly on~$\gamma$.
Consequently, both the metric~$G$ and the tangent-point energy~$\Energy$ are smooth (see \cref{thm:GIsSmooth} and \cref{thm:DE}).

Afterwards, we introduce the path length $\PathLength$ and the geodesic distance~$\dist[G]$. 
We prove that both the square root of the tangent-point energies and the square root of the arc length functional are globally Lipschitz continuous with respect to $\dist[G]$ (see \cref{thm:EnergyLipschitz} and \cref{lem:ArcLengthLipschitz}).
In particular, this shows that they are bounded on $G$-bounded sets.
Furthermore, we establish a lower bound on the geodesic distance in terms of the $\Bessel[s]$-distance that is uniform on $G$-bounded sets (see \cref{lem:GLocalDistanceBound}).
This has several consequences. 
In particular, it implies that the geodesic distance~$\dist[G]$ is definite and generates the $\Bessel[s]$-topology.
An opposite bound is not feasible in general, but at least on sufficiently small sets (see \cref{lem:GDistanceLocalDistanceBound}).

By then we will have laid the ground for proving our main results, which allows us to prove in \cref{sec:Completeness}  that $(\Mfld,G)$ has all the three completeness properties from the Hopf--Rinow theorem.
We show that $G$-bounded subsets of $\Mfld$ are weakly relatively compact in $\Mfld$, see \cref{thm:BanachAlaoglu}.
For this it is crucial that the geodesic distance has control over the tangent-point energy---which guarantees that limit points are still embeddings.
This compactness feature and the bounds derived in \cref{sec:DistanceBounds} then allow us to show that $(\Mfld,G)$ is metrically complete, see \cref{thm:MetricCompleteness}.
Geodesic completeness (see \cref{thm:LongTimeExistence}) follows from metric completeness by general principles because $G$ is a strong and smooth Riemannian metric (see \cite[Proposition~6.5]{zbMATH01246686}).

\cref{sec:ExistenceMinimizingGeodesics} 
is dedicated to the proof of existence of length-minimizing geodesics (see \cref{thm:ExistenceMinimizingGeodesics}).
We apply the direct method of the calculus of variations to the path energy functional~$\PathEnergy$ (defined in \cref{eq:PathEnergy}).
The central task here is to verify its sequential weak lower semicontinuity, see \cref{thm:DirichletEnergyIsLowerSemicontinuous}.
A technique by Bruveris \cite[Section~5]{zbMATH06502944} helps us to reduce this task to the analysis of certain $\gamma$-dependent linear operators and their (weak) continuity properties. Due to the peculiar nonlinearities in $\cB^2_\gamma$ and $\cB^3_\gamma$, this analysis is quite involved and also requires us to develop some new techniques.

Finally, in \cref{sec:ShapeSpace} we collect a few results on the \emph{shape space} $\cB \ceq \Mfld \slash \cG$, i.e., the quotient space by the action of the group $\cG$ of orientation-preserving reparameterizations of class $\Bessel[s]$.
Analogous to \cite[Section~6]{zbMATH06502944} we prove that, equipped with the \emph{quotient metric} $\dist[\cB]$, the space $(\cB,\dist[\cB])$ is a complete (extended) metric space and even a length space.

% !TEX root = Main.tex

\section{Preliminaries}
\label{sec:Preliminaries}

\subsection{Glossary of notation}\label{sec:Notation}
The following table summarizes the key mathematical notations employed throughout this paper and serves as a reference for the subsequent analysis.

{%
\setlength{\tabcolsep}{0.015\textwidth}
\renewcommand{\arraystretch}{1.3}
\begin{longtable}{ | p{0.14\textwidth}  p{0.8\textwidth}|} 
    \hline
    \textbf{Symbol}  
    &\textbf{Definition} 
\\ 
    \hline
    \endhead % all the lines above this will be repeated on every page
    \hline
    \endfoot
        $\Domain^k$
            &$(\R/\Z)^k, k\in \N$ with $\Domain\ceq \Domain^1$
\\ 
        $\Bessel[s](\Domain)$ 
            &\SoboSlobo space $W^{s,2}$ of $(\Domain, \dd x, \dist[\Domain])$, see~\cref{eq:Hsgamma}
\\
        $\Mfld$
            &the set of embeddings $\gamma\in \Bessel[s](\Domain;\AmbSpace)$, see~\cref{eq:defManifold}
\\
        $\dist[\Domain]$
            &standard distance on $\Domain$
\\
        $\dist[G]$
            &geodesic distance with respect to Riemannian metric $G$, see~\cref{eq:GeodesicDistance}
\\
        $\dist[\gamma]$
            &for $\gamma \in \Mfld$, distance function on $\Domain$ with pullback metric induced by $\gamma$, see~\cref{eq:rhogamma}
\\
        $B_r^G(\gamma)$
            &ball with center $\gamma\in \Mfld$, measured with $\dist[G]$, see~\cref{eq:GBall}
\\
        $\ArcLength(\gamma)$
            & length of a curve $\gamma\in C^1(\Domain; \AmbSpace): \ArcLength(\gamma) \ceq \int_{\Circle} \abs{\gamma'(x)}\dd x$,
            see~\cref{eq:ArcLength}            
\\
        $\PathLength(\Path)$
            &length of a path $\Path \colon \UnitInterval \rightarrow (\Mfld,G)$: $\PathLength(\Path)\ceq \int_0^1 \!\norm{\dotPath(t)}_{G_{\Path(t)}}\dd t$,
            see~\cref{eq:PathLength} 
\\
        $\distor$
            &Gromov distortion, i.e., for $\gamma \in \Mfld$: $\distor(\gamma)\ceq \sup_{x\neq y}\frac{\dist[\gamma](x,y)}{\abs{\gamma(y)-\gamma(x)}}$, see~\cref{eq:DefBiLip}
\\
        $\BiLip$
            &Bi-Lipschitz constant, see~\cref{eq:DefBiLip}
\\
        $\Energy(\gamma)$
            &tangent-point energy of $\gamma \in \Mfld$, see~\cref{eq:E}
\\
        $\PathEnergy(\Path)$
            &Dirichlet energy of $\Path \colon \UnitInterval \to (\Mfld,G)$: $\Dirichlet(\Path)\ceq\frac{1}{2}\int_0^1 \norm{\dot{\Path}(t)}_{G_{\Path(t)}}^2 \dd t$, see~\cref{eq:PathEnergy}
\\
        $\diffop{}{\sigma}$
            &$\diffop{}{\sigma} u(x,y)\ceq\frac{u(y)-u(x)}{\abs{y-x}^\sigma}$, fractional difference quotient of order $\sigma \in \intervaloc{0,1}$, see~\cref{eq:DsigmaAndMu}
\\
        $\diffop{\gamma}{\sigma}$
            &$\diffop{\gamma}{\sigma} u(x,y)\ceq\frac{u(y)-u(x)}{\dist[\gamma](x,y)^\sigma}$, fractional difference quotient of order $\sigma \in \intervaloc{0,1}$ \wrt $\dist[\gamma]$, see~\cref{eq:NuAndDsigma}
\\
        $\varLambda^\beta(\gamma)$
            &for $\gamma\in \Mfld, \beta\in \R$: $\varLambda^\beta(\gamma)(x,y)\ceq \pars[\big]{\frac{\dist[\Domain][x][y]}{\abs{\gamma(y)-\gamma(x)}}}^\beta$,
            see~\cref{eq:Lambda}
\\
        $\vol_\gamma$
            &for an immersion $\gamma\in C^1(\Domain;\AmbSpace): \dd \vol_\gamma\ceq \abs{\gamma'(x)}\dd x$,
            see~\cref{eq:vol}
\\
        $D_\gamma$
             &for an immersion $\gamma\in C^1(\Domain;\AmbSpace): D_\gamma \ceq \frac{1}{\abs{\gamma'(x)}}\partial_x$,
             see~\cref{eq:DfnDiffWrtArcLength}
\\
        $\cD_\gamma$
            &for an immersion $\gamma\in C^1(\Domain;\AmbSpace): \cD_\gamma(\cdot) \ceq D_\gamma (\cdot) (D_\gamma\gamma)^\top$,
            see~\cref{eq:OpD}
\\
        $\mu$
            &measure on $\Domain^2$, $\dd \mu(x,y) \ceq \dist[\Domain](x,y)^{-1}\dd x \dd y$,
            see~\cref{eq:DsigmaAndMu}
\\
        $\nu_\gamma$
            &measure on $\Domain^2$, $\dd \nu_\gamma(x,y)\ceq \dist[\gamma](x,y)^{-1}\dd \vol_\gamma(x)\dd \vol_\gamma(y)$, see~\cref{eq:NuAndDsigma}
\\
        $\mu_\gamma$
            &measure on $\Domain^2$, $\dd \mu_\gamma(x,y)\ceq \abs{\gamma(y)-\gamma(x)}^{-1} \dd \vol_\gamma(x)\dd \vol_\gamma(y)$ for $\gamma\in \Mfld$, see~\cref{eq:muGeometric}
\\
        $\triangle$
            &finite difference,
            $\triangle u\ceq u(y)-u(x)$,
            see~\cref{sec:BoundednessOfRs,sec:Strong,sec:Supplement}
\\
        $\norm{\cdot}_{\Lebesgue[p](\mu)}$
            &$L^p$-norm of 
            measure $\mu$, see~\cref{sec:SoboSlobo}
\\
        $\norm{\cdot}_{\Lebesgue[p](\gamma)}$
            &abbreviation for $\norm{\cdot}_{\Lebesgue[p](\dd \vol_\gamma)}$, see~\cref{eq:L2}
\\
        $\norm{\cdot}_{\Bessel[1](\gamma)}$
            &abbreviation for $\pars[\big]{\norm{\cdot}_{\Lebesgue[2](\gamma)}^2 + \norm{D_\gamma\cdot}^2_{\Lebesgue[2](\gamma)}}^{1/2}$,
            see~\cref{eq:H1}
\\
        $\norm{\cdot}_{\Bessel[s](\gamma)}$
            &abbreviation for $\pars[\big]{\norm{\cdot}_{\Lebesgue[2](\gamma)}^2 + \norm{D_\gamma\cdot}^2_{\Lebesgue[2](\gamma)}+ \norm{\delta_\gamma^{s-1} \cdot}_{\Lebesgue[2](\nu_\gamma)}^2}^{1/2}$, $s\in \intervaloo{1,2}$, see~\cref{eq:Hsgamma}
\\
        $\NProj_\gamma(x)$
            &orthoprojector onto $(\gamma'(x))^\perp$ for an immersed embedding $\gamma\in C^1(\Domain;\AmbSpace)$
\\
        $R_{\TP}(\gamma)(x,y)$
            &$\sdfrac{\abs{\gamma(y) - \gamma(x)}^2}{2 \abs{\NProj_\gamma(x) \pars{\gamma(y) - \gamma(x)}}}$ for an immersed embedding $\gamma\in C^1(\Domain;\AmbSpace)$, see~\eqref{eq:DfnTPRadius}
\\
        $\Op{\gamma}{s}v(x,y)$
            &$\abs{\gamma(x)-\gamma(y)}^{-s}\pars{v(y)-v(x)-D_\gamma v(x)\inner{D_\gamma\gamma(x), \gamma(y)-\gamma(x)}}$ for an immersed embedding $\gamma\in \Holder[1][][\Domain][\AmbDim]$, a weakly differentiable $v \colon \Domain \rightarrow \AmbSpace$ and $s\in(0,1)$, see~\eqref{eq:Op}
\\
        $B_\gamma^1(u,v)$
            &$\iint_{\Domain^2}
            \inner{\Op{\gamma}{s}u,\Op{\gamma}{s}v} \dd \mu_\gamma$, 
            for $\gamma \in \Mfld, u,v \in \Bessel[s][][\Domain][\AmbSpace], s\in \intervaloo{3/2,2}$, see~\eqref{eq:B1}
\\
        $B_\gamma^2(u,v)$
            &$\iint_{\Domain^2} 
                \abs{\Op{\gamma}{s}\gamma}^2
                    \inner[\big]{
                        \pars{u(y) - u(x)}/\abs{\gamma(y)-\gamma(x)},
                        \pars{v(y) - v(x)}/\abs{\gamma(y)-\gamma(x)}
                    }
                \dd \mu_\gamma$, for $\gamma\in \Mfld, u,v\in \Bessel[s][][\Domain][\AmbSpace]$, where $s\in \intervaloo{3/2,2}$, see~\eqref{eq:B2}
\\
       $B_\gamma^3(u,v)$
            &$\iint_{\Domain^2}\abs{\Op{\gamma}{s}\gamma}^2\pars[\big]{\inner{D_\gamma u(x),D_\gamma v(x)}+ \inner{D_\gamma u(y), D_\gamma v(y)}}\dd \mu_\gamma$ for $\gamma\in \Mfld$, $u,v\in \Bessel[s][][\Domain][\AmbSpace]$, where $s\in \intervaloo{3/2,2}$, see~\eqref{eq:B3}
\\
%\hline  // the footer is declared at the top... of course!
\end{longtable}}

\subsection{\SoboSlobo spaces}
\label{sec:SoboSlobo}

Here we introduce the \SoboSlobo space $\HSpace$ for $3/2< s < 2$ and fix a norm that topologizes it.

First we define the following bilinear forms for measurable mappings $u$, $v \colon \Domain \to \AmbSpace$:
\begin{equation*}
    \inner{u,v}_{\Lebesgue[2]}
    \ceq 
    \textstyle
    \int_\Domain \inner{u(x),v(x)} \dd x
    \qand
    \inner{u,v}_{\Bessel[1]}
    \ceq 
    \inner{u',v'}_{\Lebesgue[2]}
    =
    \textstyle
    \int_\Domain \inner{u'(x),v'(x)} \dd x
    .
\end{equation*}
On top of that, we denote the distance function in $\Domain$ by $\dist[\Circle] \colon \Domain^2 \to \intervalcc{0,\frac{1}{2}}$ and for $\sigma \in \intervalcc{0,1}$ define the operator $\diffop{}{\sigma}$ and the measure $\mu$ on $\Domain^2$ by
\begin{equation}
    \diffop{}{\sigma}  u(x,y) \ceq \sdfrac{u(y) - u(x)}{\dist[\Domain][y][x   ]^{\sigma}}
    \qand 
    \dd \mu(x,y) \ceq \sdfrac{\dd y \dd x}{\dist[\Domain][y][x]}.
    \label{eq:DsigmaAndMu}
\end{equation}
We denote the $\Lebesgue[2]$-inner product with respect to $\mu$ by $\inner{u,v}_{\Lebesgue[2](\mu)}$.
This notation allows us to write down the \emph{Gagliardo product} of order $s-1$ in the following way:
\begin{equation*}
    \inner{u,v}_{\Bessel[s-1]}
    \ceq
    \int_\Domain \int_\Domain
        \inner[\Big]{
            \sdfrac{u(y) - u(x)}{\dist[\Domain][y][x]^{s-1}},
            \sdfrac{v(y) - v(x)}{\dist[\Domain][y][x]^{s-1}}
        }
    \,
    \sdfrac{\dd y \dd x}{\dist[\Domain][y][x]}
    =
    \inner{ 
        \diffop{}{s-1} u, \diffop{}{s-1} v }_{\Lebesgue[2](\mu)}
\end{equation*}
We obtain the \emph{Gagliardo product} of order $s$ by chaining the above with the classical weak derivative $D$:
\begin{equation*}
    \inner{u,v}_{\Bessel[s]}
    \ceq
    \inner{u',v'}_{\Bessel[s-1]}
    =
    \inner{ \diffop{}{s-1} D u, \diffop{}{s-1} D v }_{\Lebesgue[2](\mu)}
    .
\end{equation*}
Let us denote the corresponding (semi)norms by 
\[
	\norm{u}_{\Lebesgue[2]} \ceq \sqrt{\inner{u,u}}_{\Lebesgue[2]},\;
	\seminorm{u}_{\Bessel[1]} \ceq \sqrt{\inner{u,u}}_{\Bessel[1]},\;
	\seminorm{u}_{\Bessel[s-1]} \ceq \sqrt{\inner{u,u}}_{\Bessel[s-1]},
	\text{ and }
	\seminorm{u}_{\Bessel[s]} \ceq \sqrt{\inner{u,u}}_{\Bessel[s]}.
\]
Finally, we define the inner product and norm
\begin{equation*}
    \iinner{u,v}_{\Bessel[s]}
    \ceq 
    \inner{u,u}_{\Lebesgue[2]} + \inner{u,u}_{\Bessel[1]} + \inner{u,u}_{\Bessel[s]}
    \qand 
    \norm{u}_{\Bessel[s]}
    \ceq 
    \sqrt{\iinner{u,v}}_{\Bessel[s]}
\end{equation*}
and the \SoboSlobo space
\begin{equation}
    \HSpace
    \ceq 
    \myset[\big]{ u \in \Lebesgue[2][][\Domain][\AmbSpace] }{ \norm{u}_{\Bessel[s]} < \infty }
    .
    \label{eq:Hsgamma}
\end{equation}
As it is well-known, $\HSpace$ together with $\iinner{u,v}_{\Bessel[s]}$ is a separable Hilbert space in which $\Holder[\infty][][\Domain][\AmbSpace]$ is a dense subset.
The Morrey inequality \cref{thm:MorreyInequality} implies that $\HSpace \subset \Holder[1][][\Domain][\AmbSpace]$.

\subsection{Geometric operators and norms}
\label{sec:GeometricOperators}

In the context of the geometry of the space $\Mfld$ of $\Bessel[s]$-embeddings it will often be convenient to work with \emph{geometric} operators and norms that depend on a particular point $\gamma \in \Mfld$. For example, we have seen already the operators $\cD_\gamma$, $D_\gamma$, and the measure $\mu_\gamma$ (cf.~\cref{eq:OpD}, \cref{eq:H1}, and \cref{eq:muGeometric}).
Now we introduce for $x$, $y \in \Circle$ the \emph{arc length distance} $\dist[\gamma](x,y)$, the arc length of the shorter arc of $\gamma(\Circle)$ between $\gamma(x)$ and $\gamma(y)$.
Assuming that $x$ and $y$ are represented by $0 \leq \tilde{x} \leq \tilde{y} < 1$, we have
\begin{equation}
    \dist[\gamma](x,y) = 
    \min \braces*{
        \textstyle \int_{\tilde x}^{\tilde y} \abs{\gamma'(z)} \dd z
        ,
        \textstyle \int_{\tilde y}^{\tilde x+1} \abs{\gamma'(z)} \dd z
    }.
    \label{eq:rhogamma}
\end{equation}
Moreover, we introduce the $\gamma$-dependent measure $\nu_\gamma$
and the fractional difference operator $\diffop{\gamma}{\sigma}$:
\begin{align}
    \dd \nu_\gamma(x,y) 
    \ceq 
    \sdfrac{\dvol_\gamma(y) \dvol_\gamma(x)}{\dist[\gamma][y][x]}
    \qand
    \diffop{\gamma}{\sigma} u (x,y)
    \ceq 
    \sdfrac{ u(y) - u(x) }{ \dist[\gamma][y][x]^\sigma }
    \quad
    \text{for $\sigma \in \intervalcc{0,1}$.}
    \label{eq:NuAndDsigma}
\end{align} 
For measurable maps $U$,~$V \colon \Domain^2 \to \AmbSpace$
the measure $\nu_\gamma$ generates the following $\Lebesgue[2]$-inner product:
\begin{align*}
    \inner{U,V}_{\Lebesgue[2](\nu_\gamma)}
    \ceq 
    \textstyle
    \iint_{\Domain^2} \inner{ U(x,y), V(x,y) }  \dd \nu_\gamma(x,y)
    \qand
    \norm{U}_{\Lebesgue[2](\nu_\gamma)}
    \ceq 
    \SmashedSqrt{ \inner{U,U}}_{\Lebesgue[2](\nu_\gamma)}
    .
\end{align*}
For $1 < s < 2$, we define the $\gamma$-dependend bilinear form and seminorm:
\begin{align}
    \inner{u,v}_{\Bessel[s](\gamma)}
    \ceq 
    \inner{
        \diffop{\gamma}{s-1} D_\gamma u , \diffop{\gamma}{s-1} D_\gamma v }_{\Lebesgue[2](\nu_\gamma)}
    \;\;\text{and}\;\;
    \seminorm{u}_{\Bessel[s](\gamma)}
    \ceq 
    \sqrt{\inner{ u, u }}_{\Bessel[s](\gamma)}
    \label{eq:DefintionHsSemi}
    .
\end{align}
Together with \cref{eq:L2} and \cref{eq:H1}, this gives rise to the $\gamma$-dependend inner product and norm
\begin{align}
    \iinner{u,v}_{\Bessel[s](\gamma)}
    \ceq 
    \inner{ u, v }_{\Lebesgue[2](\gamma)}
    +
    \inner{ u, v }_{\Bessel[1](\gamma)}
    +
    \inner{ u, v }_{\Bessel[s](\gamma)}
    \;\;\text{and}\;\;
    \norm{u}_{\Bessel[s](\gamma)}
    \ceq 
    \sqrt{\iinner{u,v}}_{\Bessel[s](\gamma)}
    \label{eq:DefintionHsNorm}
    .
\end{align}

The point of introducing these $\gamma$-dependent quantities is that they are \emph{covariant} or \emph{geometric}, i.e., they transform ``in the right way'' under reparametrizations.
This will frequently allow us to change the parametrization and assume that $\gamma$ is in  \emph{arc length parametrization}, i.e., that $\abs{\gamma'(x)} = 1$ for all $x \in \Domain$.

Let us make this more precise and briefly analyze how these entities transform under a $\Bessel[s]$-diffeomorphism $\varphi \colon \mathbb{S} \to \Domain$, $\mathbb{S} \ceq \R \slash (\ell \, \Z)$ for $\ell >0$.
For $u \colon \Domain \to \AmbSpace$ we denote the pullback along $\varphi$ by $\varphi^\pull u \ceq u \circ \varphi$.
Since $\dd \varphi(x) \colon T_x \mathbb{S} \to T_{\varphi(x)} \Circle$ is invertible and
$\dd \gamma(x) \colon T_x \Circle \to \AmbSpace$ is injective, we have that
\begin{equation}
    \begin{split}
    \cD_{(\gamma \circ \varphi)} (u \circ \varphi) (x)
    &=
    \pars[\big]{\dd (u \circ \varphi) (x)}  \pars[\big]{\dd (\gamma \circ \varphi) (x)}^{\dagger}
    \\
    &=
    \dd u( \varphi(x)) \dd \varphi(x)  \pars{ \dd \varphi(x)}^{-1} \dd \gamma( \varphi(x))^\dagger
    =
    \cD_\gamma u ( \varphi(x) )
    \end{split}
    \label{eq:cDIsGeometric}
\end{equation}
or in short: $\cD_{(\varphi^\pull \gamma)} (\varphi^\pull u) = \varphi^\pull \, \cD_\gamma u$. 
Analogously, one checks that $D_{(\varphi^\pull \gamma)} (\varphi^\pull u) = \pm \, \varphi^\pull \pars{D_\gamma u}$:
\begin{equation}
    D_{(\gamma \circ \varphi)} (u \circ \varphi) (x)
    =
    \frac{ u'(\varphi(x)) \, \varphi'(x) }{ \abs{\gamma'(\varphi(x)) \, \varphi'(x)}}
    =
    \sgn( \varphi'(x) ) \, D_\gamma u(\varphi(x))
    .
    \label{eq:DIsGeometric}
\end{equation}
The sign flip will not be of concern, because $\varphi$ will usually be orientation-preserving. (Besides, the terms involving $D_\gamma$ will always come in pairs.)
Also, the fractional difference operator is covariant
and satisfies $\diffop{(\varphi^\pull \gamma)}{\sigma} (\varphi^\pull u) = 
(\varphi \times \varphi)^\pull \, \diffop{\gamma}{\sigma} u $:
\begin{align}
    \diffop{(\gamma \circ \varphi)}{\sigma} (u \circ \varphi) (x,y)
    =
    \sdfrac{u(\varphi(y)) - u(\varphi(x))}{\dist[\gamma][\varphi(y)][\varphi(x)]^\sigma}
    =
    \diffop{\gamma}{\sigma} u (\varphi(x),\varphi(y))
    .
    \label{eq:DsigmaIsGeometric}
\end{align}
The norms and inner products of mappings are defined in terms of integrals.
By the chain rule we have $\dvol_{(\gamma \circ \varphi)}(x) = \abs{\gamma'(\varphi(x))} \abs{\varphi'(x)} \dd x$. 
Hence, we may use the transformation formula for the Lebesgue integral to compute:
\begin{align}
\begin{split}
    \inner{ \varphi^\pull u , \varphi^\pull v}_{\Lebesgue[2](\varphi^\pull\gamma )}
    &=
    \textstyle
    \int_{\mathbb{S}} 
        \inner{ u(\varphi(x)) , v(\varphi(x)) } \dvol_{(\gamma \circ \varphi)}(x)
    \\
    &=
    \textstyle
    \int_{\varphi(\Domain)} 
        \inner{ u(\varphi(x)) , v(\varphi(x)) }  \abs{\gamma'(\varphi(x))}  \abs{\varphi'(x)} 
        \dd x
    \\
    &=
    \textstyle
    \int_{\Domain} 
        \inner{ u(\xi) , v(\xi) }  \abs{\gamma'(\xi)} 
    \dd \xi        
    =
    \inner{ u,v}_{\Lebesgue[2](\gamma)}
    .
\end{split}  
    \label{eq:omegaIsGeometric}
\end{align}
Together with \cref{eq:cDIsGeometric} this immediately implies that the inner product $\inner{\cdot,\cdot}_{\Bessel[1](\gamma)}$ from \cref{eq:H1} is invariant under reparametrization, too.
In the same vein one derives the following:
\begin{align}
    \inner{ 
        (\varphi \times \varphi)^\pull \, U
        , 
        (\varphi \times \varphi)^\pull \, V 
    }_{\Lebesgue[2](\nu_{(\varphi^\pull \gamma)})}
    &=
    \inner{U, V}_{\Lebesgue[2](\nu_\gamma)}
    \quad and
    \label{eq:nuIsGeometric}
    \\
    \inner{ 
        (\varphi \times \varphi)^\pull \, U
        , 
        (\varphi \times \varphi)^\pull \, V 
    }_{\Lebesgue[2](\mu_{(\varphi^\pull \gamma)})}
    &=
    \inner{ U, V }_{\Lebesgue[2](\mu_\gamma)}
    .
    \label{eq:muIsGeometric} 
\end{align}
Combined with \cref{eq:DsigmaIsGeometric}, this shows us that the bilinear form
$\gamma \mapsto \inner{ \cdot,\cdot}_{\Bessel[s](\gamma)}$, 
the seminorm 
$\gamma \mapsto \seminorm{\cdot}_{\Bessel[s](\gamma)}$ 
and the norm 
$\gamma \mapsto \norm{\cdot}_{\Bessel[s](\gamma)}$ 
are covariant in the following sense:
\begin{align}
\begin{split}
    \inner{ \varphi^\pull u, \varphi^\pull v}_{\Bessel[s](\varphi^\pull \gamma)}
    &=
    \inner{ u, v }_{\Bessel[s](\gamma)}
    ,
    \\
    \seminorm{ \varphi^\pull u}_{\Bessel[s](\varphi^\pull \gamma)}
    &=
    \seminorm{ u }_{\Bessel[s](\gamma)}
    ,
    \\
    \norm{ \varphi^\pull u}_{\Bessel[s](\varphi^\pull \gamma)}
    &=
    \norm{ u }_{\Bessel[s](\gamma)}
    .
    \label{eq:Ws2SemimormEquivariance}
\end{split}    
\end{align}
Although the covariant norms $\norm{\cdot}_{\Bessel[s](\gamma)}$ are very compelling, at some point we have (i) to show that they generate the correct topology on $\HSpace$ and (ii) to quantify their deviation from $\norm{\cdot}_{\Bessel[s]}$. We decided that this point should be here.
A central ingredient in these proofs are the following notions of \emph{bi-Lipschitz constant} and  \emph{distortion} of a curve $\gamma \colon \Circle\to \AmbSpace$:
\begin{equation}
    \BiLip(\gamma) \ceq \esssup_{x,y \in \Circle}
        \sdfrac{\dist[\Domain](y,x)}{\abs{\gamma(y) - \gamma(x)}}
 	\quad
    \text{and}
    \quad
    \distor(\gamma)
		\ceq
		\esssup_{x,y\in\Circle} 
			\sdfrac{\dist[\gamma](y,x)}{\abs{\gamma(y)-\gamma(x)}}
        .
        \label{eq:DefBiLip}
\end{equation}
While distortion can be interpreted as the worst-case ratio between intrinsic and extrinsic length,
the bi-Lipschitz constant is the Lipschitz constant of the inverse map of~$\Curve$.
We would like to emphasize that $\distor(\gamma)$ and $\BiLip(\gamma)$ coincide if
$\gamma$ is parametrized by arc length.
In general, they differ: 
In contrast to $\BiLip(\gamma)$, the quantity $\distor(\gamma)$ is invariant under reparametrization, thus geometric,
and can in fact be bounded in terms of~$\Energy(\gamma)$.

The following norm equivalence will be pivotal for our analysis of the metric $G$. In \cref{lem:HsgammaHsOnBoundedSets} we will make these estimates uniform on sets that are bounded with respect to the geodesic distance induced by $G$.
This is why we keep precise track of the equivalence constants.
As it appears in these constants, we define the \emph{arc length} by
\begin{equation}
    \label{eq:ArcLength}
    \ArcLength(\gamma) 
    \ceq 
    \int_\Circle \abs{\gamma'(x)} \dd x
    .
\end{equation}

\begin{lemma}\label{lem:HsgammaHs}
There are continuous functions $f_{\Bessel[s]}, F_{\Bessel[s]} \colon \Mfld \rightarrow (0,\infty)$ such that the following chain of inequalities holds true:
    \begin{equation*}
        f_{\Bessel[s]}(\gamma) \norm{u}_{\Bessel[s]}
        \leq
        \norm{u}_{\Bessel[s](\gamma)}
        \leq
        F_{\Bessel[s]}(\gamma) \norm{u}_{\Bessel[s]}
        \quad 
        \text{for all $u\in \Bessel[s][][\Circle][\AmbSpace]$.}
    \end{equation*}
    In particular, $f_{\Bessel[s]}$ and $F_{\Bessel[s]}$ are given by
    \begin{align}
        \label{eqn:DefinitionF}
        F_{\Bessel[s]}(\gamma)
        &=
        \pars*{
            \norm{\Speed[\gamma]}_{\Lebesgue[\infty]}
            +
            \norm{\InvSpeed[\gamma]}_{\Lebesgue[\infty]}
            +
            2 \norm{\InvSpeed[\gamma]}_{\Lebesgue[\infty]}^{2s - 1}
            \norm{\Speed[\gamma]}_{\Lebesgue[\infty]}^2 
            \pars[\Big]{
                \norm{\InvSpeed[\gamma]}_{\Lebesgue[\infty]}^{2} 
                +
                C_{\Morrey,s-1}^2 \seminorm{\InvSpeed[\gamma]}_{\Bessel[s-1]}^2
            }
        }^{1/2}
        ,
        \\
        \label{eqn:Definitionf}
        f_{\Bessel[s]}(\gamma)
		&=
        \pars*{
            \norm{\Speed[\gamma]}_{\Lebesgue[\infty]}
            +
            \pars{
                1
                +
                2
                \norm{\Speed[\gamma]}_{\Lebesgue[\infty]}^{2s} 
            }
            \norm{\InvSpeed[\gamma]}_{\Lebesgue[\infty]} 
            +
            2 \, C_{\Morrey,s-1}^2 
            \ArcLength(\gamma)^{2 s-3} 
            \seminorm{\Speed[\gamma]}_{\Bessel[s-1]}^2
        }^{-1/2}
        ,
    \end{align}
    where $C_{\Morrey,s-1}$ is the constant from~\cref{thm:MorreyInequality}.
\end{lemma}

\begin{proof}
    We abbreviate $\sigma \ceq s-1$ and define
    $\Speed[\gamma](x) \ceq \abs{\gamma'(x)}$
    and
    $\InvSpeed[\gamma](x) \ceq 1 / \abs{\gamma'(x)}$.
    In order to bound $\seminorm{u}_{\Bessel[s](\gamma)}$ from above, we observe
    $\dist[\gamma](x,y) \geq \pars{\inf \abs{\gamma'}} \, \dist[\Circle](x,y)$.
    Hence, we have
    \[
        \sdfrac{1}{\dist[\gamma](x,y)} 
        \leq 
        \sdfrac{\norm{\InvSpeed[\gamma]}_{\Lebesgue[\infty]}}{\dist[\Circle](x,y)}
        .
    \]
    Now we calculate as follows:
    \begin{align*}        
        \abs{ \diffop{\gamma}{\sigma} D_\gamma u(x,y)}
        &=
        \abs[\Big]{
        	\sdfrac{ 
            	\InvSpeed[\gamma](y) \, u'(y) - \InvSpeed[\gamma](x) \, u'(x)
       		 }{
                \dist[\gamma](x,y)^\sigma
        	}
        }
        \leq 
        \norm{\InvSpeed[\gamma]}_{\Lebesgue[\infty]}^\sigma 
        \abs[\Big]{
        	\sdfrac{ 
            	\InvSpeed[\gamma](y) \, u'(y) - \InvSpeed[\gamma](x) \, u'(x)
       		 }{
            	\dist[\Domain][y][x]^\sigma
        	}
        }
        \\
        &\leq
        \norm{\InvSpeed[\gamma]}_{\Lebesgue[\infty]}^\sigma 
        \pars*{
            \InvSpeed[\gamma](y)
            \abs[\Big]{
                \sdfrac{ 
                    u'(y) -  u'(x)
                    }{
                    \dist[\Domain][y][x]^\sigma
                }
            }
            +
            \abs[\Big]{
                \sdfrac{ 
                    \InvSpeed[\gamma](y) - \InvSpeed[\gamma](x)
                }{
                    \dist[\Domain][y][x]^\sigma
                }
            }
            \abs{
                u'(x)
            }
        }
        \\
        &=
        \norm{\InvSpeed[\gamma]}_{\Lebesgue[\infty]}^\sigma 
        \pars[\Big]{
            \InvSpeed[\gamma](y) \abs{\diffop{}{\sigma} u'(x,y)}
            +
            \abs{\diffop{}{\sigma}  \InvSpeed[\gamma] (x,y)} \abs{u'(x)}
        }
    .
    \end{align*}
    Together with 
    $\nu_\gamma \leq  \norm{\InvSpeed[\gamma]}_{\Lebesgue[\infty]} \norm{\Speed[\gamma]}_{\Lebesgue[\infty]}^2 \, \mu$
    and $(a+b)^2 \leq 2 \pars{a^2+b^2}$, this leads us to
    \begin{align*}
    \begin{split}
        \seminorm{u}_{\Bessel[s](\gamma)}^2
        &=
        \norm{ \diffop{\gamma}{\sigma} D_\gamma u }_{\Lebesgue[2](\nu_\gamma)}^2
        \leq 
        \norm{\InvSpeed[\gamma]}_{\Lebesgue[\infty]} 
        \norm{\Speed[\gamma]}_{\Lebesgue[\infty]}^2
        \norm{ \diffop{\gamma}{\sigma} D_\gamma u }_{\Lebesgue[2](\mu)}^2
        \\
        &\leq 
        2 \norm{\InvSpeed[\gamma]}_{\Lebesgue[\infty]}^{2\sigma + 1}
        \norm{\Speed[\gamma]}_{\Lebesgue[\infty]}^2
        \pars*{
            \norm{\InvSpeed[\gamma]}_{\Lebesgue[\infty]}^2
       		\seminorm{u}_{\Bessel[s]}^2
            +
            \seminorm{\InvSpeed[\gamma]}_{\Bessel[s-1]}^2
           	\norm{u'}_{\Lebesgue[\infty]}^2
        }
        .
    \end{split}
    \end{align*}
    Using the Morrey inequality \cref{thm:MorreyInequality} we obtain
    \begin{equation}
        \seminorm{u}_{\Bessel[s](\gamma)}^2
        \leq 
        2 \norm{\InvSpeed[\gamma]}_{\Lebesgue[\infty]}^{2s - 1}
        \norm{\Speed[\gamma]}_{\Lebesgue[\infty]}^2
        \pars*{
            \norm{\InvSpeed[\gamma]}_{\Lebesgue[\infty]}^2
            +
            C_{\Morrey,\sigma}^2
            \seminorm{\InvSpeed[\gamma]}_{\Bessel[s-1]}^2
        }
        \seminorm{u}_{\Bessel[s]}^2
        .
        \label{eqn:SeminormHsGammaAgainstSeminormHs}
    \end{equation}
    For the lower order terms, we can find the following estimates:
    \begin{align*}
    	\norm{u}_{\Lebesgue[2](\gamma)}^2
	    &=
        \int_{\Domain} 
            \abs{u}^2 \abs{\gamma'} 
        \dd x
        \leq 
        \norm{\Speed[\gamma]}_{\Lebesgue[\infty]}
        \norm{u}_{\Lebesgue[2]}^2
        \quad \text{and}
        \\
        \norm{D_\gamma u}_{\Lebesgue[2](\gamma)}^2
        &= 
        \int_{\Domain} 
            \sdfrac{\abs{u'}^2 }{\abs{\gamma'}^2} \abs{\gamma'} 
        \dd x
        \leq 
        \norm{\InvSpeed[\gamma]}_{\Lebesgue[\infty]}
        \norm{u'}_{\Lebesgue[2]}^2
        .
    \end{align*}
    Combined with \cref{eqn:SeminormHsGammaAgainstSeminormHs}, this implies
    $
    	\norm{u}_{\Bessel[s](\gamma)}^2
	    \leq 
		F_{\Bessel[s]}(\gamma)^2 \norm{u}_{\Bessel[s]}^2
    $
    with $F_{\Bessel[s]}(\gamma)$ as in \cref{eqn:DefinitionF}.
	Now we move on to the reverse inequality.
	For the higher order terms, we observe $\dist[\gamma](y,x) \leq \norm{\gamma'}_{\Lebesgue[\infty]} \dist[\Domain](y,x)$,
	which implies 
	\[
		\sdfrac{1}{\dist[\Domain](y,x)} 
		\leq 
		\sdfrac{\norm{\gamma'}_{\Lebesgue[\infty]}}{\dist[\gamma](y,x)}.
	\]
	Furthermore, we find the following estimate:
    \begin{align*}
        \abs{ \diffop{}{\sigma} u'(x,y)}
        &=
        \abs[\Big]{
            \sdfrac{
                \abs{\gamma'(y)} D_\gamma u(y) - \abs{\gamma'(x)} D_\gamma u(x)
            }{
                \dist[\Circle](y,x)^\sigma
            }
        }
        \\
        &\leq 
        \abs{\gamma'(y)}
        \abs[\Big]{
            \sdfrac{
                D_\gamma u(y) - D_\gamma u(x)
            }{
                \dist[\Circle](y,x)^\sigma
            }
        }
        +
        \abs[\Big]{
            \sdfrac{
                \abs{\gamma'(y)} - \abs{\gamma'(x)}
            }{
                \dist[\Circle](y,x)^\sigma
            }
        }
        \abs{D_\gamma u(x)}
        \\
        &\leq
        \norm{\gamma'}_{\Lebesgue[\infty]}^\sigma
        \abs{\gamma'(y)} \abs{ \diffop{\gamma}{\sigma} D_\gamma u(x,y)}
        +
        \abs{\diffop{}{\sigma} \Speed[\gamma](x,y)} \abs{D_\gamma u(x)}
        .
    \end{align*}
    Together with 
    $(a+b)^2 \leq 2 \pars{a^2+b^2}$
    and the Morrey inequality \cref{lem:MorreyInequalityGeometric},
    this leads to
    \begin{align*}
        \seminorm{u}_{\Bessel[s]}^2
        &\leq
        2 \norm{\gamma'}_{\Lebesgue[\infty]}^{2\sigma} 
        \!\!\iint_{\Domain^2}
        \abs{\gamma'(y)}^2 \abs{ \diffop{\gamma}{\sigma}  D_\gamma u (x,y)}^2 
        \dd \mu(x,y)
        +
        2 \norm{\diffop{}{\sigma} \Speed[\gamma]}_{\Lebesgue[2](\mu)}^2 \norm{D_\gamma u}_{\Lebesgue[\infty]}^2 
        \\
        &=
        2 \norm{\gamma'}_{\Lebesgue[\infty]}^{2\sigma} 
        \!\!\iint_{\Circle^2}
        \, 
        \abs{\gamma'(y)} \InvSpeed[\gamma](x)
        \,
        \sdfrac{\dist[\gamma](x,y)}{\dist[\Circle](x,y)} \abs{ \diffop{\gamma}{\sigma} D_\gamma u (x,y)}^2 
        \dd \nu_\gamma(x,y)
        +
        2 \seminorm{\Speed[\gamma]}_{\Bessel[\sigma]}^2 \!\norm{D_\gamma u}_{\Lebesgue[\infty]}^2 
        \\
        &\leq
        2 \norm{\gamma'}_{\Lebesgue[\infty]}^{2\sigma+2} 
        \norm{\InvSpeed[\gamma]}_{\Lebesgue[\infty]}
        \seminorm{ D_\gamma u }_{\Bessel[\sigma](\gamma)}^2
        +
        2 \seminorm{\Speed[\gamma]}_{\Bessel[\sigma]}^2 \norm{D_\gamma u}_{\Lebesgue[\infty]}^2 
        \\
        &\leq
        \pars[\Big]{
            2 \norm{\Speed[\gamma]}_{\Lebesgue[\infty]}^{2s} 
            \norm{\InvSpeed[\gamma]}_{\Lebesgue[\infty]}
            +
            2 \, C_{\Morrey,s-1}^2 
            \ArcLength(\gamma)^{2s-3}  \seminorm{\Speed[\gamma]}_{\Bessel[s-1]}^2 
        }
        \seminorm{ u }_{\Bessel[s](\gamma)}^2
        .
    \end{align*}
    For the lower order terms, we find
	\begin{align*}
		\norm{u}_{\Lebesgue[2]}^2
		&=
		\int_\Domain 
			\abs{u}^2 
            \,
			\InvSpeed[\gamma]
			\abs{\gamma'}
		\dd x
		\leq 
		\norm{\InvSpeed[\gamma]}_{\Lebesgue[\infty]}
		\norm{u}_{\Lebesgue[2](\gamma)}^2
        \quad \text{and}
		\\
		\norm{u'}_{\Lebesgue[2]}^2
		&=
		\int_\Domain
			\abs{D_\gamma u}^2 \abs{\gamma'} ^2 
		\dd x
		\leq 
		\norm{\Speed[\gamma]}_{\Lebesgue[\infty]} 
		\norm{D_\gamma u}_{\Lebesgue[2](\gamma)}^2
        .
	\end{align*}
	Combined, we obtain
	$\norm{u}_{\Bessel[s]}^2 \leq f_{\Bessel[s]}(\gamma)^{-2} \norm{u}_{\Bessel[s](\gamma)}^2$, 
    with $f_{\Bessel[s]}(\gamma)$ as in \cref{eqn:Definitionf}.
    This concludes the proof.
\end{proof}

As side effect, the techniques from the previous proof reveal that all the measures $\mu$, $\nu_\gamma$ and $\mu_\gamma$ give rise to the same Lebesgue spaces:
\begin{corollary}\label{cor:MuGamma}
    Let $1 \leq p < \infty$.
    For $\gamma \in \Mfld$ we have
    \begin{gather*}
        c_\nu(\gamma)
        \norm{U}_{\Lebesgue[p](\mu)}
        \leq
        \norm{U}_{\Lebesgue[p](\nu_\gamma)} 
        \leq 
        C_\nu(\gamma)
        \norm{U}_{\Lebesgue[p](\mu)}
        \\
        c_\mu(\gamma)
        \norm{U}_{\Lebesgue[p](\mu)}
        \leq
        \norm{U}_{\Lebesgue[p](\mu_\gamma)} 
        \leq 
        C_\mu(\gamma)
        \norm{U}_{\Lebesgue[p](\mu)}
    \end{gather*}
    for all $U \in \Lebesgue[p][][\Domain^2;\mu][\AmbSpace]$.
    The constants are
    $
        \displaystyle
        c_\nu(\gamma) 
        =
        c_\mu(\gamma) 
        \ceq 
        \norm{ \Speed[\gamma] }_{\Lebesgue[\infty]}^{-1/p}
        \norm{\InvSpeed[\gamma] }_{\Lebesgue[\infty]}^{-2/p}
    $,
    $
        \displaystyle
        C_\nu(\gamma) 
        \ceq
        \norm{ \Speed[\gamma] }_{\Lebesgue[\infty]}^{2/p} 
        \norm{ \InvSpeed[\gamma] }_{\Lebesgue[\infty]}^{1/p}
    $,
    and
    $
        \displaystyle
        C_\mu(\gamma) 
        \ceq
        \norm{\Speed[\gamma]}_{\Lebesgue[\infty]}^{2/p} 
        \BiLip(\gamma)^{1/p}
    $.
\end{corollary}

\subsection{Geometric inequalities}

These covariant quantities allow us to geometrize some well-known inequalities that were originally derived only in special choices of coordinates.
The inequalities have to be adjusted if one has more components, i.e., if one wants to work with links.

\begin{lemma}[Morrey inequality]\label{lem:MorreyInequalityGeometric}
    Let $s \in \intervaloo{3/2,2}$ and $\alpha \ceq s - 3/2$. 
    Then there is a $C_{\Morrey,s-1} > 0$ such that
    \begin{align*}
        \norm{D_\gamma u}_{\Lebesgue[\infty]}
        \leq 
        C_{\Morrey,s-1} \, \ArcLength(\gamma)^{\alpha} \seminorm{ u }_{\Bessel[s](\gamma)}
        \quad
        \text{holds true for all $\gamma \in \Mfld$, $u \in \Bessel[s][][\Circle][\AmbSpace]$.}
    \end{align*}
\end{lemma}
\begin{proof}
    First we rescale the curve $\eta(x) \ceq \ArcLength(\gamma)^{-1} \, \gamma(x)$. 
    This is a curve of arc length $\ArcLength(\eta) = 1$.
    Now we reparametrize $\eta$ by an appropriate, orientation preserving $\Bessel[s]$-diffeomorphism $\varphi \colon \Circle \to \Circle$ to obtain a curve $\xi \ceq \eta \circ \varphi^{-1}$ in parametrization by arc length.
    With $v \ceq u \circ \varphi^{-1}$ we have
    \begin{equation*}
        \ArcLength(\gamma) \, D_\gamma u(x) 
        = 
        \sdfrac{u'(x)}{\abs{\gamma'(x)} / \ArcLength(\gamma)}
        =
        \sdfrac{u'(x)}{\abs{\eta'(x)}}
        =
        D_\eta u(x)
        =
        D_\xi v( \varphi(x) )
        =
        v'(\varphi(x))
        .
    \end{equation*}
    Due to $\Domain$ being periodic and $D_\gamma u$ being a derivative, we observe that 
    \[
        0=\int_\Domain D_\gamma u\dd \vol_\gamma
        =
        \int_\Domain \xi'(x) \dd x=\int_\Domain v'(x) \dd x
    .
    \]
    Now the Morrey embedding \cref{thm:MorreyInequality} implies 
    $
        \norm{v'}_{\Lebesgue[\infty]} 
        \leq 
        C_{\Morrey,\sigma} 
        \,
        \seminorm{ v' }_{\Bessel[s-1]}
    $.
    By the transformation behavior of the embedding-dependent $\Bessel[s]$ seminorms (see \cref{eq:Ws2SemimormEquivariance}), we have
    \begin{equation*}
        \seminorm{ v' }_{\Bessel[s-1]}
        =
        \seminorm{ v }_{\Bessel[s]}
        =
        \seminorm{ v }_{\Bessel[s](\xi)}
        =
        \seminorm{ u }_{\Bessel[s](\eta)}
        .
    \end{equation*}
    Because of $D_\eta u = \ArcLength(\gamma) \, D_\gamma u$ and $\dvol_\eta = \ArcLength(\gamma)^{-1} \dvol_\gamma$, we have
    \begin{align*}
        \seminorm{ u }_{\Bessel[s](\eta)}
        &=
        \pars[\bigg]{
            \int_\Domain \! \int_\Domain 
            \abs[\Big]{
                \sdfrac{ D_\eta u(y) - D_\eta u(x) }{ \dist[\eta](y,x)^{s-1} }
            }^2
            \,
            \sdfrac{\dvol_\eta(y) \dvol_\eta(x)}{ \dist[\eta](y,x) }
        }^{1/2}
        \\
        &=
        \pars[\bigg]{
            \ArcLength(\gamma)^{2s-1}
            \int_\Domain \! \int_\Domain 
            \abs[\Big]{
            \sdfrac{ D_\gamma u(y) - D_\gamma u(x) }{ \dist[\gamma](y,x)^{s-1} }
            }^2
            \,
            \sdfrac{\dvol_\gamma(y) \dvol_\gamma(x)}{ \dist[\gamma](y,x) }
        }^{1/2}
        =
        \ArcLength(\gamma)^{\alpha+1} \seminorm{ u }_{\Bessel[s](\gamma)}.
    \end{align*}
    Finally, we have
    \begin{equation*}
        \norm{D_\gamma u}_{\Lebesgue[\infty]}
        =
        \ArcLength(\gamma)^{-1} \norm{v'}_{\Lebesgue[\infty]}
        \leq 
        C_{\Morrey,s-1} \, \ArcLength(\gamma)^{\alpha} \seminorm{ u }_{\Bessel[s](\gamma)}
        .
    \end{equation*}
\end{proof}

In the proof of \cite[Proposition 2.5]{blattreiter1} a geometric Morrey inequality was sketched for curves in arc length parametrization.
It translates into our covariant setting as follows:
\begin{theorem}[Geometric Morrey inequality]\label{thm:MorreyInequalityGeometricGeometric}
    Let $s \in \intervaloo{3/2,2}$. Then there is a constant $\MorreyConst \geq 0$ such that the following holds true for each
    injective curve $\gamma \in \Holder[1][][\Circle][\AmbSpace]$:
    \begin{align*}
        \abs{D_\gamma \gamma(y) - D_\gamma \gamma(x)}
        \leq 
        \MorreyConst \sqrt{\Energy(\gamma)} \; \dist[\gamma][y][x]^{\alpha}
        \quad 
        \text{where}
        \quad 
        \alpha \ceq s-3/2
        .
    \end{align*}
\end{theorem}
\begin{proof}
    We reuse the rescaling $\eta$ and the reparametrized curve $\xi$ from \cref{lem:MorreyInequalityGeometric}.
    Note that $D_\gamma \gamma$, $D_\eta \eta$ and $\xi'$ are all unit tangent vectors.
    In particular, we have
    \begin{align*}
        \abs{D_\gamma \gamma(y) - D_\gamma \gamma(x)} 
        &=
        \abs{D_\eta \eta(y) - D_\eta \eta(x)} 
        =
        \abs{\xi'( \varphi(y)) - \xi'(\varphi(x))} 
        .
    \end{align*}
    Now we apply \cite[(2.7)]{blattreiter1} to $\xi$ to get 
    \begin{align*}
        \abs{\xi'( \varphi(y)) - \xi'(\varphi(x))} 
        \leq 
        \MorreyConst 
        \sqrt{\Energy(\xi)} \, \dist[\Domain](\varphi(y),\varphi(x))^\alpha
        =
        \MorreyConst 
        \sqrt{\Energy(\xi)} \, \dist[\eta](y,x)^\alpha
        .
    \end{align*}
    Recall that $\Energy$ is invariant under reparametrization, thus $\Energy(\xi) = \Energy(\eta) = \ArcLength(\gamma)^{2s - 3} \Energy(\gamma) = \ArcLength(\gamma)^{2\alpha} \Energy(\gamma)$.
    Moreover, we have $\dist[\eta] = \ArcLength(\gamma)^{-1} \dist[\gamma]$.
    Plugged into the above, we get
    \begin{align*}
        \abs{D_\gamma \gamma(y) - D_\gamma \gamma(x)} 
        \leq 
        \MorreyConst 
        \,
        \cancel{\ArcLength(\gamma)^{\alpha}}  \sqrt{\Energy(\gamma)} \, \cancel{\ArcLength(\gamma)^{-\alpha}} \, \dist[\gamma](y,x)^\alpha
        .
    \end{align*}
\end{proof}

Translating \cite[Proposition~2.7]{blattreiter1} into our covariant language, we get the following theorem:
\begin{theorem}[Uniform distortion estimate]\label{thm:UniformBiLipschitz}
    Let $s \in \intervaloo{3/2,2}$.
    Then there is a constant $0 < C_{\distor} < \infty$ such that the following holds true for the distortion (see \cref{eq:DefBiLip}) of every injective curve $\gamma \in \Holder[1][][\Circle][\AmbSpace]$:
    \begin{align*}
        \distor(\gamma) 
        \leq
        C_{\distor}
        \,
        \ArcLength(\gamma)^{(\alpha +1)/ \alpha } \, \Energy(\gamma)^{\beta}
        \quad 
        \text{for}
        \;
        \alpha \ceq s - \sdfrac{3}{2}, \;
        \beta \ceq \sdfrac{\alpha + 1}{2 \, \alpha^2}
        .
    \end{align*}
\end{theorem}
\begin{proof}
    Once more, we reuse the rescaling $\eta$ and the reparametrized curve $\xi$ from \cref{lem:MorreyInequalityGeometric}.
    Note that the rescaling does not affect the distortion, hence we have $\distor(\eta) = \distor(\gamma)$.
    Since $\ArcLength$, $\distor$, and $\Energy$ are  invariant under reparametrization, we have 
    \begin{equation*}
        \ArcLength(\xi) = \ArcLength(\eta) = 1,
        \;\;
        \distor(\xi) = \distor(\eta) = \distor(\gamma),
        \;\;\text{and}\;\;
        \Energy(\xi) = \Energy(\eta) = \ArcLength(\gamma)^{2\alpha} \, \Energy(\gamma).
    \end{equation*}
    Now we can apply \cite[Proposition~2.7]{blattreiter1}, which implies (if one meticulously collects the powers from the proof) that there is a constant $0 < C_{\distor} < \infty$ such that
    \begin{equation*}
        \distor(\xi) 
        =
        \smash{ \sup_{x,y \in \Domain} \sdfrac{ \dist[\Domain][y][x] }{ \abs{\xi(y) - \xi(x)} } }
        \leq 
        C_{\distor} \, \Energy(\xi)^{\beta}
        .
    \end{equation*}
    Thus, we have
    \begin{equation*}
        \distor(\gamma) 
        =
        \distor(\xi) 
        \leq 
        C_{\distor} \, \Energy(\xi)^{\beta}
        =
        C_{\distor} \, \ArcLength(\gamma)^{2\alpha\beta} \, \Energy(\gamma)^{\beta}
        .
    \end{equation*}
\end{proof}

As a consequence, we see that a sufficiently controlled $\Holder[1]$-convergence of embeddings with bounded energy $\Energy$ implies that the limit point is also an embedding.
We will use this property to prove various completeness properties of the metric $G$ in \cref{sec:Completeness}, in particular, to show \cref{thm:BanachAlaoglu}.

\begin{corollary}[Self-avoidance property]
    \label{lem:BoundedEnergyAndVelocityImpliesEmbeddedness}
    Let $s \in \intervaloo{3/2,2}$, $0 \leq E_0 < \infty$, and $c > 0$.
    Let $\gamma_n \in \Holder[1][][\Circle][\AmbSpace]$, $n \in \N$, be embeddings 
    and let $\gamma_\infty \in \Holder[1][][\Circle][\AmbSpace]$.
    Suppose that
    \[
        \gamma_n 
        \converges[k \to \infty][\Holder[1]]
        \gamma_\infty
        ,
        \quad
        \Energy(\gamma_n) \leq E_0 
        ,
        \qand
        \abs{\gamma_n'(x)} \geq c \quad \text{for all $x \in \Circle$ and $n \in \N$}
        .
    \]
    Then also $\gamma_\infty$ is an embedding.
\end{corollary}
\begin{proof}
    Since $\gamma_n \rightarrow \gamma_\infty$ in $C^1$,
    we have
    for every $x\in \Domain$
    that
    $\abs{\gamma_\infty'(x)} = \lim_{n \to \infty} \abs{\gamma_n'(x)} \geq c > 0$.
    Hence, $\gamma_\infty$ is an immersion.
    To prove that $\gamma_\infty$ is an embedding, we have to show that it is a homeomorphism onto its image.
    It suffices to show that $\gamma_\infty$ is injective.
    (This is true because every injective and continuous map from a compactum into a Hausdorff space is a homeomorphism onto its image.)
    Note that $\Holder[1]$-convergence of $\gamma_n$ to $\gamma_\infty$ also implies $\ArcLength(\gamma_n) \converges \ArcLength(\gamma_\infty)$.
    Hence, $\myset{ \ArcLength(\gamma_n)}{ n \in \N}$ is bounded. 
    Now the distortion bound from \cref{thm:UniformBiLipschitz} tells us that there is a constant $K < \infty$ such that 
    \[
        \sdfrac{\dist[\gamma_n](x,y)}{\abs{\gamma_n(y) - \gamma_n(x)}}
        \leq
        \distor(\gamma_n) \leq K \quad \text{for all $n \in \N$ and all $x,y \in \Circle$ satisfying $x \neq y$}
        .
    \]
    Thus, we have 
    \[
        \sdfrac{\dist[\gamma_{\infty}](x,y)}{\abs{\gamma_\infty(y) - \gamma_\infty(x)}}
        =
        \lim_{n \to \infty}
        \sdfrac{\dist[\gamma_n](x,y)}{\abs{\gamma_n(y) - \gamma_n(x)}}
        \leq K < \infty
        \quad
        \text{for $x \neq y$}
        ,
    \]
    showing that $\gamma_\infty$ must indeed be injective.
\end{proof}

We briefly state another immediate consequence of \cref{thm:UniformBiLipschitz}:
\begin{corollary}\label{cor:MorreyUniformBiLipschitz}
    Let $s \in \intervaloo{3/2,2}$ and $\alpha \ceq s - 3/2$.
    Then for $\gamma \in \Mfld$ and $u \in \Bessel[s][][\Circle][\AmbSpace]$ we have
    \begin{align*}
        \norm[\Big]{ (x,y) \mapsto \sdfrac{\abs{u(y) - u(x)}}{ \abs{\gamma(y) - \gamma(x)} } }_{\Lebesgue[\infty]}
        \leq 
        C_{\distor} \, C_{\Morrey,s-1}
        \,
        \ArcLength(\gamma)^{\frac{\alpha^2+\alpha+1}{\alpha}} \, \Energy(\gamma)^{(\alpha + 1)/(2 \alpha^2)}
        \seminorm{ u }_{\Bessel[s](\gamma)}
        .
    \end{align*}
\end{corollary}
\begin{proof}
    By the fundamental theorem of calculus, we have $\abs{u(y) - u(x)} \leq \norm{ D_\gamma u }_{\Lebesgue[\infty]} \, \dist[\gamma](y,x)$.
    Now the statement follows from combining \cref{lem:MorreyInequalityGeometric} with \cref{thm:UniformBiLipschitz}.
\end{proof}

% !TEX root = Main.tex

\section{The operator \texorpdfstring{$\Op{}{s}$}{Rs}}
\label{sec:OperatorRs}

From now on we assume that
\begin{equation}
    s \in \intervaloo[\Big]{\sdfrac{3}{2},2}
    \quad 
    \text{and globally define}
    \quad
    \alpha\ceq
    s-\sdfrac{3}{2}
    \qand
    \beta \ceq \sdfrac{\alpha+1}{2 \, \alpha^2}
    .
    \label{def:DefAlphaBeta}
\end{equation}
An essential ingredient for our metric $G$ is the family of operators
\begin{equation*}
    \Op{}{s}
    \colon 
    \Mfld \to 
    \BoundedLinOps\!\pars[\big]{ 
        \HSpace 
        , 
        \Lebesgue[2][][\Domain^2;\mu][\AmbSpace]
        }
    ,
    \quad
    \Op{}{s}(\gamma) \ceq  \Op{\gamma}{s}.
\end{equation*}
Many properties of $G$ relate to properties of $\Op{}{s}$.
The fact that $\Op{\gamma}{s}$ is always a \emph{bounded} linear operator (and thus that $\Op{}{s}$ is well-defined) will be an essential ingredient when we show that $G$ is a \emph{strong} Riemannian metric (see \cref{thm:GIsStrong}).
We deal with this in \cref{sec:BoundednessOfRs}. 
We put some extra effort into quantifying upper bounds for operator norms because these bounds will be used when we discuss the completeness properties of the metric $G$ (see \cref{sec:MetricCompleteness}).
Later we will show that $G$ is a smooth Riemannian metric (see \cref{thm:GIsSmooth}). 
That will require us to show that $\Op{}{s}$ is a smooth map (see \cref{sec:SmoothnessOfRs}).
But before we dive deeper into these technical aspects, we first analyze how $\Op{}{s}$ transforms under reparametrizations. 

\subsection{Covariance}
\label{sec:CovarianceOfRs}
The considerations below show that the operator $\Op{}{s}$ and the metric $G$ are \emph{geometric} objects, i.e., that they transform geometrically meaningfully under reparametrizations.

\begin{lemma}
    Let $M$ and $N$ be two one-dimensional smooth manifolds. 
    Let $\varphi \colon M \to N$ be a diffeomorphism of class $\Bessel[s]$,
    let $u \in \Bessel[s][][N][\AmbSpace]$, and let $\gamma \in \Bessel[s][][N][\AmbSpace]$ be an embedding. 
    Then
    \begin{equation*}
        \Op{(\varphi^\pull \gamma)}{s} (\varphi^\pull u)
        =
        (\varphi \times \varphi)^\pull \pars[\big]{\Op{\gamma}{s} u}
        .
    \end{equation*}
\end{lemma}
\begin{proof}
We merely write down the definition of $\Op{}{s}$ and employ \cref{eq:cDIsGeometric}:
\begin{gather}
    \Op{(\varphi^\pull \gamma)}{s} (\varphi^\pull u)(x,y)
    =
    \frac{
        u(\varphi(y)) 
        - 
        u(\varphi(x)) 
        - 
        \cD_{(\gamma \circ \varphi)} (u \circ \varphi)(x) \pars{ \gamma(\varphi(y)) - \gamma(\varphi(x)) } 
    }{
        \abs{ \gamma(\varphi(y)) - \gamma(\varphi(x)) }^s
    }
    \label{eq:OpIsGeometric}
    \\
    =
    \frac{
        u(\varphi(y)) 
        - 
        u(\varphi(x)) 
        - 
        \cD_\gamma u ( \varphi(x) ) \pars{ \gamma(\varphi(y)) - \gamma(\varphi(x)) } 
    }{
        \abs{ \gamma(\varphi(y)) - \gamma(\varphi(x)) }^s
    }        
    =
    \Op{\gamma}{s} u( \varphi(x),\varphi(y) )
    .
    \notag
\end{gather}    
\end{proof}

We can combine this with our observations in \cref{sec:GeometricOperators} to see that also $B_\gamma^1$, $B_\gamma^2$, and $B_\gamma^3$ are invariant under reparametrization. Thus, we have shown that $G$ is geometric:
\begin{theorem}[Reparametrization invariance of the metric]\label{thm:GIsInvariant}
    Let $M$ and $N$ be two one\hyp dimensional closed smooth manifolds.
    Let $\varphi \colon M \to N$ be a diffeomorphism of class $\Bessel[s]$ and
    let $\gamma \in \Bessel[s][][N][\AmbSpace]$ be an embedding.
    Then
    \begin{equation*}
        G_{(\varphi^\pull \gamma)}( \varphi^\pull u, \varphi^\pull v )
        =
        G_{\gamma}( u, v )
        \quad
        \text{for all $u$, $v \in \Bessel[s][][N][\AmbSpace]$.}
    \end{equation*}
\end{theorem}

\subsection{Boundedness}
\label{sec:BoundednessOfRs}

The next lemma is required for the norm bounds like \cref{lem:RsEquiBoundedII}  and \cref{lem:NormEquiv1}. The precise form of the bound becomes also important when we have to bound $DG$ in \cref{thm:LongTimeExistence}.

\begin{lemma}\label{lem:RsEquiBounded}
There is a continuous function $C_s \colon \Mfld \to \intervaloo{0,\infty}$
such that 
\begin{equation*}
    \norm{ \Op{\gamma}{s} u }_{\Lebesgue[2](\nu_\gamma)} 
    \leq 
    C_s(\gamma)  
    \seminorm{u}_{\Bessel[s](\gamma)}
    .
\end{equation*}
More precisely, with the notation of \cref{def:DefAlphaBeta}, we have
    \[
        C_s(\gamma)
        \leq
        C_{\distor}^s 
        \,
        (s-1/2)^{-1/2}
        \,
        \ArcLength(\gamma)^{s (\alpha+1)/\alpha}
        \,
        \Energy(\gamma)^{s \beta}
        \pars[\Big]{
            1
            +
            C_{\Morrey,s-1}^2 \, \ArcLength(\gamma)^{2s-3}
            \seminorm{\gamma}_{\Bessel[s](\gamma)}^2
        }^{1/2}
        .
    \]
\end{lemma}
\begin{proof}
    Since all entities involved are covariant (see also \cref{eq:OpIsGeometric}), we may 
    put $\ell \ceq \ArcLength(\gamma)$ and $\mathbb{S} \ceq \R \slash (\ell \, \Z)$ and
    assume that $\gamma \colon \mathbb{S} \to \AmbSpace$ is parametrized by arc length.
    Moreover, we pull all functions on $\mathbb{S} \times \mathbb{S} = (\R \times \R) \slash (\ell \, \Z \times \ell \, \Z)$ back to the fundamental domain 
    \begin{align}
        \FunDomain \ceq \myset[\big]{ (x,y) }{ \text{$0 \leq x \leq \ell$ and $x-\ell/2 \leq y \leq x+\ell/2$} }.
        \label{eq:FunDomain}
    \end{align}
    For $x$, $y \in \FunDomain$ 
    we have 
    $\dist[\gamma][y][x] = \abs{y-x}$,
    $D_\gamma u = u'$,
    and $\dvol_\gamma(x) = \dd x$, hence
    \begin{align*}
        \norm{ \Op{\gamma}{s} u }_{\Lebesgue[2](\nu_\gamma)}^2
        &=
        \iint_\FunDomain 
            \pars[\Big]{\sdfrac{\dist[\gamma][y][x]}{\abs{\gamma(y) - \gamma(x)}}}^{2s}
            \,
            \abs[\Big]{
                \sdfrac{
                    u(y) - u(x) - u'(x) \inner{\gamma'(x), \gamma(y) - \gamma(x)}
                }{
                    \abs{y-x}^s
                }
            }^2
        \sdfrac{\dd x \dd y}{\abs{y-x}}
        .
    \end{align*}
    We can bound the first factor of the integrand by $\distor(\gamma)^{2s}$.
    To bound the second factor, we exploit that we can work in $\varSigma \subset \R^2$ now:
    \begin{align*}
        \MoveEqLeft
        u(y) - u(x) - u'(x) \inner{\gamma'(x), \gamma(y) - \gamma(x)}
        \\
        &=
        \pars[\big]{u(y) - u(x) - u'(x) \, (y-x)}
        -
        u'(x) \inner[\big]{\gamma'(x), \gamma(y) - \gamma(x) - \gamma'(x) \, (y-x)}.
    \end{align*}
    Together with the inequality $\abs{a+b}^2 \leq 2 \abs{a}^2 + 2 \abs{b}^2$
    we obtain   
    \begin{equation}      
        \norm{ \Op{\gamma}{s} u }_{\Lebesgue[2](\nu_\gamma)}^2
        \leq 
        2 \, \distor(\gamma)^{2s} \, B(u,u)
        +
        2 \, \distor(\gamma)^{2s} \norm{D_\gamma u}_{\Lebesgue[\infty]}^2 \, B(\gamma,\gamma)
        ,
        \label{eq:RsEquiBounded1}
    \end{equation}
    where
    \begin{equation*}
        B(u,u)
        \ceq
        \iint_\FunDomain 
            \abs[\Big]{ \sdfrac{u(y) - u(x) - u'(x) \, (y-x)}{\abs{y-x}^s}}^2
        \sdfrac{\dd y \dd x}{\abs{y-x}}
        .
    \end{equation*}
    In order to bound $B(u,u)$ in terms of $\seminorm{u}_{\Bessel[s](\gamma)}^2$, we employ a technique established by Blatt in \cite[Section 2]{zbMATH06214305}.
    First we use the fundamental theorem of calculus:
    \begin{equation*}
        u(y) - u(x) - u'(x) \, (y-x)
        = \textstyle
        (y-x) \int_0^1\pars{ u'(x+\theta(y-x))-u'(x) } \dd \theta.
    \end{equation*}
    Next we substitute $X = y - x$ and use Jensen's inequality and Fubini's Theorem to obtain
    \begin{align*}
        B(u,u)
        &=
        \!
        \int_0^\ell \!\!\!\int_{-\ell/2}^{\ell/2}
        \abs[\Big]{
            \sdfrac{ 
                \textstyle
                \int_0^1 \! \pars[\big]{ u'(x{+}\theta X )-u'(x) }  \dd \theta
            }{
                \abs{X}^{s-1}
            }
        }^2
        \sdfrac{\dd X \dd x}{\abs{X}}
        \\
        &\leq
        \!
        \int_0^1 \!\!\!
        \int_0^\ell \!\!\!\int_{-\ell/2}^{\ell/2}
        \abs[\Big]{
            \sdfrac{ 
                u'(x {+} \theta X )-u'(x)
            }{
                \abs{X}^{s-1}
            }
        }^2
        \sdfrac{\dd X \dd x}{\abs{X}}  \dd \theta    
        .   
    \end{align*}
    Now we substitute
    $z = x + \theta \, X$,
    leading to
    $X = \theta^{-1} \pars{z -x}$ and
    $\dd X = \theta^{-1} \dd z$:
    \begin{align*}
        B(u,u)
        &\leq
        \int_0^1 \!\!\!
        \int_0^\ell \!\!\! \int_{x-\theta \ell/2}^{x+\theta \ell/2}
        \abs[\Big]{
            \sdfrac{ 
                u'(z)-u'(x)
            }{
                \theta^{1-s} \abs{z-x}^{s-1}
            }
        }^2
        \sdfrac{\theta^{-1}\dd z \dd x}{\theta^{-1} \!\abs{z-x}}  \dd \theta    
        \\
        &\leq
        \int_0^1 
        \!\!\!
        \iint_\FunDomain
        \theta^{2s-2}
        \abs[\Big]{
            \sdfrac{ 
                u'(z)-u'(x)
            }{
                \abs{z-x}^{s-1}
            }
        }^2
        \sdfrac{\dd z \dd x}{\abs{z-x}}  \dd \theta  
        \\
        &=
        \pars[\Big]{
            \textstyle \int_0^1 \theta^{2s-2} \dd \theta
        }
        \displaystyle
        \iint_{\Domain^2}
            \sdfrac{\abs{D_\gamma u(z_1)-D_\gamma u(z_2)}^2}
            {\dist[\gamma][z_1][z_2]^{2s-2}}
        \dd \nu_\gamma(z_1,z_2)
        \\
        &=
        \pars[\Big]{
            \textstyle \int_0^1 \theta^{2s-2} \dd \theta
        } 
        \norm{ 
            \diffop{\gamma}{s-1} D_\gamma u
        }_{\Lebesgue[2](\nu_\gamma)}^2
        =
        \displaystyle
        \sdfrac{1}{2 \, s-1} \seminorm{ u }_{\Bessel[s](\gamma)}^2
        .   
    \end{align*}  
   Substituting this bound for $B(u,u)$ and $B(\gamma,\gamma)$ back into \cref{eq:RsEquiBounded1} leads us to
    \begin{align*}      
        \norm{ \Op{\gamma}{s} u }_{\Lebesgue[2](\nu_\gamma)}^2
        &\leq 
        \sdfrac{2}{2\,s-1} \, \distor(\gamma)^{2s} \, \seminorm{ u }_{\Bessel[s](\gamma)}^2
        +
        \sdfrac{2}{2\,s-1} \, \distor(\gamma)^{2s} \norm{D_\gamma u}_{\Lebesgue[\infty]}^2 \seminorm{ \gamma }_{\Bessel[s](\gamma)}^2
        .
    \end{align*} 
    The proof is concluded by employing the  Morrey inequality \cref{lem:MorreyInequalityGeometric} to bound 
    $\norm{D_\gamma u}_{\Lebesgue[\infty]}$
    and 
    \cref{thm:UniformBiLipschitz} to bound $\distor(\gamma)$.
\end{proof}

Finally, combining \cref{lem:RsEquiBounded} with \cref{eqn:SeminormHsGammaAgainstSeminormHs} leads us immediately to the following bound.

\begin{lemma}\label{lem:RsEquiBoundedII}  
    There is a continuous function $F_{\Op{}{s}} \colon \Mfld \to \R$ such that
    \begin{equation*}
        \norm{ \Op{\gamma}{s} u }_{\Lebesgue[2](\nu_\gamma)} 
        \leq 
        F_{\Op{}{s}}(\gamma) \seminorm*{ u }_{\Bessel[s]}.
    \end{equation*}
    More precisely, we have 
    \begin{equation*}
        F_{\Op{}{s}}(\gamma)  
        = 
        C_s(\gamma)
        \pars[\Big]{
            2 \, \norm{\InvSpeed[\gamma]}_{\Lebesgue[\infty]}^{2s-1}
            \norm{\Speed[\gamma]}_{\Lebesgue[\infty]}^2 
            \pars[\big]{
                \norm{\InvSpeed[\gamma]}_{\Lebesgue[\infty]}^2 
                +
                C_{\Morrey,s-1}^2
                \seminorm{\InvSpeed[\gamma]}_{\Bessel[s-1]}^2  
            }
        }^{1/2},
    \end{equation*}
    where $\Speed[\gamma](x) \ceq \abs{\gamma'(x)}$, $\InvSpeed[\gamma](x) \ceq 1/ \abs{\gamma'(x)}$, and $C_s(\gamma)$ is the function from \cref{lem:RsEquiBounded}.
\end{lemma}

\subsection{Smoothness}
\label{sec:SmoothnessOfRs}

In this section we show that $\gamma \mapsto \Op{\gamma}{s}$ is smooth.
More precisely, our main goal is \cref{thm:RsIsSmooth} below.
But before we come to that, we first establish smoothness of a couple of auxiliary mappings.

\begin{lemma}\label{lem:AuxAreSmooth}
    For $s > 3/2$ the following maps are smooth:
    \begin{align}
        \Mfld 
        &\to 
        \Bessel[s-1][][\Domain][\R]
        ,
        &\gamma 
        &\mapsto 
        \Speed_\gamma = \abs{\gamma'}
        \label{eq:Speed}
        ,
        \\
        \Mfld 
        &\to 
        \Bessel[s-1][][\Domain][\R]
        , 
        &\gamma 
        &\mapsto 
        \InvSpeed_\gamma = 1 / \abs{\gamma'},
        \label{eq:InvSpeed}
        \\
        \Mfld 
        &\to 
        \BoundedLinOps\!\pars[\big]{\HSpace ; \Bessel[s-1][][\Domain][\AmbSpace]}
        ,
        &
        \gamma &\mapsto D_\gamma,
        \label{eq:ArcLengthDerivative}
        \\
        \Mfld 
        &\to 
        \BoundedLinOps\!\pars[\big]{\HSpace ; \Bessel[s-1][][\Domain][\BoundedLinOps(\AmbSpace;\AmbSpace)]}
        ,
        &
        \gamma &\mapsto  \cD_\gamma.
        \label{eq:FullDerivative}
    \end{align}
\end{lemma}
\begin{proof}
    Note that $s>3/2$ implies that $\Bessel[s][][\Circle][\AmbSpace] \hookrightarrow \Holder[1][][\Circle][\AmbSpace]$ embeds continuously (see \cref{thm:MorreyInequality}).
    Thus, for each $\gamma \in \Mfld$ we find a neighborhood $\BoundedSet \subset \Mfld$ and $0 < r \leq R < \infty$ such that the derivative $\xi'$ of each $\xi \in \BoundedSet$ maps $\Circle$ into the open set $V \ceq B_R(0) \setminus \overline{B_r(0)}$.
    The maps $\varPhi \colon V \to \R$, $\varPhi(X) \ceq \abs{X}$ and $\varPsi \colon V \to \R$, $\varPsi(X) \ceq 1/\abs{X}$ are smooth
    and all their derivatives $D^k \varPhi$ and $D^k \varPsi$, $k \in \N \cup \braces{0}$ are bounded. 
    Thus, the chain rule in $\Bessel[s-1]$ implies that 
    $\gamma \mapsto \abs{\gamma'} = \varPhi \circ \gamma'$ 
    and
    $\gamma \mapsto 1/\abs{\gamma'} = \varPsi \circ \gamma'$
    are smooth as maps into $\Bessel[s-1][][\Circle][\R]$. 
    This shows the smoothness of \cref{eq:Speed} and \cref{eq:InvSpeed}.
    Now the identities
    \begin{equation*}
        D_\gamma u = \InvSpeed[\gamma] \cdot u'
        \qand
        \cD_\gamma u = \pars{D_\gamma u} \cdot \pars{D_\gamma \gamma}\trans 
        \quad 
        \text{for $u \in \Bessel[s][][\Circle][\AmbSpace]$}
    \end{equation*}
    and the fact that $\Bessel[s-1][][\Circle][\R]$ is a Banach algebra
    imply the smoothness of \cref{eq:ArcLengthDerivative} and \cref{eq:FullDerivative}.
\end{proof}

\begin{remark}
    Using $\abs{D_\gamma \gamma}=1$, the derivatives of the maps in the previous lemma are now readily computed by the chain rule and product rule.
    We just state them here for reference:
    \begin{align}
        D( \gamma \mapsto \Speed[\gamma])(\gamma) \, v 
        &=
        \InvSpeed[\gamma] \inner{\gamma',v'}
        =
        \Speed[\gamma] \inner{D_\gamma \gamma, D_\gamma v}
        ,
        \label{eq:DSpeed}
        \\
        D( \gamma \mapsto \InvSpeed[\gamma])(\gamma) \, v 
        &=
        -\InvSpeed[\gamma]^3 \inner{\gamma',v'}
        =
        - \InvSpeed[\gamma] \inner{D_\gamma \gamma, D_\gamma v}
        ,
        \label{eq:DInvSpeed}
        \\
        D( \gamma \mapsto D_\gamma u)(\gamma) \, v 
        &=
        -\InvSpeed[\gamma]^3 \inner{\gamma',v'} \, u'
        =
        - \inner{ D_\gamma \gamma, D_\gamma v} D_\gamma u,
        \label{eq:DArcLengthDerivative}
        \\
        D( \gamma \mapsto \cD_\gamma u)(\gamma) \, v 
        &=
        \pars{\cD_\gamma u} \pars{ \cD_\gamma v}\trans \pars{\id_{\AmbSpace} - \cD_\gamma \gamma}
        -
        \pars{\cD_\gamma u} \pars{ \cD_\gamma v}
        ,
        \label{eq:DFullDerivative}
    \end{align}
    for all $u$, $v \in \Bessel[s][][\Circle][\AmbSpace]$.
\end{remark}

\begin{lemma}\label{lem:LambdaIsAnalytic}
    Let $\Emb[1][][\Circle][\AmbSpace] \ceq \myset{ \gamma \in \Holder[1][][\Circle][\AmbSpace] }{ \text{$\gamma$ is an embedding}}$.
    For every $\beta \in \R$ the following function is analytic:
    \begin{equation}
        \varLambda^\beta \colon 
        \Emb[1][][\Circle][\AmbSpace]
        \to 
       \Holder[0][][\Domain^2][\R],
        \quad
        \varLambda^\beta(\gamma)(x,y) 
        \ceq 
        \pars*{
            \sdfrac{ \dist[\Domain][y][x] }{ \abs{\gamma(y)-\gamma(x)}}
        }^{\beta}.
        \label{eq:Lambda}
    \end{equation}
\end{lemma}
\begin{proof}
    Since we may write
    $\varLambda^{-2}(\gamma)
        =
        \abs{ \delta^1 \gamma }^2
    $,
    it follows that the map $\varLambda^{-2}$ is analytic.
    Furthermore, since $\gamma\in C^1$, we conclude that $\varLambda^{-2}(\gamma)\in \Holder[0][][\Domain^2][\R]$.
    Now we define the open set 
    \[
        \cU \ceq \myset[\big]{ 
            \eta \in 
            \Holder[1][][\Domain][\AmbSpace]
        }{ 
            \norm{\eta'-\gamma'}_{\Lebesgue[\infty]} 
            <
            \min\!\braces[\big]{\Lip(\gamma),\tfrac{1}{2} \BiLip(\gamma)^{-1}}
        } 
        .
    \]
    We claim that it is also a subset of $\Emb[1][][\Circle][\AmbSpace]$.
    Let $\eta \in \cU$ and put $u \ceq \eta - \gamma$.
    Then for $(x,y) \in \Domain^2$ with $x \neq y$ we have
    \begin{align*}
        \abs{\eta(y) - \eta(x)}
        &\geq 
        \abs[\big]{ \abs{\gamma(y) - \gamma(x)} - \abs{u(y) - u(x)} }
        \geq 
        \abs{\gamma(y) - \gamma(x)} - \Lip(u) \, \dist[\Domain](y,x)
        \\
        &\geq
        \abs{\gamma(y) - \gamma(x)} 
        - 
        \abs{\gamma(y) - \gamma(x)}\,  \Lip(u) \, 
        \sdfrac{\dist[\Domain](y,x)}{\abs{\gamma(y) - \gamma(x)}}
        \\
        &\geq
        \abs{\gamma(y) - \gamma(x)} 
        - 
        \abs{\gamma(y) - \gamma(x)}\,  \Lip(u) \, 
        \BiLip(\gamma)
        \\
        &>
        \tfrac{1}{2} \abs{\gamma(y) - \gamma(x)} 
        \geq
        \tfrac{1}{2} \BiLip(\gamma)^{-1} \, \dist[\Domain](y,x)
        .
    \end{align*}
    In particular, this shows that $\eta$ is injective.
    Dividing by $\dist[\Domain](y,x)$,  we get 
    \[
        \varLambda^{-1}(\eta)(x,y) 
        = 
        \abs{\delta^1 \eta(x,y)} 
        \geq 
        b
        \ceq 
        \tfrac{1}{2} \BiLip(\gamma)^{-1} > 0.
    \]
    In the limit $y \to x$, this also yields $\abs{\eta'(x)} \geq \sqrt{b}$. 
    This shows that $\eta$ is an embedding, thus 
    $\cU \subset \Emb[1][][\Circle][\AmbSpace]$.
    Moreover, we have 
    \[
        \varLambda^{-1}(\eta)(x,y)
        \leq
        \Lip(\eta)
        \leq
        \Lip(\gamma)
        +
        \Lip(u)
        \leq
        \Lip(\gamma)
        +
        \norm{\eta' - \gamma'}_{\Lebesgue[\infty]}
        <
        a
        \ceq 
        2 \Lip(\gamma)
        .
    \]
    Combined, we get the pointwise estimate
    \[
        a^2 < \varLambda^{-2}(\eta)(x,y) < b^2
        \quad 
        \text{for all $\eta \in \cU$ and all $(x,y) \in \Domain^2$ with $x \neq y$.}
    \]
    Hence, the image of $\cU$ under $\varLambda^{-2}$ lies in the open set
    \[
        \cV \ceq 
        \myset[\big]{ 
            \zeta \in \Holder[0][][\Domain^2][\R]
        }{ 
            \text{for almost all $(x,y) \in \Domain^2$: }
            a^2 < \zeta(x,y) < b^2
        } 
        .
    \]
    For $\alpha \in \R$ we define the function $\psi_\alpha \colon \intervaloo{0,\infty} \to \R$, $\psi_\alpha(t) =  t^{\alpha}$ and the map 
    \[
        \varPsi_\alpha \colon \cV \to \Holder[0][][\Domain^2][\R],
        \quad 
        \varPsi_\alpha (\zeta) \ceq \zeta^\alpha.
    \]
    Using the local version of~\cite[Theorem 6.2]{zbMATH00194094}, we conclude that $\psi_\alpha$ is analytic in $\cV$.
    Furthermore,~\cite{zbMATH03230715} states that the composition of analytic functions is analytic.
    For $\alpha = -\beta/2$, this shows that the composition $\varLambda^\beta \ceq \varPsi_{-\beta/2} \circ \varLambda^{-2}$ is analytic on the open neighborhood $\cU$ of $\gamma$ in 
    $\Emb[1][][\Circle][\AmbSpace]$.
    Since $\gamma$ was arbitrary, this shows that $\varLambda^\beta$ is analytic on 
    $\Emb[1][][\Circle][\AmbSpace]$.
\end{proof}

\begin{remark}
    As a byproduct, we have shown that  $\gamma \mapsto \BiLip(\gamma) = \norm{ \varLambda^1(\gamma) }_{\Lebesgue[\infty]}$ is continuous and that $\Mfld \subset \Bessel[s][][\Circle][\AmbSpace]$ is indeed an open set.
\end{remark}

\begin{remark}
    A short computation reveals that the Fréchet derivative of $\varLambda^\beta$ is given by 
    \begin{equation}
        \label{eq:DiffLambda}
        D \varLambda^\beta(\gamma) \, v 
        =
        -\beta \; \varLambda^{\beta}(\gamma) 
        \inner[\Big]{ \sdfrac{\triangle \gamma}{\abs{\triangle \gamma}}, \sdfrac{\triangle v}{\abs{\triangle \gamma} } }
        ,
        \quad \text{where} \quad
        \triangle v(x,y) \ceq v(y) - v(x)
        .
    \end{equation} 
\end{remark}

\begin{corollary}\label{cor:DifferenceOpIsAnalytic}
    For $s > 3/2$ the following map is analytic:
    \begin{align*}  
        \Mfld 
        \to 
        \BoundedLinOps\!\pars[\big]{ 
            \HSpace 
            ; 
            \Lebesgue[\infty][][\Domain^2][\AmbSpace]
        },
        \quad 
        \gamma \mapsto \pars[\Big]{ u \mapsto  \sdfrac{u(y)-u(x) }{\abs{ \gamma(y)-\gamma(x)}}
        =
        \varLambda^1(\gamma) \, \sdfrac{ u(y)-u(x)}{\dist[\Domain](y,x)}
        }.
    \end{align*}
\end{corollary}

Finally, we have aggregated all requirements for showing the main theorem in this section.

\begin{theorem}\label{thm:RsIsSmooth}
    For $s \in \intervaloo{3/2,2}$ the map
    $
        \Op{}{s}
        \colon 
        \Mfld \to 
        \BoundedLinOps\!\pars[\big]{ 
            \HSpace 
            , 
            \Lebesgue[2][][\Domain^2;\mu][\AmbSpace]
        }
    $
    is smooth.
\end{theorem}
\begin{proof}
Let $\eta \colon \Circle \to \AmbSpace$ be the standard embedding of the round circle (or any other fixed $\eta \in \Mfld$) and let $u \in \Bessel[s][][\Circle][\AmbSpace]$.
First we show that we can write
\begin{equation}
    \Op{\gamma}{s} u
    =
    \sdfrac{ \varLambda^s(\gamma) }{ \varLambda^s(\eta) }
    \pars[\big]{    
        \Op{\eta}{s} u
        -
        \pars{\pullx \, \cD_{\gamma} u} \, \Op{\eta}{s} \gamma
    }
    \label{eq:RDecomposition}
    ,
\end{equation}
where $\pullx$ denotes pullback along the projection onto the first factor 
$\prx  \colon \Domain \times \Domain \to \Domain$, $\prx(x,y) = x$. 
(Likewise, we denote the projection onto the second factor by $\pry$.)
The crucial observation here is that $\dd \gamma(x)\pinv \dd \gamma(x) = \id_{(T_x \Circle)}$. Hence, we have
\begin{equation}
    \cD_\eta u(x) 
    =  \dd u(x) \dd \eta(x)\pinv
    =  \dd u(x) \dd \gamma(x)\pinv \dd \gamma(x) \dd \eta(x)\pinv
    =
    \cD_\gamma u(x)  \, \cD_\eta \gamma(x) 
    .
    \label{eq:DDecomposition}
\end{equation}
Now we have
\begin{align*}
    \abs{\gamma(y)-\gamma(x)}^s \, \Op{\gamma} u (x,y)
    &=
    u(y) - u(x) - \cD_\gamma u(x) \pars{\gamma(x) - \gamma(y)}
    \\
    &=
    u(y) - u(x) - \cD_\eta u(x) \pars{\eta(x) - \eta(y)}
    \\
    &\qquad
    +
    \pars[\big]{
        \cD_\gamma u(x) \, \cD_\eta \gamma(x) \pars{\eta(x) - \eta(y)} - \cD_\gamma u(x) \pars{\gamma(x) - \gamma(y)}
    }
    \\
    &=
    \abs{\eta(y)-\eta(x)}^s \pars[\big]{
        \Op{\eta}{s} u (x,y)
        -
        \cD_\gamma u(x) \, \Op{\eta}{s} u (x,y)
    }
    .
\end{align*}
Hence, \cref{eq:RDecomposition} follows from dividing this by $\abs{\gamma(y)-\gamma(x)}^s$ and observing that 
\begin{equation*}
    \sdfrac{ \abs{\eta(y)-\eta(x)}^s }{ \abs{\gamma(y)-\gamma(x)}^s }
    =
    \sdfrac{ \varLambda^s(\gamma) }{ \varLambda^s(\eta) }
    .
\end{equation*}

Now let us abbreviate $H \ceq \BoundedLinOps(\AmbSpace;\AmbSpace)$ and $\Bessel[s] \ceq \Bessel[s][][\Circle][\AmbSpace]$.
The pullback map
$\pullx \colon \Lebesgue[\infty][][\Circle][H] \to \Lebesgue[\infty][][\Circle^2][H]$,
$v \mapsto v \circ \prx$, is well-defined, linear, and continuous.
Combined with \cref{lem:AuxAreSmooth}, the Morrey embedding theorem, and \cref{lem:LambdaIsAnalytic}, this shows that the following maps are smooth:
\[
    \Mfld \times \Bessel[s]
    \to 
    \Lebesgue[\infty][][\Circle^2][H], 
    \;
    (\gamma,u) \mapsto \pullx \, \cD_\gamma u
    \qand
    \Mfld
    \to 
    \Lebesgue[\infty][][\Circle^2][\R],
    \;
    \gamma \mapsto \smash{\sdfrac{ \varLambda^s(\gamma) }{ \varLambda^s(\eta) }}
    .
\]
For the sake of compact notation, we abbreviate $\Lebesgue[2][\mu] \ceq \Lebesgue[2][][\Circle^2;\mu][\AmbSpace]$ throughout the rest of this proof.
By \cref{cor:MuGamma} and \cref{lem:RsEquiBoundedII}, also the following maps are linear and continuous, hence, smooth:
\[
    \Mfld
    \to 
    \Lebesgue[2][\mu],
    \;
    \gamma \mapsto \Op{\eta}{s} \gamma
    \qand
    \Bessel[s]
    \to 
    \Lebesgue[2][\mu],
    \; 
    u \mapsto \Op{\eta}{s} u
    .
\]
We define the map $F \colon \Mfld \times \Bessel[s] \to \Lebesgue[2][\mu]$ by $F(\gamma,u) \ceq \Op{\gamma}{s} u$.
By \cref{eq:RDecomposition}, $F(\gamma,u)$ depends polynomially on the smooth maps above. 
More precisely, $F(\gamma,u)$ is obtained by employing  well-defined and continuous bilinear forms
\begin{gather*}
    \Lebesgue[\infty][][\Domain^2][\R] 
    \times 
    \Lebesgue[2][\mu]
    \to 
    \Lebesgue[2][\mu]    
    \qand
    \Lebesgue[\infty][][\Domain^2][H] 
    \times 
    \Lebesgue[2][\mu]
    \to 
    \Lebesgue[2][\mu]
\end{gather*}
induced by pointwise multiplication $\R \times \AmbSpace \to \AmbSpace$
and  $H \times \AmbSpace \to \AmbSpace$.
Hence, $F$ is smooth.
Recalling that $T_\gamma \Mfld = \Bessel[s]$, we see that also $DF \colon \Mfld \times \Bessel[s] \to \BoundedLinOps\!\pars{ 
    \Bessel[s] \times \Bessel[s]; \Lebesgue[2][\mu]
}$ is smooth.
With the bounded linear projector $Q \colon \Bessel[s] \times \Bessel[s] \to \Bessel[s]$, $Q(\gamma,u) \ceq \gamma$, 
we may write $\Op{\gamma}{s} = DF(\gamma,0) \cdot Q$.
Thus, also the map $\Mfld \to \BoundedLinOps\!\pars{ \Bessel[s]; \Lebesgue[2][\mu]}$, 
$\gamma \mapsto \Op{\gamma}{s}$ is smooth.
\end{proof}

\begin{remark}
    With \cref{eq:RDecomposition} we can readily compute the derivative of $\Op{}{s}$.
    This will be useful later, e.g., when we compute the derivative of the tangent-point energy (see \cref{thm:DE}).
    Let $u$, $v \in \Bessel[s][][\Circle][\AmbSpace]$ and abbreviate 
    $U(\gamma) \ceq \pullx \, \cD_\gamma u$, 
    $V(\gamma) \ceq \pullx \, \cD_\gamma v$, 
    $\triangle \gamma \ceq \gamma(y)-\gamma(x)$,
    $\triangle v \ceq v(y)-v(x)$,
    and 
    $W(\gamma) \ceq \pullx \, \cD_\gamma \gamma$.
    Now the product rule implies
    \begin{align*}
        \MoveEqLeft
        \pars[\big]{D\Op{}{s}(\gamma) \, v} \, u
        = 
        D \!\pars[\big]{ \xi \mapsto \Op{\xi}{s} u } (\gamma) \, v
        \\
        &=
        \sdfrac{ \pars{ D\varLambda^s(\gamma)\,v } }{ \varLambda^s(\eta) }
        \pars[\big]{    
            \Op{\eta}{s} u
            -
            U(\gamma) \, \Op{\eta}{s} \gamma
        }
        -
        \sdfrac{ \varLambda^s(\gamma) }{ \varLambda^s(\eta) } \, U(\gamma) \, 
        \Op{\eta}{s}    v
        -
        \sdfrac{ \varLambda^s(\gamma) }{ \varLambda^s(\eta) } \pars{DU(\gamma) \, v}
        \, \Op{\eta}{s} \gamma
        .
    \end{align*}
    Plugging in \cref{eq:DiffLambda} and \cref{eq:DFullDerivative}, which takes the form 
    \begin{equation*}
        DU(\gamma) \, v =
        - U(\gamma) \,  V(\gamma)
        +
        U(\gamma) \, V(\gamma)\trans \pars{ \id_{\AmbSpace} -  \, W(\gamma)}
        ,
    \end{equation*}
    and rearranging terms leads to
    \begin{align*}
        \pars[\big]{D\Op{}{s}(\gamma) \, v} \, u
        &=
        -
        s
        \inner[\Big]{ \sdfrac{\triangle \gamma}{\abs{\triangle \gamma}}, \sdfrac{\triangle v}{\abs{\triangle \gamma} } }
        \sdfrac{ \varLambda^s(\gamma) }{ \varLambda^s(\eta) }
        \pars[\big]{    
            \Op{\eta}{s} u
            -
            U(\gamma) \, \Op{\eta}{s} \gamma
        }
        \\
        &\qquad
        -
        \sdfrac{ \varLambda^s(\gamma) }{ \varLambda^s(\eta) }
        U(\gamma)
        \pars[\big]{
            \Op{\eta}{s} v
            -
            V(\gamma) \, \Op{\eta}{s} \gamma
        }
        -
        \sdfrac{ \varLambda^s(\gamma) }{ \varLambda^s(\eta) }
        U(\gamma) \, V(\gamma)\trans
        \pars[\big]{
            \Op{\eta}{s} \gamma
            -
            W(\gamma) \, \Op{\eta}{s} \gamma
        }
        .
    \end{align*}
    Finally, using \cref{eq:RDecomposition} in reverse shows that
    \begin{equation}
        \pars[\big]{D\Op{}{s}(\gamma) \, v} \, u
        =
        - 
        s \, \Op{\gamma}{s} u
        \inner[\Big]{ 
            \sdfrac{ \triangle \gamma }{ \abs{\triangle \gamma} }
            , 
            \sdfrac{ \triangle v }{ \abs{\triangle \gamma} }
        }
        - \pars{\pullx \, \cD_\gamma u} \, \Op{\gamma}{s} v
        - \pars{\pullx \, \cD_\gamma u} \pars{\pullx \, \cD_\gamma v}\trans \, \Op{\gamma}{s} \gamma
    .
    \label{eq:RsGateaux}
    \end{equation}
\end{remark}

% !TEX root = Main.tex

\section{The metric}
\label{sec:GenProp}

We have introduced the metric $G$ on $\Mfld$ already in \cref{eq:G} in the introduction.
In this section we collect several important properties of the metric.

In \cref{sec:Strong} we show that $G$ is a \emph{strong Riemannian metric} on $\Mfld$, i.e., it is positive-definite and its induced norm generates the topology of $T_\gamma \Mfld = \Bessel[s][][\Domain][\AmbSpace]$.
In \cref{sec:MetricSmooth} we show that $G_\gamma$ depends smoothly on $\gamma$.
Afterwards, we introduce the geodesic distance function $\dist[G]$ of $G$ (see \cref{sec:GeodesicDistance}) and observe that it makes many relevant functionals \emph{globally} Lipschitz continuous. 
Most importantly, this applies to the tangent-point energy $\Energy$ (see \cref{thm:EnergyLipschitz}) and to the arc length functional $\ArcLength$ (see \cref{lem:ArcLengthLipschitz}). 
We then show that quantities like $\Energy$, $\abs{\gamma'}$, and $1 / \abs{\gamma'}$ are bounded on $G$-bounded sets (see \cref{lem:BiLipBound}). 
Summarized, these bounds will show that the Riemannian metric $G$ has sufficient control to keep curves away from various modes of degeneration. 
Finally, we compare our distance to the normal $\Bessel[s]$-distance, quantify their relation, and show that the identity map $\id \colon (\Mfld, \dist[G]) \to (\Mfld, \dist[\Bessel[s]])$ is locally bi-Lipschitz continuous (see \cref{lem:GLocalDistanceBound}, \cref{lem:GDistanceLocalDistanceBound}, and \cref{cor:SameTopology}). 
These findings will be crucial 
for the discussion of compactness properties of bounded sets (see \cref{sec:BanachAlaoglu})
and for  metric completeness (\cref{sec:MetricCompleteness}).

% !TEX root = Main.tex

\subsection{The metric is strong}
\label{sec:Strong}

Our aim here is to show the following theorem:
\begin{theorem}
    \label{thm:GIsStrong}
    The metric $G$ defined in \cref{eq:G} is a strong Riemannian metric on $\Mfld$.
    \newline
    In particular, the \emph{Riesz isomorphism} 
\begin{align*}
    \Riesz_\gamma \colon T_\gamma \Mfld \to T\dual_\gamma\Mfld,
	\quad 
	\inner{ \Riesz_\gamma u , v} \ceq G_\gamma(u,v)
\end{align*}
    is an isomorphism between the tangent space $T_\gamma \Mfld = \HSpace$
    and the continuous cotangent space $T\dual_\gamma\Mfld \ceq (\HSpace)\dual = \Bessel[-s][][\Domain][\AmbSpace]$.
\end{theorem}
\begin{proof}
    To this end we have to show that the induced norm 
    $
        \norm{u}_{\smash{G_\gamma}} \ceq \sqrt{ G_\gamma(u,u)}
    $
    for
    $
        u \in T_\gamma \Mfld
    $
    generates the right topology on $T_\gamma \Mfld = \HSpace$.
    By \cref{lem:HsgammaHs} we already know that $\norm{\cdot}_{\Bessel[s](\gamma)}$ generates this topology.
    So it suffices to show that the norms $\norm{\cdot}_{\Bessel[s](\gamma)}$ and $\norm{\cdot}_{\smash{G_\gamma}}$ are equivalent for fixed $\gamma$. 
    As this is very technical and because we will later need also some detailed information on the proportionality factors in this norm estimates, we delegate this task to \cref{lem:NormEquiv2} and \cref{lem:NormEquiv1} below.

    The norm bounds imply that $G_\gamma$ is a bounded and coercive symmetric bilinear form on $\HSpace$.
    So the Lax--Milgram theorem shows that $\Riesz_\gamma \colon \HSpace \to \pars{\HSpace}'$ is indeed an isomorphism of Hilbert spaces.
\end{proof}

We spend the remainder of this section on showing \cref{lem:NormEquiv2} and \cref{lem:NormEquiv1}.
To this end we heavily rely on the techniques established in \cite[Theorem~2.5]{blattreiter1}.
The following is a preparation for \cref{lem:NormEquiv2}.

\begin{lemma}\label{lem:DeltaBound}
    Let $s\in \intervaloo{3/2,2}$, $ \alpha, \beta$ as in \cref{def:DefAlphaBeta} and $\gamma \in \Mfld$. Put $\delta(\gamma) \ceq \pars[\big]{2 \, \MorreyConst \SmashedSqrt{\Energy(\gamma)} }^{-1/\alpha}$.
    Then for all $(x,y) \in \Domain^2$ satisfying $\dist[\gamma][y][x] < \delta(\gamma)$
    we have
    \begin{align}
        \begin{split}
            \abs{\inner{ D_\gamma \gamma (x), \gamma(y) -\gamma(x)}}
            &\geq 
            \sdfrac{1}{2} \abs{\gamma(y) -\gamma(x)}
            \quad \text{and}
            \\
            \abs{\inner{ D_\gamma \gamma (y), \gamma(y) -\gamma(x)}}
            &\geq 
            \sdfrac{1}{2} \abs{\gamma(y) -\gamma(x)}
            .
        \end{split}
        \label{eq:TangetLowerBound}
        \end{align}
\end{lemma}
\begin{proof}
    Let $x$, $y \in \Circle$ be two distinct points.
    We denote by $I \subset \Circle$ the shortest closed arc with respect to $\vol_\gamma$.
    By swapping $x$ and $y$ if necessary, we may assume that $x$ is the left end point and that $y$ is the right end point of $I$ with respect to the orientation chosen on $\Circle$.
    Fix an $\eta \in I$. 
    Now we exploit $\abs{\gamma(y)-\gamma(x)} \leq \dist[\gamma](y,x)$ (the chord length is less or equal to the arc length),
    $\abs{D_\gamma \gamma(\eta)} = 1$ for any $\eta\in I$, and the inverse triangle inequality:
     \[
        \abs[\Big]{\inner[\Big]{
                D_\gamma \gamma(\eta), 
                \sdfrac{\gamma(y)-\gamma(x)}{\abs{\gamma(y)-\gamma(x)}}
        }}
        \geq
            \abs[\Big]{\inner[\Big]{
                D_\gamma \gamma(\eta), 
                \sdfrac{\gamma(y)-\gamma(x)}{\dist[\gamma](y,x)}
        }}
        \geq
            1
            -
            \abs[\Big]{\inner[\Big]{
                D_\gamma \gamma(\eta),
                D_\gamma\gamma(\eta)- \sdfrac{\gamma(y)-\gamma(x)}{\dist[\gamma](y,x)}
            }}
        .
    \]
    By construction of $I$, we have $\gamma(y) - \gamma(x) = \int_I D_\gamma \gamma(z) \dvol_\gamma(z)$ and $\dist[\gamma](y,x) = \int_I \dvol_\gamma(z)$.
    Now the Cauchy-Schwarz inequality, the Jensen inequality, and the geometric Morrey inequality \cref{thm:MorreyInequalityGeometricGeometric} imply
    \begin{gather*}
        \abs[\Big]{\inner[\Big]{
			D_\gamma \gamma(\eta),
            D_\gamma\gamma(\eta)- \sdfrac{\gamma(y)-\gamma(x)}{\dist[\gamma](y,x)}
		}}
        \leq
        \abs[\Big]{
            D_\gamma\gamma(\eta)- \sdfrac{\gamma(y)-\gamma(x)}{\dist[\gamma](y,x)}
		}
        \\
        =
        \abs[\Big]{
            \fint_I \pars[\big]{D_\gamma\gamma(\eta) - D_\gamma \gamma(z)} \dvol_\gamma(z)
		}
        \leq 
        \fint_I \abs[\big]{D_\gamma\gamma(\eta)- D_\gamma \gamma(z)} \dvol_\gamma(z)
        \\
        \leq
        \fint_I \MorreyConst \sqrt{\Energy(\gamma)} \dist[\gamma](y,x)^{\alpha} \dvol_\gamma(z)
        \leq
        \MorreyConst \sqrt{\Energy(\gamma)}\,\delta(\gamma)^{\alpha}
        \leq \sdfrac{1}{2}
        .
    \end{gather*}
    Inserting this into the above leads to
    \[
        \abs[\Big]{\inner[\Big]{
                D_\gamma \gamma(\eta), 
                \sdfrac{\gamma(y)-\gamma(x)}{\abs{\gamma(y)-\gamma(x)}}
        }}
        \geq
        1 - \sdfrac{1}{2}
        = \sdfrac{1}{2}.
    \]
\end{proof}

\begin{lemma}\label{lem:NormEquiv2}
    There is a constant $0 < C < \infty$ depending on $s\in \intervaloo{3/2,2}$ such that the following holds true
    for all $\gamma \in \Mfld$ and $u \in \HSpace$:
    \begin{equation}
        \seminorm{u}_{\Bessel[s](\gamma)}^2
        \leq 
        32 \, B_\gamma^1(u,u) + 128 \, B_\gamma^2(u,u)
        +
        C\,\ArcLength(\gamma) \, \Energy(\gamma)^{(\alpha + 1)/\alpha} \seminorm{u}_{\Bessel[1](\gamma)}^2
        .
        \label{eq:HsgammaNormControlledByG0}
    \end{equation}
    Moreover, we have
    \bignegskip%
    \begin{equation}
        \norm{u}_{\Bessel[s](\gamma)}
        \leq 
        \sqrt{32 + C \,  \ArcLength(\gamma) \,\Energy(\gamma)^{(\alpha+1)/\alpha} }
        \norm{u}_{G_\gamma}
        .
        \label{eq:HsgammaNormControlledByG}
    \end{equation}
\end{lemma}
\begin{proof}
    Let $\delta \ceq \delta(\gamma)$ as in \cref{lem:DeltaBound} and
    define 
    \begin{align*}
        U \ceq \myset[\big]{ (x,y) \in \Domain^2 }{ \dist[\gamma][y][x] < \delta }
        \qand 
        U^c \ceq \myset[\big]{ (x,y) \in \Domain^2 }{ \dist[\gamma][y][x] \geq \delta }
    \end{align*}
    and split $\seminorm{u}_{\Bessel[s](\gamma)}^2$ as follows:
    \begin{align*}
        \seminorm{u}_{\Bessel[s](\gamma)}^2
        =
        \iint_{U}
        \abs{ \diffop{\gamma}{s-1} D_\gamma u }^2
        \dd \nu_\gamma
        +
        \iint_{U^c}
            \abs{ \diffop{\gamma}{s-1} D_\gamma u }^2
        \dd \nu_\gamma
        .
    \end{align*}
    Using $(a - b)^2 \leq 2 \, a^2 + 2 \, b^2$ and the definition of $U^c$, we can bound the integral over $U^c$ by
    \begin{align}
    \begin{split}
        \iint_{U^c}
            \abs{ \diffop{\gamma}{s-1} D_\gamma u }^2
        \dd \nu_\gamma
        &=
        \iint_{U^c} 
            \sdfrac{
                \abs{D_\gamma u(y) - D_\gamma u(x)}^2 
            }{
                \dist[\gamma](y,x)^{2(s-1)}
            }
        \sdfrac{
            \dvol_\gamma(y) \dvol_\gamma(x)
        }{
                \dist[\gamma](y,x)
            }
        \\
        &\leq
        \frac{2}{\delta^{2s-1}}
        \iint_{U^c} 
            \pars[\big]{
                \abs{D_\gamma u(x)}^2 + \abs{D_\gamma u(y)}^2 
            }
        \dvol_\gamma(y) \dvol_\gamma(x)
        \\
        &
        \leq
        \frac{4 \, \ArcLength(\gamma)}{\delta^{2s-1}} 
        \norm{D_\gamma u}_{\Lebesgue[2](\gamma)}^2
        = 
        C \, \ArcLength(\gamma) \, \Energy(\gamma)^{(\alpha + 1)/\alpha} 
        \seminorm{u}_{\Bessel[1](\gamma)}^2
        .
    \end{split}
    \label{eq:NormEquiv2a}
    \end{align}
    Next we deal with the integral over $U$.
    Let $(x,y) \in U$ and abbreviate $\tau \ceq D_\gamma \gamma$,
    $\triangle \gamma \ceq \triangle \gamma(x,y) = \gamma(y) - \gamma(x)$,
    and
    $\triangle u \ceq \triangle u(x,y) =  u(y) - u(x)$.
    Note that since $\delta$ was chosen according to \cref{lem:DeltaBound}, we have 
    $\abs{\inner{\tau(x), \triangle\gamma}} > 0$ and $\abs{\inner{\tau(y), \triangle\gamma}} > 0$ for all $x,y \in U$.
    We compute
    \begin{align}
        \begin{split}
        \MoveEqLeft
        \abs{D_\gamma u(y) {-} D_\gamma u(x)}
        \leq
        \abs[\Big]{D_\gamma u(y) {-} \sdfrac{\triangle u}{\inner{\tau(y),\triangle \gamma}}}
        +
        \abs[\Big]{\sdfrac{\triangle u}{\inner{\tau(y),\triangle \gamma}} {-} \sdfrac{\triangle u}{\inner{\tau(x),\triangle \gamma}} }
        +
        \abs[\Big]{\sdfrac{\triangle u}{\inner{\tau(x),\triangle \gamma}} {-} D_\gamma u(x)}
        .
        \end{split}
        \label{eq:hPrimeDifference}
    \end{align}    
    Using
    $D_\gamma u(x) \inner{\tau(x), \triangle \gamma} = \cD_\gamma u (x) \, \triangle \gamma$, we get the following for the last summand:
    \begin{align*}
        \abs[\Big]{\sdfrac{\triangle u}{\inner{\tau(x),\triangle \gamma}} - D_\gamma u(x)}
        =
        \abs[\Big]{
            \sdfrac{\triangle u - \smash{\cD_\gamma u(x)} \, \triangle \gamma}{\inner{\tau(x),\triangle \gamma}}
        }
        =
        \sdfrac{
            \abs*{u(y) - u(x) - \smash{\cD_\gamma u(x)} \, (\gamma(y) - \gamma(x))}
        }{
            \abs{\inner{ D_\gamma \gamma(x), \gamma(y) - \gamma(x)}}
        }
        .
    \end{align*}  
    Now  \cref{eq:TangetLowerBound} implies 
    \[
        \abs[\Big]{\sdfrac{\triangle u}{\inner{\tau(x),\triangle \gamma}} - D_\gamma u(x)}
        \leq
        2 \abs{\gamma(y)- \gamma(x)}^{s-1} \abs{\Op{\gamma}{s} u(y,x)}
        .
    \]
    Likewise, we obtain the following for the first summand in \cref{eq:hPrimeDifference}:
    \[
        \abs[\Big]{D_\gamma u(y) - \sdfrac{\triangle u}{\inner{\tau(y),\triangle \gamma}}}
        \leq
        2 \abs{\gamma(y)- \gamma(x)}^{s-1} \abs{\Op{\gamma}{s} u(x,y)}
        .
    \]
    For the center term in \cref{eq:hPrimeDifference}
    we use \cref{eq:TangetLowerBound} once more:
    \begin{align*}
        \abs[\Big]{\sdfrac{\triangle u}{\inner{\tau(y),\triangle \gamma}} - \sdfrac{\triangle u}{\inner{\tau(x),\triangle \gamma}} }
        =
        \sdfrac{
            \abs{ \inner{\tau(x),\triangle \gamma} {-} \inner{\tau(y),\triangle \gamma} }
        }{
            \abs{\inner{\tau(y),\triangle \gamma}} \abs{\inner{\tau(x),\triangle \gamma}}
        }
        \abs{\triangle u}
        \leq 
        4 
        \sdfrac{
            \abs{ \inner{\tau(x),\triangle \gamma} {-} \inner{\tau(y),\triangle \gamma} }
        }{
            \abs{\triangle \gamma}
        }
        \sdfrac{\abs{\triangle u}}{\abs{\triangle \gamma}}
        .
    \end{align*}  
    Because of $\cD_\gamma \gamma =  \tau(x) \, \tau(x)\trans$, we may proceed with the numerator as follows:
    \begin{align*}
        \MoveEqLeft
        \abs{ \inner{\tau(x),\triangle \gamma} - \inner{\tau(y),\triangle \gamma} }
        =
        \abs[\big]{ 
            \abs{ \cD_\gamma \gamma(x) \, \triangle \gamma} 
            - 
            \abs{ \cD_\gamma \gamma(y) \, \triangle \gamma}
        }
        \\
        &\leq 
        \abs[\big]{ 
            \cD_\gamma \gamma(x) \, \triangle \gamma
            - 
            \cD_\gamma \gamma(y) \, \triangle \gamma
        }
        =
        \abs[\big]{ 
            \pars{\triangle \gamma - \cD_\gamma \gamma(y) \, \triangle \gamma}
            - 
            \pars{\triangle \gamma - \cD_\gamma \gamma(x) \, \triangle \gamma}
        }
        \\
        &\leq
        \abs{\gamma(x) - \gamma(y) - \cD_\gamma \gamma(y) \, (\gamma(x) - \gamma(y))}
        +
        \abs{\gamma(y) - \gamma(x) - \cD_\gamma \gamma(x) \, (\gamma(y) - \gamma(x))}
        \\
        &=
        \abs{\gamma(y)- \gamma(x)}^{s}
        \abs{\Op{\gamma}{s} \gamma(y,x)} 
        + 
        \abs{\gamma(y)- \gamma(x)}^{s}
        \abs{\Op{\gamma}{s} \gamma(x,y)}
        .
    \end{align*}
    Now we insert all this into \cref{eq:hPrimeDifference}
    and divide by $\dist[\gamma][y][x]^{s-1} \geq \abs{\gamma(y) - \gamma(x)}^{s-1}$:
    \begin{align*}
        \abs{
            \diffop{\gamma}{s-1}  D_\gamma u(x,y)
        }
        \leq
        2 \abs{ \Op{\gamma}{s} u(y,x)}
        +
        4 \abs{ \Op{\gamma}{s} \gamma(y,x)} \sdfrac{\abs{\triangle u}}{\abs{\triangle \gamma}} 
        +
        4 \abs{ \Op{\gamma}{s} \gamma(x,y)} \sdfrac{\abs{\triangle u}}{\abs{\triangle \gamma}} 
        +
        2 \abs{ \Op{\gamma}{s} u(x,y)}
        .
    \end{align*} 
    Squaring and the inequality $\pars{a+b+c+d}^2 \leq 4 \pars{a^2+b^2+c^2+d^2}$ lead us to
\begin{align*}
        \abs{
            \diffop{\gamma}{s-1}  D_\gamma u(x,y)
        }^2
        &\leq
        16  \abs{ \Op{\gamma}{s} u(y,x)}^2
        +
        64 \abs{ \Op{\gamma}{s} \gamma(y,x)}^2 \sdfrac{\abs{\triangle u}^2}{\abs{\triangle \gamma}^2}
        \\
        &\qquad
        +
        64 \abs{ \Op{\gamma}{s} \gamma(x,y)}^2 \sdfrac{\abs{\triangle u}^2}{\abs{\triangle \gamma}^2}
        +
        16 \abs{ \Op{\gamma}{s} u(x,y)}^2
        .
\end{align*}
    Now we integrate against $\dd \nu_\gamma \leq \dd \mu_\gamma$
    and exploit that $U$ is symmetric with respect to $(x,y) \mapsto (y,x)$:
    \begin{align*}
        &\iint_{U}
            \abs{
                \diffop{\gamma}{s-1}  D_\gamma u
            }^2
        \dd \nu_\gamma
        \leq 
        \iint_{U}
            \abs{
                \diffop{\gamma}{s-1}  D_\gamma u
            }^2
        \dd \mu_\gamma
        \\
        &\leq 
        32
        \iint_{U} \abs{ \Op{\gamma}{s} u}^2 \dd \mu_\gamma
        +
        128
        \iint_{U} \abs{ \Op{\gamma}{s} \gamma}^2 \abs[\Big]{ \sdfrac{\triangle u}{\abs{\triangle \gamma}}}^2\dd \mu_\gamma
        \leq 
        32 \, B_\gamma^1(u,u) + 128 \, B_\gamma^2(u,u)
        .
    \end{align*}
    Combined with \cref{eq:NormEquiv2a}, this shows \cref{eq:HsgammaNormControlledByG0}. 
    Recalling the definition \cref{eq:DefintionHsNorm} of $\norm{u}_{\Bessel[s](\gamma)}$, we have
    \begin{equation*}
        \norm{u}_{\Bessel[s](\gamma)}^2
        \leq
        32 \, B_\gamma^1(u,u) 
        + 
        128 \, B_\gamma^2(u,u)
        +
        \pars[\big]{
            1 
            + 
            C \, \ArcLength(\gamma) \, \Energy(\gamma)^{(\alpha + 1)/\alpha}
        } 
        \seminorm{u}_{\Bessel[1](\gamma)}^2
        +
        \norm{u}_{\Lebesgue[2](\gamma)}^2
        .
    \end{equation*}
    Recalling also the definition \cref{eq:G} of $\norm{u}_{G_\gamma}^2 = G_\gamma(u,u)$,
    \cref{eq:HsgammaNormControlledByG} follows by $\pars{2s+1} \geq 4$.
\end{proof}

\begin{lemma}\label{lem:NormEquiv1}
    Let $\gamma\in \Mfld$ be an embedding.
    Then there exists an $F_{G}(\ArcLength(\gamma),\Energy(\gamma))>0$ such that 
    \begin{align*}
        \norm{u}_{G_\gamma}\leq F_{G}(\ArcLength(\gamma),\Energy(\gamma)) \norm{u}_{\Bessel[s](\gamma)}
        \quad
        \text{holds true for all $u \in \HSpace$.}
    \end{align*}
\end{lemma}
\begin{proof}
    Both $\norm{u}_{G_\gamma}^2$ and $\norm{u}_{\Bessel[s](\gamma)}^2$
    have the summands $\seminorm{u}_{\Bessel[1](\gamma)}^2 + \norm{u}_{\Lebesgue[2](\gamma)}^2$ in common.
    Hence, the proof boils down to bounding $B_\gamma^k(u,u)$, $k \in \braces{1,2,3}$ in terms of $\seminorm{u}_{\Bessel[s](\gamma)}^2$.
    Fortunately, we have done the most difficult part of this already. 
    In \cref{lem:RsEquiBounded} in we found $C_s(\gamma)$ such that
    \begin{equation*}
        B_\gamma^1(u,u)
        =
        \norm{ \Op{\gamma}{s} u }_{\Lebesgue[2](\nu_\gamma)}^2
        \leq 
        C_s(\gamma)^2
        \seminorm*{ u }_{\Bessel[s](\gamma)}^2
        .
    \end{equation*}
    Because of $\abs{u(y) - u(x)}  \leq \norm{D_\gamma u}_{\Lebesgue[\infty]} \dist[\gamma][y][x]$, we may bound $B_\gamma^2$ as follows:
    \begin{align*}
        B_\gamma^2(u,u)
        &=
        \iint_{\Domain^2}
            \pars[\Big]{
                \sdfrac{\dist[\gamma][y][x]}{\abs{\gamma(y) - \gamma(x)}}
            }^{2}
            \abs{\Op{\gamma}{s} \gamma}^2
            \abs[\Big]{
             \sdfrac{u(y) - u(x)}{\dist[\gamma][y][x]}
            }^2
        \dd \mu_\gamma(x,y)
        \leq 
        \distor(\gamma)^{2}
        \,
        \Energy(\gamma)
        \norm{D_\gamma u}_{\Lebesgue[\infty]}^2
        .
    \end{align*}
    Closer inspection of the definition of $B_\gamma^3$ reveals that
    $B_\gamma^3(u,u) \leq 2\, \Energy(\gamma) \norm{D_\gamma u}_{\Lebesgue[\infty]}^2$.
    Finally, using the notation from \cref{def:DefAlphaBeta}, we apply the Morrey inequality \cref{lem:MorreyInequalityGeometric} to bound $\norm{D_\gamma u}_{\Lebesgue[\infty]}$ and \cref{thm:UniformBiLipschitz} to bound $\distor(\gamma)$:
    \begin{align*}
        B_\gamma^2(u,u)
        +
        B_\gamma^3(u,u)
        &\leq
          C_{\Morrey,s-1}^2
            \,
            \Energy(\gamma)
            \,
            \ArcLength(\gamma)^{2s-3}
            \pars[\big]{
                2 
                + 
                C_{\text{distor}}^2\;
                \ArcLength(\gamma)^{\frac{2\alpha+2}{\alpha}}
                \Energy(\gamma)^{2\beta}
            }
        \seminorm{ u }_{\Bessel[s](\gamma)}^2
        \\
        &=
        C_G(\gamma)\seminorm{ u }_{\Bessel[s](\gamma)}^2
        .
    \end{align*}
    Recalling the definition of $C_s(\gamma)$, we conclude that
    \begin{align}
    	C_{G}(\gamma)^2
        &=
        C_{\Morrey,s-1}^2
        \,
        \Energy(\gamma)
        \,
        \ArcLength(\gamma)^{2s-3}
        \pars[\big]{
            2 
            + 
            C_{\text{distor}}^2\;
            \ArcLength(\gamma)^{\frac{2\alpha+2}{\alpha}}
            \Energy(\gamma)^{2\beta}
        }
        \label{eqn:DefinitionFG}
        \\
        &+
        C_{\text{distor}}^{2s} \frac{1}{s-1/2}
        \ArcLength(\gamma)^{\frac{2s(\alpha+1)}{\alpha}}
        \Energy(\gamma)^{2s\beta}
        \pars[\big]{
            1
            +
            C_{\Morrey,s-1}^2
            \ArcLength(\gamma)^{2s-3}
            \seminorm{\gamma}_{\Bessel[s](\gamma)}^2
        }.
        \notag
    \end{align}
In order to estimate $\seminorm{\gamma}_{\Bessel[s](\gamma)}$, we exploit the invariance with respect to reparametrization of the seminorm similarly as in \cref{lem:MorreyInequalityGeometric}:
after a suitable reparametrization, we may assume that $\gamma$ is parametrized by constant speed.
Applying \cref{lem:NormEquiv2} with $u=\gamma$ and observing that $B_\gamma^1(\gamma,\gamma)=B_\gamma^3(\gamma, \gamma)=\frac{1}{2}B_\gamma^2(\gamma,\gamma)=\Energy(\gamma)$ and $\seminorm{\gamma}_{H^1(\gamma)}^2=\ArcLength(\gamma)^2$ yields the precise constant.
Hence, we conclude that $C_G(\gamma)\leq F_G(\ArcLength(\gamma), \Energy(\gamma))$.
\end{proof}

% !TEX root = Main.tex

\subsection{Smoothness}
\label{sec:MetricSmooth}

We are now ready to collect the fruits from the previous sections.
Next we prove smoothness of the metric.
\emph{Some} amount of regularity of the metric is crucial for the existence of geodesics, since the argument relies on the Picard--Lindelöf theorem.
Since proving smoothness does not pose any extra problems, we prove the stronger statement.
As a byproduct, we also prove the smoothness of $\Energy$.

\begin{theorem}\label{thm:GIsSmooth}
    The bilinear forms $B_\gamma^1$, $B_\gamma^2$, and $B_\gamma^3$ from \cref{eq:B1,eq:B2,eq:B3}
    and the metric $G$ from \cref{eq:G} are smooth on $\Mfld$.
\end{theorem}
\begin{proof}
We can write $B_\gamma^i(u,v) = \iint_{\Domain^2} b_\gamma^i(u,v) (x,y) \dd \mu(x,y)$ with the densities
\begin{align*}
    b_\gamma^1(u,v) (x,y)
    &\ceq 
    \inner{\Op{\gamma}{s} u (x,y),\Op{\gamma}{s} v (x,y)} 
    \,
    \varLambda^1(\gamma) (x,y) 
    \abs{\gamma'(x)} \abs{\gamma'(y)}
    ,
    \\
    b_\gamma^2(u,v) (x,y)
    &\ceq 
    b_\gamma^1(\gamma,\gamma) (x,y)
    \inner[\Big]{
        \sdfrac{u(y) - u(x)}{\dist[\Domain][y][x]},
        \sdfrac{v(y) - v(x)}{\dist[\Domain][y][x]}} 
    \, \varLambda^2(\gamma) (x,y)  
    ,
    \\
    b_\gamma^3(u,v) (x,y)
    &\ceq
    b_\gamma^1(\gamma,\gamma) (x,y)
    \pars*{ 
        \inner{D_\gamma u(x),D_\gamma v(x)} 
        +   
        \inner{D_\gamma u(y),D_\gamma v(y)} 
    }
    ,
\end{align*}
where $\varLambda^\beta$ was defined in \cref{lem:LambdaIsAnalytic}. 
Combining  
\cref{lem:AuxAreSmooth}, \cref{lem:LambdaIsAnalytic}, \cref{cor:DifferenceOpIsAnalytic}, and \cref{thm:RsIsSmooth} 
with the well-known rules of differentiation and the Hölder inequalities, it becomes now apparent that each of the maps $(\gamma,u,v) \mapsto b_\gamma^i(u,v)$ is a smooth map with values in $\Lebesgue[1][][\Domain^2;\mu][\R]$.
\end{proof}

As a consequence of this, we can reproduce the following result about the first derivative of the energy $\Energy$. 
Fréchet differentiability and the formula for the derivative have of course been shown before (see \cite{blattreiter1}). However, the formula that we obtain involves the terms $B_\gamma^1$, $B_\gamma^2$, $B_\gamma^3$, and this is a major reason why we included them into the metric.

\begin{theorem}\label{thm:DE}
    The tangent-point energy $\Energy \colon \Mfld \to \R$ is smooth. Its first derivative is given by
    \begin{align*}
        D\Energy(\gamma) \, v 
        =
        2\, B_\gamma^1(\gamma,v) - \DenomExp \, B_\gamma^2(\gamma,v) + B_\gamma^3(\gamma,v)
        \quad
        \text{with}
        \quad
        \DenomExp = 2\,s + 1
        .
    \end{align*}
\end{theorem}
\begin{proof}
    The identity
    $
        \Energy(\gamma)
        =
        B^1_\gamma(\gamma,\gamma)
    $
    and \cref{thm:GIsSmooth} imply that $\Energy$ is smooth.
    To compute the derivative, we notice that the product rule implies
    \begin{equation*}
        D\Energy(\gamma) \, v 
        =
        \iint_{\Domain^2}
            2 \inner*{ 
                \Op{\gamma}{s} \gamma 
                , 
                D\!\pars[\big]{ \eta \mapsto \Op{\eta}{s} \eta} (\gamma) \, v
            }
        \dd \mu_\gamma
        +
        \iint_{\Domain^2}
            \abs{\Op{\gamma}{s}}^2
            \,
            D\!\pars[\big]{ \eta \mapsto \dd \mu_\eta} (\gamma) \, v
        .
    \end{equation*}
    Once again, we abbreviate
    $\triangle \gamma= \gamma(y)-\gamma(x)$ and 
    $\triangle v = v(y)-v(x)$.
    It is straightforward to check that
    \begin{equation*}
        D\!\pars[\big]{ \eta \mapsto \dd \mu_\eta} (\gamma) \, v
        =
        - \inner[\Big]{\sdfrac{\triangle \gamma}{\abs{\triangle \gamma}}, \sdfrac{\triangle v}{\abs{\triangle \gamma}} }
        \dd \mu_\gamma
        +
        \pars*{ 
            \pullx \inner{D_\gamma \gamma, D_\gamma v}
            +
            \pully \inner{D_\gamma \gamma, D_\gamma v}
        }
        \dd \mu_\gamma
        .
    \end{equation*}
    Using \cref{eq:RsGateaux} with $u=\gamma$, we obtain
    \begin{align*}
        \MoveEqLeft
        D\!\pars[\big]{ \eta \mapsto \Op{\eta}{s} \eta } (\gamma) \, v
        \\
        &
        =     
        \Op{\gamma}{s} v 
        - \pars{\pullx \, \cD_\gamma \gamma} \, \Op{\gamma}{s} v
        - \pars{\pullx \, \cD_\gamma \gamma} \pars{\pullx \, \cD_\gamma v}\trans \, \Op{\gamma}{s} \gamma
        - 
        s \, \Op{\gamma}{s} \gamma
        \inner[\Big]{ 
            \sdfrac{ \triangle \gamma }{ \abs{\triangle \gamma} }
            , 
            \sdfrac{ \triangle v }{ \abs{\triangle \gamma} }
        }.
    \end{align*}
    Now we observe that $\cD_\gamma \gamma(x)$ is tangent to $\gamma'(x)$ and that $\Op{\gamma}{s} \gamma(x,y)$ is perpendicular to $\gamma'(x)$. Hence,
    \bignegskip%
    \begin{equation*}
        2 \inner*{ 
            \Op{\gamma}{s} \gamma 
            , 
            D\pars[\big]{ \eta \mapsto \Op{\eta}{s} \eta} (\gamma) \, v
        }
        =
        2 \inner{ \Op{\gamma}{s}\gamma , \Op{\gamma}{s} v  }
        - 2 \, s  \abs{\Op{\gamma}{s}\gamma}^2
        \inner[\Big]{ 
            \sdfrac{ \triangle \gamma }{ \abs{\triangle \gamma} }
            , 
            \sdfrac{ \triangle v }{ \abs{\triangle \gamma} }
        }.
    \end{equation*}
    Combining all this leads directly to the stated formula for $D\Energy(\gamma) \,v$.
\end{proof}

\begin{remark}
    The same techniques can be applied to the other generalized tangent-point energies $\TP^{(p,q)}$ (see \cref{eq:DfnTPpq}).
    If $q$ is an even integer, then $\TP^{(p,q)}$ is smooth. 
    Otherwise, it has at least $\ceil{q-1}$ continuous derivatives. 
    We leave the details to the reader.
\end{remark}

\begin{remark}
    In~\cite{2511.07214} Freches and the first author have strengthened this by showing that the mappings $\gamma \mapsto \smash{\Op{\gamma}{s}}$, $\gamma \mapsto \smash{\mu_\gamma}$, $\gamma \mapsto \smash{\cB^{i}_\gamma}$, $i \in \braces{1,2,3}$, and $\gamma \mapsto \smash{G_\gamma}$ are real analytic.
    Analogously, it follows that $\smash{\TP^{(p,q)}}$ is real analytic if $q$ is an even integer.
\end{remark}

% !TEX root = Main.tex

\subsection{Geodesic distance}
\label{sec:GeodesicDistance}

With a strong Riemannian metric at disposal, we may introduce the \emph{geodesic distance}
\begin{align*}
    \dist[G] \colon \Mfld \times \Mfld \to \intervalcc{0,\infty}. 
\end{align*}
To this end, let $\gamma_0$ and $\gamma_1$ be two points in $\Mfld$. 
Suppose there is a \emph{path} $\Path$ connecting them, i.e., a map $\Path \in \Holder[1][][\UnitInterval][\Mfld]$ such that $\Path(0) = \gamma_0$ and $\Path(1) = \gamma_1$. 
Moreover, all curves $\Path(t)$, $t \in \intervalcc{0,1}$, are embeddings of class $\Holder[1]$. Thus, by definition, such a path is always an isotopy between $\gamma_0$ and $\gamma_1$.
For a ``time'' $t \in \UnitInterval$ we denote the \emph{velocity} by $\dot \Path(t) \ceq \frac{\dd}{\dd t} \Path(t)$.
We define the \emph{path length} $\PathLength$ 
of the path $\Path$ by
\begin{equation}
    \PathLength(\Path) 
    \ceq \int_0^1 \!\! \SmashedSqrt{G_{\Path(t)}( \dot{\Path}(t), \dot{\Path}(t) )} \dd t
    \label{eq:PathLength}
    .
\end{equation}
Then the \emph{geodesic distance} between $\gamma_0$ and $\gamma_1$ is defined as the infimum over the lengths of all paths connecting $\gamma_0$ and $\gamma_1$:
\begin{equation}
    \dist[G][\gamma_1][\gamma_0]
    \ceq
    \inf \myset[\big]{ 
        \;
        \PathLength(\Path)  
    }{
        \text{$\Path \in \Holder[1][][\UnitInterval][\Mfld]$, $\Path(0) = \gamma_0$, $\Path(1) = \gamma_1$}
    }
    \label{eq:GeodesicDistance}
    .
\end{equation}
Of course, if no such path exists, the geodesic distance is $\infty$. 

Clearly $\gamma_0 = \gamma_1$ implies $\dist[G](\gamma,\gamma) = 0$.
Moreover, we have $\dist[G](\gamma_0,\gamma_1) = \dist[G](\gamma_1,\gamma_0)$ because we can just reverse every path from $\gamma_0$ to $\gamma_1$ to get a path from $\gamma_1$ to $\gamma_0$.
One can also show that $\dist[G]$ satisfies the triangle inequality 
$\dist[G][\gamma_2][\gamma_0] \leq \dist[G][\gamma_2][\gamma_1] + \dist[G][\gamma_1][\gamma_0]$:
one merely has to concatenate two $\Holder[1]$-paths from $\gamma_0$ to $\gamma_1$ and from $\gamma_1$ to $\gamma_2$ and to smooth them ever so slightly around the junction. 
Thus, $\dist[G] \colon \Mfld \times \Mfld \to \intervalcc{0,\infty}$ is what is sometimes called an \emph{extended pseudometric} or \emph{extended semimetric}. 
``Extended'' because the value $\infty$ may indeed be attained when $\gamma_0$ and $\gamma_1$ lie in distinct path components or knot classes.
And ``pseudo'' or ``semi'' because we have not yet shown that $\dist[G]$ is definite, i.e., that $\dist[G](\gamma_0,\gamma_1) = 0$ implies $\gamma_0 = \gamma_1$. 
In \cref{cor:GeodesicDistanceIsExtendedMetric} we will obtain definiteness of $\dist[G]$ as a byproduct of our analysis.

For $\gamma \in \Mfld$ and $r > 0$ we denote by
\begin{equation}
    B_r^G(\gamma) 
    \ceq 
    \myset[\big]{ \eta \in \Mfld }{ \dist[G](\eta, \gamma)  < r}
    \label{eq:GBall}
\end{equation}
the ball of radius $r$ centered at $\gamma$.
For brevity, we say a set $\BoundedSet \subset \Mfld$  is \emph{$G$-bounded} if and only if there is a $\gamma \in \Mfld$ and a radius $0 < r < \infty$ such that $\BoundedSet \subset B_{r}^{G}(\gamma)$.

\subsection{Lipschitz continuity}
\label{sec:LipschitzContinuity}

As we mentioned in the introduction, we chose the constituents of $G$ carefully to make the metric $G$ control the tangent-point energy $\Energy$ and the arc length functional $\ArcLength$. We make this more precise now.

\begin{theorem}\label{thm:EnergyLipschitz}
    The function $\sqrt{ \smash[b]{\Energy}} \colon (\Mfld,\dist[G]) \to \R$ is globally Lipschitz continuous:
    \begin{align*}
        \abs[\big]{\SmashedSqrt{\Energy(\gamma_1)} - \SmashedSqrt{\Energy(\gamma_0)} \, }
        \leq 
        \sqrt{1+\DenomExp/4} \; \dist[G][\gamma_1][\gamma_0],
        \quad
        \DenomExp = 2 \, s + 1,
        \quad 
        \text{for all $\gamma_0$, $\gamma_1 \in \Mfld$.}
    \end{align*}
    In particular, we learn: $\Energy$ is bounded on $G$-bounded sets.
\end{theorem}
\begin{proof}
    Let $\Path$ be an arbitrary path.
    Using the fundamental theorem of calculus, the chain rule, and the well-known formula for the derivative of $\ArcLength$, we obtain 
    \begin{equation}
        \abs[\big]{\SmashedSqrt{\Energy(\gamma_1)} - \SmashedSqrt{\Energy(\gamma_0)} \, }
        =
        \abs*{\int_0^1 \frac{\dd}{\dd t} \SmashedSqrt{\Energy(\Path(t))} \dd t}
        \leq
        \int_0^1 \frac{\abs{ D\Energy(\Path(t)) \, \dot{\Path}(t) }}{2 \SmashedSqrt{\Energy(\Path(t))}} \dd t
        .
        \label{eq:EnergyLipschitz1}
    \end{equation}
    We now bound the integrand. To this end, let us fix some $t$ for the moment and abbreviate $\gamma \ceq \Path(t)$ and $u \ceq \dot \Path(t)$. 
    From \cref{thm:DE} we know that
    \begin{equation*}
        D\Energy(\gamma) \, u
        =
        2 \, B_\gamma^1(\gamma,u)
        -
        \DenomExp \, B_\gamma^2(\gamma,u)
        +
        B_\gamma^3(\gamma,u)
        .
    \end{equation*}
    Using the Cauchy-Schwartz inequality first for each of the $B_\gamma^i$ and then for the Euclidean inner product on $\R^3$,
    and the definition \cref{eq:G} of $G$, we obtain
    \begin{align*}
        \abs{ D\Energy(\gamma) \, u }
        &\leq 
        \abs{2 \, B_\gamma^1(\gamma,u)}
        +
        \abs{\DenomExp \, B_\gamma^2(\gamma,u)}
        +
        \abs{B_\gamma^3(\gamma,u)}
        \\
        &\leq
        \SmashedSqrt{2 \, B_\gamma^1(\gamma,\gamma)}
        \SmashedSqrt{2 \, B_\gamma^1(u,u)}
        +
        \SmashedSqrt{\DenomExp \, B_\gamma^2(\gamma,\gamma)}
        \SmashedSqrt{\DenomExp \, B_\gamma^2(u,u)}
        +
        \SmashedSqrt{B_\gamma^3(\gamma,\gamma)}
        \SmashedSqrt{B_\gamma^3(u,u)}
        \\
        &\leq 
        \SmashedSqrt{
            2 \, B_\gamma^1(\gamma,\gamma) + \DenomExp \, B_\gamma^2(\gamma,\gamma) + B_\gamma^3(\gamma,\gamma)
        }
        \SmashedSqrt{
            2 \, B_\gamma^1(u,u) + \DenomExp \, B_\gamma^2(u,u) + B_\gamma^3(u,u)
        }
        .
    \end{align*}
    Most curiously, each of the terms $B_\gamma^i(\gamma,\gamma)$ is tightly connected to the energy $\Energy(\gamma)$.
    Indeed, when we insert the identities
    \begin{align*}
            \inner[\Big]{ 
                \sdfrac{\gamma(y)-\gamma(x)}{\abs{\gamma(y)-\gamma(x)}}, 
                \sdfrac{\gamma(y)-\gamma(x)}{\abs{\gamma(y)-\gamma(x)}} 
            }  = 1
            \qand
            \inner[\big]{D_\gamma \gamma(x) , D_\gamma \gamma(x) } = 1,
    \end{align*}
    into the definitions \cref{eq:B1}, \cref{eq:B2}, and \cref{eq:B3}, then we can make the following observations:
    \begin{equation*}
        B_\gamma^1(\gamma,\gamma) = \Energy(\gamma),
        \quad 
        B_\gamma^2(\gamma,\gamma) = \Energy(\gamma),
        \qand
        B_\gamma^3(\gamma,\gamma) = 2 \, \Energy(\gamma).
    \end{equation*}
    Hence, it turns out that
    \begin{align}
        \begin{split}
        \abs{ D\Energy(\gamma) \, u }
        &\leq 
        \SmashedSqrt{4 + \DenomExp}
        \SmashedSqrt{
            \Energy(\gamma)
        }
        \SmashedSqrt{
            2 \, B_\gamma^1(u,u) + \DenomExp \, B_\gamma^2(u,u) + B_\gamma^3(u,u)
        }
        \\
        &\leq
        \SmashedSqrt{4 + \DenomExp}
        \SmashedSqrt{
            \Energy(\gamma)
        }
        \SmashedSqrt{
            G_\gamma(u,u)
        }
        .
        \end{split}
        \label{eq:DEBound}
    \end{align}
    Substituting this and $\gamma=\Path(t), u=\dotPath$ back into \cref{eq:EnergyLipschitz1}, $\sqrt{\Energy}$ cancels and one obtains
    \begin{equation*}
        \abs[\big]{\SmashedSqrt{\Energy(\gamma_1)} - \SmashedSqrt{\Energy(\gamma_0)} \, }
        \leq
        \SmashedSqrt{1 + \DenomExp/4} \int_0^1 \!\! 
            \SmashedSqrt{G_{\Path(t)} \pars{ \dot{\Path}(t) , \dot{\Path}(t) } } 
        \dd t
        =
        \SmashedSqrt{1 + \DenomExp/4} \, \PathLength(\Path)
        .
    \end{equation*}
    Finally, taking the infimum over all such paths leads to the statement of the theorem.
\end{proof}
Note that this theorem holds true for every Riemannian metric that includes at least the terms $B_\gamma^1$, $B_\gamma^2$, and $B_\gamma^3$.
Neither the $\Lebesgue[2]$-term nor the $\Bessel[1]$-term are needed for this.

This theorem's proof is very much in spirit of the following result.
The latter is of course well-known for metrics that contain the $\Bessel[1]$-term, see for example \cite[Section 7.4]{zbMATH06391675}.
It is also the reason we included the $\Bessel[1]$-term in our metric.
We give a short proof for the reader's convenience.

\begin{lemma}\label{lem:ArcLengthLipschitz}
    The function $\sqrt{ \smash[b]{\ArcLength}} \colon (\Mfld,\dist[G]) \to \R$ is globally Lipschitz continuous:
    \begin{align*}
        \abs[\big]{ 
            \sqrt{ \smash[b]{\ArcLength(\gamma_1)}} - \sqrt{ \smash[b]{\ArcLength(\gamma_0)}}
        \,}
        \leq 
        \tfrac{1}{2} \, \dist[G][\gamma_1][\gamma_0]
        \quad 
        \text{for all $\gamma_0$, $\gamma_1 \in \Mfld$.}
    \end{align*}
    In particular, we learn: $\ArcLength$ is bounded on $G$-bounded sets.
\end{lemma}
\begin{proof}
    Let $\Path$ be a path from $\gamma_0$ to $\gamma_1$.
    In the same way as in the previous proof one shows
    \begin{equation*}
        \abs[\big]{\SmashedSqrt{\ArcLength(\gamma_1)} - \SmashedSqrt{\ArcLength(\gamma_0)}}
        \leq
        \smash{\int_0^1 \frac{ \abs{D\ArcLength(\Path(t)) \, \dot{\Path}(t)} }{ 2 \SmashedSqrt{\ArcLength(\Path(t))}} \dd t}
        .
    \end{equation*}
    For $\gamma \ceq \Path(t)$ and $u \ceq \dotPath(t)$ we have
    \begin{gather*}
        D \ArcLength(\gamma) \, u 
        =
        \textstyle \int_\Domain \inner{ D_\gamma \gamma(x), D_\gamma u(x)} \dvol_\gamma(x)
        \leq 
        \pars*{ \textstyle \int_\Domain \abs{D_\gamma \gamma}^2  \dvol_\gamma(x)}^{1/2}
        \pars*{ \textstyle \int_\Domain \abs{D_\gamma u}^2  \dvol_\gamma(x)}^{1/2}
        \\
        =
        \pars*{ \textstyle \int_\Domain 1 \dvol_\gamma(x)}^{1/2}
        \SmashedSqrt{\inner{ u, u}}_{\Bessel[1](\gamma)}
        =
        \SmashedSqrt{ \ArcLength(\gamma) } \SmashedSqrt{\inner{u, u}}_{\Bessel[1](\gamma)}
        \leq 
        \SmashedSqrt{ \ArcLength(\gamma) } \SmashedSqrt{G_\gamma (u, u)}
        .
    \end{gather*}
    So we get
    $
        \SmashedSqrt{\ArcLength(\gamma_1)} - \SmashedSqrt{\ArcLength(\gamma_0)}
        \leq
        \int_0^1 \!\! \SmashedSqrt{G_\gamma (\dotPath, \dotPath)} \dd t,
    $
    and the lemma follows from taking the infimum over all such paths.
\end{proof}

% !TEX root = Main.tex

\subsection{Norm equivalences}\label{sec:NormEquivOnBoundedSets}

\cref{thm:EnergyLipschitz} and \cref{lem:ArcLengthLipschitz} show that the metric $G$ has good control on arc length $\ArcLength$ and energy $\Energy$.
Both these quantities appear in the norm bounds \cref{lem:NormEquiv2} and \cref{lem:NormEquiv1}.
Our next aim is to make the resulting norm equivalence uniform on $G$-bounded sets (see \cref{lem:NormEquivOnBoundedSets}).
But first, to find a way to control $1/\ArcLength$, we employ a result by Volkmann (see~\cite[Corollary~5.12]{volkmann1}) for curves of length $1$ and 
exploit the scaling behavior of the tangent-point energy $\Energy$.
\begin{lemma}\label{lem:LowerEnergyBound}
	For $s \in \intervaloo{3/2,2}$ every embedding $\gamma \in \Bessel[s](\Domain,\AmbSpace)$ satisfies
	$
		\uppi^2 \leq \Energy(\gamma) \, \ArcLength(\gamma)^{2s-3}
	$.
\end{lemma}
\begin{proof}
    Let $\gamma \in \Mfld$ and rescale it to $\eta \ceq \gamma / \ArcLength(\gamma)$.    
    This is a curve of length $1$.
    Its tangent-point energy is given by $\Energy(\eta) = \Energy(\gamma) \, \ArcLength(\gamma)^{2s-3}$.
    Recall that $\NProj_\eta(x)$ is the orthoprojector onto the orthogonal complement to $\eta'(x)$, i.e., $\NProj_\eta(x)= \id_{\AmbSpace}-\cD_\eta \eta(x)$.
    From~\cite[Corollary~5.12]{volkmann1} we know that
	\begin{align*}
		4 \, \uppi^2 
		\leq 
        \int_{\Domain} \! \int_{\Domain} 
            \sdfrac{1}{R_{\TP}(\eta)(x,y)^2} 
        \dvol_\eta(y) \dvol_\eta(x) 
        =		
        4
        \!
        \int_{\Domain} \! \int_{\Domain} 
            \sdfrac{\abs{\NProj_\eta(x) \pars{\eta(y)-\eta(x)}}^2}{\abs{\eta(y)-\eta(x)}^4} 
        \dvol_\eta(y) \dvol_\eta(x)
        ,
	\end{align*}
    where $R_{\TP}$ denotes the \emph{tangent-point radius}.
	Because of $\abs{\eta(x)-\eta(y)}\leq 1$, we obtain
	\begin{equation*}
		\uppi^2 
		\leq
		\iint_{\Domain^2} 
			\sdfrac{\abs{\NProj_\eta(x) \pars{\eta(y)-\eta(x)}}^2}{\abs{\eta(y)-\eta(x)}^{2s}} 
            \abs{\eta(y)-\eta(x)}^{2s-3}
        \dd \mu_{\eta}(x,y)
        \leq 
        \Energy(\eta) 
        = 
        \Energy(\gamma) \, \ArcLength(\gamma)^{2s-3}
        .
	\end{equation*}
\end{proof}

\begin{corollary}\label{lem:LengthBoundedFromBelow}
	The function 
    $\gamma \mapsto 1 / \ArcLength(\gamma)$ is bounded on $G$-bounded sets.
\end{corollary}
\begin{proof}
    Let $\BoundedSet \subset \Mfld$ be $G$-bounded.
	By \cref{thm:EnergyLipschitz} and \cref{lem:ArcLengthLipschitz}, there is $0<\Energy_{\max}<\infty$ such that $\Energy(\gamma)\leq \Energy_{\max}$ for all $\gamma\in\BoundedSet$.
    Now \cref{lem:LowerEnergyBound} and $2 \, s - 3 > 0$ imply
    \begin{align*}
        \ArcLength(\gamma) 
        \geq 
        \uppi^{2/(2s-3)} \Energy(\gamma)^{-1/(2s-3)}
        &\geq 
        \ArcLength_{\min} 
        \ceq 
        \uppi^{2/(2s-3)} \Energy_{\max}^{-1/(2s-3)}
        .
    \end{align*}
\end{proof}

Note that this approach for bounding $1/\ArcLength$ fails in the scale-invariant case $s=3/2$.
Indeed, for $s=3/2$ the tangent-point energy does not penalize scaling down and therefore does not prevent the embedding from shrinking to a single point.
Note also that it is crucial here that all the curves in $\Mfld$ are \emph{closed} because the tangent-point energy vanishes on straight line segments.

\begin{remark}
    In the same vein one may show,
    relying on $\Energy(\gamma)\geq \pi^2 \ArcLength(\gamma)^{3-2s}$,
    that
    the function $\gamma \mapsto 1 / \Energy(\gamma)$ is bounded on $G$-bounded sets.
    This fact might be of interest on its own. 
    However, in what follows we only require upper bounds on $\Energy$, no lower bounds.
\end{remark}

Now we are in the position of making the norm equivalence from \cref{lem:NormEquiv1}  uniform on $G$-bounded sets:

\begin{lemma}\label{lem:NormEquivOnBoundedSets}
For each $G$-bounded set $\BoundedSet \subset \Mfld$ there are $0 < c \leq C <\infty$ such that 
\begin{equation*}
    c \norm{u}_{G_{\gamma}} 
    \leq
    \norm{u}_{\Bessel[s](\gamma)}
    \leq
    C \norm{u}_{G_{\gamma}} 
    \quad 
    \text{holds true for all $\gamma \in \BoundedSet$ and all $u\in \HSpace$.}
\end{equation*}
\end{lemma}
\begin{proof}
From \cref{lem:NormEquiv2} and \cref{lem:NormEquiv1}
we know that there are continuous functions $F_1$, $F_2 \colon \intervaloo{0,\infty} \times \intervalco{0,\infty} \to \intervalco{0,\infty}$ such that
\begin{equation*}
    \norm{u}_{\Bessel[s](\gamma)}
    \leq 
    F_1(\ArcLength(\gamma),\Energy(\gamma))
    \norm{u}_{G_\gamma}
    \qand 
    \norm{u}_{G_\gamma}
    \leq 
    F_2(\ArcLength(\gamma),\Energy(\gamma))
    \norm{u}_{\Bessel[s](\gamma)}
    .
\end{equation*}

By \cref{thm:EnergyLipschitz}, \cref{lem:LengthBoundedFromBelow}, and \cref{lem:ArcLengthLipschitz} there are $0 < \ArcLength_{\min} \leq \ArcLength_{\max} < \infty$ 
and
 $0 < \Energy_{\max} < \infty$ such that 
$\ArcLength(\gamma) \in \intervalcc{\ArcLength_{\min},\ArcLength_{\max}}$ and $\Energy(\gamma) \in \intervalcc{0,\Energy_{\max}}$ hold true for all $\gamma \in \BoundedSet$.
The set $\intervalcc{\ArcLength_{\min},\ArcLength_{\max}} \times \intervalcc{0,\Energy_{\max}} \subset \intervaloo{0,\infty} \times \intervalco{0,\infty}$ is compact;
hence the continuous functions $F_1$ and $F_2$ are bounded on this set.
This completes the proof.
\end{proof}

Our next aim is to make also the norm equivalence \cref{lem:HsgammaHs} uniform on $G$-bounded sets.

\begin{lemma}\label{lem:HsgammaHsOnBoundedSets}
    For each $G$-bounded set $\BoundedSet \subset \Mfld$ there are $0 < c \leq C <\infty$ such that 
\begin{equation*}
    c \norm{u}_{\Bessel[s]}
    \leq
    \norm{u}_{\Bessel[s](\gamma)}
    \leq
    C \norm{u}_{\Bessel[s]}
    \quad 
    \text{holds true for all $\gamma \in \BoundedSet$ and all $u\in \HSpace$.}
\end{equation*}
\end{lemma}
\begin{proof}
    From \cref{lem:HsgammaHs} we know that 
    $f_{\Bessel[s]}(\gamma) \norm{u}_{\Bessel[s]}
    \leq
    \norm{u}_{\Bessel[s](\gamma)}
    \leq
    F_{\Bessel[s]}(\gamma) \norm{u}_{\Bessel[s]}$.
    Closer inspection of $f_{\Bessel[s]}(\gamma)$ and $F_{\Bessel[s]}(\gamma)$ reveals that we also need control over 
    $\norm{\Speed[\gamma]}_{\Lebesgue[\infty]}$, $\seminorm{\Speed[\gamma]}_{\Bessel[s-1]}$,
    $\norm{\InvSpeed[\gamma]}_{\Lebesgue[\infty]}$, and $\seminorm{\InvSpeed[\gamma]}_{\Bessel[s-1]}$.  
    This will be provided in \cref{lem:BiLipBound} and \cref{lem:HIsBounded} below.
\end{proof}

The remainder of this subsection will be devoted to completing the proof of \cref{lem:HsgammaHsOnBoundedSets}. 
We start with the following auxiliary lemma.
\begin{lemma}\label{lem:LinftyDominatedByG}
For each $G$-bounded set $\BoundedSet \subset \Mfld$ there is $0 < K <\infty$ such that 
\begin{equation*}
    \norm{D_\gamma u}_{\Lebesgue[\infty]}
    \leq
    K \norm{u}_{G_{\gamma}} 
    \quad 
    \text{holds true for all $\gamma \in \BoundedSet$ and all $u\in \HSpace$.}
\end{equation*}
\end{lemma}
\begin{proof}
    From combining the above arguments with the Morrey inequality (\cref{lem:MorreyInequalityGeometric}), we infer
	\begin{equation*}
		\norm{D_\gamma u}_{\Lebesgue[\infty]}
		\leq
		C_{\Morrey,s-1} \, 
		\ArcLength(\gamma)^{s-3/2} 
		\seminorm{ u }_{\Bessel[s](\gamma)}
		\leq 
		C_{\Morrey,s-1} \, 
        \pars[\big]{
            \textstyle
            \sup_{\eta \in \BoundedSet} \ArcLength(\eta) 
        }^{s-3/2} 
        C
		\norm{u}_{G_\gamma}
        .
	\end{equation*}
\end{proof}

The second statement of the following lemma reveals why $G$ prevents immersions from degenerating into curves with vanishing velocity.

\begin{lemma}\label{lem:BiLipBound}
    The maps $\gamma \mapsto \norm{\Speed[\gamma]}_{\Lebesgue[\infty]} = \norm{\gamma'}_{\Lebesgue[\infty]}$
    and 
    $\gamma \mapsto \norm{\InvSpeed[\gamma]}_{\Lebesgue[\infty]} = \norm{\,1/\abs{\gamma'}\,}_{\Lebesgue[\infty]}$
    are bounded on $G$-bounded sets.
\end{lemma}
\begin{proof}
    Let $\BoundedSet \subset \Mfld$ be a $G$-bounded set.
    Then there are $\gamma_0\in \Mfld$ and $0 < R < \infty$ such that $\BoundedSet \subset B_R^{G}(\gamma_0)$.
    For $\gamma_1 \in \BoundedSet$, we may pick a  path $\Path \in \Holder[1][][\intervalcc{0,1}][\Mfld]$ of  length $\leq 2 \, R$ that connects $\gamma_0$ with $\gamma_1$.
    Note that $\Path(\intervalcc{0,1}) \subset B_{2 R}^{G}(\gamma_0)$.
    We define the functions $f$, $g \colon \intervalcc{0,1} \times \Circle \to \R$ by
    \[
        f(t,x) \ceq \Speed[\Path(t)](x) = \abs{\partial_x \Path(t,x)}
        \qand 
        g(t,x) \ceq \InvSpeed[\Path(t)](x) = 1/\abs{\partial_x \Path(t,x)}
        .
    \]
    We intend to apply a Grönwall argument to bound $f$ and $g$.
    From \cref{eq:DSpeed} and \cref{eq:DInvSpeed} we infer
    \begin{align*}
        \abs*{\partial_t f(t,x)}
        &\leq
        f(t,x) \norm{ D_{\Path(t)} \dot \Path(t)}_{\Lebesgue[\infty]}
        \qand
        \abs*{\partial_t g(t,x)}
        \leq
        g(t,x) \norm{ D_{\Path(t)} \dot \Path(t)}_{\Lebesgue[\infty]}
        .
    \end{align*}
    Employing \cref{lem:LinftyDominatedByG} on the bounded set $B_{2 R}^{G}(\gamma_0)$ with $\gamma = \Path(t)$, and $u = \dot \Path(t)$ yields a $0 < K < \infty$ such that
    \begin{equation*}
        \abs*{\partial_t f(t,x)} 
        \leq 
        K \, f(t,x) \norm{ D_{\Path(t)} \dot \Path(t)}_{G_{\Path(t)}}
        \qand
        \abs*{\partial_t g(t,x)} 
        \leq 
        K \, g(t,x) \norm{ D_{\Path(t)} \dot \Path(t)}_{G_{\Path(t)}}
        .
    \end{equation*}
    Now the Grönwall inequality (\cref{thm:GronwallDifferential}) leads us to
	\begin{equation*}
        f(1,x)
        \leq
        f(0,x) \exp \!\pars[\big]{ \textstyle K \int_0^1 \norm{\dot \Path}_{G_{\Path}} \dd t }
        \qand
        g(1,x)
        \leq
        g(0,x) \exp \!\pars[\big]{ \textstyle K \int_0^1 \norm{\dot \Path}_{G_{\Path}} \dd t }
        .
    \end{equation*}
    Substituting the definitions of $f$ and $g$ and recalling that the path length of $\Path$ in $\Mfld$ is bounded by $2\,R$, we obtain
    \begin{equation*}
        \Speed[\gamma_1](x)
        \leq
        \ee^{2 K R} \, \Speed[\gamma_0](x)
        \qand
        \InvSpeed[\gamma_1](x)
        \leq
        \ee^{2 K R} \, \InvSpeed[\gamma_0](x)
        .
    \end{equation*}
    Passing to the suprema over $x \in \Circle$ proves the claim.
\end{proof}

\begin{lemma}\label{lem:HIsBounded}
    The maps 
    $\gamma \mapsto \seminorm{\Speed[\gamma]}_{\Bessel[s-1]}$ 
    and
    $\gamma \mapsto \seminorm{\InvSpeed[\gamma]}_{\Bessel[s-1]}$ are bounded on $G$-bounded sets.
\end{lemma}
\begin{proof}
  	We start by proving that $\gamma \mapsto \seminorm{\Speed[\gamma]}_{\Bessel[s-1]}$ is bounded on $G$-bounded sets.
    Let $\BoundedSet$, $R$, $\gamma_0$, $\gamma_1$, and $\Path$ be as in the previous proof.
    Due to the reverse triangle inequality, we have 
    \[
        \seminorm{\Speed[\gamma]}_{\Bessel[s-1]}^2
        =
        \iint_{\Domain^2}
            \sdfrac{ \abs[\big]{\abs{\gamma'(y)} - \abs{\gamma'(x)}}^2 }{ \dist[\Circle](y,x)^{2s-2} }
            \, \dd \mu(x,y)
        \leq 
        \iint_{\Domain^2}
            \sdfrac{ \abs{\gamma'(y) - \gamma'(x)}^2 }{ \dist[\Circle](y,x)^{2s-2} }
            \, \dd \mu(x,y)
        =
        \seminorm{\gamma}_{\Bessel[s]}^2,
    \]
    hence it suffices to bound the latter.
    In order to apply another Grönwall argument, we define 
    $
        f(t) \ceq \seminorm{\Path(t)}_{\Bessel[s]}
    $.
    Note that $\gamma \mapsto \seminorm{\gamma}_{\Bessel[s]}$ is $1$-Lipschitz with respect to the seminorm $\seminorm{\cdot}_{\Bessel[s]}$, hence $f$ is weakly differentiable, and we have
    \begin{equation}
    	\tfrac{\dd}{\dd t} f(t)
		\leq 
		\seminorm{\dot \Path(t)}_{\Bessel[s]}.
        \label{eq:NormEquivOnBoundedSetsA}
    \end{equation}
    Using \cref{lem:NormEquivOnBoundedSets} on the bounded set $B_{2R}^{G}(\gamma_0)$, there exists a uniform $C>0$ such that
    \[
    	\norm{\dot \Path(t)}_{\Bessel[s](\Path)} \leq C \norm{\dot \Path(t)}_{G_{\Path(t)}}
        \quad 
        \text{for all }
        t\in [0,1]
        .
    \]
    Next we use \cref{lem:HsgammaHs} to bound the stationary norm by the covariant norm:
    \begin{equation*}
    	\norm{\dot  \Path(t)}_{\Bessel[s]}
	    \leq
		f_{\Bessel[s]}(\Path(t))^{-1}
		\norm{\dot \Path(t)}_{\Bessel[s](\Path(t))}
        \leq 
        C \, f_{\Bessel[s]}(\Path(t))^{-1} \norm{\dot \Path(t)}_{G_{\Path(t)}}
        ,
        \quad 
        \text{where}
    \end{equation*}
    \begin{equation*}
    	f_{\Bessel[s]}(\gamma)^{-1}
		=
        \pars*{
            \norm{\Speed[\gamma]}_{\Lebesgue[\infty]}
            +
            \pars{
                1
                +
                2
                \norm{\Speed[\gamma]}_{\Lebesgue[\infty]}^{2s} 
            }
            \norm{\InvSpeed[\gamma]}_{\Lebesgue[\infty]} 
            +
            2 \, C_{\Morrey,s-1}^2 
            \ArcLength(\gamma)^{2 s-3} 
            \seminorm{\Speed[\gamma]}_{\Bessel[s-1]}^2
        }^{1/2}
        .
    \end{equation*}
    With our knowledge from \cref{lem:BiLipBound}, we can bound $\norm{\Speed[\gamma]}_{\Lebesgue[\infty]}$ and  $\norm{\InvSpeed[\gamma]}_{\Lebesgue[\infty]}$ uniformly in $\gamma \in B_{2R}^{G}(\gamma_0)$.
    Moreover, we can also bound $\ArcLength(\gamma)$ uniformly on $B_{2R}^{G}(\gamma_0)$ from above and below.
    Thus, we can find $0 < k_1,\,k_2,\,K_1,\,K_2 < \infty$ so that
    \begin{align*}
    	f_{\Bessel[s]}(\Path(t))^{-1}
        &\leq
        \pars[\big]{
            k_1
            +
            k_2 
            \seminorm{\Speed[\Path(t)]}_{\Bessel[s-1]}^{2}
        }^{1/2}
        \leq
        K_1
        +
        K_2 
        \seminorm{\Path(t)}_{\Bessel[s]}
        =
        K_1
        +
        K_2 \, f(t)
        .
    \end{align*}
    Inserting this into \cref{eq:NormEquivOnBoundedSetsA}, we obtain
    \begin{align*}
        \abs{\dot f(t)}
        &\leq
        C \, K_1 \norm{\dotPath(t)}_{G_{\Path(t)}}
        +
        C \, K_2 \norm{\dotPath(t)}_{G_{\Path(t)}}
        \, f(t) 
        \quad
        \text{for all $t \in \intervalcc{0,1}$.}
    \end{align*}
    Hence, the fundamental theorem of calculus implies 
    \begin{align*}
        f(t) 
        &\leq
        f(0)+
        \textstyle \int_0^t 
            \abs{\dot f(s)}
        \dd s
        \\
        &\leq
		f(0)
		+
		C \, K_1
		\textstyle \int_0^t 
			\norm{\dot \Path(s)}_{G_{\Path(s)}}
		\dd s
		+
		\int_0^t 
            C \, K_2
			\norm{\dot \Path(s)}_{G_{\Path(s)}}
            f(s)
		\dd s
	\\ 
	&\leq
		\underset{\alpha(t)}{\underbrace{
            f(0)
            + 
            2 \, C \, K_1 \,  R
        }}
		+
		\textstyle \int_0^t 
            \underset{\beta(t)}{\underbrace{
                C \, K_2
			    \norm{\dotPath(s)}_{G_{\Path(s)
            }}}}
            \,
            f(s)
		\dd s
		\quad 
		\text{ for all } t\in [0,1].
    \end{align*}
    This allows us to use the integral version of the famous Grönwall inequality (see \cref{thm:GronwallIntegral}):
    \begin{align*}
        f(1) 
        &\leq
        \alpha(1)
        \exp \pars[\Big]{
            \textstyle
            \int_0^1 \beta(s)
            \dd s
        }
        \leq 
        \pars{ f(0) + 2 \, C \, K_1 R}
        \,
        \ee^{C \, K_2 R}
        .
    \end{align*} 
    This proves the first claim. 
    Now we show that $\gamma \mapsto \seminorm{\InvSpeed[\gamma]}_{\Bessel[s-1]}$ is bounded on $G$-bounded sets.
    From
    \begin{equation*}
        \abs{\InvSpeed[\gamma](y) - \InvSpeed[\gamma](x)}
        =
        \abs[\Big]{ 
            \sdfrac{\abs{\gamma'(x)} - \abs{\gamma'(y)}}{\abs{\gamma'(y)} \abs{\gamma'(x)}}
        }
        =
        \InvSpeed[\gamma](x) \, \InvSpeed[\gamma](y) \abs{\Speed[\gamma](x) - \Speed[\gamma](y)}
        ,
    \end{equation*}
    we obtain
    \begin{equation}
        \seminorm{\InvSpeed[\gamma]}_{\Bessel[s-1]}
        =
        \pars[\Big]{
            \iint_{\Domain^2}
                \abs[\Big]{\sdfrac{\InvSpeed[\gamma](y) - \InvSpeed[\gamma](x)}{\dist[\Domain][y][x]^{s-1}}}^2 
            \dd \mu(x,y)
        }^{1/2}
        \leq 
        \norm{\InvSpeed[\gamma]}_{\Lebesgue[\infty]}^2 \seminorm{\Speed[\gamma]}_{\Bessel[s-1]}
        .
        \label{eq:InverseSpeedSeminormBound}
    \end{equation}
    By the previous claim, $\gamma \mapsto \seminorm{\Speed[\gamma]}_{\Bessel[s-1]}$ is bounded on $G$-bounded sets.
    Applying \cref{lem:BiLipBound} once more completes the proof.
\end{proof}

\subsection{Distance bounds}\label{sec:DistanceBounds}

The uniform norm equivalences \cref{lem:NormEquivOnBoundedSets} and \cref{lem:HsgammaHsOnBoundedSets} (uniform on $G$-bounded subsets) finally allow us to derive estimates between the $\Bessel[s]$-\SoboSlobo distance
on the one hand 
and the geodesic distance $\dist[G]$ on the other hand.

\begin{lemma}\label{lem:GLocalDistanceBound}
    For each $G$-bounded subset $\BoundedSet \subset \Mfld$ there is  $0 < K < \infty$ such that 
    \[
        \dist[\Bessel[s]](\gamma_1,\gamma_0)
        \ceq
        \norm{\gamma_1 - \gamma_0}_{\Bessel[s]} 
        \leq K \, \dist[G](\gamma_1,\gamma_0) 
        \quad 
        \text{ for all $\gamma_0$, $\gamma_1 \in \BoundedSet$.}
    \]
\end{lemma}
\begin{proof}
    Let $\gamma \in \Mfld$ and $0 < R < \infty$ such that $\BoundedSet \subset B_{R}^{G}(\gamma)$.
    Clearly, we have $\dist[G](\gamma_1,\gamma_0) \leq \dist[G](\gamma_1,\gamma) + \dist[G](\gamma,\gamma_0) < 2 \, R$.
    Let $\Path \in \Holder[1][][\intervalcc{0,1}][\Mfld]$ be an arbitrary path from $\gamma_0$ to $\gamma_1$ with path length $\PathLength(\Path) \leq 2 \, R$.
    Note that we have $\dist[G]( \Path(t), \gamma ) \leq \dist[G]( \Path(t), \gamma_0 ) + \dist[G]( \gamma_0, \gamma ) < 3 \, R$ for all $t \in \intervalcc{0,1}$, so the image of the path lies in $B_{3R}^{G}(\gamma)$, which is of course a bounded set.
    By \cref{lem:NormEquivOnBoundedSets} and \cref{lem:HsgammaHsOnBoundedSets} we find a $0 < K < \infty$ such that 
    $
        \norm{ \dot \Path(t) }_{\Bessel[s]}
        \leq
        K\norm{ \dot \Path(t) }_{G_{\Path(t)}}
    $ for all $t \in \intervalcc{0,1}$.
    The Jensen inequality yields 
    \begin{equation*}
        \norm{\gamma_1 - \gamma_0}_{\Bessel[s]}
        =
        \norm[\big]{\textstyle\int_0^1 \dot \Path(t) \, \dd t}_{\Bessel[s]}
        \leq 
        \int_0^1 \norm{ \dot \Path(t) }_{\Bessel[s]} \, \dd t
        \leq 
        K \, \int_0^1 \norm{ \dot \Path(t) }_{G_{\Path(t)}} \, \dd t
        =
        K \, \PathLength(\Path).
    \end{equation*}
    Taking the infimum on the right-hand side over all paths with $\PathLength(\Path) \leq 2 \, R$ (which is equal to the infimum over \emph{all} such paths), yields the claim.
\end{proof}   

As direct consequences, we state the following corollaries.

\begin{corollary}\label{cor:GeodesicDistanceGeneratesTopology}
    Let $(\gamma_k)_{k \in N}$ be a sequence in $\Mfld$ that converges to $\gamma \in \Mfld$ with respect to the geodesic distance, i.e., $\dist[G](\gamma_k, \gamma) \longrightarrow 0$ as $k \longrightarrow \infty$.
    Then $(\gamma_k)_{k \in N}$ converges to $\gamma$ also in $\HSpace$.
\end{corollary}

\begin{corollary}\label{cor:GeodesicDistanceIsExtendedMetric}
    The geodesic distance $\dist[G]$ is definite, i.e., 
    $\dist[G](\gamma_0,\gamma_1) = 0$ implies $\gamma_0 = \gamma_1$.
\end{corollary}

\begin{corollary}\label{cor:BoundedIsBounded}
    Every $G$-bounded set is also bounded in $\Bessel[s]$.
\end{corollary}

\begin{remark}
Beware that the reverse statement does not hold true in general:     
Not every $\norm{\cdot}_{\Bessel[s]}$-bounded set is $G$-bounded.
To see this, consider some $\gamma\in \Mfld$ and some $\eta \in \Bessel[s][][\Circle][\AmbSpace] \setminus \Mfld$ that is not an embedding. Define $R \ceq 2 \norm{\eta-\gamma}_{\Bessel[s]}$. Then the $R$-ball with respect to $\norm{\cdot}_{\Bessel[s]}$ centered at $\gamma$ cannot be bounded with respect to $\dist[G]$.
\emph{Assume it is}.
Then there is a path $\Path$ of finite length $\ell$ with respect to $G$ connecting $\Path(0) = \eta$ to $ \Path(1) = \gamma$.
We put $\gamma_n \ceq \Path(1/n)$ and observe that $\gamma_n \longrightarrow \eta$ in $\Holder[1]$.
\cref{thm:EnergyLipschitz} leads to
$\Energy(\gamma_n)^{1/2} \leq \Energy_0^{1/2} \ceq \Energy(\gamma)^{1/2} + \sqrt{1+\DenomExp/4} \, \ell < \infty$.
And \cref{lem:BiLipBound} implies that there is a $c > 0$ such that $\abs{\gamma_n'(x)} \geq c$ for all $x \in \Circle$ and $n \in \N$.
Now \cref{lem:BoundedEnergyAndVelocityImpliesEmbeddedness} implies $\eta \in \Mfld$. \emph{Contradiction.}
\end{remark}

Nonetheless, $\dist[\Bessel[s]]$ dominates $\dist[G]$ \emph{on sufficiently small sets}, which shall be shown next. 
This will be relevant for our proof of metric completeness (see \cref{thm:MetricCompleteness}).
We point out that the following lemma can be derived by more general arguments, just by using that $\iinner{\cdot,\cdot}_{\Bessel[s]}$ and $G$ are both strong and continuous Riemannian metrics on $\Mfld$. 
But we find it notable that our estimates allow us to quantify the radius $r(\gamma)$ and the bi-Lipschitz constants, if needed.

\begin{lemma}\label{lem:GDistanceLocalDistanceBound}
	For every $\gamma\in \Mfld$
	there is an $r = r(\gamma)>0$ and a $C > 0$ such that 
	\begin{align*}
		\dist[G](\gamma_1,\gamma_0) \leq C \, \dist[\Bessel[s]](\gamma_1,\gamma_0   )
        \quad
        \text{for all $\gamma_0$, $\gamma_1 \in B_r(\gamma)$,}
	\end{align*}
	where $B_r(\gamma)$ denotes the open $r$-ball with respect to the $\Bessel[s]$-distance $\dist[\Bessel[s]]$.
\end{lemma}
\begin{proof}
    Let $\gamma \in \Mfld$. 
    Since $\Mfld \subset \Bessel[s][][\Domain][\AmbSpace]$ is open, there is a radius $r>0$ such 
    $B_r(\gamma) \subset \Mfld$.
	Because of continuity, we may shrink $r$ such that the following holds true for every $\eta \in B_r(\gamma)$:
    \begin{equation*} 
        \abs{\eta'(x)}
        \in 
        \intervalcc[\big]{\tfrac{1}{2} \abs{\gamma'(x)}, 2 \abs{\gamma'(x)}},
        \;\, 
        \norm{\eta}_{\Bessel[s]}\in \intervalcc[\big]{\tfrac{1}{2} \norm{\gamma}_{\Bessel[s]}, 2 \norm{\gamma}_{\Bessel[s]}},
        \;\,
        \text{and}
        \;\,
        \Energy(\eta)\in \intervalcc[\big]{\tfrac{1}{2} \Energy(\gamma), 2\Energy(\gamma)}.
    \end{equation*}
    In particular, this also implies that each $\eta \in B_r(\gamma)$ has arc length $\ArcLength(\eta) \in \intervalcc[\big]{\tfrac{1}{2} \ArcLength(\gamma), 2\ArcLength(\gamma)}$.
    By \cref{lem:NormEquiv1} and \cref{lem:HsgammaHs}, there are nonnegative functions $F_G$ and $F_{\Bessel[s]}$ such that 
	\begin{equation*}
		\norm{u}_{G_\eta} 
		\leq 
		F_{G}(\ArcLength(\eta),\Energy(\eta)) 
        \norm{u}_{\Bessel[s](\eta)} 
        \qand 
        \norm{u}_{\Bessel[s](\eta)} 
		\leq 
		F_{\Bessel[s]}(\eta) \norm{u}_{\Bessel[s]} 
	\end{equation*}
    hold true for all $\eta \in B_r(\gamma)$ and all $u\in \HSpace$.
    Next we show that $F_G$ and $F_{\Bessel[s]}$ are bounded on $B_r(\gamma)$. 
    Now, a closer inspection of $F_{G}(\ArcLength(\eta),\Energy(\eta)) $ reveals that it depends continuously on $\ArcLength(\eta)$
    and $\Energy(\eta)$.
    The functions $\ArcLength$
    and $\Energy$ are bounded on $B_r(\gamma)$ by choice of $r$ and therefore, $F_G$ is bounded.
    By \cref{eqn:DefinitionF}, the latter is given by
	\begin{align*}
        F_{\Bessel[s]}(\eta)^2
        =
        \norm{\Speed[\eta]}_{\Lebesgue[\infty]}
        +
        \norm{\InvSpeed[\eta]}_{\Lebesgue[\infty]}
        +
        2  
        \norm{\InvSpeed[\eta]}_{\Lebesgue[\infty]}^{2s-1}
        \norm{\Speed[\eta]}_{\Lebesgue[\infty]}^2 
        \pars[\Big]{
            \norm{\InvSpeed[\eta]}_{\Lebesgue[\infty]}^{2} 
            +
            C_{\Morrey,s-1} \seminorm{\InvSpeed[\eta]}_{\Bessel[s-1]}^2
        }
		.
	\end{align*}
    Again, by choice of $r$, the functions $\eta \mapsto \norm{\Speed_\eta}_{\Lebesgue[\infty]}$ and $\eta \mapsto \norm{\InvSpeed[\eta]}_{\Lebesgue[\infty]}$ are bounded on $B_r(\gamma)$.
    Moreover, with \cref{eq:InverseSpeedSeminormBound} we may bound
    \[ 
        \seminorm{\InvSpeed[\eta]}_{\Bessel[s-1]} 
        \leq 
        \norm{\InvSpeed[\eta]}_{\Lebesgue[\infty]}^2 \seminorm{\Speed[\eta]}_{\Bessel[s-1]}
        \leq
        \norm{\InvSpeed[\eta]}_{\Lebesgue[\infty]}^2 \seminorm{\eta}_{\Bessel[s]}
        \leq
        \norm{\InvSpeed[\eta]}_{\Lebesgue[\infty]}^2 \pars{\norm{\gamma}_{\Bessel[s]}  + r}.
    \]
    This shows that both $F_G$ and $F_{\Bessel[s]}$ are bounded on $B_r(\gamma)$.
    Hence, we may find a $C > 0$ such that
	\begin{equation*}
		\norm{u}_{G_\eta} 
		\leq 
		C \norm{u}_{\Bessel[s]}
        \quad 
        \text{holds true for all $\eta \in B_r(\gamma)$ and all $u\in \HSpace$.}
	\end{equation*}
    This uniform norm bound allows us to tackle the distance bound:
    Let $\gamma_0$, $\gamma_1 \in B_r(\gamma)$.
	Since $B_r(\gamma)\subset \Mfld \subset \HSpace$ is convex, the path 
    $\Path(t) \ceq (1-t) \, \gamma_0 + t \,\gamma_1$ is a path in $ B_r(\gamma) \subset \Mfld$.
    By construction, $\Path$ is the shortest path from $\gamma_0$ to $\gamma_1$ with respect to $\dist[\Bessel[s]]$. Thus, we get
    \begin{align*}
        \dist[G](\gamma_1,\gamma_0)
        \leq 
       \textstyle \int_0^1 \!
            \norm{\dot \Path}_{G_{\Path}}
        \dd t
        \leq
        C \textstyle\int_0^1 \!
			\norm{ 
				\dot\Path
				}_{\Bessel[s]}
		\dd t
        =
		C \norm{\gamma_1 - \gamma_0}_{\Bessel[s]}
        .
	\end{align*}
\end{proof}

\begin{corollary}\label{cor:SameTopology}
The identity map $\id \colon (\Mfld, \dist[G]) \to (\Mfld, \dist[\Bessel[s]])$ is locally bi-Lipschitz continuous.
\end{corollary}

\section{Completeness properties}\label{sec:Completeness}

% !TEX root = Main.tex

\subsection{Banach--Alaoglu property}
\label{sec:BanachAlaoglu}

As we have mentioned in the introduction, Riesz's lemma tells us that we cannot expect the Heine--Borel property to hold  true in our infinite-dimensional Riemannian manifold $(\Mfld,G)$. 
The best we can hope for is some form of \emph{weak} compactness. 
And indeed, the metric $G$ has such a property.

\begin{theorem}[Banach--Alaoglu property]
    \label{thm:BanachAlaoglu}
    Bounded sets in $(\Mfld,\dist[G])$ are relatively compact in $\Mfld$ with respect to the topology inherited from the weak topology on $\HSpace$.
    More precisely, every $G$-bounded set is relatively compact with respect to the weak topology in $\HSpace$ and the weak closure is still contained in $\Mfld$.
\end{theorem}
\begin{proof}
    By the Eberlein--\v{S}mulian theorem (see \cite[Theorem 2.8.6]{zbMATH01196615}), it suffices to show that every
    $G$-bounded sequence $(\gamma_n)_{n \in \N}$ in $\Mfld$ contains a weakly convergent subsequence with limit point $\gamma \in \Mfld$.
    By \cref{cor:BoundedIsBounded}, the sequence $(\gamma_n)_{n \in \N}$ is bounded also in $\norm{\cdot}_{\Bessel[s]}$.
    Hence, the Banach--Alaoglu theorem implies that there is a weakly convergent subsequence with limit point $\gamma \in \HSpace$.
    We are left to show that $\gamma$ lies indeed in $\Mfld$, i.e., that $\gamma$  is an embedding.    
    To this end, we pass over to a convergent subsequence, so that we may assume that $(\gamma_n)_{n \in \N}$ converges weakly to $\gamma$.
    Since the Morrey embedding $\Bessel[s][][\Domain][\AmbSpace] \hookrightarrow \Holder[1][][\Domain][\AmbSpace]$ is compact, the sequence $(\gamma_n)_{n \in \N}$ converges strongly to $\gamma$ in $\Holder[1][][\Domain][\AmbSpace]$.
    Since $(\gamma_n)_{n \in \N}$ is $G$-bounded, \cref{thm:EnergyLipschitz} and \cref{lem:BiLipBound} imply that there exist constants $E_0 < \infty$ and  $c > 0$ such that 
    $
        E(\gamma_n) \leq E_0 
    $
    and
    $
        \abs{\gamma_n'(x)} \geq c
    $
    for all $x \in \Domain$ and $n \in \N$.
    Now \cref{lem:BoundedEnergyAndVelocityImpliesEmbeddedness} proves the claim.
\end{proof}

% !TEX root = Main.tex

\subsection{Metric completeness}
\label{sec:MetricCompleteness}
We are now able to move on to the next of the Hopf--Rinow properties. 

\begin{theorem}[Metric completeness]
    \label{thm:MetricCompleteness}
    With the geodesic distance function $\dist[G] \colon \Mfld \times \Mfld \to \intervalcc{0,\infty}$ from \cref{eq:GeodesicDistance}, the extended metric space $(\Mfld,\dist[G])$ is complete.
\end{theorem}

\begin{proof}
    Let $(\gamma_n)_{n\in \N}$ be a Cauchy sequence in $\Mfld$ with respect to $\dist[G]$.
    By dropping at most finitely many initial elements, we may assume that  $(\gamma_n)_{n\in \N}$ is $G$-bounded.
    Hence, we are allowed to use \cref{lem:GLocalDistanceBound}, which implies that $(\gamma_n)_{n\in \N}$ is a Cauchy sequence in $\Bessel[s]$.
    Hence, the sequence converges strongly (thus weakly) in $\Bessel[s]$ to some $\gamma \in \HSpace$.
    Thus, \cref{thm:BanachAlaoglu} shows that $\gamma \in \Mfld$.
    We are left to show that $(\gamma_n)_{n\in \N}$ converges to $\gamma$ also in $\dist[G]$.
    We apply \cref{lem:GDistanceLocalDistanceBound} to find an $r>0$ and a $K > 0$ such that
    $\dist[G](\eta,\xi) \leq  K \norm{\eta-\xi }_{\Bessel[s]}$ 
    for all $\eta, \xi \in B_r(\gamma)$.
    Since $(\gamma_n)_{n\in \N}$ converges to $\gamma$ with respect to $\dist[\Bessel[s]]$, we know that there is an $N\in \N$ such that
    $\gamma_n \in B_r(\gamma)$ for all $n\geq N$.
    Therefore,
    \begin{equation*}
    	\lim_{n\rightarrow \infty} 
		\dist[G](\gamma,\gamma_n)
        =
        \lim_{N\leq n\rightarrow \infty} 
            \dist[G](\gamma,\gamma_n)
        \leq 
        K 
        \lim_{N\leq n\rightarrow \infty} 
        \norm{\gamma_n-\gamma}_{\Bessel[s]}
        =0.
    \end{equation*}
    Thus, the sequence converges strongly in $\dist[G]$ to $\gamma\in \Mfld$.
\end{proof}

% !TEX root = Main.tex

\subsection{Geodesic completeness}
\label{sec:GeodesicCompleteness}

Let $\gamma_0 \in \Mfld$ and $\gamma_1 \in \Mfld$ be two embeddings. 
Our aim is to find the shortest path between them, i.e., we consider the following problem:
\begin{align}
	\text{Minimize}
	\;
	\PathLength(\Path) 
	\;
	\text{among all}
	\;
	\Path \in \Holder[1][][\UnitInterval][\Mfld]
	\;
	\text{s.t.}
	\;
	\Path(0) = \gamma_0,
	\;
	\Path(1) = \gamma_1.
	\label{eq:MinPathLength}
\end{align}
As it is well-known, it is easier to use the
\emph{Dirichlet} or \emph{path energy} $\PathEnergy$ 
\begin{equation}
	\label{eq:PathEnergy}
	\PathEnergy(\Path) 
	\ceq
	\frac{1}{2} \int_0^1 \norm{\dot{\Path}(t)}_{G_{\Path(t)}}^2 \dd t
\end{equation}
and to solve the following variational problem:
\begin{align}
	\text{Minimize}
	\;
	\PathEnergy(\Path) 
	\;
	\text{among all}
	\;
	\Path \in \Bessel[1](\UnitInterval;\Mfld)
	\;
	\text{s.t.}
	\;
	\Path(0) = \gamma_0
	,
	\;
	\Path(1) = \gamma_1.
	\label{eq:MinPathEnergy}
\end{align}
Here we relaxed the feasible set from $\Holder[1][][\UnitInterval][\Mfld]$ to the larger space
\begin{align*}
	\Bessel[1](\UnitInterval;\Mfld) 
	\ceq 
	\myset[\big]{ 
		\Path \in  \Bessel[1][][\UnitInterval][\HSpace]
	}{
		\text{$\Path(t) \in \Mfld$ for almost all $t \in \UnitInterval$}
	}.
\end{align*}
In contrast to $\Holder[1][][\UnitInterval][\HSpace]$, the space $\Bessel[1][][\UnitInterval][\HSpace]$ is reflexive because the following norm turns it into a Hilbert space:
\begin{equation*}
	\norm{U}_{\Bessel[1]\Bessel[s]}
	\ceq
	\pars[\Big]{
		\textstyle
		\int_0^1 
		\pars[\big]{
			\norm{U(t)}_{\Bessel[s]}^2 
			+
			\norm{\dot{U}(t)}_{\Bessel[s]}^2
		}
	\dd t
	}^{1/2}
	.
\end{equation*}
Minimizing the path energy will essentially lead to the same results:
The Hölder inequality implies
\begin{align}
	\PathLength(\Path)
	=
	\textstyle\int_0^1 1 \cdot \norm{\dot{\Path}(t)}_{G_{\Path(t)}} \dd t 
	\leq 
	\pars[\Big]{\int_0^1 1^2 \dd t}^{1/2}
	\pars[\Big]{\int_0^1 \norm{\dot{\Path}(t)}_{G_{\Path(t)}}^2 \dd t}^{1/2}
	=
	\sqrt{2\, \PathEnergy(\Path) }
	\label{eq:DirichletConstrollsArcLength}
\end{align}
with equality if and only if $\norm{\dot{\Path}(t)}_{G_{\Path(t)}}$ is constant, i.e., if the path $\Path$ has constant speed.
And it is well-known that each critical point of problem~\cref{eq:MinPathEnergy} has this property.	
So each minimizer of problem~\cref{eq:MinPathLength}, once reparametrized to constant speed, will be a minimizer of problem~\cref{eq:MinPathEnergy}. 
And thus every minimizer of problem~\cref{eq:MinPathEnergy} will be a minimizer of problem~\cref{eq:MinPathLength}.

We will show the existence of minimizers of problem \cref{eq:MinPathEnergy} in \cref{sec:ExistenceMinimizingGeodesics}.
For now, we want to analyze the \emph{geodesic equation}. 
That is an ordinary differential equation of second order that characterizes critical points of problem \cref{eq:MinPathEnergy}.
It can be derived from the criticality equation of $\PathEnergy$ via standard techniques from the calculus of variations. 
So we keep the exposition brief in this regard.
Let $\Path \in \Bessel[1](\intervalcc{0,1};\Mfld)$ be a critical point of problem~\cref{eq:MinPathEnergy}.
Then for every $W \in \Holder[\infty][][\intervalcc{0,1}][\Mfld]$ with $W(0) = 0$ and $W(1) = 0$, we have $D\PathEnergy(\Path) \, W = 0$.
After writing down this equation, integrating by parts in the variable $t$, and applying the fundamental lemma of the calculus of variations, we are lead to the following equation:
\begin{align*}
	G_{\gamma_t} \!\pars[\big]{ \ddot{\gamma}_t, w }
	+
	\pars[\big]{ DG(\gamma_t) \, \dot{\gamma}_t } \!\pars[\big]{\dot{\gamma}_t, w }
	-
	\sdfrac{1}{2}  \pars[\big]{ D G(\gamma_t) \, w } \!\pars[\big]{ \dot{\gamma}_t, \dot{\gamma}_t }
	=
	0
	\quad
	\text{for all $w \in T_{\gamma_t} \Mfld$}
	.
\end{align*}
Here we abbreviate $\gamma_t \ceq \Path(t)$, $\dot{\gamma}_t \ceq \dotPath(t)$, and $\ddot{\gamma}_t \ceq \ddotPath(t)$, and we denote with $DG(\gamma) \, u$ the differential of the map $\gamma \mapsto G_\gamma$ in direction $u \in T_{\gamma} \Mfld$. 
From this, we may deduce the \emph{geodesic equation}
\begin{equation}
	\nabla_{\dot \Path} \dot \Path
	\ceq
	\ddot{\Path} + A_{\Path} (\dot{\Path}) \, \dot{\Path} = 0,
	\label{eq:GeodesicEqII}
\end{equation}
where the linear map $A_\gamma \colon T_\gamma \Mfld \to L( T_\gamma \Mfld; T_\gamma \Mfld )$
is implicitly given by
\begin{align}
	G_\gamma \!\pars[\big]{
		A_\gamma (u) \, v, w 
	}
	=
	\sdfrac{1}{2} 
	\pars[\Big]{
		\pars[\big]{DG(\gamma) \, u }(v,w)
		+
		\pars[\big]{DG(\gamma) \, v }(u,w)
		- 
		\pars[\big]{DG(\gamma) \, w }(u,v)
	}
	.
	\label{eq:Koszul}
\end{align}
In Riemannian geometry, $A_\gamma$ is called the \emph{Christoffel symbol} of the so-called  Levi--Civita connection of $G$ at point $\gamma$.
Equation~\cref{eq:Koszul} is known as \emph{Koszul formula}.

The \emph{geodesic initial value problem} or the \emph{geodesic shooting problem} refers to
finding a path $\Path \colon \intervalco{0,T} \to \Mfld$ satisfying
\begin{align}
	\ddot{\Path} + A_{\Path} (\dot{\Path}) \, \dot{\Path} = 0,
	\quad 
	\Path(0) = \gamma_0,
	\qand 
	\dot{\Path}(0) = v_0
	\label{eq:GeodesicShooting}
\end{align}
for some time $T>0$, 
some starting point $\gamma_0 \in \Mfld$, and some starting velocity $v_0 \in T_{\gamma_0} \Mfld$. 
Local existence and uniqueness of solutions follow from standard techniques 
when $G$ is a sufficiently smooth and strong Riemannian metric.
This boils down to rewriting the second order initial value problem \cref{eq:GeodesicShooting}
into the following first order initial value problem by setting $V(t) \ceq \dotPath(t)$:
\begin{align}
	\begin{split}
	\pars[\big]{ \dotPath(t), \dot{V}(t) }
	=
	\pars[\big]{
		V(t)
		,
		-\pars[\big]{A_{\Path(t)}  V(t)} \, V(t)
	}
	\qand
	\pars[\big]{\Path(0), V(0)}
	=
	\pars[\big]{\gamma_0,v_0}
	.
	\end{split}
	\label{eq:GeodesicSpray}
\end{align}
If $G$ is of class $\Holder[2]$, then the right-hand side is of class $\Holder[1]$ (see \cref{eq:Koszul}), thus locally Lipschitz continuous. 
Due to~\cref{thm:GIsSmooth}, the Picard--Lindelöf theorem on Banach spaces applies and shows that this system of ordinary differential equations gives rise to a unique short-time solution.

\begin{theorem}[Short-time existence]\label{thm:ShortTimeExistence}
	For every $\gamma_0 \in \Mfld$ and every $v_0 \in T_{\gamma_0} \Mfld = \HSpace$ there is a $T>0$ so that the initial value problem
	\cref{eq:GeodesicShooting} has a unique solution for $t \in \intervalco{0,T}$.
\end{theorem}
A Riemannian manifold $(X,g)$ is called \emph{geodesically complete} if 
for every choice of initial data
the geodesic initial value has a solution that exists for all times $t > 0$.

If $X$ is finite-dimensional and if $g$ is sufficiently smooth, then the Hopf--Rinow theorem states: $(X,g)$ is metrically complete if and only of $(X,g)$ is geodesically complete.
If $X$ is infinite-dimensional, at least one direction of this is still valid:

\begin{theorem}\label{thm:MetricCompletenessImpliesGeodesicCompleteness}
	Let $X$ be a smooth manifold and let $g$ be a strong and smooth  Riemannian metric on~$X$.
	If $(X,g)$ is metrically complete, then $(X,g)$ is also geodesically complete.
\end{theorem}
A proof of this is given in Proposition~5.1 and Proposition~6.5 in \cite{zbMATH01246686}.
Our manifold $\Mfld$ is smooth (it is an open set in a Hilbert space); our metric $G$ is strong (see \cref{thm:GIsStrong}) and smooth (see \cref{thm:GIsSmooth}); 
and we have already shown that $(\Mfld,G)$ is metrically complete (see \cref{thm:MetricCompleteness}). 
Thus, the above theorem immediately implies the following:

\begin{theorem}[Geodesic completeness]\label{thm:LongTimeExistence}
    The Riemannian manifold $(\Mfld,G)$ is geodesically complete.
\end{theorem}

\begin{remark}\label{rem:Submanifolds}
	We would also like to point out that this can be generalized to submanifolds of $\Mfld$. 
	Suppose that $\cN \subset \Mfld$ is a smooth submanifold and that $\cN$ is a closed set in $\Mfld$.
	Denote by $G_{\cN}$ the Riemannian metric inherited from $G$, i.e., 
	$G_{\cN}(\gamma)(u,v) \ceq G(\gamma)(u,v)$ for $\gamma \in \cN$, $u,v \in T_\gamma \cN$.
	Then $(\cN,G_{\cN})$ is also metrically complete and geodesically complete.
	By \cref{thm:MetricCompletenessImpliesGeodesicCompleteness}, it suffices to show the former.
	Notice that
	$
		\dist[G](\gamma_0,\gamma_1)
		\leq
		\dist[G_{\cN}](\gamma_0,\gamma_1)
	$
	for all $\gamma_0,\gamma_1 \in \cN$.
	Hence, every Cauchy sequence with respect to $\dist[G_{\cN}]$
	is also a Cauchy sequence with respect to $\dist[G]$.
	Since $(\Mfld,\dist[G])$ is complete, the limit point exists in $\Mfld$.
	And since $\cN \subset \Mfld$ is closed, the limit point lies also in $\cN$.
	This shows that $(\cN,G_{\cN})$ is metrically complete (and thus geodesically complete).

	In particular, this applies to submanifolds of the form $\cN = \myset{ \gamma \in \Mfld}{ \varPhi(\gamma) = y}$, where $\varPhi \colon \Mfld \to \cY$ is a smooth map into some other manifold $\cY$ and $y \in \cY$ is a regular point.
	For example, by setting $\cY = \Bessel[s-1][][\Circle][\R]$, $y = 0 \in \cY$, and $\varPhi(\gamma)(x) \ceq \abs{\gamma(x)}^2 -1$,
	this applies to
	\[
		\cN = 
		\myset[\big]{ \gamma \in \Bessel[s][][\Circle][\Sphere^3] }{\text{$\gamma$ is an embedding}}
		\subset 
		\myset[\big]{ \gamma \in \Bessel[s][][\Circle][\R^4] }{\text{$\gamma$ is an embedding}}
		=
		\Mfld,
	\]
	a setting frequently requested by knot theorists.
\end{remark}

% !TEX root = Main.tex

\section{Minimizing geodesics}
\label{sec:ExistenceMinimizingGeodesics}

In this section we consider many one-parameter families in functions spaces on the domains $\Domain$ and $\Domain^2 \ceq \Domain \times \Domain$. 
So, for the sake of brevity, we write $I \ceq \intervalcc{0,1}$ and use the following abbreviations
for general $1 \leq p, q \leq \infty$, $\sigma$, $\tau \in \R$ and a finite-dimensional Hilbert space $Z$:
\begin{align*}
    \Lebesgue[p]\Lebesgue[q]
    &\ceq 
    \Lebesgue[p][][I][\Lebesgue[q][][\Domain][Z]],
    &
    \Bessel[\sigma]\Bessel[\tau]
    &\ceq 
    \Bessel[\sigma][][I][\Bessel[\tau][][\Domain][Z]],
    &
    \Lebesgue[p]\Lebesgue[q][\mu]
    &\ceq 
    \Lebesgue[p][][I][\Lebesgue[][q][\Domain^2; \mu][Z]],
    \\
    \Lebesgue[p]\Bessel[\tau]
    &\ceq 
    \Lebesgue[p][][I][\Bessel[\tau][][\Domain][Z]],
    &
    \Bessel[\sigma]\Lebesgue[q]
    &\ceq 
    \Bessel[\sigma][][I][\Lebesgue[q][][\Domain][Z]],
    &
    \Bessel[\sigma]\Lebesgue[q][\mu]
    &\ceq 
    \Bessel[\sigma][][I][\Lebesgue[q][][\Domain^2; \mu][Z]]
    ,
\end{align*}
where $\mu$ is the measure defined in \cref{eq:DsigmaAndMu}.
Notice that $\mu$ has the same nullsets as the Lebesgue measure on $\Domain^2$, hence their essential supremum norms coincide and we have
$\Lebesgue[\infty][][\Domain^2; \mu][Z] = \Lebesgue[\infty][][\Domain^2][Z]$.

Most of the time we will have $Z = \AmbSpace$ or $Z = \R$.
Occasionally, we will also use $Z = \Hom(\AmbSpace;\AmbSpace)$ (the space of linear maps $\AmbSpace \to \AmbSpace$) with the Frobenius inner product.
The concrete choice will be clear from the context.

\begin{theorem}[Existence of minimizing geodesics]
    \label{thm:ExistenceMinimizingGeodesics}
    Let $\gamma_0 \in \Mfld$ and $\gamma_1 \in \Mfld$ lie in the same path component. 
    Then there is a length-minimizing $G$-geodesic $\Path$ from $\gamma_0$ to $\gamma_1$.
\end{theorem}
\begin{proof}
    We use the direct method of calculus of variations and the classical approach to minimize the path energy $\PathEnergy$ (see \cref{eq:PathEnergy}) instead of the path length $\PathLength$. That means, we consider the variational problem \cref{eq:MinPathEnergy} instead of \cref{eq:MinPathLength}.
    Remember: every minimizer of the former will also be a minimizer of the latter.
    Let $\Path[n]$, $n \in \N$, be paths such that $\Path[n](0) = \gamma_0$ and $\Path[n](1) = \gamma_1$ and 
    \begin{equation}
        \PathEnergy(\Path[n]) \leq \cE_0+ 1/n
        \quad \text{with} \quad
        \cE_0 \ceq \inf_{\Path} \PathEnergy(\Path),
        \label{eq:MinimizingSequence}
    \end{equation}
    where the infimum is taken over all paths $\Path \in \Bessel[1][][I][\Mfld] \subset \Bessel[1]\Bessel[s]$ satisfying $\Path(0) = \gamma_0$ and $\Path(1) = \gamma_1$.
    The Hölder inequality \cref{eq:DirichletConstrollsArcLength} implies the following bound on the path lengths:
    \begin{equation}
        \PathLength(\Path[n]) 
        \leq 
        \cL_0 
        \ceq 
        \sqrt{2 \, \cE_0 + 2}.
        \label{eq:LengthBoundedByDirichlet}
    \end{equation}
    So for each $n \in \N$ and each $t \in I$ we have
    $
        \dist[G]( \Path[n](t), \gamma_0 )
        \leq \PathLength(\Path[n]) \leq \cL_0
    $,
    showing that $\BoundedSet\ceq\myset{ \Path[n](t) }{ n \in \N,\ t \in I} $ is $G$-bounded.
    By combining \cref{lem:NormEquivOnBoundedSets,lem:HsgammaHsOnBoundedSets,lem:GLocalDistanceBound}
    we can find $0 < K < \infty$ such that
    \begin{align*}
        \norm{ \dotPath[n](t)}_{\Bessel[s]}
        \leq 
        K \norm{ \dotPath[n](t)}_{G(\Path[n](t))}
        \qand 
        \dist[\Bessel[s]](\Path[n](t),\gamma_0)
        \leq
        K \, \dist[G](\Path[n](t),\gamma_0)
    \end{align*}
    hold true for all $n \in \N$ and all $t \in I$.
    We claim that $(\Path[n])_{n \in \N}$ is bounded in $\Bessel[1]\Bessel[s]$. 
    Indeed, we have
    \[
        \norm{\Path[n](t) - \gamma_0}_{\Bessel[s]}
        = 
        \dist[\Bessel[s]](\Path[n](t),\gamma_0)
        \leq 
        K \, \dist[G](\Path[n](t),\gamma_0)
        \leq 
        K \, \PathLength(\Path[n]) 
        \leq 
        K  \, \cL_0
    \]
    and
    \begin{align*}
        \norm{\Path[n]}_{\Bessel[1]\Bessel[s]}
        &\leq
        \pars[\big]{
            \textstyle
            \int_0^1 
                \norm{\Path[n](t)}_{\Bessel[s]}^2 
            \dd t
        }^{1/2}
        +
        \pars[\big]{
            \textstyle
            \int_0^1 
                \norm{\dotPath[n](t)}_{\Bessel[s]}^2 
            \dd t
        }^{1/2}
        \\
        &\leq 
        \sup_{t \in I}  \norm{\Path[n](t)}_{\Bessel[s]}
        +
        \,K \pars[\big]{ 
            \textstyle
            \int_0^1 
                \norm{ \dotPath[n](t)}_{G(\Path[n](t)) }^2                
            \dd t 
        }^{1/2}
        \\
        &\leq 
        \norm{\gamma_0}_{\Bessel[s]}
        +
        \sup_{t \in I}  
        \norm{\Path[n](t) - \gamma_0}_{\Bessel[s]}
        +
        K \, \cL_0
        \leq
        \norm{\gamma_0}_{\Bessel[s]}
        +
        2 \, K \, \cL_0.
    \end{align*}
    Since $\Bessel[1]\Bessel[s]$ is a Hilbert space and thus reflexive, and because of the Banach--Alaoglu theorem, there is a subsequence $(\Path[n_k])_{k \in \N}$ that converges weakly to some $\Path[\infty]$ in $\Bessel[1]\Bessel[s]$.
    We have yet to show that $\Path[\infty](t) \in \Mfld$ for all $t \in I$ and that $\Path[\infty]$ is a minimizer of~$\PathEnergy$.

    Due to a generalization of the Arzelà--Ascoli theorem~\cite[Thm.~47.1]{MunkresTopo}, the embedding
    \begin{equation*}
        \Bessel[1]\Bessel[s]
        \hookrightarrow
        \Holder[0]\Holder[1] \ceq \Holder[0][][I][\Holder[1][][\Domain][\AmbSpace]]
    \end{equation*}
    is compact, hence $(\Path[n_k])_{k \in \N}$ converges in $\Holder[0]\Holder[1]$. 
    This implies 
    \begin{equation*}
        \Path[\infty](0) = \lim_{k \to \infty} \Path[n_k](0) = \gamma_0
        \qand
        \Path[\infty](1) = \lim_{k \to \infty} \Path[n_k](1) = \gamma_1
        .
    \end{equation*}
    The set $\BoundedSet$ is $G$-bounded, so \cref{thm:EnergyLipschitz} and \cref{lem:BiLipBound} provide $0 \leq E_0 < \infty$ and $0 < c \leq C < \infty$ such that 
    \begin{equation}
        \label{eq:ExistenceMinimizingGeodesics:VelocityBounds}
        \Energy(\Path[n](t)) \leq E_0 
        \qand
        c \leq \abs{\partial_x \Path[n](t,x)} \leq C 
        \quad \text{for all $n \in \N$, $t \in I$, $x \in \Domain$.}
    \end{equation}
    Since $\Path[n_k](t)$ converges to $\Path[\infty](t)$ in $\Holder[1]$,
    \cref{lem:BoundedEnergyAndVelocityImpliesEmbeddedness}
    implies that $\Path[\infty](t)$ is an embedding.
    By the Lions--Magenes lemma, see~\cite[Chapter 1, §3.1]{zbMATH03396393}, the $\Path$  has a continuous representative, i.e., we may assume that $\Path \in \Holder[0]\Bessel[s]$. 
    Hence, for every $t \in I$ we have $\Path[\infty](t) \in \HSpace$, thus $\Path[\infty](t) \in \Mfld$.
    So this shows that $\Path[\infty] \in \Bessel[1][][I][\Mfld]$ is indeed a path in the space $\Mfld$ of embeddings.

    Finally, we show that $\Path[\infty]$ is a minimizer of the path energy $\PathEnergy$.
    To this end, we employ \cref{thm:DirichletEnergyIsLowerSemicontinuous} below. It states that  $\PathEnergy$ is weakly sequentially lower semicontinuous, hence
    \begin{equation*}
        \PathEnergy(\Path[\infty]) 
        \leq 
        \liminf_{k \to \infty} \PathEnergy(\Path[n_k]) 
        =
        \liminf_{k \to \infty} \pars*{\PathEnergy_0 + 1/n_k }
        =
        \PathEnergy_0.
    \end{equation*}
    Thus, $\Path[\infty]$ is indeed a minimizing geodesic.
\end{proof}

\begin{remark}\label{rem:ConstraintMinimizingGeodesics}
    One can extend this proof also to the case of minimizing geodesics in a differentiable submanifold $\ConstraintMfld \subset \Mfld$ (see also~\cref{rem:Submanifolds}).
    The only point of failure here is that the weak limit $\Path$ of the minimizing subsequence $\Path_{n_k}$ may not be contained in $\ConstraintMfld$ anymore.
    So one needs some requirement that guarantees that $\Bessel[1][][I][\ConstraintMfld]$ is weakly closed in $\Bessel[1][][I][\Bessel[s]]$.
    We give an example construction that often occurs in practice. (This can be skipped on first reading.)

    Suppose that $\ConstraintMfld$ is given as a zero set of a submersion $\ConstraintMap \colon \Mfld \to \cX$ into some Banach space~$\cX$.
    Moreover, suppose that $\varPhi$ factors through some $\Bessel[s-\varepsilon] \ceq \Bessel[s-\varepsilon][][\Circle][\AmbSpace]$, $\varepsilon > 0$, i.e., there is a continuous map $\varPsi \colon \Bessel[s-\varepsilon] \to \cY$ into some Banach space~$\cY$ and that there is some continuous embedding $\iota \colon \cX \hookrightarrow \cY$ such that the following diagram commutes:
    \begin{equation*}
        \begin{tikzcd}[]
            \Mfld
                \ar[r, "{\ConstraintMap}"]
                \ar[d,hook, "{}"]
            &\cX
                \ar[d,hook, "{\iota}"]	
            \\
            \Bessel[s-\varepsilon]
                \ar[r, "{\varPsi}"]	
            &\cY
            \nospaceperiod
        \end{tikzcd}	
    \end{equation*}
    Then the minimizing subsequence satisfies
    $
        \varPsi(\Path_{n_k}(t))
        = 
        \iota(\ConstraintMap(\Path_{n_k}(t)))
        =
        0
    $
    for all $t \in I$ and all $k \in \N$.
    Because the embedding $\Bessel[1]\Bessel[s] \hookrightarrow \Holder[0]\Bessel[s-\varepsilon]$ is compact, and because $\varPsi$ is continuous, we get
    \[
        \iota(\varPhi(\Path(t)))
        =
        \varPsi(\Path(t)) = \lim_{k \to \infty} \varPsi(\Path_{n_k}(t)) = 0
        \quad \text{for all $t \in I$.}
    \]
    Since $\iota$ is injective, this implies $\varPhi(\Path(t)) = 0$ and thus $\Path(t) \in \ConstraintMfld$ for all $t \in I$.
    For example, one may use $\varPsi \colon \Bessel[1][][\Circle][\AmbSpace] \to \cY = \R$, $\varPsi(\gamma) = \ArcLength(\gamma) - L_0$
    to constrain the arc length of curves.
    Or one may require that curves are parametrized by arc length by using 
    $\varPsi \colon \Bessel[s-\varepsilon][][\Circle][\AmbSpace] \hookrightarrow \Holder[1][][\Circle][\AmbSpace] \to \cY = \Holder[0][][\Circle][\R]$,
    $\varPsi(\gamma) \ceq \abs{\gamma'}^2-1$.
    Moreover, we may constrain our curves to lie on a submanifold $\varSigma \subset \AmbSpace$ by representing $\varSigma$ as a level set of a function $\psi$ and put $\varPsi(\gamma) \ceq \psi \circ \gamma - c$, where $c$ is a regular value of $\psi$ such that $\varSigma = \psi^{-1}(\braces{c})$.
    In particular, $\AmbDim = 4$ and the unit sphere $\varSigma = \Sphere^3 \subset \R^4$ can be realized this way.
\end{remark}

\subsection{Weak sequential lower semicontinuity of path energy}

We write $u_n \wconverges[n \to \infty][X] u_\infty$ if the sequence $(u_n)_{n \in \N}$ converges weakly to $u_\infty$ in the Banach space $X$.

\begin{theorem}
    \label{thm:DirichletEnergyIsLowerSemicontinuous}
    Let $\gamma_0$, $\gamma_1 \in \Mfld$ and let 
    $\Path[n] \in \Bessel[1][][I][\Mfld]$, $n \in \N \cup \braces{\infty}$ be paths in $\Mfld$.
    Moreover, suppose that $\Path[n] \smash{\wconverges[n \to \infty]} \Path[\infty]$ in $\Bessel[1]\Bessel[s]$.
    Then 
    $
        \PathEnergy(\Path[\infty]) \leq \liminf_{n \to \infty} \PathEnergy(\Path[n]).
    $
\end{theorem}
\begin{proof}
    We point out that it is part of the assumptions that $\Path[\infty] \in \Bessel[1][][I][\Mfld]$, i.e., that $\Path_{\infty}$ is a path in $\Mfld$.
    The key idea, which we learnt from \cite[Section~5, in particular, Remark~5.4]{zbMATH06502944}, 
    is to write the path energy $\PathEnergy(\Path)$ (see \cref{eq:PathEnergy}) of a general path $\Path$ as a sum of squared $\Lebesgue[2]$-norms: 
    \begin{equation}
        2 \, \PathEnergy(\Path)
        =
        \int_0^1 \norm{\dot{\Path}(t)}_{G_{\Path(t)}}^2 \dd t
        =
        \sum_{i=1}^4
        \norm{ A^{(i)}_{\Path} \dotPath}_{\Lebesgue[2]\Lebesgue[2][\mu]}^2
            +
            \norm{ A^{(5)}_{\Path} \dotPath }_{\Lebesgue[2]\Lebesgue[2]}^2
            +
            \norm{ A^{(6)}_{\Path} \dotPath }_{\Lebesgue[2]\Lebesgue[2]}^2
            ,
            \label{eq:SumOfSquares}
    \end{equation}
    where the operators $A^{(k)}$ are defined as follows:
    \begin{align*}
        A^{(1)}_{\Path} u (t,x,y) 
        &\ceq 
        \varXi_{\Path}(t,x,y)
        \;
        \Op{\Path}{s} u(t,x,y),
        \\
        A^{(2)}_{\Path} u (t,x,y) 
        &\ceq 
        \varXi_{\Path}(t,x,y)
        \;
        \Op{\Path}{s} \Path(t,x,y)
        \pars[\Big]{
            \sdfrac{
                u(t,y) - u(t,x)
            }{
                \abs{\Path(t,y) - \Path(t,x)}
            }
        }^\transp
        ,
        \\
        A^{(3)}_{\Path} u (t,x,y) 
        &\ceq 
        \varXi_{\Path}(t,x,y)
        \;
        \Op{\Path}{s} \Path(t,x,y)
        \otimes
        \pars{D_{\Path} u(t,x)}
        ,
        \\
        A^{(4)}_{\Path} u (t,x,y) 
        &\ceq 
        \varXi_{\Path}(t,x,y)
        \;
        \Op{\Path}{s} \Path(t,x,y)
        \otimes
        \pars{D_{\Path} u(t,y)}
        ,
        \\
        A^{(5)}_{\Path} u (t,x) 
        &\ceq 
        \SmashedSqrt{\abs{\partial_x \Path(t,x)}}
        \,
        D_{\Path} u(t,x)
        ,
        \\
        A^{(6)}_{\Path} u (t,x) 
        &\ceq 
        \SmashedSqrt{\abs{\partial_x \Path(t,x)}}
        \, 
        u(x),
    \end{align*}
    where 
    \[
        \varXi_{\Path}(t,x,y)
        \ceq 
        \SmashedSqrt{\abs{\partial_x \Path(t,x)}}
        \SmashedSqrt{\abs{\partial_x \Path(t,y)}}
        \,
        \varLambda^{1/2}(\Path(t))(x,y)
    \]
    and with $\varLambda^{1/2}$ from \cref{lem:LambdaIsAnalytic}.
    Since the squared norms
    $\smash{\norm{ \cdot }_{\Lebesgue[2]\Lebesgue[\infty][\mu]}^2}$ 
    and 
    $\smash{\norm{ \cdot }_{\Lebesgue[2]\Lebesgue[2]}^2}$
    are weakly sequentially lower semicontinuous, it suffices to show that 
    \[
        A^{(k)}_{\Path[n]} \dotPath[n] 
        \wconverges[n \to \infty][\Lebesgue[2]\Lebesgue[2][\mu]] 
        A^{(k)}_{\Path[\infty]} \dotPath[\infty]
        \;\;\text{ for $k \in \set{1,2,3,4}$}
        \qqand
        A^{(k)}_{\Path[n]} \dotPath[n] 
        \wconverges[n \to \infty][\Lebesgue[2]\Lebesgue[2]]
        A^{(k)}_{\Path[\infty]} 
        \dotPath[\infty]
        \;\;\text{for $k \in \set{5,6}$}
        .
    \]
    Recall that the embedding $\Bessel[1]\Bessel[s] \hookrightarrow \Holder[0]\Holder[1]$ is compact due to the Arzelà--Ascoli theorem (see \cite[Thm.~47.1]{MunkresTopo}).
    This yields strong convergence
    $\Path[n] \converges[n \to \infty] \Path[\infty]$ in $\Holder[0]\Holder[1]$ and thus
    \[
        \SmashedSqrt{\abs{\partial_x \Path[n]}}
        \converges[n \to \infty][\Lebesgue[\infty]\Lebesgue[\infty]]
        \SmashedSqrt{\abs{\partial_x \Path[\infty]}}
        .
    \]
    We have $\Path[\infty] \in \Bessel[1][][I][\Mfld] \subset \Holder[0][][I][\Emb[1][][\Circle][\AmbSpace]]$.
    This shows strong convergence
    \begin{equation*}
        1/\abs{\partial_x \Path[n]}
        \converges[n \to \infty][\Lebesgue[\infty]\Lebesgue[\infty]]
        1/\abs{\partial_x \Path[\infty]}
        .
    \end{equation*}
    By assumption, we have weak convergence
    \[
        \dotPath[n]
        \wconverges[n \to \infty][\Lebesgue[2]\Lebesgue[2]]
        \dotPath[\infty]
        \qand
        \partial_x \dotPath[n]
        \wconverges[n \to \infty][\Lebesgue[2]\Lebesgue[2]]
        \partial_x \dotPath[\infty]
    \]
    Hence, \cref{lem:StronglyConvergentTimesWeaklyConvergent} 
    and
    $D_{\Path} \dotPath = \pars{\partial_{x} \dotPath} / \abs{\partial_{x}\Path}$
    imply
    $
        A^{(k)}_{\Path[n]} \dotPath[n] 
        \wconverges[n \to \infty]
        A^{(k)}_{\Path[\infty]} \dotPath[\infty]
    $
    in $\Lebesgue[2]\Lebesgue[2]$ for $k \in \braces{5,6}$.

    So let us move on to $k \in \braces{1,2,3,4}$.
    By \cref{lem:LambdaIsAnalytic}, the map $\varLambda^{1/2} \colon \Emb[1][][\Circle][\AmbSpace] \to \pars{\Lebesgue[\infty]\Lebesgue[\infty][\mu]}$ is continuous.
    Hence, strong convergence $\Path[n] \converges[n \to \infty] \Path[\infty]$ in $\Holder[0]\Holder[1]$ implies
    \begin{align*}
        \varLambda^{1/2}(\Path[n]) 
        \converges[n \to \infty][\Lebesgue[\infty]\Lebesgue[\infty][\mu]] \varLambda^{1/2}(\Path[\infty])
        \qand 
        \varXi_{\Path[n]} \converges[n \to \infty][\Lebesgue[\infty]\Lebesgue[\infty][\mu]] \varXi_{\Path[\infty]}
        .
    \end{align*}
    So, in light of \cref{lem:StronglyConvergentTimesWeaklyConvergent}, it suffices to show that 
    \begin{align*}
        \Op{\Path[n]}{s} \dotPath[n]
        \wconverges[n \to \infty][\Lebesgue[2]\Lebesgue[2][\mu]]
        \Op{\Path[\infty]}{s} \dotPath[\infty]
        \qqand
        \pars{\Op{\Path[n]}{s} \Path[n]} \otimes w_n
        \wconverges[n \to \infty][\Lebesgue[2]\Lebesgue[2][\mu]]
        \pars{\Op{\Path[\infty]}{s} \Path[\infty]} \otimes w_\infty
        ,
    \end{align*}
    where $w_n$ can be any of the following:
    \[
        w_n(x,y,t) = D_{\Path[n]} \dotPath[n](t,x)
        ,
        \quad
        w_n(x,y,t) = D_{\Path[n]} \dotPath[n](t,y)
        ,
        \qor 
        w_n(x,y,t) = \sdfrac{
            \dotPath[n](t,y) - \dotPath[n](t,x)
        }{
            \abs{\Path[n](t,y) - \Path[n](t,x)}
        }
        .
    \]
    This is quite technical, so we delegate this to \cref{lem:WeakConvergenceRs} below.
\end{proof}

\newpage 

\subsection{Supplement}\label{sec:Supplement}

We briefly recall the difference notation $\triangle u(x,y) \ceq u(y) - u(x)$.

\begin{lemma}\label{lem:WeakConvergenceRs} 
    Let
    $\Path[n] \in \Bessel[1][][I][\Mfld] \subset \Bessel[1]\Bessel[s]$, 
    $v_n \in \Bessel[1]\Bessel[s]$, 
    and 
    $u_n \in \Lebesgue[2]\Bessel[s]$, 
    $n \in \N \cup \braces{\infty}$
    satisfy
    \[
        \Path[n] \wconverges[n \to \infty][\Bessel[1]\Bessel[s]] \Path[\infty]
        ,
        \quad
        v_n \wconverges[n \to \infty][\Bessel[1]\Bessel[s]] v_\infty
        ,
        \qand
        u_n \wconverges[n \to \infty][\Lebesgue[2]\Bessel[s]] u_\infty
        .
    \]
    Let 
    \[
        w_n \in 
        \smash{
        \braces[\Big]{
            \pullx \, D_{\Path[n]} u_n,
            \pully \, D_{\Path[n]} u_n,
            \pullx \, \cD_{\Path[n]} u_n,
            \pully \, \cD_{\Path[n]} u_n,
            \sdfrac{\triangle u_n}{\abs{\triangle \Path[n]}}
        }}
        .
    \]
    Then we have weak convergence 
    \[
        \Op{\Path[n]}{s} u_n 
        \wconverges[n \to \infty][\Lebesgue[2]\Lebesgue[2][\mu]] 
        \Op{\Path_\infty}{s} u_\infty
        \qand
        \pars{\Op{\Path[n]}{s} v_n} \otimes w_n
        \wconverges[n \to \infty][\Lebesgue[2]\Lebesgue[2][\mu]]
        \pars{\Op{\Path[\infty]}{s} v_\infty} \otimes w_\infty
        .
    \]
    Here, $\otimes$ denotes the pointwise tensor product in finite-dimensional vector spaces, i.e.,
    \[
        \pars[\big]{\pars{\Op{\Path[n]}{s} v_n} \otimes w_n}(t,x,y) 
        = 
        \Op{\Path[n]}{s} v_n(t,x,y) \otimes w_n(t,x,y).
    \]
\end{lemma}
\begin{proof}
    First we pick a smooth embedding $\eta \colon \Domain \to \AmbSpace$,
    and define the following functions:
    \begin{align*}
        \varTheta_{n}(t,x,y)
        \ceq
        \sdfrac{ 
            \abs{\eta(y) - \eta(x) }^s
        }{ 
            \abs{\Path[n](t,y) - \Path[n](t,x)}^s
        }
        =
        \sdfrac{ \varLambda^s(\Path[n](t))(x,y) }{ \varLambda^s(\eta)(x,y) }
        \quad 
        \text{for $t \in I$, $x$, $y \in \Domain$, $n \in \N \cup \braces{\infty}$,}
    \end{align*}
    where we use $\varLambda^s$ from \cref{lem:LambdaIsAnalytic} again.
    Using \cref{eq:RDecomposition}, we may write
    \begin{align}
        \Op{\Path[n]}{s} u_n
        &=
        \varTheta_{n} \pars{\Op{\eta}{s} \, u_n}
        -
        \varTheta_{n} 
        \pars{\pullx \, \cD_{\Path[n]} u_n}
        \pars{\Op{\eta}{s} \, \Path[n]}
        .
        \label{eq:WeakConvergenceRs:2a}
    \end{align}
    Since $\Op{\eta}{s}$ is a bounded linear operator, cf.~\cref{thm:RsIsSmooth}, we have
    \begin{equation}
        \Op{\eta}{s} \, \Path[n] 
        \wconverges[n \to \infty][\Bessel[1]\Lebesgue[2][\mu]] 
        \Op{\eta}{s} \, \Path[\infty]
        ,
        \quad
        \Op{\eta}{s} \, v_n 
        \wconverges[n \to \infty][\Bessel[1]\Lebesgue[2][\mu]] 
        \Op{\eta}{s} \, v_\infty
        ,
        \qand 
        \Op{\eta}{s} \, u_n \wconverges[n \to \infty][\Lebesgue[2]\Lebesgue[2][\mu]] \Op{\eta}{s} \, u_\infty 
        .
        \label{eq:WeakConvergenceRs:3a}
    \end{equation}
    Compactness of $\Bessel[1]\Bessel[s] \hookrightarrow \Holder[0]\Holder[1]$ (see \cite[Thm.~47.1]{MunkresTopo})
    and \cref{lem:LambdaIsAnalytic}
    imply
    \begin{align}
        \varTheta_{n} \converges[n \to \infty]
        [\Lebesgue[\infty]\Lebesgue[\infty][\mu]] 
        \varTheta_\infty 
        \qand 
        \cD_{\Path[n]} v_n
        \converges[n \to \infty][\Lebesgue[\infty]\Lebesgue[\infty]]
        \cD_{\Path[n]} v_\infty
        , \quad \text{hence} \quad 
        \pullx \, \cD_{\Path[n]} v_n
        \converges[n \to \infty][\Lebesgue[\infty]\Lebesgue[\infty][\mu]]
        \pullx \, \cD_{\Path[n]} v_\infty
        .
        \label{eq:WeakConvergenceRs:4a}
    \end{align}
    Applying
    \cref{lem:StronglyConvergentTimesWeaklyConvergent} 
    to the first summand in \cref{eq:WeakConvergenceRs:2a},
    we are lead to
    \begin{equation*}
        \varTheta_n \pars{\Op{\eta}{s} \, u_n}
        \wconverges[n \to \infty][\Lebesgue[2]\Lebesgue[2][\mu]] 
        \varTheta_\infty \pars{\Op{\eta}{s} \, u_\infty}
        .
    \end{equation*}
    However, a more sophisticated technique is required for the second summand in \cref{eq:WeakConvergenceRs:2a} as $\pullx \, \cD_{\Path[n]} u_n$ need not converge strongly in $\Lebesgue[\infty]\Lebesgue[\infty][\mu]$ because $u_n$ may lack smoothness in the time variable.
    We are going to employ \cref{lem:WeakConvergenceHelper}, which allows us to ``borrow'' that smoothness from $\Op{\eta}{s} \, \Path[n]$.
    To this end, we use the identity \cref{eq:DDecomposition} to further decompose the expression:
    \[
        \varTheta_{n} 
        \pars{\pullx \, \cD_{\Path[n]} u_n}
        \pars{\Op{\eta}{s} \, \Path[n]}
        =
        \pars{\pullx \, \cD_{\eta} u_n}
        \pars{\pullx \, \cD_{\Path[n]} \eta}
        \, \varTheta_{n} 
        \pars{\Op{\eta}{s} \, \Path[n]}
        .
    \]
    Moreover, we define
    \[
        U_n(t,x) \ceq \textstyle \int_0^t u_n(t,r) \dd r
        ,
        \quad 
        \varphi_n \ceq \pullx \, \cD_{\eta} U_n
        ,
        \quad
        \zeta_n \ceq \pars{\pullx \, \cD_{\Path[n]} \eta} \, \varTheta_n
        ,
        \qand 
        \psi_n \ceq \Op{\eta}{s} \, \Path[n]
        .
    \]
    This way we have 
    \[
        \varTheta_{n} 
        \pars{\pullx \, \cD_{\Path[n]} u_n}
        \pars{\Op{\eta}{s} \, \Path[n]}
        =
        \pars{\partial_t \varphi_n} \, \zeta_n \, \psi_n
        .
    \]
    By the construction of $U_n$, by \cref{eq:WeakConvergenceRs:4a}, and by \cref{eq:WeakConvergenceRs:3a}, 
    we have
    \[
        U_n 
        \wconverges[n \to \infty][\Bessel[1]\Bessel[s]]
        U_\infty
        ,
        \quad
        \zeta_n
        \converges[n \to \infty][\Lebesgue[\infty]\Lebesgue[\infty][\mu]]
        \zeta_\infty
        ,
        \quad
        \psi_n
        \wconverges[n \to \infty][\Bessel[1]\Lebesgue[2][\mu]]
        \psi_\infty
        .
    \]
    The boundedness of $\Lebesgue[2]\Bessel[s] \hookrightarrow \Lebesgue[2]\Holder[1]$ 
    and 
    the compactness of $\Bessel[1]\Bessel[s] \hookrightarrow \Holder[0]\Holder[1]$ yield
    \[
        u_n 
        \wconverges[n \to \infty][\Lebesgue[2]\Holder[1]]
        u_\infty
        \qand 
        U_n 
        \converges[n \to \infty][\Holder[0]\Holder[1]]
        U_\infty
        .
    \]
    Since $\pullx \, \cD_{\eta} \colon \Holder[1][][\Domain][\AmbSpace] \to \Lebesgue[\infty][\mu] \ceq \Lebesgue[\infty][][\Domain^2;\mu][\AmbSpace]$ is continuous and commutes with $\partial_t$, this implies 
    that 
    \[
        \partial_t \varphi_n 
        = 
        \pullx \, \cD_{\eta} u_n
        \wconverges[n \to \infty][\Lebesgue[2]\Lebesgue[\infty][\mu]]
        \pullx \, \cD_{\eta} u_\infty
        = 
        \partial_t \varphi_\infty
        \qand 
        \varphi_n 
        = 
        \pullx \, \cD_{\eta} U_n
        \converges[n \to \infty][\Lebesgue[\infty]\Lebesgue[\infty][\mu]]
        \pullx \, \cD_{\eta} U_\infty 
        =
        \varphi_\infty
        .
    \]
    Thus, we may apply \cref{lem:WeakConvergenceHelper}, showing the first claim.

    To show the second claim, we use a similar approach. 
    First we use \cref{eq:RDecomposition} to expand
    \[
        \pars{\Op{\Path[n]}{s} v_n} \otimes w_n
        =
        \varTheta_{n} \pars{\Op{\eta}{s} \, v_n} \otimes w_n
        -
        \varTheta_{n} \pars{\pullx \, \cD_{\Path[n]} v_n} \pars{\Op{\eta}{s} \, \Path[n]} \otimes w_n
        .
    \]
    By \cref{lem:StronglyConvergentTimesWeaklyConvergent} and \cref{eq:WeakConvergenceRs:4a}, it suffices to show that 
    \begin{equation}
        \label{eq:WeakConvergenceRs:6a}
        \chi_n \otimes w_n
        \wconverges[n \to \infty][\Lebesgue[2]\Lebesgue[2][\mu]]
        \chi_{\infty} \otimes w_\infty
        \quad \text{for} \quad
         \chi_n
        \in  
        \braces[\big]{ 
            \Op{\eta}{s} \, v_n,  
            \Op{\eta}{s} \, \Path[n]
        }
        .
    \end{equation}    
    In each of the cases we may write 
    $
        w_n  
        =
        \pars{A \, u_n} \, \zeta_n
    $
    with a suitable bounded linear operator $A \colon \Holder[1][][\Domain][\AmbSpace] \to \Lebesgue[\infty][\mu]$
    and a function $\zeta_n$ 
    so that $\zeta_n \converges[n \to \infty] \zeta_\infty$ in $\Lebesgue[\infty]\Lebesgue[\infty][\mu]$.
    More precisely, we use
    \begin{alignat*}{5}
        A &= \pi_i^\pull \, \cD_{\eta}
        &&\qand 
        &\zeta_n &= \pi_i^\pull \, \cD_{\Path[n]} \eta
        \vphantom{\sdfrac{\abs{\eta'}}{\abs{\partial_x\Path[n]}}}
        &&\quad\text{for}\quad
        &w_n &= \pi_i^\pull \, \cD_{\Path[n]} u_n, \; i \in \braces{1,2};
        \\
        A &= \pi_i^\pull \, D_{\eta}
        &&\qand 
        &\zeta_n &= \pi_i^\pull \, \sdfrac{\abs{\partial_x \eta}}{\abs{\partial_x\Path[n]}}
        &&\quad\text{for}\quad
        &w_n &= \pi_i^\pull \, D_{\Path[n]} u_n, \; i \in \braces{1,2}; \quad \text{and}
        \\
        A &= \diffop{}{1} = \sdfrac{\triangle}{\dist[\Domain]}
        &&\qand 
        &\zeta_n &= \sdfrac{\dist[\Domain]}{\abs{\triangle \Path[n]}}
        &&\quad\text{for}\quad
        &w_n &= \sdfrac{\triangle u_n}{\abs{\triangle \Path[n]}}.
    \end{alignat*}
    (See \cref{eq:DsigmaAndMu} for the definition of $\diffop{}{1}$.)
    Since for each choice of $w_n$ the operator $A$ is time independent, the operators $\partial_t$ and $A$ commute.
    With $\kappa_n \ceq A \, U_n$, we obtain
    \[
        w_n  
        =
        \pars{A \, u_n} \, \zeta_n
        =
        \pars{A \, \partial_t U_n} \, \zeta_n
        =
        \pars{\partial_t \!\pars{A \, U_n}} \, \zeta_n
        =
        \pars{\partial_t \kappa_n} \, \zeta_n
        ,
    \]
    hence 
    $ \chi_n \otimes w_n =  \chi_n \otimes \pars[\big]{ \pars{\partial_t \kappa_n} \, \zeta_n }
    $.
    Since $A \colon \Holder[1][][\Domain][\AmbSpace] \to \Lebesgue[\infty][\mu]$ is bounded, we obtain
    \[
        \partial_t \kappa_n 
        = 
        A \, u_n
        \wconverges[n \to \infty][\Lebesgue[2]\Lebesgue[\infty][\mu]]
        A \, u_\infty
        = 
        \partial_t \kappa_\infty
        \qand 
        \kappa_n 
        = 
        A \, U_n
        \converges[n \to \infty][\Lebesgue[\infty]\Lebesgue[\infty][\mu]]
        A \, U_\infty 
        =
        \kappa_\infty
        .
    \]
    Moreover, we know from \cref{eq:WeakConvergenceRs:3a} that
    $ \chi_n\wconverges[]  \chi_\infty$ in $\Bessel[1]\Lebesgue[2][\mu]$.
    Thus, we may apply \cref{lem:WeakConvergenceHelper} once again, showing \cref{eq:WeakConvergenceRs:6a} and, thus, the second claim.
\end{proof}

The following abstracts the trick of ``borrowing one time derivative'' that we use in \cref{lem:WeakConvergenceRs}. 

\begin{lemma}\label{lem:WeakConvergenceHelper}
Let $X$, $Y$, $Z$, $W$ be finite-dimensional vector spaces
and let $\alpha \colon X \times Y \times Z \to W$ be a trilinear map.
Let 
$\varphi_n \in \Bessel[1]\Lebesgue[\infty][\mu] 
= \Bessel[1][][I][\Lebesgue[\infty][][\Domain^2;\mu][X]]$,
$\zeta_n \in \Lebesgue[\infty]\Lebesgue[\infty][\mu] 
= \Lebesgue[\infty][][I][\Lebesgue[\infty][][\Domain^2;\mu][Y]]$, 
and 
$\psi_n \in \Bessel[1]\Lebesgue[2][\mu] 
= \Bessel[1][][I][\Lebesgue[2][][\Domain^2;\mu][Z]]$,
$n \in \Ninfty$
satisfy
\[
    \partial_t \varphi_n 
    \wconverges[n \to \infty][\Lebesgue[2]\Lebesgue[\infty][\mu]] 
    \partial_t \varphi_\infty
    ,
    \quad
    \varphi_n 
    \converges[n \to \infty][\Lebesgue[\infty]\Lebesgue[\infty][\mu]] 
    \varphi_\infty
    ,
    \quad
    \zeta_n 
    \converges[n \to \infty][\Lebesgue[\infty]\Lebesgue[\infty][\mu]] 
    \zeta_\infty
    ,
    \qand 
    \psi_n 
    \wconverges[n \to \infty][\Bessel[1]\Lebesgue[2][\mu]] 
    \psi_\infty
    .
\]
Then we have
\[
    \alpha(\partial_t\varphi_n,\zeta_n,\psi_n)
    \wconverges[n \to \infty][\Lebesgue[2]\Lebesgue[2][\mu]] 
    \alpha(\partial_t\varphi_\infty,\zeta_\infty,\psi_\infty)
    .
\]
\end{lemma}
\begin{proof}
    Since $X$, $Y$, $Z$, $W$ are finite-dimensional, we may fix bases and show the statement for each of the coordinate functions.
    This way, we may reduce the problem to the case where $X = Y = Z = W = \R$ and where $\alpha$ is the product $\R \times \R \times \R \to \R$.
    Since the latter is commutative, it remains to show that
    \[
        \pars{\partial_t\varphi_n} \, \psi_n \, \zeta_n
        \wconverges[n \to \infty][\Lebesgue[2]\Lebesgue[2][\mu]] 
        \pars{\partial_t\varphi_\infty} \, \psi_\infty \, \zeta_\infty
        .
    \]
    Since $\zeta_n \converges[n\to \infty] \zeta_\infty$ converges strongly in $\Lebesgue[\infty]\Lebesgue[\infty][\mu]$,
    by \cref{lem:StronglyConvergentTimesWeaklyConvergent}, it suffices to show that 
    \[
        \beta_n 
        \ceq 
        \pars{\partial_t\varphi_n} \, \psi_n
        \wconverges[n \to \infty][\Lebesgue[2]\Lebesgue[2][\mu]] 
        \pars{\partial_t\varphi_\infty} \, \psi_\infty
        \qec 
        \beta_\infty
        .
    \]
    Because they are weakly convergent, the sequences 
    $(\partial_t \varphi_n)_{n \in \N}$
    and
    $(\psi_n)_{n \in \N}$
    are bounded in 
    $\Lebesgue[2]\Lebesgue[\infty][\mu]$
    and
    $\Bessel[1]\Lebesgue[2][\mu]$, 
    respectively.
    The latter embeds continuously into $\Lebesgue[\infty]\Lebesgue[2][\mu]$.
    Hence, $(\psi_n)_{n \in \N}$ is also bounded in $\Lebesgue[\infty]\Lebesgue[2][\mu]$.
    Now the Hölder inequality implies that
    \[
        \norm{\beta_n}_{\Lebesgue[2]\Lebesgue[2][\mu]}
        =
        \norm{\pars{\partial_t \varphi_n} \,\psi_n}_{\Lebesgue[2]\Lebesgue[2][\mu]}
        \leq
        \norm{\partial_t \varphi_n}_{\Lebesgue[2]\Lebesgue[\infty][\mu]}
        \norm{\psi_n}_{\Lebesgue[\infty]\Lebesgue[2][\mu]}
        .
    \]
    This shows not only that $\beta_k$ is well-defined, 
    but also that the sequence $(\beta_k)_{k \in \N}$ is bounded in $\smash{\Lebesgue[2]\Lebesgue[2][\mu]}$.
    Hence, we have 
    \[
        C 
        \ceq 
        \sup_{k \in \N} \norm{\beta_k - \beta_\infty}_{\Lebesgue[2]\Lebesgue[2][\mu]} < \infty
    \]
    To check weak convergence $\beta_k \wconverges[n \to \infty] \beta_\infty$ in $\Lebesgue[2]\Lebesgue[2][\mu]$,
    it suffices to test against elements from the dense subset $\Holder[\infty][0]\Lebesgue[2][\mu]$, the set of smooth functions $I \to \Lebesgue[2][\mu]$ with compact support in $\intervaloo{0,1}$ (see \cref{lem:TestingOnDenseSubset}).
    So it remains to show that $\inner{\chi,\beta_n - \beta_\infty}_{\smash{\Lebesgue[2]\Lebesgue[2][\mu]}} \longrightarrow 0$ for each $\chi \in \Holder[\infty][0]\Lebesgue[2][\mu]$.
    By integration by parts in $t$, we obtain
    \begin{align}
        \label{eq:WeakConvergenceHelper:1}
        \begin{split}
        \inner{ 
            \chi,
            \beta_n
        }_{\smash{\Lebesgue[2]\Lebesgue[2][\mu]}}
        =
        \textstyle
        \int_0^1 
            \inner{\chi, \pars{\partial_t \varphi_n} \, \psi_n}_{\Lebesgue[2][\mu]}
        \dd t
        &=
        -
        \textstyle
        \int_0^1 
            \pars[\big]{
                \inner{\chi, \varphi_n \pars{\partial_t \psi_n}}_{\Lebesgue[2][\mu]}
                +
                \inner{\partial_t \chi, \varphi_n \, \psi_n}_{\Lebesgue[2][\mu]}
            }
        \dd t
        \\
        &=
        - \inner{
            \chi, 
            \varphi_n \pars{\partial_t \psi_n} 
        }_{\Lebesgue[2]\Lebesgue[2][\mu]}
        -
        \inner{
            \partial_t \chi, 
            \varphi_n \, \psi_n
        }_{\Lebesgue[2]\Lebesgue[2][\mu]}
        .
        \end{split}
    \end{align}  
    Now \cref{lem:StronglyConvergentTimesWeaklyConvergent} implies 
    \[
        \varphi_n \pars{\partial_t \psi_n} 
        \wconverges[n \to \infty][\Lebesgue[2]\Lebesgue[2][\mu]] 
        \varphi_\infty \pars{\partial_t \psi_\infty} 
        \qand
        \varphi_n  \, \psi_n 
        \wconverges[n \to \infty][\Lebesgue[2]\Lebesgue[2][\mu]] 
        \varphi_\infty \, \psi_\infty 
        .
    \]
    Hence, \cref{eq:WeakConvergenceHelper:1} converges to
    \[
        - \inner{
            \chi,
            \varphi_\infty \pars{\partial_t \psi_\infty}
        }_{\Lebesgue[2]\Lebesgue[2][\mu]}
        -
        \inner{
            \partial_t \chi,
            \varphi_\infty \, \psi_\infty
        }_{\Lebesgue[2]\Lebesgue[2][\mu]}
        =
        \inner{ 
            \chi,
            \beta_\infty 
        }_{\Lebesgue[2]\Lebesgue[2][\mu]}        
        .
    \]    
    This shows that $\beta_n \wconverges[n \to \infty] \beta_\infty$ in $\Lebesgue[2]\Lebesgue[2][\mu]$, which completes the proof.
\end{proof}

We conclude this section with two auxiliary lemmata.
Their statements are well-known, of course. We include them here only for the sake of self-containedness.

\begin{lemma}\label{lem:StronglyConvergentTimesWeaklyConvergent}
    Let $\cX$, $\cY$, $\cZ$ be Banach spaces and let $B \colon \cX \times \cY \to \cZ$ be a bounded, bilinear map.
    Let $\varphi_n \wconverges[n \to \infty] \varphi_\infty$ in $\cX$ and $\psi_n \converges[n \to \infty] \psi_\infty$ in $\cY$.
    Then $B(\varphi_n,\psi_n) 
    \wconverges[n \to \infty] 
    B(\varphi_\infty,\psi_\infty)$ in $\cZ$.
\end{lemma}
\begin{proof}
    The uniform boundedness principle implies that weakly convergent sequences are bounded.
    Hence, $C \ceq \sup_{n \in \N} \norm{\varphi_n}$ is finite.
    Thus, for every continuous and linear functional $\zeta \in \cZ\dual$ we have
    \begin{align*}
        \abs{ \inner{\zeta,B(\varphi_n,\psi_n)} - \inner{\zeta,B(\varphi_\infty,\psi_\infty)} }
        &\leq 
        \abs{ \inner{\zeta,B(\varphi_n,\psi_n - \psi_\infty)} }
        +
        \abs{\inner{\zeta,B(\varphi_n - \varphi_\infty,\psi_\infty)} }
        \\
        &\leq 
        C \norm{\zeta} \norm{B} \norm{\psi_n-\psi_\infty}
        +
        \abs{\inner{\xi,\varphi_n-\varphi_\infty} }
        \converges[n \to \infty] 0,
    \end{align*}
    where $\xi \in \cX'$ is the linear functional defined by $\inner{\xi,\varphi} \ceq \inner{\zeta,B(\varphi,\psi_\infty)}$.
\end{proof}

In the following statement it is easy to overlook that the boundedness of the sequence is a necessary requirement. (We thank the anonymous referee for pointing this out.)
\begin{lemma}\label{lem:TestingOnDenseSubset}
    Let $\cX$ be a Banach space and let $\cU \subset \cX'$ be a dense subset.  
    Let $\smash{(\beta_n)_{n \in \N}}$ be a bounded sequence in $\cX$ and let $\beta_\infty \in \cX$ such that 
    $\inner{\chi,\beta_n}
    \longrightarrow
    \inner{\chi,\beta_\infty}$
    for all $\chi \in \cU$.
    Then $(\beta_n)_{n\in\N}$ converges weakly to $\beta_\infty$.
\end{lemma}
\begin{proof}
    Let $C \ceq \sup_{n \in \N} \norm{\beta_n}_{\cX} < \infty$.
    For a fixed $\chi \in \cX'$ we choose $\chi_\ell \in \cU$ such that $\norm{\chi -\chi_\ell}_{\cX'} \leq 1/\ell$ for each $\ell \in \N$.
    Then we have
    \begin{align*}
        \abs{\inner{\chi,\beta_n - \beta_\infty}}
        &\leq 
        \abs{\inner{\chi_\ell,\beta_n - \beta_\infty}}
        +
        \abs{\inner{\chi - \chi_\ell,\beta_n - \beta_\infty}}
        \\
        &\leq
        \abs{\inner{\chi_\ell,\beta_n - \beta_\infty}}
        +
        \norm{\chi - \chi_\ell}_{\cX'} 
        \norm{\beta_n - \beta_\infty}_{\cX}
        \leq
        \abs{\inner{\chi_\ell,\beta_n - \beta_\infty}}
        +
        2C / \ell
        .
    \end{align*}
    This shows 
    $
        \limsup_{n \to \infty} \abs{\inner{\beta_n - \beta_\infty,\chi}}
        \leq
        C / \ell
    $
    for all $\ell \in \N$.
    Hence, we obtain
    \[
        \smash{\limsup_{n \to \infty} \abs{\inner{\beta_n - \beta_\infty,\chi}}}
        =
        0,
    \]
    which implies 
    $\inner{\beta_n - \beta_\infty,\chi} \longrightarrow 0$ as $n \to \infty$.
    We can do this for every $\chi \in \cX'$, so this yields weak convergence $\beta_n \wconverges[][] \beta_\infty$ in $\cX$.
\end{proof}

\section{Shape space}\label{sec:ShapeSpace}

We would like to close this work with a few observations on the quotient space and the quotient metric with respect to the action of the reparametrization group.
We point out that we work with \SoboSlobo spaces $\Bessel[s]$ for $3/2 < s < 2$.
So we define the \emph{diffeomorphism group} by
\begin{equation}
    \DiffGroup \ceq \myset[\big]{ \varphi \in \Bessel[s][][\Circle][\Circle] }{
        \text{$\varphi$ is an orientation-preserving diffeomorphism}
    }
    .
    \label{eq:DiffeomorphismGroup}
\end{equation}
By \cite[Theorem B.2.]{zbMATH06304950}, the group $\DiffGroup$ is a topological group, and the group action 
\[
    \Mfld \times \DiffGroup \to \Mfld, 
    \quad
    (\gamma,\varphi) \mapsto \varphi^\pull \gamma = \gamma\circ \varphi
\]
is well-defined and continuous.
Note that $\gamma \mapsto \varphi^\pull \gamma$ is a restriction of a bounded linear operator to $\Mfld$, so this map is differentiable, and we have 
$D(\eta \mapsto \varphi^\pull \eta)(\gamma) \, u = \varphi^\pull u$
for $u \in T_\gamma \Mfld = \HSpace$.
\cref{thm:GIsInvariant} shows that each $\varphi \in \DiffGroup$ acts on $\Mfld$ through a Riemannian isometry $\varphi^\pull \colon (\Mfld,G) \to (\Mfld,G)$. And it is straightforward to show that also $\varphi^\pull \colon (\Mfld,\dist[G]) \to (\Mfld,\dist[G])$ is an isometry, too.
We define the \emph{shape space} $\QuotientSpace$ of oriented unparametrized embeddings of class~$\Bessel[s]$ as the quotient space of this group action:
\[
    \QuotientSpace \ceq \Mfld \slash \DiffGroup
    .
\]
We denote the quotient map by $\QuotientMap \colon \Mfld \to \QuotientSpace$
and
define the \emph{quotient metric} $\dist[\QuotientSpace]$ on $\QuotientSpace$ by 
\begin{equation}
    \dist[\QuotientSpace]\!\pars[\big]{
        \QuotientMap(\gamma),
        \QuotientMap(\eta)
    }
    \ceq 
    \inf_{\varphi, \psi \in \DiffGroup} \, 
    \dist[G]({\varphi^\pull} \gamma , {\psi^\pull} \eta)
    \quad 
    \text{for $\gamma,\eta \in \Mfld$}
    .
    \label{eq:QuotientMetric}
\end{equation}
By construction, the map $\QuotientMap$ is Lipschitz continuous with Lipschitz constant $1$.
Since $\DiffGroup$ acts through isometries on $\Mfld$, we have
\begin{equation}
    \dist[\QuotientSpace]\!\pars[\big]{
        \QuotientMap(\gamma),
        \QuotientMap(\eta)
    }  
    =
    \inf_{\varphi,\psi \in \DiffGroup} \,  
    \dist[G]\!\pars[\big]{
        (\psi^{-1})^\pull \varphi^\pull\gamma,  
        (\psi^{-1})^\pull \psi^\pull\eta
    }
    =
    \inf_{\varphi \in \DiffGroup} \,  \dist[G](\varphi^\pull \gamma, \eta)
    .
    \label{eq:QuotientMetric2}
\end{equation}

In general, quotient metrics need not be definite.
As a simple counterexample, we let $a \in \R \setminus \Q$ be an irrational number and consider the action of the group $(a \, \Z,+)$
on $\Circle$ defined by $\Circle \times (a \, \Z) \to \Circle$, $(x, a \, n) \mapsto x + a \, n$.
Note that this group action is isometric with respect to the standard distance function $\dist[\Circle]$.
The orbits of the group action are dense in $\Domain$. Hence, the quotient metric is the zero map. But the quotient space consists of more than one point, so the quotient metric is not definite.

We recall that the metric $G$ is strong and smooth, which implies that $\dist[G]$ generates the $\Bessel[s]$-topology on $\Mfld$. 
Moreover, the latter is stronger than the $\Holder[1]$-topology. This suffices to prove the following:

\begin{lemma}\label{lem:QuotientIsMetricSpace}
    The quotient metric $\dist[\QuotientSpace]$ is definite. 
    Hence, $(\QuotientSpace,\dist[\QuotientSpace])$ is an extended metric space.
\end{lemma}
\begin{proof}
    Let $\gamma$, $\eta \in \Mfld$ be two curves satisfying $\dist[\QuotientSpace](\QuotientMap(\gamma),\QuotientMap(\eta))  = 0$.
    We have to show that there is a $\varphi \in \DiffGroup$ such that $\eta = \varphi^\pull \gamma$.
    Because of \cref{eq:QuotientMetric2}, there are $\varphi_n \in \DiffGroup$ such that $\dist[G](\varphi_n^\pull \gamma,\eta) \longrightarrow 0$ as $n \longrightarrow \infty$.
    This means that $\gamma_n \ceq \varphi_n^\pull \gamma$ converges to $\eta$ in $\dist[G]$. \cref{cor:GeodesicDistanceGeneratesTopology} implies that $\gamma_n$ converges to $\eta$ in $\HSpace$.
    Because the Morrey embedding $\Bessel[s][][\Circle][\AmbSpace] \hookrightarrow \Holder[1][][\Circle][\AmbSpace]$ is continuous, we get
    $\gamma_n \longrightarrow \eta$ in~$\smash{\Holder[1]}$.
    Since $\gamma_n$ is a mere reparametrization of $\gamma$, we have 
    $\gamma_n(\Circle) = \gamma(\Circle)$ for every $n \in \N$.
    The $\Holder[1]$-convergence of $\gamma_n$ to $\eta$ implies 
    convergence of $\gamma_n(\Circle)$ to $\eta(\Circle)$
    in Hausdorff distance.
    Hence, we get $\gamma(\Circle) = \eta(\Circle)$.
    Since $\gamma$ and $\eta$ are embeddings of class~$\Holder[1]$ with the same image, the map $\varphi \ceq \gamma^{-1} \circ \eta \colon \Circle \to \Circle$ is well-defined and a diffeomorphism of class~$\Holder[1]$.
    We have $\eta = \varphi^\pull \gamma$ by construction.
    Each of the diffeomorphisms $\varphi_n$ is orientation-preserving.
    Moreover, we have
    $
        \varphi_n 
        = 
        \gamma^{-1} \circ \gamma_n 
        \longrightarrow
        \gamma^{-1} \circ \eta
        =
        \varphi
    $ 
    in $\Holder[1]$.
    Thus, $\varphi$ must be orientation-preserving, too.
    Finally, we show that $\varphi \in \Bessel[s][][\Circle][\Circle]$.
    For this it suffices to verify $\varphi' \in \Bessel[s-1][][\Circle][\R]$.
    By the chain rule, we have $\eta' = (\gamma \circ \varphi)' = (\gamma' \circ \varphi) \cdot  \varphi'$.
    We define $u \ceq \gamma' / \abs{\gamma'}^2$ and form
    the pointwise inner product with $u \circ \varphi$ to get $\varphi' = \inner{ u \circ \varphi, \eta'}$.
    Since $s-1 > 1/2$, the space $\Bessel[s-1][][\Circle][\R]$ is a Banach algebra. 
    With $H_\gamma \ceq 1/ \abs{\gamma'}$ and \cref{eq:InverseSpeedSeminormBound}, we see $u = \gamma' \, H_\gamma  \, H_\gamma \in \Bessel[s-1][][\Circle][\AmbSpace]$.
    Being a $\Holder[1]$-diffeomorphism, the map $\varphi \colon \Circle \to \Circle$ is also a bi-Lipschitz map.
    Because of $s-1 < 1$, 
    both factors $u \circ \varphi$ and $\eta'$ are members of $\Bessel[s-1][][\Circle][\AmbSpace]$.
    Hence, $\varphi' = \inner{ u \circ \varphi, \eta'} \in \Bessel[s-1][][\Circle][\R]$.
    This shows $\varphi \in \DiffGroup$ and completes the proof.
\end{proof}

Recall that $\DiffGroup$ acts isometrically on $\Mfld$. This allows one to lift (a certain subsequence of) a Cauchy sequence in $(\QuotientSpace,\dist[\QuotientSpace])$ to a Cauchy sequence in $(\Mfld,\dist[G])$ (see \cite[Lemma~6.5]{zbMATH06502944}).
Together with \cref{thm:MetricCompleteness}, this immediately leads to the following observation.
\begin{lemma}\label{lem:QuotientIsComplete}
The shape space $(\QuotientSpace,\dist[\QuotientSpace])$ is complete.
\end{lemma}

Utilizing the variational techniques from \cref{sec:ExistenceMinimizingGeodesics}, we can also show that $\dist[\QuotientSpace](\QuotientMap(\gamma),\QuotientMap(\eta))$ is realized by a particular reparametrization $\varphi$ and by a particular path $\Path$.
This corresponds to \cite[Theorem~6.6 and Theorem~6.8]{zbMATH06502944} and, of course, we use very similar techniques.

\begin{theorem}
    Let $\gamma_0 \in \Mfld$ and $\gamma_1 \in \Mfld$ lie in the same path component. 
    Then there is a $\varphi \in \DiffGroup$ and a path $\Path \in \Bessel[1][][I][\Mfld]$ from $\Path(0) = \gamma_0$ to $\Path(1) = \varphi^\pull \gamma_1$ so that 
    \begin{equation}
        \dist[\QuotientSpace]\!\pars[\big]{
            \QuotientMap(\gamma_0),
            \QuotientMap(\gamma_1)
        }
        = 
        \dist[G](\gamma_0,\varphi^\pull \gamma_1)
        =
        \PathLength(\Path)
        \label{eq:MinimizingReparametrization}
    \end{equation}
    and so that $H \ceq \QuotientMap \circ \Path$ is a minimizing geodesic in the sense
    \begin{equation}
        \dist[\QuotientSpace]\!\pars[\big]{H(t_1), H(t_2)}
        =
        \dist[\QuotientSpace]\!\pars[\big]{H(0), H(1)}
        \abs{t_1-t_2}
        \quad 
        \text{for all $t_1,t_2 \in I$.}
        \label{eq:LengthSpace}
    \end{equation}
    In particular, $(\QuotientSpace,\dist[\QuotientSpace])$ is a length space.
\end{theorem}
\begin{proof}
    By \cref{eq:QuotientMetric2}, it suffices to minimize $\dist[G](\gamma_0,\varphi^\pull\gamma_1)$ over all $\varphi \in \cG$.
    That means, it suffices to minimize $\PathLength(\Path)$ over $\Path \in \cP(\cG)$, where, for a subset $\cH \subset \cG$, we define 
    \[
        \cP(\cH) 
        \ceq 
        \myset[\big]{ \Path \in \Bessel[1][][I][\Mfld]}{
            \text{$\Path(0) = \gamma_0$ and there is $\varphi \in \cH$ such that 
            $\Path(1) = \varphi^\pull \gamma_1$}
        }.
    \]
    As in \cref{thm:ExistenceMinimizingGeodesics}, we apply the direct method to the path energy $\PathEnergy$;
    we will explain below, why this leads also to a minimizer of $\PathLength$.
    Let $\Path[n] \in \cP(\cG)$, $n \in \N$
    such that 
    \[
        \PathEnergy(\Path[n]) \leq \cE_0+ 1/n
        \quad \text{with} \quad
        \cE_0 \ceq 
        \inf\myset{ \PathEnergy(\Path) }{ \Path \in \cP(\cG)}
        .
    \]
    By the same reasoning as in \cref{thm:ExistenceMinimizingGeodesics}, we may pass over to a weakly convergent subsequence and assume that
    $(\Path[n])_{n \in \N}$ converges weakly in $\Bessel[1]\Bessel[s]$ to $\Path[\infty] \in \Bessel[1][][I][\Mfld]$. 
    Due to compactness $\Bessel[1]\Bessel[s] \hookrightarrow \Holder[0]\Holder[1]$, we have
    $\Path[n] \longrightarrow \Path[\infty]$ in $\Holder[0]\Holder[1]$ as $n \longrightarrow \infty$.
    This implies $\Path[\infty](0) = \gamma_0$ and that the image $\Path[n](1)(\Circle)$ always coincides with $\gamma_1(\Circle)$. 
    Hence, we have $\Path[\infty](1)(\Circle) = \gamma_1(\Circle)$.
    Using the same argument as in~\cref{lem:QuotientIsMetricSpace}, we conclude that $\varphi_\infty \ceq \gamma_1^{-1} \circ \Path[\infty](1)$ is a well-defined diffeomorphism of class $\Bessel[s]$, showing 
    $\Path[\infty](1) = \gamma_1 \circ \varphi_\infty = \varphi_\infty^\pull \gamma_1$. 
    Moreover, we have $\smash{\varphi_{n} \ceq \gamma_1^{-1} \circ \Path[n](1) \longrightarrow \gamma_1^{-1} \circ \Path[\infty](1) = \varphi_\infty}$ in $\Holder[1]$
    and $\varphi_{n}$ is orientation-preserving. 
    Thus, also the limit $\varphi_\infty$ must be orientation-preserving.
    This shows $\varphi \in \cG$ and $\Path[\infty] \in \cP(\cG)$.
    By lower semicontinuity (\cref{thm:DirichletEnergyIsLowerSemicontinuous}), $\Path[\infty]$ is also a minimizer of $\PathEnergy$ on $\cP(\cG)$.
    Then it is also a minimizer of $\PathEnergy$ on the smaller set $\cP(\braces{\varphi_\infty})$, which means that $\Path[\infty]$ is a minimizing geodesic from $\gamma_0$ to $\varphi_\infty^\pull \gamma_1$, that 
    $
        \dist[G](\gamma_0, \varphi_\infty^\pull \gamma_1) 
        = 
        \PathLength(\Path[\infty])
    $,
    and that $\Path[\infty]$ is parametrized by constant speed.
    Let $\Path \in \cP(\cG) \cap \Holder[1][][I][\Mfld]$ be an arbitrary path.
    After suitable reparametrization, we may assume that it is parametrized by constant speed.
    By \cref{eq:DirichletConstrollsArcLength}, we have
    \[
        \smash{\PathLength(\Path[\infty]) 
        = 
        \sqrt{2\, \PathEnergy(\Path[\infty]) }
        \leq
        \sqrt{2\, \PathEnergy(\Path) }
        =
        \PathLength(\Path)}
        .
    \]
    Since $\cP(\cG) \cap \Holder[1][][I][\Mfld] \subset \cP(\cG)$ is dense (we can smoothen without changing the end points) and since $\PathEnergy$ is continuous, we obtain \cref{eq:MinimizingReparametrization}.

    To show \cref{eq:LengthSpace}, we abbreviate 
    $H \ceq \QuotientMap \circ \Path[\infty]$.
    Since the Lipschitz constant of the quotient map $\QuotientMap$ equals $1$, 
    and since $\Path[\infty]$ is parametrized by constant speed equal to $C \ceq \PathLength(\Path[\infty])$,
    we have
    \begin{align*}
        \dist[\QuotientSpace]\!\pars{H(t_1),H(t_2)}
        =
        \dist[\QuotientSpace]\!\pars[\big]{
            \QuotientMap(\Path[\infty](t_1)),
            \QuotientMap(\Path[\infty](t_2))
        }
        \leq
        \dist[G]\!\pars[\big]{
            \Path[\infty](t_1),
            \Path[\infty](t_2)
        }
        \leq
        C \pars{t_2-t_1}
    \end{align*}
    for all $0 \leq t_1 \leq t_2 \leq 1$.
    \emph{Assume} there are $0 \leq t_1 \leq t_2 \leq 1$ such that 
    $\dist[\QuotientSpace]\!\pars{H(t_1),H(t_2)}  < C \pars{t_2-t_1}$.
    Then the construction of $\Path[\infty]$ and the triangle inequality yield
    \begin{align*}
        C 
        =
        \dist[\QuotientSpace]\!\pars{H(0),H(1)}
        &\leq
        \dist[\QuotientSpace]\!\pars{H(0),H(t_1)}
        +
        \dist[\QuotientSpace]\!\pars{H(t_1),H(t_2)}
        +
        \dist[\QuotientSpace]\!\pars{H(t_2),H(1)}
        \\
        &< C \, t_1 + C \pars{t_2 - t_1} + C \pars{1-t_2}
        =
        C.
    \end{align*}
    \emph{Contradiction.}
    Hence, $\Path \ceq \Path[\infty]$ satisfies \cref{eq:LengthSpace}.
\end{proof}

\paragraph*{Acknowledgments}
We would like to thank Cy Maor for stimulating discussions and for bringing the papers~\cite{zbMATH07845668,MR4142275,MR4702745,MR3264258} to our attention.
Our thanks go also to Jason Cantarella for vivid discussions and for a never-ending supply of inspiration for numerical simulations and illustrations.
Moreover, we are indebted to Roland Herzog for his interest at an early stage of this project and for referring us to~\cite{zbMATH03362632} and~\cite{zbMATH07603624}.
Finally, we would like to express our deepest gratitude towards the anonymous referees whose numerous comments, suggestions, and corrections were invaluable to us and helped us to improve our manuscript considerably.

\paragraph*{Funding}
Elias~Döhrer gratefully acknowledges the funding support from the European Union
and the Free State of Saxony (ESF).
While working on this article Henrik Schumacher was employed at Chemnitz University of Technology, the University of Georgia, RWTH Aachen University, and Graz University of Technology. In particular, he gratefully acknowledges support by the Research Training Group ``Energy, Entropy, and Dissipative Dynamics'', funded by the Deutsche Forschungsgemeinschaft (DFG, German Research Foundation), project no.{} 320021702/GRK2326.

\paragraph*{Data availability statement}
We do not analyze or generate any datasets, because our work proceeds within a theoretical and mathematical approach.

\paragraph*{Conflict-of-interest statement}
The authors have no competing interests to declare that are relevant to the content of this article.

\appendix
\crefalias{section}{appendix} % \from now on, all newly created sections will be referred to by \cref as Appendix A, Appendix B,etc. instead of Section A, Section B, etc.

\section{Pull-tights}\label{sec:PullTight}

Our aim here is to show the statement below.
It is somewhat folklore, see, e.g., \cite[Remark~1.3]{blatt-reiter-schikorra-vorderobermeier}.
However, we have not found any detailed reference for this. Hence, we provide a brief proof.

\begin{lemma}
    Suppose that $q \geq 0$, $p \geq 0$, $ 2\,q - p > - 1$, and $-2 \leq q - p \leq 0$. Then the generalized tangent-point energy $\TP^{(p,q)}$ (see \cref{eq:DfnTPpq}) does not necessarily blow up under pull-tight.
\end{lemma}
\begin{proof}
    For $0 < t < 1$, we construct the knot $K(t) \subset \R^3$ as union $K(t) \ceq A \cup B \cup C \cup D$, where the curves $A$, $B$, $C$, $D \subset \R^3$ are defined in \cref{fig:PullTightConstruction}.
    More precisely, let $S$ be a stadium curve with radius $1$ and with straight edges of length $4$.
    We define $A$ as $S$ with one of its straight segments removed;
    $B$ as the union of two line segments of length $2 - 2\,t$ lying on that removed segment;
    and $C$ as the union of two line segments of length $t$, also lying on that removed segment.
    Moreover, we define $D \ceq t \, \varSigma$, where  $\varSigma \subset \R^3$ is a smooth knotted arc contained in the unit ball whose ends are attached to a pair of antipodal points on the unit sphere, meeting the sphere orthogonally.
    We also define $\widetilde{\varSigma}$ as the $\Holder[1]$-continuous extension of $\varSigma$ by two line segments of length $1$ so that $C \cup D = t \, \widetilde{\varSigma}$.
    If $\varSigma$ is chosen suitably, $K(t)$ lies in the same nontrivial knot class for all $0 < t < 1$. However, for $t \searrow 0$, the set $K(t)$ converges in Hausdorff distance to $S$, which is an unknot. 
    So the knottedness of $K(t)$ ``nozzles off'' for $t \searrow 0$.
    We claim that the tangent-point energy $\TP^{(p,q)}$ (for any parametrization of $K(t)$) stays finite during this process.

    \begin{figure}
    \centering%
    \includegraphics[
        width=0.5\textwidth,
        trim = 0 190 0 90, 
        clip = true,  
    ]{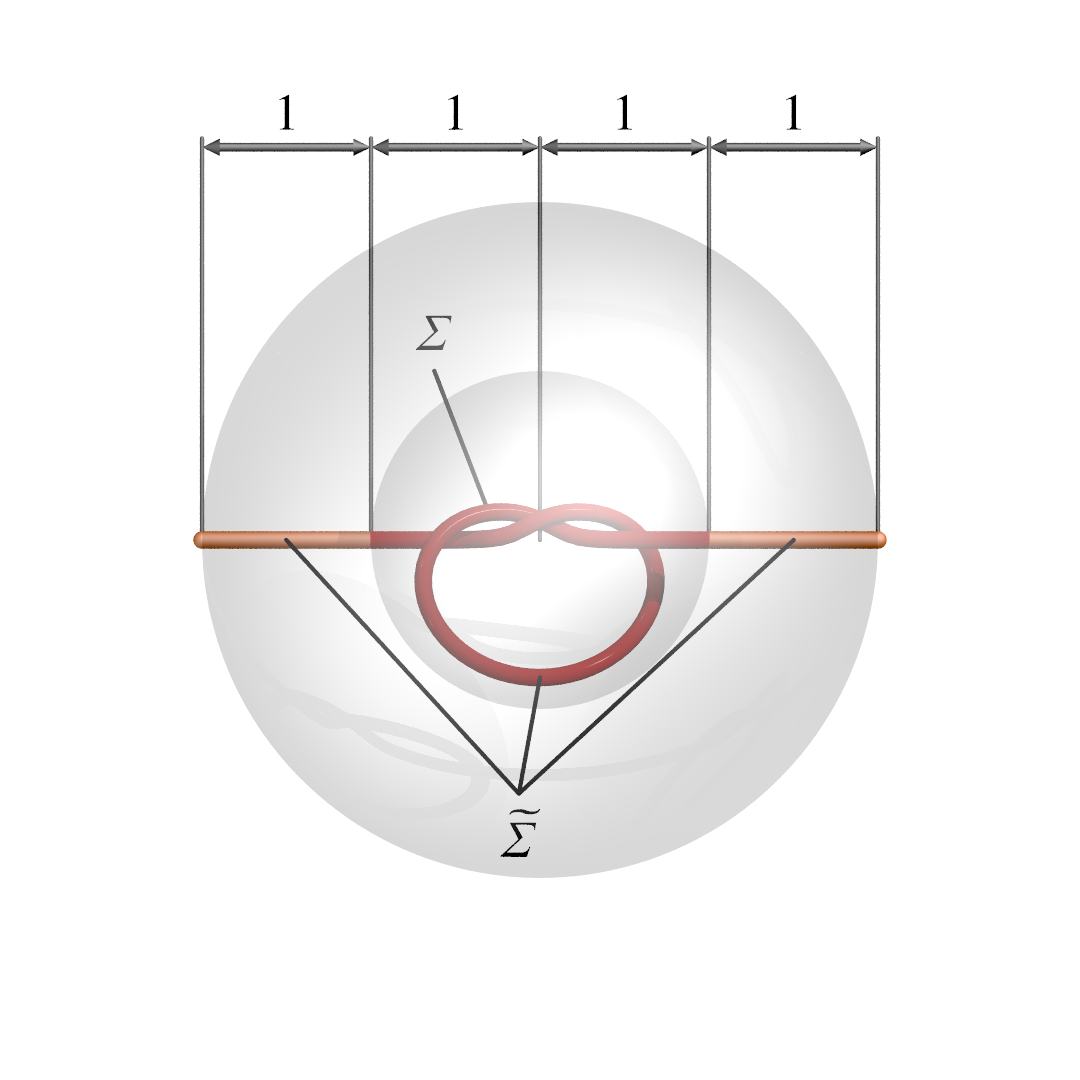}%
    \includegraphics[
        width=0.5\textwidth,
        trim = 0 180 0 100, 
        clip = true,  
    ]{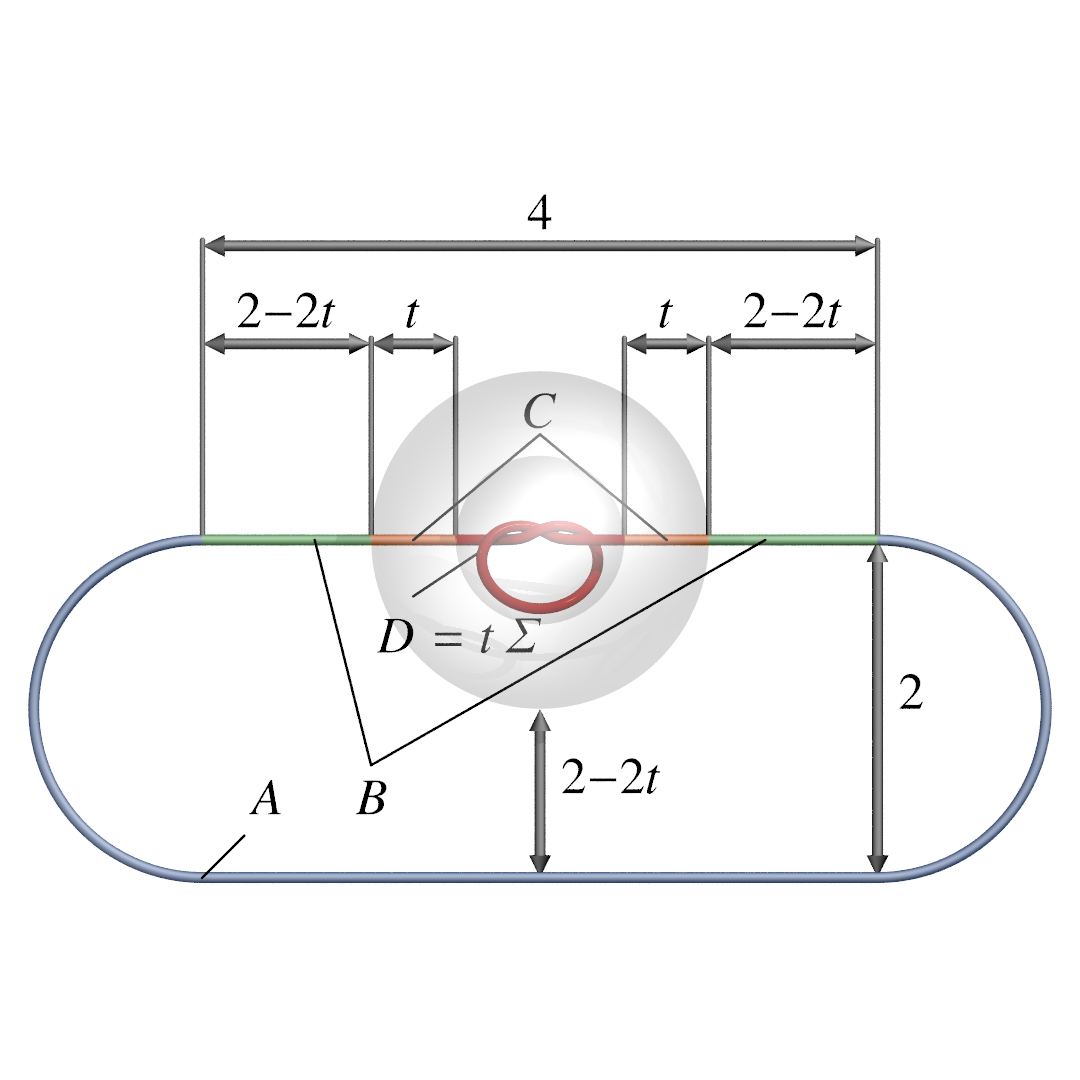}%
    \caption{\emph{Left:} The knotted arc $\varSigma$ (red) enclosed in a ball of radius $1$ and its extension $\tilde{\varSigma}$ (red and orange) by two line segments of length $1$ (orange).
    \emph{Right:} For $0 < t < 1$, we construct the knot $K(t)$  by concatenating the curves $A$ (blue), $B$ (green), $C$ (orange), and $D = t \, \varSigma$ (red).}%
    \label{fig:PullTightConstruction}
\end{figure}

    To this end, we introduce some convenient notation.
    For two embedded curves $C_1,C_2 \subset \R^3$ of class $\Holder[1]$, we pick parametrizations $\gamma_1 \colon \intervalcc{0,1} \to \R^3$ and $\gamma_2 \colon \intervalcc{0,1} \to \R^3$ and define
    \[
        F(C_1,C_2)
        \ceq
        F(\gamma_1,\gamma_2)
        \ceq
        \int_0^1 \int_0^1
            \sdfrac{
                \abs{P_{\gamma_1}(x) \pars{\gamma_2(y) - \gamma_1(x)}}^q
            }{ 
                \abs{\gamma_2(y) - \gamma_1(x)}^p
            }
        \dvol_{\gamma_1}(x)
        \dvol_{\gamma_2}(y)
        .
    \]
    Note that this does not depend on the particular choice of parametrizations.
    Moreover, we define the tangent-point energy of $C_1$ by $F(C_1) \ceq \TP^{(p,q)}(\gamma_1)$ and observe that $F(C_1) = F(C_1,C_1)$.
    Since $q \geq 0$, we have 
    $\abs{P_{\gamma_1}(x) \pars{\gamma_2(y) - \gamma_1(x)}}^q \leq \abs{\gamma_2(y) - \gamma_1(x)}^q$, leading us to
    \begin{equation}
        F(C_1,C_2)
        \leq
        \textstyle \int_{C_2} \int_{C_1}  \abs{Y - X}^{q-p} 
        \dd \mathcal{H}^1(X) \dd \mathcal{H}^1(Y) 
        ,
        \label{eq:PullTightHelper}
    \end{equation}    
    where $\mathcal{H}^1$ denotes the one-dimensional Hausdorff measure.
    With this notation we observe that
    \[
        F(K(t)) 
        =
        F(A \cup B \cup C) + F(C \cup D) + F(A,D) + F(D,A) + F(B,D) + F(D,B).
    \]
    Here we used that $F(C,C) = 0$, which holds true because $C$ is contained in a straight line.
    Since the integrand of the tangent-point energy is nonnegative and because $A \cup B \cup C \subset S$, we have 
    $
        F(A \cup B \cup C) \leq F(S).
    $
    By construction, $S$ can be parametrized by an embedding $\gamma \colon \Circle \to \R^3$ of class $\Holder[1,1] = \Sobo[2,\infty]$. 
    So we have $\abs{P_\gamma(x) \pars{\gamma(y) - \gamma(x)}} \leq \norm{D_\gamma^2 \gamma}_{\Lebesgue[\infty]} \, \dist[\gamma](x,y)^2$ and 
    \[
        F(S) 
        \leq
        \norm{D_\gamma^2 \gamma}_{\Lebesgue[\infty]}^q
        \distor(\gamma)^{p}
        \textstyle
        \int_\Circle \int_\Circle
            \dist[\gamma](x,y)^{2q-p}
        \dvol_\gamma(x) \dvol_\gamma(y)
        .
    \]
    This is finite because $2\,q - p  > -1$. Analogously, one shows $F(\tilde \varSigma) < \infty$.
    The scaling behavior of the tangent-point energy is such that
    \[
        F(C \cup D) = F( t \, \widetilde{\varSigma} )
        =
        t^{q-p+2} \, F(\widetilde{\varSigma})
        \stackrel{\smash{t \searrow 0}}{\longrightarrow}
        0
        .
    \]
    The minimum distance between points in $A$ and $D$ is $\geq 2 - 2 \, t$.
    With \cref{eq:PullTightHelper} and $q - p \leq 0$, we get
    \begin{align*}
        \Energy(A,D) + \Energy(D,A)
        &\leq
        \textstyle
        2 \, \int_{A} \int_{D} \abs{Y - X}^{q-p} \dd \mathcal{H}^1(Y) \dd \mathcal{H}^1(X)
        \leq
        2 \, \mathcal{H}^1(A) \, \mathcal{H}^1(D) \pars{2-2\,t}^{q-p}
        \\
        &\leq
        2 \, t \, \mathcal{H}^1(A) \, \mathcal{H}^1(\varSigma) \pars{2-2\,t}^{q-p}
        \stackrel{\smash{t \searrow 0}}{\longrightarrow} 
        0
        .
    \end{align*} 
    With \cref{eq:PullTightHelper} and $q - p \leq 0$ again, and since $D = t \, \varSigma$ is contained in $B_t(0)$, we obtain 
    \begin{align*}
        F(B,D) + F(D,B)
        &\leq
        \textstyle
        2 \, \int_{B} \int_{D} \abs{Y - X}^{q-p} \dd \mathcal{H}^1(Y) \dd \mathcal{H}^1(X)
        \leq
        2 \, t \, \mathcal{H}^1(\varSigma) \int_{2t}^{2} \pars{x-t}^{q-p} \dd x
        .
    \end{align*} 
    We leave it to the reader to verify that $q - p + 2 > 0$ implies that this converges to $0$ as $t \searrow 0$.
    Summarized, we obtain $\limsup_{t \searrow 0} F(K(t)) \leq F(S) < \infty$, which proves the claim.
\end{proof}

\begin{remark}
    With \cref{eq:PullTightHelper} one can easily see that $\TP^{(p,q)}$ for $q - p \geq 0$ does not blow up when self-intersections form. This is why we excluded that case from the above lemma.
\end{remark}
% !TEX root = Main.tex

\section{Grönwall inequality and Morrey inequality}

For the reader's convenience, we provide precise statements of some well-known results that are applied in this paper.

\begin{theorem}[Grönwall inequality, differential version]\label{thm:GronwallDifferential}
	Let $I\subset \R$ be an interval of the form 
	$\intervalco{a,\infty}$, $\intervalcc{a,b}$, or $\intervalco{a,b}$ with $a<b$.
	Furthermore, let $\beta \in \Holder[0][][I][\R]$.
	Suppose that $u \in \Holder[0][][I][\R]$ is differentiable on the interior $I^\circ$ of $I$ and that it satisfies
	\begin{equation*}
		\sdfrac{\dd}{\dd t} u(t) \leq \beta(t) \, u(t) \quad \text{for all } t\in I^\circ.
	\end{equation*}
	\mednegskip%
	Then 
	\begin{equation*}
		u(t) 
		\leq
		u(a)
		\exp \pars[\Big]{
			\textstyle 
			\int_a^t \beta(s) \dd s
		 } 
		 \quad 
		 \text{for all $t\in I$.}
	\end{equation*}
\end{theorem}

 \begin{theorem}[Grönwall inequality, integral version]\label{thm:GronwallIntegral}
 	Let $I\subset \R$ be an interval of the form $\intervalco{a,\infty}, \intervalcc{a,b},\intervalco{a,b}$ with $a<b$.
	Furthermore, let $\beta \in \Holder[0][][I][\R]$ be nonnegative and $\alpha \colon I \to \R$ be non-decreasing and such that $\min(0,\alpha) \in \Lebesgue[1][\loc]$.
	Suppose that $u \in \Holder[0][][I][\R]$ satisfies
	\begin{equation*}
 		u(t)
 		\leq 
 		\alpha(t)
 		+
 		\textstyle \int_a^t
 			\beta(s) \, u(s)
 		\dd s
 		\quad \text{ for all $t\in I$.}
 	\end{equation*}
	\mednegskip%
	Then
	\begin{equation*}
 		u(t) 
 		\leq
 		\alpha(t)
 		\exp
		\pars[\Big]{
		\textstyle
			\int_a^t 
				\beta(s)
			\dd s
		}
 		\quad \text{for all $t\in I$.} 
 	\end{equation*}
 \end{theorem}

A proof of the following result can be found, e.g., in~\cite[Thm.~8.2]{hitchhiker}.
A careful inspection of the cited argument reveals that $C_{\Morrey,\sigma}=1+2^{1+\sigma}$.

\begin{theorem}[Morrey inequality]\label{thm:MorreyInequality}
	For every $\sigma \in \intervaloo{1/2,1}$ there is a $C=C(\sigma) >0$ such that
	the following holds true:
	\[
		\norm{u}_{\Holder[0,\sigma-1/2]}
		\leq C
		\norm{u}_{\Bessel[\sigma](\Domain; \AmbSpace)}
		\quad 
		\text{for all $u \in \Bessel[\sigma][][\Domain][\AmbSpace]$.}
	\]
	Furthermore, if $u$ has mean value $0$, there exists $C_{\Morrey, \sigma}>0$, such that
	\begin{equation*}
		\norm{u}_{\Lebesgue[\infty]}
		\leq 
		C_{\Morrey,\sigma} \,
		\pars[\bigg]{
			\int_\Domain \int_\Domain
				\abs[\Big]{
					\sdfrac{u(y) - u(x)}{\dist[\Domain][y][x]^\sigma}
				}^2
			\sdfrac{\dd y \dd x}{\dist[\Domain][y][x]}
		}^{1/2}
		.
	\end{equation*}
	In particular, every equivalence class in $\Bessel[\sigma](\Domain, \AmbSpace)$ contains a Hölder continuous representative.
\end{theorem}
\begin{proof}
	The first statement follows directly from~\cite[Thm.~8.2]{hitchhiker}.
	The second part of the statement follows from equation 8.2 in~\cite[Thm.~8.2]{hitchhiker} with $R=1/4$.
	A detailed proof of the central inequality for this equation can be found in~\cite[Lemma 2.2]{MR1962933}.
\end{proof}
\let\oldthebibliography\thebibliography
\let\endoldthebibliography\endthebibliography
\renewenvironment{thebibliography}[1]{
  \begin{oldthebibliography}{#1}
	\footnotesize
    \setlength{\itemsep}{0em}
    \setlength{\parskip}{0em}
}
{
  \end{oldthebibliography}
}

\interlinepenalty=10000 % prevents page breaks from occurring in bibliography items
\bibliographystyle{abbrvhref}
\bibliography{references}

\end{document}